\documentclass[11pt]{article}
\usepackage[cm]{fullpage}

\usepackage{amsmath,amssymb,amsthm}

\usepackage[table]{xcolor}
\usepackage{graphicx}
\usepackage{arydshln}
\usepackage{subcaption}
\usepackage[]{mathtools}
\usepackage[numbers,compress]{natbib}
\usepackage{algorithmicx}
\usepackage{algorithm}
\usepackage{soul}
\usepackage[toc,page]{appendix}
\usepackage[]{algpseudocode}
\algnewcommand\algorithmicinput{\textbf{Input:}}
\algnewcommand\INPUT{\item[\algorithmicinput]}
\usepackage{comment}
\usepackage{multirow}

\usepackage{tikz} 
\usetikzlibrary{backgrounds,fit,shapes,snakes,arrows,shapes.geometric,positioning}
\usetikzlibrary{intersections,patterns,shapes.misc}
\usetikzlibrary{decorations.pathmorphing}

\tikzstyle{block} = [rectangle, rounded corners, minimum width=3cm, minimum height=1cm,text centered, draw=black, fill=red!30]
\tikzstyle{new} = [rectangle, rounded corners, minimum width=1cm, minimum
height=1cm,text centered, draw=black, fill=blue!10!white, dashed]
\tikzstyle{arrow} = [thick,->,>=stealth]
\usetikzlibrary{calc, quotes}

\tikzstyle{fblock} = [rectangle, draw, fill=gray!20, 
text width=8em, text centered, rounded corners, minimum height=4em, minimum width = 8em]
\tikzstyle{line} = [draw, -latex']

\usepackage{fontawesome}
\DeclareFontFamily{OT1}{pzc}{}
\DeclareFontShape{OT1}{pzc}{m}{it}{<-> s * [1.200] pzcmi7t}{}
\DeclareMathAlphabet{\mathpzc}{OT1}{pzc}{m}{it}

\AtBeginDocument{%
  \DeclareMathAlphabet\PazoBB{U}{fplmbb}{m}{n}%
}

\renewcommand{\ttdefault}{cmtt}
\usepackage{dsfont}
\usepackage{microtype}

\usepackage{diagbox}
\usepackage{multirow}
\usepackage{dashbox}

\usepackage{sidecap}
\usepackage{soul}

\newcommand{\norm}[1]{\lVert#1\rVert}

\DeclareMathOperator{\Sym}{\mathbb{S}}

\def \PSD{\mathbb{S}_{+}}

\DeclareMathOperator{\tr}{tr}
\DeclareMathOperator{\Span}{span}
\DeclareMathOperator{\Diag}{Diag}

\def \transpose{^\mathsf{T}}

\def \pinv {\dagger}
\DeclareMathOperator{\rank}{rank}
\DeclareMathOperator{\SymOp}{Sym}

\DeclareMathOperator{\kernel}{ker}

\DeclareMathOperator{\image}{image}

\DeclareMathOperator{\Stiefel}{St}
\DeclareMathOperator{\Real}{\mathbb{R}}

\DeclareMathOperator{\rgrad}{grad}
\DeclareMathOperator{\Hess}{Hess}
\DeclareMathOperator{\proj}{{Proj}}
\DeclareMathOperator{\Precon}{Precon}
\DeclareMathOperator{\Retr}{Retr}

\DeclareMathOperator{\Orthogonal}{O}
\DeclareMathOperator{\SOd}{SO}
\DeclareMathOperator{\SE}{SE}

\DeclareMathOperator{\BDiag}{BlockDiag}
\DeclareMathOperator{\SymBlockDiag}{SymBlockDiag}

\DeclareMathOperator{\SBD}{SBD}  

\DeclareMathOperator{\vectorize}{vec}

\DeclareMathOperator{\Langevin}{Langevin}
\newcommand{\ST}{S} 

\newcommand{\Rset}{\mathbb{R}}

\newcommand{\Scal}{\mathcal{S}}
\newcommand{\Ecal}{\mathcal{E}}
\newcommand{\Gcal}{\mathcal{G}}
\newcommand{\Lcal}{\mathcal{L}}

\newcommand{\Mcal}{\mathcal{M}}
\newcommand{\Acal}{\mathcal{A}}

\newcommand{\Ocal}{\mathcal{O}}

\newcommand{\Runder}{\underline{R}}

\newcommand{\tunder}{\underline{t}}

\newcommand{\Rtilde}{\widetilde{R}}
\newcommand{\ttilde}{\widetilde{t}}

\newcommand{\mhat}{\widehat{m}}
\newcommand{\fhat}{\widehat{f}}
\newcommand{\DirEdges}{{\mathcal{E}}}
\newcommand{\Vcal}{\mathcal{V}}

\newcommand{\ConLap}{L_\text{R}}
\newcommand{\ConLapT}{Q}
\newcommand{\ConLapM}{Q_\text{R}}
\newcommand{\ZR}{Z_\text{R}}
\newcommand{\ZT}{Z}
\newcommand{\MLE}{\diamond}
\newcommand{\RBCD}{\text{RBCD}}
\newcommand{\ARBCD}{\text{RBCD}\texttt{++}}
\newcommand{\BlockUpdate}{\textsc{BlockUpdate}}
\newcommand{\Manifold}{\mathcal{M}_{\text{\normalfont\tiny PGO}}}
\newcommand{\vnull}{{v}_0}

\newcommand{\edit}[2]{{\color{black} #2}}

\def \AlgName{DC2-PGO}

\usepackage{bbm}
\usepackage[sc]{mathpazo}
\usepackage{microtype}
\usepackage[T1]{fontenc}
\usepackage{setspace}
\usepackage[affil-it]{authblk} 
\usepackage{etoolbox}
\usepackage{lmodern}
\usepackage{tocloft}
\newcommand{\listtodotitle}{List of Comments}
\newlistof{todos}{tod}{\listtodotitle}
\newcommand{\todolist}[1]{%
		\refstepcounter{todos}
		\addcontentsline{tod}{todos}
{\protect\numberline{\thetodos} #1}}
\usepackage{hyperref}
\hypersetup{bookmarksopen,bookmarksnumbered,colorlinks=true, linkcolor=blue, citecolor = blue, urlcolor = cyan, filecolor=black , pagebackref=false, hypertexnames=false, plainpages=false, pdfpagelabels}

\begin{document}

\theoremstyle{plain}
\newtheorem{theorem}{Theorem}
\newtheorem{lemma}{Lemma}
\newtheorem{proposition}{Proposition}
\newtheorem{corollary}{Corollary}
\newtheorem{innercustomthm}{Theorem}
\newenvironment{customthm}[1]
{\renewcommand\theinnercustomthm{#1}\innercustomthm}
{\endinnercustomthm}

\theoremstyle{definition}
\newtheorem{definition}{Definition}
\newtheorem{assumption}{Assumption}
\newtheorem{problem}{Problem}
\newtheorem{remark}{Remark}

\newcommand\Ycomment[1]{\todolist{#1} \textcolor{green!50!black}{@Y: #1}}
\newcommand\Kcomment[1]{\todolist{#1} \textcolor{red}{@K: #1}}

\newcommand\why[1]{&& {\color{gray!80!black} ( #1)}}

\renewcommand{\ttdefault}{cmtt}

\makeatletter
\patchcmd{\@maketitle}{\LARGE \@title}{\fontsize{16}{19.2}\selectfont\@title}{}{}
\makeatother

\title{\bf
 Distributed Certifiably Correct Pose-Graph Optimization
}

\author{Yulun Tian, Kasra Khosoussi, David M. Rosen, and Jonathan P. How
\thanks{The authors are with the Laboratory for Information and Decision Systems (LIDS), Massachusetts Institute of Technology, Cambridge, MA. {\tt\small \{yulun, kasra, dmrosen, jhow\}@mit.edu}}%
}

\maketitle


\begin{abstract}
This paper presents the first \emph{certifiably correct} algorithm for \emph{distributed} pose-graph optimization (PGO), the backbone of modern collaborative simultaneous localization and mapping (CSLAM) and camera network localization (CNL) systems.  Our method is based upon a sparse semidefinite relaxation that we prove provides globally-optimal PGO solutions under moderate measurement noise (matching the guarantees enjoyed by state-of-the-art centralized methods), but is amenable to distributed optimization using the low-rank Riemannian Staircase framework.  To implement the Riemannian Staircase in the distributed setting, we develop \emph{Riemannian block coordinate descent} (RBCD), a novel method for (locally) minimizing a function over a product of Riemannian manifolds.  We also propose the first distributed solution verification and saddle escape methods to certify the global optimality of critical points recovered via RBCD, and to descend from suboptimal critical points (if necessary).  All components of our approach are inherently decentralized: they require only local communication, provide privacy protection, and are easily parallelizable.  Extensive evaluations on synthetic and  real-world datasets demonstrate that the proposed method correctly recovers globally optimal solutions under moderate noise, and outperforms alternative distributed techniques in terms of solution precision and convergence speed.

\end{abstract}

\clearpage
\tableofcontents

\section{Introduction}
Collaborative multi-robot missions require \emph{consistent} \emph{collective} spatial
perception across the entire team. In unknown GPS-denied environments, this is achieved by
\emph{collaborative} simultaneous localization and mapping (CSLAM), in which a team of agents \emph{jointly} constructs a \emph{common} model of an environment via exploration.
At the heart of CSLAM, robots must solve a \emph{pose-graph optimization} (PGO)
problem (also known as \emph{pose synchronization}) to estimate their trajectories based on noisy relative \emph{inter-robot} and \emph{intra-robot} measurements.

While several prior approaches to CSLAM have appeared in the literature, to date
no method has been proposed that is capable of \emph{guaranteeing} the recovery
of an optimal solution in the distributed setting.  One general line of prior
work
\citep{schmuck2018ccm,deutsch2016framework,morrison2016moarslam,kim2010multiple} proposes to solve CSLAM in a \emph{centralized} manner. While this approach is
straightforward (as it enables the use of off-the-shelf methods for PGO), it
imposes several practically-restrictive requirements, including: a central node
that is capable of solving the \emph{entire team's} SLAM problem, a sufficiently
reliable communication channel that connects the central node to the team, and
sufficient resources (particularly energy and bandwidth) to regularly relay the
team's (raw or processed) observations to the central node. Moreover, these
schemes \emph{by construction} are unable to protect the \emph{spatial privacy}
of individual robots, and lack robustness due to having a single point of
failure.  An alternative line of work proposes fully distributed algorithms
\cite{Choudhary17IJRR,fan2020majorization,CunninghamDDFSAM2,cunningham2010ddf}; however, at present these methods are all based upon applying distributed \emph{local} optimization methods to the \emph{nonconvex} PGO problem \cite{Choudhary17IJRR,Fan2019ISRR}, which renders them vulnerable to convergence to significantly suboptimal solutions \cite{Carlone2015Verification}.

In parallel, over the last several years recent  work has developed a novel class of \emph{certifiably correct} estimation methods that \emph{are} capable of efficiently recovering \emph{provably globally optimal} solutions of generally-intractable estimation problems within a restricted (but still practically-relevant) operational regime \cite{Bandeira2016Certifiable}.  In particular, \citet{Rosen19IJRR} developed SE-Sync, a certifiably correct algorithm for pose-graph optimization.  SE-Sync is based upon a (convex) semidefinite relaxation whose minimizer is \emph{guaranteed} to provide a \emph{globally optimal} PGO solution whenever the noise on the available measurements falls below a certain critical threshold; moreover, in the (typical) case that this occurs, it is possible to \emph{computationally verify} this fact \emph{a posteriori}, thereby \emph{certifying} the correctness (optimality) of the recovered estimate.  However, SE-Sync is not directly amenable to a \emph{decentralized} implementation because its semidefinite relaxation generically has \emph{dense} data matrices, and it solves this relaxation using a second-order optimization scheme, both of which would require an impractical degree of communication among the agents in a distributed setting.

In this paper, we advance the state of the art in CSLAM by proposing the first PGO algorithm that is both \emph{fully distributed} and \emph{certifiably correct}.   Our method leverages the same semidefinite relaxation strategy that underpins  current state-of-the-art (centralized) certifiably correct PGO algorithms \cite{Rosen19IJRR}, but employs novel decentralized optimization and solution verification techniques that enable these relaxations to be  solved efficiently in the distributed setting.  Specifically, we make the following contributions:

\begin{itemize}
 \item We prove that a sparse semidefinite relaxation of PGO employed by \citet{Briales17CartanSync} enjoys the same exactness guarantees as the one used in SE-Sync \cite{Rosen19IJRR}: namely, that its minimizers are \emph{low-rank} and provide \emph{exact} solutions of the original PGO problem under moderate measurement noise.
 
 \item We describe an efficient low-rank optimization scheme to solve this semidefinite relaxation in the distributed setting.  Specifically, we employ a \emph{distributed} Riemannian Staircase approach \cite{BoumalBlockDiagonal}, and propose \emph{Riemannian block coordinate descent} (RBCD), a novel method for minimizing a function over a product of Riemannian manifolds, to solve the resulting low-rank subproblems in the distributed setting.  We prove that RBCD converges to first-order critical points with a \emph{global} sublinear rate under standard (mild) conditions,  and that these are in particular \emph{always} satisfied for the low-rank PGO subproblems.  We also describe Nestorov-accelerated variants of RBCD that significantly improve its convergence speed in practice.
 
 \item We propose the first \emph{distributed solution verification} and \emph{saddle escape} methods to certify the optimality of low-rank critical points recovered via RBCD, and to descend from suboptimal critical points (if necessary). 
 
\item Finally, we describe simple distributed procedures for initializing the distributed Riemannian Staircase optimization, and for \emph{rounding} the resulting low-rank factor to extract a final PGO estimate.
\end{itemize}

Each of these algorithmic components has the same communication, computation, and privacy properties enjoyed by current distributed CSLAM methods \citep{Choudhary17IJRR,cunningham2010ddf,CunninghamDDFSAM2}, including:

\begin{enumerate}
	\item \emph{Communication and computational efficiency:} Robots need only 
	communicate with their neighbors in the pose graph. 
	To this end, the minimum requirement is that robots form a \emph{connected} network, so that information can flow between any pair of robots (possibly relayed by intermediate robots).
	The payload size in each round of communication is only $\Ocal(m_\text{inter})$ where
	$m_\text{inter}$ is the number of inter-robot loop closures.
	Moreover, local updates in RBCD can be performed efficiently and in
	\emph{parallel}, and the solution is produced in an
	anytime fashion.
	\item \emph{Spatial privacy protection:} Robots are not required to 
	reveal \emph{any} information about their own observations or their \emph{private} poses
	(those poses that are not \emph{directly} observed by other robots).
\end{enumerate}

Our overall algorithm, \emph{Distributed Certifiably Correct Pose-Graph Optimization} (\AlgName), thus preserves the desirable computational properties of existing state-of-the-art CSLAM methods while enabling the recovery of \emph{provably globally optimal} solutions  in the distributed setting.

The rest of this paper is organized as follows. 
In the remainder of this section we introduce necessary notations and mathematical preliminaries.
In Section~\ref{sec:related_work}, we review state-of-the-art centralized and
distributed PGO solvers, as well as recent advances in block-coordinate optimization methods.
In Section~\ref{sec:background}, we formally define the distributed PGO problem and present its sparse SDP relaxation.
We present a distributed procedure to solve this SDP via the Riemannian Staircase framework.
On the theoretical front, we establish formal \emph{exactness} guarantees for the SDP relaxation. 
In Section~\ref{sec:bcd}, we present our distributed local search method to solve the rank-restricted SDPs using block-coordinate descent. 
In Section~\ref{sec:convergence}, we prove convergence of the proposed local search method and analyze its global convergence rate.
In Section~\ref{sec:verification}, we present a distributed solution verification procedure that checks the global optimality of our local solutions, and enables us to \emph{escape} from suboptimal critical points, if necessary. 
We discuss distributed initialization and rounding in Section~\ref{sec:initialization_and_rounding}. 
Finally, we conclude with extensive experimental evaluations in Section~\ref{sec:experiments}.


\subsection*{Notations and Preliminaries}

\subsubsection*{General Notations}
Unless stated otherwise, lowercase and uppercase letters are generally used for vectors and matrices, respectively.
We define $[n] \triangleq \{1, 2, \hdots, n\}$. 

\subsubsection*{Linear Algebra}

$\Sym^d$ and $\PSD^d$ denote the set of $d \times d$ symmetric and symmetric positive semidefinite matrices, 
respectively. 
$0_{d\times d}, I_{d \times d} \in \Real^{d \times d}$ are the zero and identity matrix, 
and $0_d, 1_d \in \Real^d$ represent the vectors of all zeros and all ones, respectively.
To ease the burden of notations, we also drop the subscript $d$ when the dimension is clear from context.
For a matrix $A$, we use $A_{(i,j)}$ to index its $(i,j)$-th entry.
$A^\dagger$ denotes the Moore-Penrose inverse of $A$.
Given a ($d \times d$)-block-structured matrix $Z \in \Real^{dn \times dn}$,
$Z_{[i,j]} \in \Real^{d \times d}$ refers to its $(i,j)$-th block. 
\edit{
	Following \cite{Rosen19IJRR}, we define
	$\BDiag(M)$ as the linear operator that extracts the diagonal blocks of $M$ and zeros out all remaining blocks,
	and $\SymBlockDiag_d(M)$
	as its symmetric version;
	see \cite[Equations~(4)-(5)]{Rosen19IJRR}.}{}
$\proj_\Scal$ denotes the orthogonal projection operator onto a given set $\Scal$ with respect to the Frobenius norm.

We define several linear operators that will be useful in the rest of this paper. 
Given a list of square matrices $A_1, \hdots, A_n$ (possibly with different dimensions), $\Diag$ assembles them into a block-diagonal matrix: 
\begin{equation}
	\Diag(A_1, \hdots, A_n) \triangleq \begin{bmatrix}
	A_1 &        &    \\
	    & \ddots &    \\
	    &        & A_n
	\end{bmatrix}.
\end{equation}
For an arbitrary square matrix $A$, $\SymOp$ returns its symmetric part $\SymOp(A) \triangleq (A + A^\top)/2$. 
For a $[d \times d]$-block-structured matrix $Z \in \Real^{dn \times dn}$, 
we define the following linear operator that outputs another $[d \times d]$-block-diagonal matrix as follows,
\begin{equation}
\SymBlockDiag_d (Z)_{[i,j]} \triangleq
\begin{cases}
\SymOp(Z_{[i,i]}), & \text{if } i = j, \\
0_{d \times d}, & \text{otherwise}.
\end{cases}
\label{eq:SymBlockDiag}
\end{equation}
In addition, for a $[(d+1) \times (d+1)]$-block-structured matrix $Z \in \Real^{(d+1)n \times (d+1)n}$, we define a similar linear operator:
\begin{equation}
\SymBlockDiag_d^{+} (Z)_{[i,j]} \triangleq
\begin{cases}
\begin{bmatrix}
\SymOp(Z_{[i,i](1:d,1:d)}) & 0_d \\
0_d^\top & 0
\end{bmatrix}, & \text{if } i = j, \\
0_{(d+1) \times (d+1)}, & \text{otherwise}.
\end{cases}
\label{eq:SymBlockDiagPlus}
\end{equation}

\subsubsection*{Differential Geometry of Riemannian Manifolds}
\edit{}
{
We refer the reader to \citep{boumal2020intromanifolds,absil2009optimization}
for outstanding introductions to Riemannian optimization on matrix submanifolds. 	
In general, we use $\Mcal$ to denote a smooth matrix submanifold of a real Euclidean space.
$T_{x} \Mcal$ (or $T_x$ for brevity) denotes
the tangent space at $x \in \Mcal$. 
The tangent space is endowed with the
standard Riemannian metric induced from the ambient (Euclidean) space, 
i.e., $\langle \eta_1, \eta_2 \rangle \triangleq \tr(\eta_1^\top \eta_2)$,
and the induced norm is $\norm{\eta} \triangleq \sqrt{\langle \eta, \eta \rangle}$. 
$\Retr$ denotes a retraction operator, with $\Retr_x: T_{x} \Mcal \to \Mcal$ being its restriction to $T_x \Mcal$.
For a function $f: \Mcal \to \Real$, we use $\nabla f(x)$ and $\rgrad f(x)$ to denote the Euclidean and Riemannian gradients of $f$ at $x \in \Mcal$.
For matrix submanifolds, the Riemannian gradient is obtained as the orthogonal projection of the Euclidean gradient onto the associated tangent space: 
\begin{equation}
	\rgrad f(x) = \proj_{T_x} (\nabla f(x)).
	\label{eq:egrad_to_rgrad}
\end{equation}
We call $x^\star \in \Mcal$ a \emph{first-order critical point} if $\rgrad f(x^\star) = 0$. 
For the second-order geometry, we are primarily concerned with the Riemannian Hessian of a function $f: \Mcal \to \Real$. For matrix submanifolds, the Riemannian Hessian is a linear mapping on the tangent space which captures the directional derivative of the Riemannian gradient:
\begin{equation}
	\begin{aligned}
	\Hess f(x) : \; & T_x \Mcal \to T_x \Mcal, \\
	& \eta \mapsto \proj_{T_x} (\mathrm{D} \rgrad f(x) [\eta]).
	\end{aligned}
	\label{eq:riemannian_hessian}
\end{equation}
Above, the operator $\mathrm{D}$ denotes the standard directional derivative in the Euclidean space; see \cite[Chapter~5]{absil2009optimization}. 

\subsubsection*{Matrix submanifolds used in this work}

In SLAM, standard matrix submanifolds that frequently appear include the special orthogonal group $\SOd(d) \triangleq \{ R \in \Real^{d \times d} | R^\top R = I_d, \det(R) = 1\}$ and the special Euclidean group $\SE(d) \triangleq \{(R, t) | R \in \SOd(d), t \in \Real^d\}$.
We also make use of the Stiefel manifold $\Stiefel(d,r) \triangleq \{Y \in \Real^{r \times d} \mid Y^\top Y = I_d\}$. The geometry of these manifolds can be found in standard textbooks (e.g.\ \cite{absil2009optimization}). Given a matrix $A \in \Real^{r \times d}$, if $A = U \Sigma V^\top$ is the singular value decomposition (SVD) of $A$, the projection of $A$ onto $\Stiefel(d,r)$ can be obtained as:
\begin{equation}
    \proj_{\Stiefel(d,r)} (A) = U V^\top.
    \label{eq:projection_to_stiefel}
\end{equation}
In this work, the following product manifold is also used extensively and we give it a specific name:
\begin{equation}
    \Manifold(r,n) \triangleq (\Stiefel(d,r) \times \Real^r)^n \subset \Real^{r \times (d+1)n}.
	\label{eq:my_manifold}
\end{equation}
In our notation \eqref{eq:my_manifold}, we highlight the constants $r$ and $n$
and omit the constant $d$, as the latter is essentially fixed (i.e., $d \in
\{2,3\}$).
As a product manifold, we can readily characterize the first-order geometry of $\Manifold(r,n)$ using those of $\Stiefel(d,r)$ and $\Real^r$. 
In particular, the tangent space of $\Manifold(r,n)$ is given as the Cartesian product of the tangent spaces of individual components. 
In matrix form, we can write the tangent space as:
\begin{equation}
	T_X \Manifold(r,n) = \{\dot{X} \in \Real^{r \times (d+1)n} \mid \SymBlockDiag_d^+(\dot{X}^\top X) = 0\}. 
	\label{eq:tangent_space}
\end{equation}
The normal space is the orthogonal complement of $T_X \Manifold(r,n)$ in the ambient space. It can be shown that the normal space takes the form, 
\begin{equation}
	T_X^{\perp} \Manifold(r,n) = \{ XS \mid S = \Diag(S_1, 0, \hdots, S_n, 0), S_i \in \Sym^d \}.
	\label{eq:normal_space}
\end{equation} 
Finally, given a matrix $U \in \Real^{r \times (d+1)n}$ in the ambient space, 
the orthogonal projection onto the tangent space at $X$ \eqref{eq:tangent_space} is given by:
\begin{equation}
\begin{aligned}
	\proj_{T_X}: & \Real^{r \times (d+1)n} \to T_X \Manifold(r,n), \\
	           & U \mapsto  U - X \SymBlockDiag_d^{+} (X^\top U). 
\end{aligned}
\label{eq:projection_onto_tangent_space}
\end{equation}
}

\section{Related Work}
\label{sec:related_work}

\subsection{Centralized Pose-Graph Optimization Algorithms}
In prior work \citet{Rosen19IJRR} developed SE-Sync, a state-of-the-art certifiably correct algorithm for pose-graph optimization.  SE-Sync is based upon a (convex) semidefinite relaxation that its authors prove admits a \emph{unique, low-rank} minimizer providing an \emph{exact, globally-optimal} solution to the original PGO problem whenever the noise on the available measurements is not too large; moreover, in the (typical) case that exactness obtains, it is possible to \emph{verify} this fact \emph{a posteriori} \cite{Carlone2015Verification}, thereby \emph{certifying} the correctness (optimality) of the recovered estimate.  To solve the resulting semidefinite relaxation efficiently, SE-Sync employs the Riemannian Staircase \citep{BoumalBlockDiagonal}, which leverages symmetric low-rank (Burer-Monteiro) factorization \cite{Burer2003ANP} to directly search for a symmetric low-rank factor of the SDP solution, and implements this low-dimensional search using the truncated-Newton \emph{Riemannian trust-region} (RTR) method \citep{absil2007trust,absil2009optimization}.  This combination of \emph{low-rank factorization} and \emph{fast local search} (via truncated-Newton RTR) enables SE-Sync to recover \emph{certifiably globally optimal} PGO solutions at speeds comparable to (and frequently significantly faster than) standard state-of-the-art \emph{local} search methods (e.g.\ Gauss-Newton) \cite{Rosen19IJRR}.  

Unfortunately, the specific computational synthesis proposed in SE-Sync does not admit a straightforward distributed implementation.  In particular, the semidefinite relaxation underlying SE-Sync is obtained after exploiting the separable structure of PGO to \emph{analytically eliminate} the translational variables \cite{KasraTRO}; while this has the advantage of reducing the problem to a lower-dimensional optimization over a compact search space, it also means that the objective matrix of the resulting SDP is generically \emph{fully dense}.  Interpreted in the setting of CSLAM, this reduction has the effect of requiring all poses to be \emph{public} poses; i.e., \emph{every pose} must be available to \emph{every agent} in the team.  In addition to violating privacy, implementing this approach in a distributed setting would thus require an intractable degree of communication among the agents.

A similar centralized solver, Cartan-Sync, is proposed by \citet{Briales17CartanSync}. 
The main difference between SE-Sync and Cartan-Sync is that the latter directly relaxes the pose-graph optimization problem \emph{without} first analytically eliminating the translations; consequently, the resulting relaxation  retains the sparsity present in the original PGO problem.  However, this alternative SDP relaxation (and consequently Cartan-Sync itself) has \emph{not} previously been shown to enjoy any exactness guarantees; in particular, its minimizers, and their relation to solutions of PGO, have not previously been characterized.  As one of the main contributions of this work, we derive sharp correspondences between minimizers of Cartan-Sync's relaxation and the original relaxation employed by SE-Sync (Theorem \ref{thm:sdp_equivalence}); in particular, this correspondence enables us to \emph{extend} the exactness guarantees of the latter to cover the former, thereby justifying its use as a basis for distributed certifiably correct PGO algorithms.

As a related note, similar SDP relaxations
\citep{SingerPhaseSync2011,Bandeira2017Angular,eriksson2019rotation,Dellaert2020Shonan} have also
been proposed for \emph{rotation averaging} \citep{Hartley2013}.
This problem arises in a number of important applications
such as CNL
\citep{tron2012distributed,Tron2014,Tron2016}, structure from motion 
\citep{Hartley2013}, and other domains such as cryo-electron microscopy in structural biology
\citep{SingerCryoEM}.  Mathematically, rotation averaging can be derived as a \emph{specialization} of PGO obtained by setting the measurement precisions for the translational observations to $0$ (practically, \emph{ignoring} translational observations and states); consequently, the algorithms proposed in this work immediately apply, \emph{a fortiori}, to distributed rotation averaging.

\subsection{Decentralized Pose-Graph Optimization Algorithms}

The work by \citet{Choudhary17IJRR} is currently
the state of the art in distributed PGO solvers and has been recently
used by modern decentralized CSLAM systems
\citep{cieslewski2018data,lajoie2019door}. 
\citet{Choudhary17IJRR} propose a two-stage approach for
finding approximate solutions to PGO. 
The first stage approximately solves the underlying rotation averaging problem by 
relaxing the non-convex $\SOd(d)$ constraints, solving the
resulting (unconstrained) linear least squares problem, and projecting the results back to
$\SOd(d)$.
The rotation estimates are then used in the second stage to initialize a single
Gauss-Newton iteration on the full PGO problem.
In both stages, iterative and distributable linear solvers such as 
Jacobi over-relaxation (JOR) and successive over-relaxation (SOR) 
\citep{bertsekas1989parallel} are used to solve the normal equations. 
The experimental evaluations presented in \citep{Choudhary17IJRR}
demonstrate that this approach significantly outperforms prior techniques
\citep{cunningham2010ddf,CunninghamDDFSAM2}. Nonetheless, the proposed approach
is still performing incomplete local search on a non-convex problem, and thus
cannot offer any performance guarantees.

In another line of work,
Tron~et~al.~\cite{tron2012distributed,Tron2014,Tron2016} propose a multi-stage distributed
Riemannian consensus protocol for CNL based on distributed 
execution of Riemannian gradient descent over $\Mcal^n$ where $\Mcal = \SOd(3)
\times \Rset^3$ and $n$ is the number of cameras (agents). CNL can be seen as a
special instance of collaborative PGO where each agent owns a \emph{single}
pose rather than an entire trajectory. In these works, the authors establish convergence to
critical points and, under \emph{perfect} (noiseless) measurements, convergence to
globally optimal solutions. See Remark~\ref{rem:bcm} for a specialized
form of our algorithm suitable for solving CNL.

Recently, Fan and Murphy \cite{Fan2019ISRR,Fan2020Arxiv} proposed a majorization-minimization approach to solve distributed PGO. 
Each iteration constructs a quadratic upper bound on the cost function, and minimization of this upper bound is carried out in a distributed and parallel fashion. 
The core benefits of this approach are that it is \emph{guaranteed} to converge to a first-order critical point of the PGO problem, and that it allows one to incorporate Nesterov's acceleration technique, which provides significant empirical speedup on typical PGO problems. 
In this work, we propose a different local search method that performs block-coordinate descent over Riemannian manifolds.
Similar to \cite{Fan2019ISRR,Fan2020Arxiv}, we achieve significant empirical speedup by adapting Nesterov's accelerated coordinate descent scheme \cite{Nesterov2012ACD} while ensuring global convergence using adaptive restart \cite{Odonoghue2012}.
However, compared to \cite{Fan2020Arxiv}, our method enjoys the computational advantage that each iteration requires only an inexpensive \emph{approximate} solution of each local subproblem (as opposed to fully solving the local subproblem to a first-order critical point).
Lastly, while all distributed methods \cite{Choudhary17IJRR,tron2012distributed,Tron2014,Tron2016,Fan2019ISRR,Fan2020Arxiv} reviewed thus far are \emph{local} search techniques, our approach is the first \emph{global} and \emph{certifiably correct} distributed PGO algorithm.

\subsection{Block-Coordinate Descent Methods}
Block-coordinate descent (BCD) methods (also known as Gauss-Seidel-type
methods) are classical techniques
\citep{bertsekas1989parallel,Wright2015Survey} that have recently regained popularity in
large-scale machine learning and numerical optimization
\citep{Nutini2017BCD,marevcek2015distributed,beck2013convergence,nesterov2012efficiency}.
These methods are popular due to their simplicity, cheap iterations,
and flexibility in the parallel and distributed settings \citep{bertsekas1989parallel}.

BCD is a natural choice for solving PGO in the distributed setting due to the graphical decomposition of the underlying optimization problem.
In fact, BCD-type 
techniques have been applied in the past
\citep{duckett2000learning,frese2005multilevel} to solve SLAM.
Similarly, in computer vision, variants of the Weiszfeld algorithm have also been used for robust rotation
averaging \citep{Hartley2013,hartley2011l1}. The abovementioned works,
however, use BCD for \emph{local} search and thus cannot
guarantee global optimality.
More recently, \citet{eriksson2019rotation} propose a BCD-type algorithm for
solving the SDP relaxation of rotation averaging. Their
row-by-row (RBR) solver extends the approach of
\citet{wen2009row} from SDPs with diagonal constraints to block-diagonal
constraints. In small problems
(with up to $300$ rotations), RBR is shown to be 
comparable or better than the Riemannian Staircase
approach \citep{BoumalBlockDiagonal} in terms of runtime. This approach,
however, needs to store and manipulate a \emph{dense} $dn \times dn$  matrix,
which is not scalable to SLAM applications (where $n$ is usually one to two orders of magnitude larger than the problems considered in \citep{eriksson2019rotation}). 
Furthermore, although in principle RBR can be executed distributedly, 
it requires each agent to estimate a block-row of the SDP variable, which in the case of rotation averaging corresponds to both public and private rotations of other agents.
In contrast, the method proposed in this work does not leak information about robots' private poses throughout the optimization process. 

This work is originally inspired by recent block-coordinate minimization algorithms 
for solving SDPs with diagonal constraints via the Burer-Monteiro
approach \citep{WangBCM2017,ErdogduBCM2018}. Our recent technical report \citep{Tian2019BCM} extends these
algorithms and the global convergence rate analysis provided by
\citet{ErdogduBCM2018} from the unit sphere
(SDPs with diagonal constraints) to the Stiefel manifold (SDPs with
\emph{block}-diagonal constraints).
In this work, we further extend our initial results by
providing a unified Riemannian BCD algorithm and its global convergence rate
analysis.

\section{Certifiably Correct Pose-Graph Optimization}
\label{sec:background}
\begin{figure*}
		\centering
		\begin{tikzpicture}[node distance = 6.5cm, auto]
				\node [fblock,fill=red!10!white, thick] (se_sync) {\footnotesize
				PGO \\ (Problem~\ref{prob:se_sync})};
				\node[fblock, fill=blue!5!white, right of=se_sync, thick] (se_sdp)
				{\footnotesize SDP Relaxation\\ (Problem~\ref{prob:se_sdp})};
				\node[fblock, fill=green!10!white, right=4.5cm of se_sdp, densely dashed,
				thick] (se_sync_riemannian) {\footnotesize Rank-restricted SDP\\ (Problem~\ref{prob:se_sync_riemannian})};
				\node[fblock, fill=green!10!white, above right of=se_sync_riemannian, node
				distance = 0.2cm, densely dashed, thick] (se_sync_riemannian2)
				{\footnotesize Rank-restricted SDP\\ (Problem~\ref{prob:se_sync_riemannian})};
				\node[fblock, fill=green!10!white, above right of=se_sync_riemannian2,node
				distance = 0.2cm, thick]
				(se_sync_riemannian3) {\footnotesize Rank-$r$ restriction of SDP\\ (Problem~\ref{prob:se_sync_riemannian})};
				\draw [line] (se_sync.5) -- node[above]{\small SDP relaxation} (se_sdp.175);
				\draw [line, dashed] (se_sdp.185) -- node[below] (thm1) {\footnotesize low-noise guarantees} node[below of=thm1,node distance = 0.7cm] {\small (Theorem~\ref{thm:tightness_informal})} (se_sync.-5);
				\draw [line] (se_sdp.5) -- node[above]{\small Dist.\ Riemannian
				Staircase} (se_sync_riemannian.175);
				\draw [line, dashed] (se_sync_riemannian.185) -- node[below]
				(thm5) {\footnotesize convergence guarantees} node[below of=thm5,node distance =
				0.7cm] {\small
				(Theorem~\ref{Convergence_of_first_order_Riemannian_staircase_Theorem})} (se_sdp.-5);			
		\end{tikzpicture}
		\caption
		{\small	Relations between problems considered in this work. From the MLE formulation of PGO (Problem~\ref{prob:se_sync}), applying semidefinite relaxation yields the SDP (Problem~\ref{prob:se_sdp}). 
				Applying a distributed implementation of the Riemannian staircase
				algorithm \cite{BoumalBlockDiagonal} and Burer-Monteiro (BM) factorization \cite{Burer2003ANP} on the SDP
				then yields a set of rank-restricted problems
				(Problem~\ref{prob:se_sync_riemannian}) which can be \emph{locally} optimized using
				our distributed Riemannian local search method
				(Section~\ref{sec:bcd}). After local search, the \emph{global} optimality of the recovered first-order critical points of Problem~\ref{prob:se_sync_riemannian} as solutions of the original SDP (Problem~\ref{prob:se_sdp})	can then be checked via \emph{post hoc} verification, and if necessary, a descent direction can be constructed from a suboptimal critical point to continue the search (Section~\ref{sec:verification}).
				Finally, under sufficiently low noise, SDP relaxations are guaranteed to find \emph{global} minimizers of PGO (Theorem~\ref{thm:tightness_informal}).
		}
		\label{fig:flow_chart}
\end{figure*}
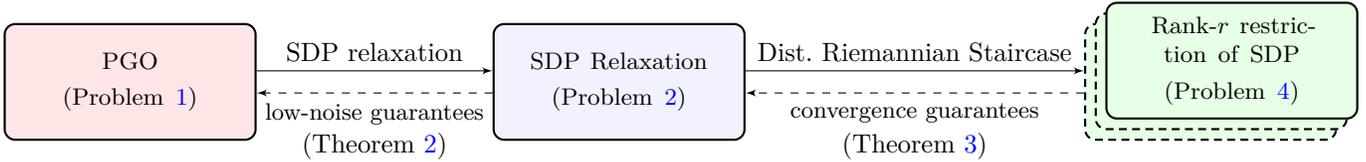

In this section, we formally introduce the pose-graph optimization problem, its semidefinite relaxation, and our certifiably correct algorithm for solving these in the distributed setting.
Figure~\ref{fig:flow_chart} summarizes the problems introduced in this section and how they relate to each other.

\subsection{Pose-Graph Optimization}
\emph{Pose-graph optimization} (PGO) is the problem of estimating unknown poses
from noisy relative measurements. 
PGO can be modeled as a directed graph (a \emph{pose graph})
$\Gcal = (\Vcal, \DirEdges)$, 
where $\Vcal = [n]$ and $\DirEdges \subseteq \Vcal \times \Vcal$ correspond to
the sets of unknown poses and relative measurements, respectively.
In the rest of this paper,
we make the standard assumption that $\Gcal$ is weakly connected.
Let $T_1, T_2, \hdots, T_n \in \SE(d)$ denote the poses that need to be
		estimated, where each $T_i = (R_i, t_i)$ consists of a rotation
		component $R_i \in \SOd(d)$ and a translation component $t_i \in \Real^d$. 
Following \cite{Rosen19IJRR}, we assume that for each edge $(i,j) \in \DirEdges$, the 
corresponding relative measurement $\widetilde{T}_{ij} = (\widetilde{R}_{ij}, \widetilde{t}_{ij})$ from pose $T_i$ to $T_j$ is generated according to:
\begin{align} 
\widetilde{R}_{ij} &= \Runder_{ij}R^{\epsilon}_{ij}, \; R^{\epsilon}_{ij} \sim  \Langevin(I_d, \kappa_{ij}),
\label{eq:rotation_noise_model} \\
\widetilde{t}_{ij} &= \tunder_{ij} + t^{\epsilon}_{ij}, \; t^{\epsilon}_{ij} \sim \mathcal{N}(0, \tau_{ij}^{-1} I_d).
\label{eq:translation_noise_model}
\end{align}
Above, $\Runder_{ij} \triangleq \Runder_i^\top \Runder_j$ and $\tunder_{ij} \triangleq \Runder_i^\top(\tunder_j - \tunder_i)$ denote the true (noiseless) relative rotation and translation, respectively.
Under the noise model
\eqref{eq:rotation_noise_model}-\eqref{eq:translation_noise_model}, it can be
shown that a maximum likelihood estimate (MLE) is obtained as a minimizer of the following non-convex optimization problem \cite{Rosen19IJRR}:

\begin{problem}[Pose-Graph Optimization]
	\normalfont
	\begin{subequations}
		\label{eq:se_sync}
		\begin{align}
		\underset{\{R_i\}, \{t_i\}}{\text{minimize}}
		& \quad \sum_{(i,j) \in \DirEdges} \kappa_{ij} \norm{R_j - R_i \widetilde{R}_{ij}}_F^2 
		+ \tau_{ij}\norm{t_j - t_i - R_i \ttilde_{ij}}^2_2,
		\label{eq:se_sync_cost}  \\
		\text{subject to} 
		&  \quad R_i \in \SOd(d), \, t_i \in \Real^d, \, \forall i \in [n]. 
		\label{eq:se_sync_const}
		\end{align}
	\end{subequations}
	\label{prob:se_sync}
\end{problem}
\vspace{-0.5cm}

\subsubsection*{Collaborative PGO}
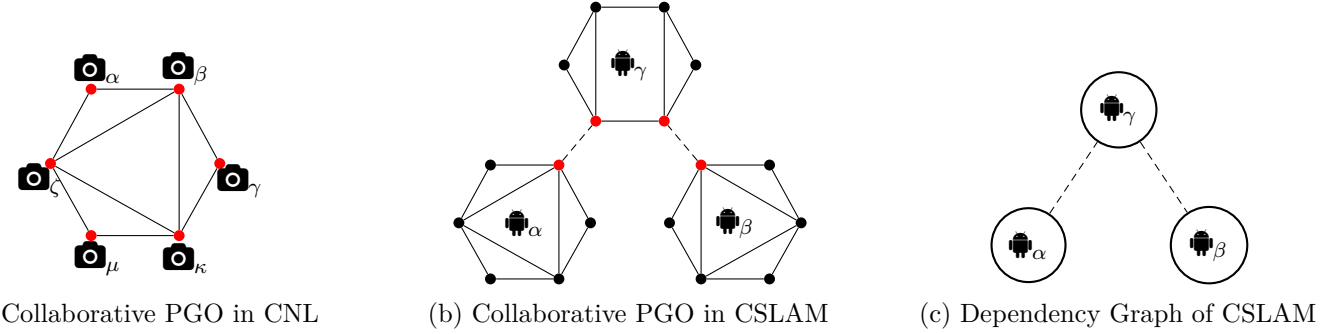
\begin{figure*}[t]
	\centering
	\begin{subfigure}[t]{0.30\textwidth}
		\centering
		\begin{tikzpicture}[scale=0.9]
		\tikzstyle{vertex}=[circle,fill,scale=0.4]
		\tikzstyle{special vertex}=[circle,fill=red,scale=0.4]
		\tikzstyle{square vertex}=[rectangle,fill,scale=0.5,draw]
		\tikzstyle{diamond vertex}=[regular polygon,regular polygon
		sides=3,rotate=45,fill,scale=0.3,draw]
		\node[special vertex] at (0.85,2.598076) (v1) {};
		\node[special vertex] at (2.15,2.598076) (v2) {};
		\node[special vertex] at (2.75,1.499037) (v3) {};
		\node[special vertex] at (2.15, 0.433013) (v4) {};
		\node[special vertex] at (0.85, 0.433013) (v5) {};
		\node[special vertex] at (0.25,1.499037) (v6) {};
		\node (c1) at (0.95,2.90) [] {\faCamera$_\alpha$};
		\node (c2) at (2.25,2.90) [] {\faCamera$_\beta$};
		\node (c4) at (3.05,1.25) [] {\faCamera$_\gamma$};
		\node (c6) at (2.25,0.10) [] {\faCamera$_\kappa$};
		\node (c5) at (0.95,0.10) [] {\faCamera$_\mu$};
		\node (c3) at (0.10,1.25) [] {\faCamera$_\zeta$};
		\draw[](v1) -- (v2);
		\draw[](v2) -- (v3);
		\draw[](v3) -- (v4);
		\draw[](v4) -- (v5);
		\draw[](v5) -- (v6);
		\draw[](v6) -- (v1);
		\draw[](v2) -- (v4);
		\draw[](v2) -- (v6);
		\draw[](v4) -- (v6);
		\end{tikzpicture}
		\captionsetup{justification=centering}
		\caption{Collaborative PGO in CNL}
		\label{fig:cnl}
	\end{subfigure}
	\hspace*{\fill}
	\begin{subfigure}[t]{0.30\textwidth}
		\centering
		\begin{tikzpicture}[scale=0.7]
				\tikzstyle{vertex}=[circle,fill,scale=0.4]
				\tikzstyle{special vertex}=[circle,fill=red,scale=0.4]
				\tikzstyle{square vertex}=[rectangle,fill,scale=0.5,draw]
				\tikzstyle{diamond vertex}=[regular polygon,regular polygon
				sides=3,rotate=45,fill,scale=0.3,draw]
				\node[vertex] at (0.85,2.598076) (Av1) {};
				\node[special vertex] at (2.15,2.598076) (Av2) {};
				\node[vertex] at (2.75,1.499037) (Av3) {};
				\node[vertex] at (2.15, 0.433013) (Av4) {};
				\node[vertex] at (0.85, 0.433013) (Av5) {};
				\node[vertex] at (0.25,1.499037) (Av6) {};

				\node[special vertex] at (4.85,2.598076) (Bv1) {};
				\node[vertex] at (6.15,2.598076) (Bv2) {};
				\node[vertex] at (6.75,1.499037) (Bv3) {};
				\node[vertex] at (6.15, 0.433013) (Bv4) {};
				\node[vertex] at (4.85, 0.433013) (Bv5) {};
				\node[vertex] at (4.25,1.499037) (Bv6) {};

				\node[vertex] at (2.85,5.598076) (Cv1) {};
				\node[vertex] at (4.15,5.598076) (Cv2) {};
				\node[vertex] at (4.75,4.499037) (Cv3) {};
				\node[special vertex] at (4.15, 3.433013) (Cv4) {};
				\node[special vertex] at (2.85, 3.433013) (Cv5) {};
				\node[vertex] at (2.25,4.499037) (Cv6) {};

				\node (r1) at (1.5,1.5) [] {\faAndroid$_{\alpha}$};
				\node (r2) at (5.5,1.5) [] {\faAndroid$_{\beta}$};
				\node (r3) at (3.5, 4.5) [] {\faAndroid$_{\gamma}$};
				\draw[densely dashed](Av2) -- (Cv5);
				\draw[densely dashed](Bv1) -- (Cv4);

				\draw[](Av1) -- (Av2);
				\draw[](Av2) -- (Av3);
				\draw[](Av3) -- (Av4);
				\draw[](Av4) -- (Av5);
				\draw[](Av5) -- (Av6);
				\draw[](Av6) -- (Av1);
				\draw[](Av2) -- (Av4);
				\draw[](Av2) -- (Av6);
				\draw[](Av4) -- (Av6);

				\draw[](Bv1) -- (Bv2);
				\draw[](Bv2) -- (Bv3);
				\draw[](Bv3) -- (Bv4);
				\draw[](Bv4) -- (Bv5);
				\draw[](Bv5) -- (Bv6);
				\draw[](Bv6) -- (Bv1);
				\draw[](Bv1) -- (Bv5);
				\draw[](Bv1) -- (Bv3);
				\draw[](Bv5) -- (Bv3);

				\draw[](Cv1) -- (Cv2);
				\draw[](Cv2) -- (Cv3);
				\draw[](Cv3) -- (Cv4);
				\draw[](Cv4) -- (Cv5);
				\draw[](Cv5) -- (Cv6);
				\draw[](Cv6) -- (Cv1);
		%
				\draw[](Cv2) -- (Cv4);
				\draw[](Cv1) -- (Cv5);
		\end{tikzpicture}
		\captionsetup{justification=centering}
		\caption{Collaborative PGO in CSLAM}
		\label{fig:cslam}
	\end{subfigure}
	\hspace*{\fill}
	\begin{subfigure}[t]{0.30\textwidth}
		\centering
		\begin{tikzpicture}[scale=0.6]
				\tikzstyle{robot}=[circle,scale=1,draw,thick]
				\tikzstyle{special vertex}=[circle,fill=red,scale=0.4]
				\tikzstyle{square vertex}=[rectangle,fill,scale=0.5,draw]
				\tikzstyle{diamond vertex}=[regular polygon,regular polygon
				sides=3,rotate=45,fill,scale=0.3,draw]

				\node[robot] (r1) at (1.5,1.5) [] {\faAndroid$_{\alpha}$};
				\node[robot] (r2) at (5.5,1.5) [] {\faAndroid$_{\beta}$};
				\node[robot] (r3) at (3.5, 4.5) [] {\faAndroid$_{\gamma}$};
				\draw[densely dashed](r3) -- (r2);
				\draw[densely dashed](r1) -- (r3);

		\end{tikzpicture}
		\captionsetup{justification=centering}
		\caption{Dependency Graph of CSLAM}
		\label{fig:dependency_graph}
	\end{subfigure}
	\caption
	{\small
		Two applications of distributed PGO:
		(a) In CNL,
		a group of cameras need to localize each other in a common reference frame. 
		Each vertex in the pose graph denotes the pose of a single camera. Cameras that share overlapping fields of view are connected by relative measurements.
		(b) In CSLAM, multiple robots need to 
		jointly estimate their trajectories in the same frame.
		Each robot has multiple pose variables that are connected by odometry measurements
		and loop closures.
		We refer to poses that have inter-robot loop closures (dashed edges) as
		\emph{public} (marked in \textcolor{red}{red}),
		and all other poses as \emph{private} (marked in black). (c) Dependency graph
		for the CSLAM pose graph shown in (b). Each vertex corresponds to a robot, and two
		vertices are adjacent if and only if there exists at least one inter-robot loop
		closure between the corresponding robots.
	}
	\label{fig:diagram}
\end{figure*}

In collaborative PGO, multiple robots must collaboratively estimate their trajectories in a common
reference frame by solving the \emph{collective} PGO problem
\emph{distributedly} (i.e., without 
outsourcing data to a single ``central'' node)
using inter-robot collaboration.
Each vertex of the collective pose graph represents the pose of a robot at a certain time step.
Odometry measurements and \emph{intra-robot} loop closures connect poses within a single robot's trajectory.
When two robots visit the same place (not necessarily at the
same time), 
they establish \emph{inter-robot} loop closures that link their respective
poses; see \cite{cieslewski2018data,giamou2017talk,tian2019resource} and 
references therein for resource-efficient distributed inter-robot loop closure
detection techniques. Figure~\ref{fig:diagram} illustrates simple examples based
on CNL and CSLAM.
Inter-robot loop closures induce a natural partitioning of collective pose graph nodes into public and
private poses, marked in red and black in Figure~\ref{fig:diagram},
respectively. Additionally, these inter-agent measurements also create 
dependencies between the robots. This is captured by the \emph{dependency graph}
shown in Figure~\ref{fig:dependency_graph}.
\begin{definition}[Public and private poses]
	Poses that share inter-robot
	loop closures with poses of other robots are called \emph{public} poses
	(or \emph{separators} \cite{CunninghamDDFSAM2}).
	All other poses are \emph{private} poses. 
	\label{def:publicparallel}
\end{definition}

\subsection{SDP Relaxation for PGO}
\label{sec:sdp_relaxation}
Traditionally, Problem~\ref{prob:se_sync} is solved with local search algorithms such as Gauss-Newton. 
However, depending on the noise level and the quality of initialization,
local search algorithms are susceptible to local minima \citep{Carlone2015Verification}. 
To address this critical issue, recent works aim to develop \emph{certifiably
correct} PGO solvers.
In particular, techniques based on SDP relaxation demonstrate empirical state-of-the-art
performance while providing theoretical correctness (global optimality) guarantees under low noise regimes \citep{Rosen19IJRR, Bandeira2017Angular, eriksson2019rotation}.

In this section, we present a semidefinite relaxation of Problem~\ref{prob:se_sync} that was first studied in \cite{Briales17CartanSync}.
Let $T \triangleq [R_1 \,\, t_1 \,\, \hdots \,\, R_n \,\, t_n] \in (\SOd(d) \times \Real^d)^n$ be the block-row matrix
obtained by aggregating all rotation and translation variables. \citet{Briales17CartanSync} show that
the cost function \eqref{eq:se_sync_cost} in Problem~\ref{prob:se_sync} can be
written in matrix form as $f(T) = \langle \ConLapT, T^\top T \rangle$, where
$\ConLapT \in \Sym^{(d+1)n}$ is a symmetric matrix known as the \emph{connection
Laplacian} formed using all relative measurements.
Consider the ``lifted'' variable $\ZT = T^\top T \in \PSD^{(d+1)n}$. 
Treating $\ZT$ as a $(d+1) \times (d+1)$-block-structured matrix, we see that several necessary conditions for $\ZT$ to satisfy the original constraints \eqref{eq:se_sync_const} in PGO are,
\begin{align}
	&\ZT \succeq 0, \\
	&\rank(\ZT) = d, \label{eq:rank_constraint} \\
	&{\ZT}_{[i,i](1:d, 1:d)} = R_i^\top R_i = I_{d \times d}, \; \forall i \in [n], \\
	&\det({\ZT}_{[i,j](1:d, 1:d)}) = \det(R_i^\top R_j) = 1, \; \forall i, j \in [n], i \neq j.
	\label{eq:determinant_constraint}
\end{align}
Dropping the non-convex rank and determinant constraints
\eqref{eq:rank_constraint} and \eqref{eq:determinant_constraint} yields an SDP relaxation of Problem~\ref{prob:se_sync}.

\begin{problem}[SDP Relaxation for Pose-Graph Optimization \citep{Briales17CartanSync}]
	\label{prob:SDP_relaxation_for_PGO}
	\begin{subequations}
		\label{eq:se_sdp}
		\begin{align}
		\underset{\ZT \in \PSD^{n+dn}}{\text{minimize}}
		& \quad \langle \ConLapT, \ZT \rangle   \\
		\text{subject to} 
		& \quad  {\ZT}_{[i,i](1:d,1:d)} = I_{d \times d}, \forall i \in [n]. \label{eq:se_sdp_const}
		\end{align}
	\end{subequations}
	\label{prob:se_sdp}
\end{problem}
\vspace{-0.5cm}

The original SE-Sync algorithm \citep{Rosen19IJRR}
employs a different SDP relaxation for Problem~\ref{prob:se_sync}, obtained by first exploiting the so-called \emph{separable} structure of PGO \citep{KasraTRO} to analytically eliminate
the translation variables,
and \emph{then} performing convex relaxation over the resulting rotation-only problem.
This approach yields:
\begin{problem}[Rotation-only SDP Relaxation for Pose-Graph Optimization \citep{Rosen19IJRR}]
	\begin{subequations}
		\label{eq:se_sdp_marg}
		\begin{align}
		\underset{\ZR \in \PSD^{dn}}{\text{minimize}}
		& \quad \langle \ConLapM , \ZR \rangle \label{eq:se_sdp_marg_cost}  \\
		\text{subject to} 
		& \quad  {\ZR}_{[i,i]} = I_{d \times d}, \forall i \in [n], \label{eq:se_sdp_marg_const}
		\end{align}
	\end{subequations}
where $\ConLapM$ is obtained by computing a 
generalized Schur complement of the connection Laplacian $\ConLapT$ (Appendix~\ref{SDP_equivalence_appendix}).
	\label{prob:se_sdp_marg}
\end{problem}

\begin{remark}[Choosing the right SDP for Distributed PGO]
	Problem~\ref{prob:se_sdp_marg} has several advantages over
	Problem~\ref{prob:se_sdp}, including a compact search space and better numerical 
	conditioning. Nevertheless, unlike $\ConLapT$ in Problem~\ref{prob:se_sdp}, the cost matrix $\ConLapM$ in
	Problem~\ref{prob:se_sdp_marg} is generally \emph{dense} (Appendix~\ref{sec:equivalence_proof}).
	In graphical terms, eliminating the translation variables makes the underlying dependency graph 
	\emph{fully connected}. This is a major drawback in the
	distributed setting, since it corresponds to making all of the poses \emph{public}, thereby substantially increasing the required communication.
	As we shall see in the following sections, our proposed algorithm
	relies on and exploits the \emph{sparse} graphical structure (both
	intra-robot and inter-robot) of the problem  to
	achieve computational and communication efficiency, and to preserve the privacy
	of participating robots.
	Therefore, in this work we seek to solve Problem~\ref{prob:se_sdp} as a sparse convex
	relaxation to PGO.
	\label{rem:sdp}
\end{remark}

However, in contrast to the SE-Sync relaxation (Problem~\ref{prob:se_sdp_marg}) \cite{Rosen19IJRR}, Problem~\ref{prob:se_sdp} has \emph{not} previously been shown to enjoy any exactness guarantees.  We now present new results to characterize the
connection between the solutions of these problems, thereby \emph{extending} the guarantee of exactness from Problem~\ref{prob:se_sdp_marg} to  Problem~\ref{prob:se_sdp}.
\begin{theorem}[Connection between Problems~\ref{prob:se_sdp} and
		\ref{prob:se_sdp_marg}]
	Problem~\ref{prob:se_sdp} admits a minimizer $\ZT^\star$ with $\rank(\ZT^\star) = r$
	\emph{if and only if} 
	Problem~\ref{prob:se_sdp_marg} admits a minimizer $\ZR^\star$ with the same rank. 
	Furthermore, $\langle \ConLapT, \ZT^\star \rangle = \langle \ConLapM,
	\ZR^\star \rangle$ for \emph{all} minimizers of
	Problem~\ref{prob:se_sdp} and Problem~\ref{prob:se_sdp_marg}. 
	\label{thm:sdp_equivalence}
\end{theorem}
Theorem~\ref{thm:sdp_equivalence} indicates that relaxing the additional translational variables when forming Problem \ref{prob:SDP_relaxation_for_PGO} does not weaken the relaxation versus the SE-Sync relaxation (Problem \ref{prob:se_sdp_marg}), \emph{nor} (crucially) introduce any additional minimizers that do \emph{not} correspond to PGO solutions.  In particular, Theorem~\ref{thm:sdp_equivalence} and \cite[Proposition~2]{Rosen19IJRR} together imply the following \emph{exactness} guarantee for Problem~\ref{prob:SDP_relaxation_for_PGO} under low measurement noise. 

\begin{theorem}[Exact recovery via Problem~\ref{prob:se_sdp}]
	Let $\underline{\ConLapT}$ be the connection Laplacian in Problem~\ref{prob:se_sdp}
	, constructed using the true (latent) relative transformations $(\Runder_{ij}, \tunder_{ij})$.
	There exists a constant $\delta > 0$ such that
	if $\norm{\ConLapT - \underline{\ConLapT}}_2 < \delta$, 
	{every minimizer} $\ZT^\star$ to Problem~\ref{prob:se_sdp}
	has its first $d \times (n+dn)$ block row 
	given by,
	\begin{equation}
		{\ZT^\star}_{(1:d,:)} = \begin{bmatrix}
		R_1^\star & t_1^\star & \hdots 	& R_n^\star & t_n^\star
		\end{bmatrix},
	\end{equation}
	where 
	$\{R_i^\star, t_i^\star\}$ 
	is an optimal solution to Problem~\ref{prob:se_sync}.
	\label{thm:tightness_informal}
\end{theorem}
Theorem~\ref{thm:tightness_informal} provides a crucial missing piece for achieving
certifiably correct \emph{distributed} PGO solvers: under low noise (quantified by the deviation in spectral norm of the connection Laplacian $\ConLapT$ from its latent value), 
one can directly read off a global minimizer to PGO (Problem~\ref{prob:se_sync})
from the first block row of any solution $\ZT^\star$ of the sparse SDP
relaxation (Problem~\ref{prob:se_sdp}).
As empirically shown in \cite{Rosen19IJRR,Briales17CartanSync}, both SDP
relaxations are exact in real-world scenarios (see Section~\ref{sec:experiments}
for additional empirical evidence).

\subsection{Solving the Relaxation: The Distributed Riemannian Staircase}

In typical CSLAM scenarios, the dimension of the SDP relaxation can be quite large (e.g.\ $\dim(\ZT) > 10^4$),
and thus it is often impractical to solve Problem~\ref{prob:se_sdp} using standard (interior-point) methods.
In a seminal paper, \citet{Burer2003ANP} proposed a more scalable approach to search for \emph{low-rank} solutions $\ZT^\star$ in particular: \emph{assume} that some solution admits a symmetric low-rank factorization of the form 
$\ZT^\star = {X^\star}^\top X^\star$, where $X^\star \in \Real^{r \times n}$ and $r \ll n$, and then directly search for the low-rank factor $X^\star$. This substitution has the two-fold effect of (i) dramatically reducing the dimension of the search space, and (ii) rendering the positive semidefiniteness constraint on $\ZT$ \emph{redundant}, since $X^\top X \succeq 0$ for \emph{any} $X \in \Real^{r \times n}$.  In consequence, the \emph{rank-restricted} version of the original semidefinite program obtained by performing the substitution $\ZT = X^\top X$ is actually a lower-dimensional \emph{nonlinear program}, and so can be processed much more efficiently using standard (local) NLP methods.

For SDPs with block-diagonal constraints,
\citet{BoumalBlockDiagonal} extends the general approach of \citet{Burer2003ANP} 
by further exploiting the \emph{geometric} structure of the constraints in the Burer-Monteiro-factored problem.
The result is an elegant algorithm known as the \emph{Riemannian Staircase},
which is used to solve the (large-scale) semidefinite relaxations in SE-Sync \cite{Rosen19IJRR} and Cartan-Sync \citep{Briales17CartanSync}.

\begin{algorithm}[t]
	\caption{\textsc{\small Distributed Riemannian Staircase }}
	\label{alg:riemannian_staircase}
	\begin{algorithmic}[1]
		\renewcommand{\algorithmicrequire}{\textbf{Input:}}
		\renewcommand{\algorithmicensure}{\textbf{Output:}}
		\Require
		\Statex - Initial point $X \in \Manifold(r_0, n)$.
		\Ensure
		\Statex - A minimizer $X^\star \in \Manifold(r,n)$ of Problem~\ref{prob:se_sync_riemannian} such that $Z^\star = (X^\star)^\top (X^\star)$ is a solution to Problem~\ref{prob:se_sdp}.
		\vspace{0.1cm}
		
		\For{$r=r_0, r_0 + 1, \dotsc$}
		
		\State 
		Compute first-order critical point of Problem~\ref{prob:se_sync_riemannian}: $X^\star \leftarrow \textsc{RBCD}(X)$ (Section~\ref{sec:bcd}). 
		\label{alg:rs_local_search}
		
		\State Lift to first-order critical point at next level: $
		X^\star \leftarrow [
		(X^\star)^\top \,\, 0]^\top$ as in \eqref{lifted_suboptimal_critical_point}.
		\label{alg:rs_lift_rank}
		
		\State Construct corresponding dual certificate matrix $\ST(X^\star)$ in \eqref{verification_equations}.
		\label{alg:rs_dual_certificate}
		
		\State {Compute minimum eigenpair: $(\lambda, v) \leftarrow \Call{MinEig}{\ST(X^\star)}$ (Algorithm \ref{alg:MinEig}).} \label{alg:min_eigenpair_computation}
		
		\If{$\lambda \geq 0$}  \label{alg:min_eig_nonnegativity_check}
		\State \textbf{Return} $X^\star$. \label{alg:rs_return}
		\Else
		\State {Construct second-order descent direction $\dot{X}_{+}$ in \eqref{second_order_descent_direction}.} \label{alg:descent_direction}
		\State {Descend from $X^\star$: $X \leftarrow \textsc{EscapeSaddle}(X^\star, \dot{X}_{+})$ (Algorithm \ref{alg:EscapeSaddle}).}
		\label{alg:rs_escape}
		\EndIf
		\EndFor
	\end{algorithmic}
\end{algorithm}

In this work, we show how to implement the Riemannian Staircase approach in a \emph{distributed} manner, thereby enabling us to solve collaborative PGO.  Algorithm~\ref{alg:riemannian_staircase} presents our distributed Riemannian Staircase algorithm. 
 In each iteration of the Riemannian Staircase, we assume a symmetric rank-$r$ factorization $\ZT = X^\top X$ where $X \in \Real^{r \times (n+dn)}$. 
Writing the blocks of $X$ as $X = [Y_1 \,\, p_1 \,\, \hdots \,\, Y_n \,\,p_n]$, 
the SDP constraints \eqref{eq:se_sdp_const} require that $Y_i \in \Stiefel(d,r)$ and $p_i \in \Real^r$, for all $i \in [n]$.
Equivalently, the aggregate variable $X$ is constrained to live on the product manifold $\Manifold(r,n) \triangleq (\Stiefel(d,r)
\times \Real^r)^n$.
Imposing the rank-$r$ factorization thus transforms the original SDP into the following \emph{rank-restricted} problem:

\begin{problem}[Rank-restricted SDP for Pose-Graph Optimization]
	\label{prob:Rank_restricted_SDP}
	\begin{equation}
	\label{eq:se_sync_riemannian}
	\underset{X \in \Manifold(r,n)}{\text{minimize}}
	\quad \langle \ConLapT, X^\top X \rangle .
	\end{equation}
	\label{prob:se_sync_riemannian}
\end{problem}
\vspace{-0.5cm}

In Section~\ref{sec:bcd}, we develop a novel local search algorithm, \emph{Riemannian block coordinate descent} ($\RBCD$), which we will use to recover first-order critical points of the rank-restricted SDP (Problem \ref{prob:Rank_restricted_SDP}) in the distributed setting.
Inspired by Nesterov's accelerated coordinate descent \cite{Nesterov2012ACD}, we also propose an accelerated variant, $\ARBCD$.
In Section~\ref{sec:convergence}, we establish global first-order convergence guarantees for both $\RBCD$ and $\ARBCD$. 

From a first-order critical point $X^\star$, we ultimately wish to recover a solution to the SDP relaxation. 
To do so, we first lift $X^\star$ to the next level of the staircase (i.e., increment rank $r$ by one).
This can be trivially done by padding $X^\star$ with a row of zeros (line~\ref{alg:rs_lift_rank}). 
The motivations behind this operation will become clear later. 
By construction, the matrix $\ZT = {X^\star}^\top X^\star$ is feasible in Problem~\ref{prob:se_sdp}. 
We may verify the global optimality of $Z$ by checking the (necessary and sufficient) Karush-Kuhn-Tucker (KKT) conditions; for a first-order critical point $X^\star$, this amounts to verifying that a certain dual certificate matrix $S(X^\star)$ is positive semidefinite (see line~\ref{alg:rs_dual_certificate}).  
In Section~\ref{sec:verification}, we present the first distributed procedure to carry out this verification. 
If the dual certificate has a negative eigenvalue, then $\ZT$ is \emph{not} a minimizer of the SDP, and $X^\star$ is in fact a \emph{saddle point} to Problem~\ref{prob:se_sync_riemannian}. 
Fortunately, in this case, the procedure in Section~\ref{sec:verification} also returns a  descent direction, with which we can escape the saddle point (line~\ref{alg:rs_escape}) and restart distributed local search.

\begin{remark}[First- vs.\ second-order optimization in the Riemannian Staircase]
The formulation of the Riemannian Staircase presented in Algorithm
\ref{alg:riemannian_staircase} differs \emph{slightly} from its original
presentation in \cite{BoumalBlockDiagonal}: specifically, the latter presupposes
access to an algorithm that is capable of computing \emph{second}-order critical
points of Problem \ref{prob:Rank_restricted_SDP}, whereas the Riemannian block
coordinate descent method we employ in line \ref{alg:rs_local_search}  only
guarantees convergence to \emph{first}-order critical points.  This has
implications for the convergence properties of the overall algorithm: while one
can show that the second-order version of the Riemannian Staircase \cite[Alg.\
1]{BoumalBlockDiagonal}  is guaranteed to terminate at a level $r \le n$ when
applied to Problem \ref{prob:SDP_relaxation_for_PGO} \cite[Thm.\
3.8]{BoumalBlockDiagonal},\footnote{Strictly speaking, the finite-termination
guarantees provided by \cite[Thm.\ 3.8]{BoumalBlockDiagonal} only hold for the
SE-Sync relaxation Problem \ref{prob:se_sdp_marg} (cf.\ \cite[Prop.\
3]{Rosen19IJRR});  however, we can extend these guarantees to Problem
\ref{prob:SDP_relaxation_for_PGO} by exploiting the correspondence between
critical points of the low-rank factorizations of Problems
\ref{prob:SDP_relaxation_for_PGO} and \ref{prob:se_sdp_marg} that we establish
in Lemma \ref{lem:rank_restricted_SDP_equivalence} (Appendix
\ref{SDP_equivalence_appendix}). }  the weaker (first-order) guarantees provided
by RBCD are reflected in a correspondingly weaker set of convergence guarantees
for our (first-order) Algorithm \ref{alg:riemannian_staircase} provided in the
following theorem.

\begin{theorem}[Convergence of Algorithm \ref{alg:riemannian_staircase}]
\label{Convergence_of_first_order_Riemannian_staircase_Theorem}
Let $\lbrace X^{(r)} \rbrace$ denote the sequence of low-rank factors generated
by Algorithm \ref{alg:riemannian_staircase} in line \ref{alg:rs_lift_rank} using
a particular saddle escape procedure described in Appendix~\ref{staircase_convergence_appendix}. Then exactly one of the following two cases holds:
\begin{itemize}
 \item [$(i)$]  Algorithm \ref{alg:riemannian_staircase} terminates after finitely many iterations and returns a symmetric factor $X^{(r)}$ for a minimizer $Z^\star = (X^{(r)})^\top X^{(r)}$ of Problem \ref{prob:se_sdp} in line \ref{alg:rs_return}.
 \item [$(ii)$]  Algorithm \ref{alg:riemannian_staircase} generates an infinite sequence $\lbrace X^{(r)} \rbrace$ satisfying  $f(X^{(r_2)}) < f(X^{(r_1)})$ for all $r_2 > r_1$, with
 \begin{equation}
\lim_{r \to \infty} f(X^{(r)}) = f_\textnormal{SDP}^\star,
\label{subsequence_converging_to_optimal_value}
 \end{equation}
and there exists an infinite subsequence $\lbrace X^{(r_k)} \rbrace \subset \lbrace X^{(r)} \rbrace$ satisfying:
\begin{equation}
\label{subsequence_converging_to_nonnegative_min_eig}
\lim_{k \to \infty} \lambda_{\textnormal{min}}\left(S(X^{(r_k)}) \right) = 0.
\end{equation}
\end{itemize}
\end{theorem}
\end{remark}

\noindent In a nutshell, Theorem
\ref{Convergence_of_first_order_Riemannian_staircase_Theorem} states that
Algorithm \ref{alg:riemannian_staircase}---with a particular version of saddle
		escape procedure described in
		Appendix~\ref{staircase_convergence_appendix}---either terminates after a finite number of iterations, or generates an infinite sequence of factors $\lbrace X^{(r)} \rbrace$ that monotonically strictly decrease the objective to the optimal value $f_\textnormal{SDP}^\star$ \emph{and} that can arbitrarily well-approximate the satisfaction of the KKT condition $\lambda_{\textnormal{min}}(S(X^{(r)})) \ge 0$. We prove this theorem in Appendix \ref{staircase_convergence_appendix}.

We remark that while the convergence guarantees of Theorem \ref{Convergence_of_first_order_Riemannian_staircase_Theorem} are \emph{formally} weaker than those achievable using a second-order local search method, as a \emph{practical} matter these differences are inconsequential.  In any numerical implementation of the Riemannian Staircase framework, both the second-order criticality of a stationary point (in the second-order version) and the nonnegativity of the minimum eigenvalue $\lambda$ (in Algorithm \ref{alg:riemannian_staircase}) are checked subject to some numerical tolerance $\epsilon_{\textnormal{tol}} > 0$; this accounts for both the finite precision of real-world computers, and the fact that the low-rank factors $X$ computed via local search in line \ref{alg:rs_local_search} are \emph{themselves} only \emph{approximations} to critical points, as they are obtained using iterative local optimization methods.  In particular, practical implementations of Algorithm \ref{alg:riemannian_staircase} (including ours) would replace line \ref{alg:min_eig_nonnegativity_check} with a termination condition of the form  ``$\lambda \ge -\epsilon_{\textnormal{tol}}$''\footnote{This is analogous to the standard stopping criterion $\lVert \nabla f(x) \rVert < \epsilon_{\textnormal{tol}}$ for local optimization methods.}, and \eqref{subsequence_converging_to_nonnegative_min_eig} guarantees that this condition is satisfied after \emph{finitely} many iterations for \emph{any} $\epsilon_{\textnormal{tol}} > 0$.  As a practical matter,  the behavior of Algorithm \ref{alg:riemannian_staircase} is far from the pessimistic case described in part $(ii)$ of Theorem \ref{Convergence_of_first_order_Riemannian_staircase_Theorem}; as we show empirically in Section \ref{sec:experiments}, in real-world applications typically only 1-3 iterations suffice.

\subsection{The Complete Algorithm}
The distributed Riemannian Staircase (Algorithm~\ref{alg:riemannian_staircase}) is the core computational procedure of our overall algorithm.  Nevertheless, to implement a \emph{complete} distributed method for solving the original PGO problem (Problem~\ref{prob:se_sync}), we must still specify procedures for (i) initializing the Riemannian Staircase by constructing an initial
point $X \in \Manifold(r_0, n)$, and (ii) \emph{rounding} the low-rank factor $X^\star$ returned by the
Riemannian Staircase to extract a feasible solution $T \in \SE(d)^n$ of the PGO problem. 
We discuss the details of both distributed initialization and rounding in Section~\ref{sec:initialization_and_rounding}. 
Combining these procedures produces our complete distributed certifiably correct algorithm,  \AlgName\ (Algorithm~\ref{alg:dpgo}).

\begin{algorithm}[t]
	\caption{\textsc{\small Distributed Certifiably Correct Pose Graph Optimization (\AlgName) }}
	\label{alg:dpgo}
	\begin{algorithmic}[1]
		\renewcommand{\algorithmicrequire}{\textbf{Input:}}
		\renewcommand{\algorithmicensure}{\textbf{Output:}}
		\Require
		\Statex - Initial rank $r_0 \geq d$ for the Riemannian Staircase.
		\Ensure
		\Statex - A feasible solution $T \in \SE(d)^n$ to Problem~\ref{prob:se_sync} and the lower bound $f^\star_\text{SDP}$ on Problem~\ref{prob:se_sync}'s optimal value.
		\vspace{0.1cm}
		
		\State 
		Obtain initial point $X \in \Manifold(r,n)$ through distributed initialization.
		\label{alg:initialization}
		
		\State $X^\star \leftarrow \text{DistributedRiemannianStaircase}(X)$.
		
		\State Recover optimal value of the SDP relaxation $f^\star_\text{SDP} = \langle \ConLapT, {X^\star}^\top X^\star \rangle$. 
		
		\State From $X^\star$, obtain feasible $T \in \SE(d)^n$ through distributed rounding.
		\label{alg:rounding}
		
		\State \Return $T, f_\text{SDP}^\star$.
	\end{algorithmic}
\end{algorithm}

Since the SDP (Problem~\ref{prob:se_sdp}) is a convex relaxation of PGO, its optimal value 
$f^\star_\text{SDP}$ is necessarily a lower bound on the global minimum of PGO. 
Using this fact, we may obtain an \emph{upper} bound on the suboptimality of the solution returned by \AlgName.
Specifically, let $f(T)$ denote the objective achieved by the final estimate, and let $f^\star_\text{MLE}$ denote the optimal value of Problem~\ref{prob:se_sync}. Then: 
\begin{equation}
	f(T) - f^\star_\text{MLE} \leq f(T) - f^\star_\text{SDP}. 
	\label{eq:certificate}
\end{equation}
In particular, if $f(T) = f^\star_\text{SDP}$, then the SDP relaxation is \emph{exact}, and $f(T) = f^\star_\text{MLE}$. In this case, \eqref{eq:certificate} serves as a \emph{certificate} of the \emph{global} optimality of $T$.

\section{Distributed Local Search via Riemannian Block-Coordinate Descent}
\label{sec:bcd}

In this section, we introduce a new \emph{distributed} local search algorithm to
identify a first-order critical point of the rank-restricted SDP 
relaxation (Problem~\ref{prob:se_sync_riemannian}), which is needed by the
Distributed Riemannian Staircase framework
(Algorithm~\ref{alg:riemannian_staircase}, line~\ref{alg:rs_local_search}).
Our algorithm is applicable to a broad class of smooth
optimization problems defined over the Cartesian product of matrix manifolds:
\begin{equation}
\label{eq:manopt_problem}
\underset{X \in \Mcal}{\text{minimize}}
\quad f(X), \quad \Mcal \triangleq \Mcal_1 \times \hdots \times \Mcal_N.
\end{equation}
The above problem contains Problem~\ref{prob:se_sync_riemannian} as a special case.
Specifically, in distributed PGO, each block $b$ exactly corresponds to a robot and 
$\Mcal_b = \Manifold(r, n_b) \triangleq (\Stiefel(d,r) \times
\Real^r)^{n_b}$ corresponds to the search space of this robot's trajectory. 
Here, $n_b$ is the number of poses owned by the robot
associated to block $b$, and $N$ is the total number of robots. 
For this reason, unless otherwise mentioned, in this section we use the words ``block'' and ``robot'' interchangeably.

To solve \eqref{eq:manopt_problem}, we leverage the product structure of the underlying manifold and propose a distributed \emph{block-coordinate descent} algorithm that we call $\RBCD$ (Algorithm~\ref{alg:BCD}).
In each iteration of $\RBCD$, a block $b \in [N]$ is selected to be optimized.
Specifically, let $X_b \in \Mcal_b$ be the component of $X$ corresponding to the
selected block, and let $\widehat{X}_{[N] \setminus \{b\}}$ be the (fixed) values of remaining blocks.
We update $X_b$ by minimizing the following \emph{reduced cost function}.
\begin{equation}
\begin{aligned}
& \underset{X_b \in \Mcal_b}{\text{minimize}}
&& f_b(X_b) \triangleq f(X_b, \widehat{X}_{[N]\setminus \{b\}}).
\end{aligned}
\label{eq:subproblem_b}
\end{equation}
For the rank-restricted SDP (Problem~\ref{prob:se_sync_riemannian}) in PGO, the reduced problem for block $b$ takes the form,
\begin{equation}
f_b(X_b) = \langle {\ConLapT}_{b}, X_b^\top X_b \rangle + 2 \langle F_b, X_b \rangle + \text{const}.
\label{eq:reduced_cost}
\end{equation}
In the above equation, ${\ConLapT}_{b}$ is the submatrix of $\ConLapT$ formed with the rows and columns that correspond to block $b$ (i.e., the trajectory of robot $b$), and 
$F_b \in \Real^{r \times (d+1)n_b}$ is a constant matrix that depends on the
(fixed) public variables of robot $b$'s neighbors in the pose graph. 

\begin{remark}[Communication requirements of $\RBCD$]
	$\RBCD$ is designed such that it can be easily implemented by a network of robots. 
	At each iteration, the team first coordinates to select the next block (robot) to update (Algorithm~\ref{alg:BCD}, line~\ref{alg:bcd_block_selection}). 
	Then, to update the selected block (Algorithm~\ref{alg:BCD}, line~\ref{alg:bcd_block_update}), the robot corresponding to this block receives public variables from its neighboring robots in the pose graph.
	Afterwards, this robot forms and solves its local optimization problem~\eqref{eq:subproblem_b}, which does not require further communications.
	Finally, to determine when to terminate $\RBCD$ (Algorithm~\ref{alg:BCD}, line~\ref{alg:bcd_termination}), robots need to collaboratively evaluate their total gradient norm. In practice, checking the termination condition may be done periodically (instead of after every iteration) to save communication resources. 
\end{remark}

\begin{remark}[Block-coordinate minimization on product manifolds]
	Prior works (e.g., \cite{WangBCM2017,ErdogduBCM2018,Tian2019BCM}) have proposed similar \emph{block-coordinate minimization} (BCM) algorithms to solve low-rank factorizations of SDPs with diagonal or block-diagonal constraints. 
	Our approach generalizes these methods in two major ways.
	First, while prior methods are explicitly designed for problems over the product of spheres \cite{WangBCM2017,ErdogduBCM2018} or Stiefel manifolds \cite{Tian2019BCM}, 
	our algorithm is applicable to the product of \emph{any} matrix submanifolds.
	Secondly, prior works \cite{WangBCM2017,ErdogduBCM2018,Tian2019BCM} require
	that the cost function to have a certain quadratic form, so that exact
	minimization of each variable block admits a closed-form solution. In
	contrast, our algorithm does not seek to perform exact minimization, but
	instead computes an inexpensive \emph{approximate} update that achieves a \emph{sufficient reduction} of the objective (see Section~\ref{sec:multiple_pose_update}).
	This makes our method more general and applicable 
	to a much broader class of smooth cost functions that satisfy a Lipschitz-type condition. 
	We discuss this point in greater detail in Section~\ref{sec:convergence}.
\end{remark}

The rest of this section is organized to discuss each step of $\RBCD$ in detail.
We begin in Section~\ref{sec:block_selection_rule} by discussing \emph{block
selection rules}. These rules determine how blocks are selected at each
iteration (Algorithm~\ref{alg:BCD}, line~\ref{alg:bcd_block_selection}).
In Section~\ref{sec:multiple_pose_update}, we propose a general
\emph{block update rule} (Algorithm~\ref{alg:BCD},
line~\ref{alg:bcd_block_update}) based on approximate minimization of trust-region subproblems \cite{absil2007trust}. 
In Section~\ref{sec:acceleration}, we further develop an \emph{accelerated}
variant of $\RBCD$ based on Nesterov's celebrated accelerated coordinate-descent
algorithm \cite{Nesterov2012ACD}, which greatly speeds up convergence near critical points.
Finally, in Section~\ref{sec:parallel_execution} with show that, using a
slight modification of our blocking scheme, we can allow multiple robots to update their
coordinates in parallel, thereby speeding up our distributed local search.

\begin{algorithm}[t]
	\caption{\textsc{\small Riemannian Block-Coordinate Descent (RBCD)} }
	\label{alg:BCD}
	\edit{}{
	\begin{algorithmic}[1]
		\renewcommand{\algorithmicrequire}{\textbf{Input:}}
		\renewcommand{\algorithmicensure}{\textbf{Output:}}
		\Require
		\Statex - Global cost function $f: \Mcal \triangleq \Mcal_1 \times \hdots \times \Mcal_{N} \to \Real$.
		\Statex - Initial solution $X^0 \in \Mcal$.
		\Statex - Stopping condition on gradient norm $\epsilon$.
		\Ensure
		\Statex - First-order critical point $X^\star$.
		\vspace{0.1cm}
		\State $k\leftarrow 0$.
		\While{$\norm{\rgrad f(X^k)} > \epsilon$} \label{alg:bcd_termination}
		\State Select next block $b_k \in [N]$.
		\label{alg:bcd_block_selection}
		\State Update the selected block $X^{k+1}_{b_t} \leftarrow \textsc{BlockUpdate}(f_{b_k}, X_{b_k}^t)$.
		\label{alg:bcd_block_update}
		\State Carry over all other blocks $X^{k+1}_{{b'}} = X^{k}_{{b'}}, \, \forall b' \neq b_k$.
		\State $k \leftarrow k+1$.
		\EndWhile
		\State \Return $X^\star = X^k$.
	\end{algorithmic}
	}
\end{algorithm}

\subsection{Block Selection Rules}
\label{sec:block_selection_rule}
In this section, 
we describe three mechanisms for selecting which block to update at each
iteration of $\RBCD$ (Algorithm~\ref{alg:BCD}, line~\ref{alg:bcd_block_selection}).
We note that similar rules have been proposed in the past; see, e.g.,
\citep{Nutini2017BCD, Javanmard16PNAS, ErdogduBCM2018}.
\begin{itemize}
	\item \textbf{Uniform Sampling}. 
	The first rule is based on the idea of uniform sampling.
	At each iteration, each block $b \in [N]$ is selected with equal probability $p_b = 1 / N$.
	\item \textbf{Importance Sampling}. 
	In practice, it is often the case that selecting certain blocks leads to
	significantly better performance compared to others \cite{Nutini2017BCD}. 
	Therefore, it is natural to assign these blocks higher weights during the sampling process. 
	We refer to this block selection rule as \emph{importance sampling}. 
	In this work, we set the probability of selecting each block to be proportional to the squared gradient norm,
	i.e., $p_b \propto \norm{\rgrad_{b} f (X)}^2, \forall b \in [N]$.
	Here, $\rgrad_b f(X)$ denotes the component of the Riemannian gradient of $f(X)$ that corresponds to block $b$.
	Under Lipschitz-type conditions, the squared gradient norm can be used to construct a lower bound on the achieved cost decrement; see Lemma~\ref{lem:sufficient_decrease}.
	\item \textbf{Greedy (Gauss-Southwell)}. 
	We can also modify importance sampling into a {deterministic} strategy that simply selects the block with the largest squared
	gradient norm, i.e., $b \in \arg\max \norm{\rgrad_{b} f(X)}^2$. We refer to
	this strategy as \emph{greedy selection} or the \emph{Gauss-Southwell (GS)} rule \citep{Nutini2017BCD}.
	Recent works also propose other variants of greedy selection such as 
	Gauss-Southwell-Lipschitz (GSL) and  Gauss-Southwell-Quadratic (GSQ) \citep{Nutini2017BCD}. 
	However, such rules require additional knowledge about the block Lipschitz constants that are hard to obtain in our application. 
	For this reason, we restrict our deterministic selection rule to GS. 
	Despite its simplicity, empirically the GS rule exhibits satisfactory performance; see
	Section~\ref{sec:experiments}. 
\end{itemize}

\begin{remark}[Communication requirements of different block selection rules]
	In practice, uniform sampling does not incur communication overhead, and can be approximately implemented using 
	synchronized clocks on each robot (to conduct and coordinate BCD rounds) 
	and a common random seed for the pseudorandom number generator (to agree on
	which robot should update in the next round).
	In contrast, importance sampling
	and greedy selection require additional communication overhead at each round, as
	robots need to evaluate and exchange local gradient norms. In particular, the greedy
	selection rule can be implemented via flooding gradient norms; see, e.g., the FloodMax
	algorithm for leader election in general synchronized networks
	\cite[Chapter 4]{lynch1996distributed}. This requires robots to have unique IDs
	and communicate in synchronized rounds.
	While greedy and importance rules have higher communication overhead than uniform
	sampling, they also produce more effective iterations and thus
	converge faster (see Section~\ref{sec:experiments}). 
\end{remark}

\subsection{Computing a Block Update}
\label{sec:multiple_pose_update}

\begin{algorithm}[h]
	\caption{\textsc{\small BlockUpdate} }
	\label{alg:BlockUpdate}
	\edit{}{
		\begin{algorithmic}[1]
			\renewcommand{\algorithmicrequire}{\textbf{Input:}}
			\renewcommand{\algorithmicensure}{\textbf{Output:}}
			\Require
			\Statex - Reduced cost function $f_b: \Mcal_b \to \Real$.
			\Statex - Current block estimate $X^k_b \in \Mcal_b$.
			\Statex - User-specified mapping on tangent space $H: T_{X^k_b} \to T_{X^k_b}$ (default to Riemannian Hessian).
			\Statex - Initial trust-region radius $\Delta_0$.
			\Ensure
			\Statex - Updated block estimate $X^{k+1}_b \in \Mcal_b$.
			\vspace{0.1cm}
			\State $\Delta \leftarrow \Delta_0$.
			\State Form model function $\hat{m}_b(\eta_b) = f_b(X_b^k) + \langle \rgrad f_b(X_b^k), \eta_b \rangle + \frac{1}{2}
			\langle \eta_b, H [\eta_b] \rangle$.
			\While{true}
				\State Compute an \emph{approximate} solution $\eta_b^\star \in T_{X_b^k} \Mcal_b$ to the trust-region subproblem \eqref{eq:trust_region_subproblem}.
				\If{$\rho(\eta_b^\star) > 1/4$}
				\label{alg:block_update_termination}
					\State \Return $X_b^{k+1} = \Retr_{X_b^k}(\eta_b^\star)$.
				\Else
					\State Decrease trust-region radius $\Delta \leftarrow \Delta / 4$.
				\EndIf
			\EndWhile
		\end{algorithmic}
	}
\end{algorithm}

Note that since \eqref{eq:subproblem_b} is in general a nonconvex minimization, computing a block update by \emph{exactly} solving this problem is intractable.  In this section, we describe how to implement a cheaper approach that permits the use of \emph{approximate} solutions of \eqref{eq:subproblem_b} by requiring only that they produce a \emph{sufficient decrease} of the objective.
In Section~\ref{sec:convergence}, we show that under mild conditions, such approximate updates are sufficient to ensure global first-order convergence of $\RBCD$.
While there are many options to achieve sufficient descent, in this work we propose to (approximately) solve a single trust-region subproblem using the truncated preconditioned conjugate gradient method \cite{absil2007trust,absil2009optimization}. 
Compared to a full minimization that would solve \eqref{eq:subproblem_b} to first-order critical point, our approach greatly reduces the computational cost.
On the other hand, unlike other approximate update methods such as Riemannian gradient descent, our method allows us to leverage (local) second-order information of the reduced cost which leads to more effective updates.

Let $b$ be the block that we select to update, and denote the current value of this block (at iteration $k$) as $X_b^k$. We define the \emph{pullback} of the reduced
cost function \eqref{eq:subproblem_b} as follows \cite{absil2009optimization,Boumal2018Convergence},
\begin{equation}
	\begin{aligned}
		\widehat{f}_b : T_{X^k_b} & \to \Real,  \\
		\eta_b & \mapsto f_b \circ \Retr_{X^k_b} (\eta_b).
	\end{aligned}
	\label{eq:reduced_pullback}
\end{equation}
Note that the pullback is conveniently defined on the tangent space which itself is a vector space. 
However, since directly minimizing the pullback is nontrivial, it is
approximated with a quadratic \emph{model} function \cite{absil2009optimization,Boumal2018Convergence}, as defined below.
\begin{equation}
	\begin{aligned}
	\widehat{m}_b(\eta_b) \triangleq f_b(X_b^k) + \langle \rgrad f_b(X_b^k), \eta_b \rangle + \frac{1}{2}
	\langle \eta_b, H [\eta_b] \rangle.
	\end{aligned}
	\label{eq:reduced_model}
\end{equation}
In \eqref{eq:reduced_model}, $H: T_{X_b^k} \to T_{X_b^k}$ is a user-specified mapping on the tangent space. 
By default, we use the Riemannian Hessian $H = \Hess f_b(X_b^k)$ so that the model function is a second-order approximation of the pullback. 
Then, we compute an update direction $\eta_b^\star$ on the tangent space by approximately solving the following trust-region subproblem,
\begin{equation}
	\begin{aligned}
				\text{minimize}_{\eta_b \in T_{X_b^k} \Mcal_b} \;\;
				\mhat_{b}(\eta_b) \;\;
				\text{subject to} \;\;
				\norm{\eta_b} \leq \Delta.
		\end{aligned}
		\label{eq:trust_region_subproblem}
\end{equation}

To ensure that the obtained update direction yields sufficient descent on the original pullback, we follow standard procedure \cite{absil2007trust,absil2009optimization} and evaluate the following ratio that quantifies the
\emph{agreement} between model decrease (predicted reduction) and pullback
decrease (actual reduction),
\begin{equation}
\rho(\eta_b^\star) 
\triangleq
\frac{\fhat_b(0) - \fhat_{b}(\eta_{b}^\star)}{\mhat_b(0) - \mhat_{b}(\eta_{b}^\star)}
=
\frac{f_b(X_b^k) - \fhat_{b}(\eta_{b}^\star)}{f_b(X_b^k) - \mhat_{b}(\eta_{b}^\star)}.
\label{eq:rtr_quotient}
\end{equation}
If the above ratio is larger than a constant threshold (default
to $1/4$), we accept the current update direction and set $X_b^{k+1} = \Retr_{X_b^k} (\eta_b^\star)$ as the updated value of this block.
Otherwise, we reduce the trust-region radius $\Delta$ and solve the trust-region subproblem again. 
Algorithm~\ref{alg:BlockUpdate} gives the pseudocode for the entire block update procedure.

In Appendix~\ref{apx:convergence_proof}, we prove that under mild conditions, we can always find an update direction (the so-called \emph{Cauchy step} \cite{absil2007trust,absil2009optimization}) that satisfies the required termination condition (Algorithm~\ref{alg:BlockUpdate}, line~\ref{alg:block_update_termination}). 
Furthermore, the returned solution is guaranteed to produce \emph{sufficient descent} on the cost function, which is crucial to establish global convergence rate of $\RBCD$ (Algorithm~\ref{alg:BCD}). We discuss the details of convergence analysis in Section~\ref{sec:convergence}.

\begin{remark}[Solving the trust-region subproblem \eqref{eq:trust_region_subproblem}]
		Following \cite{absil2007trust,absil2009optimization}, we also use the
		truncated conjugate-gradient (tCG) method to solve the trust-region subproblem \eqref{eq:trust_region_subproblem} inside Algorithm~\ref{alg:BlockUpdate}. 
	tCG is an efficient ``inverse-free'' method, i.e., it does not require
	inverting the Hessian itself, and instead only requires evaluating Hessian-vector products. 
	Furthermore, tCG can be significantly accelerated using a suitable \emph{preconditioner}.
	Formally, a preconditioner is a linear, symmetric, and positive-definite operator on the tangent space
	that approximates the inverse of the Riemannian Hessian.
	Both SE-Sync \cite{Rosen17IROS,Rosen19IJRR} and Cartan-Sync \cite{Briales17CartanSync} have already proposed empirically effective preconditioners for problems similar  to \eqref{eq:reduced_cost}.\footnote{The only difference is the additional linear terms in our cost functions, as a result of anchoring variables owned by other robots.}
	Drawing similar intuitions from these works, we design our preconditioner as,
	\begin{align}
			\Precon f(X_b^k): T_{X_b^k} &\to T_{X_b^k}, \nonumber \\ \eta_b
			&\mapsto \proj_{T_{X_b^k}} (\eta_b ({\ConLapT}_{b} + \lambda I)^{-1}).
	\label{eq:pose_preconditioner}
	\end{align}
	The small constant $\lambda > 0$ ensures that the proposed preconditioner is
	positive-definite. 
	In practice, we can store and reuse the Cholesky decomposition of ${\ConLapT}_{b} + \lambda I$ for improved numerical efficiency.
\end{remark}

\begin{remark}[Block update via exact minimization]
		Algorithm~\ref{alg:BlockUpdate} employs RTR \cite{absil2007trust} to
		\emph{sufficiently reduce} the cost function along a block coordinate.
		It is worth noting that in some special cases, one can \emph{exactly} minimize the cost
		function along any single block coordinate \cite{ErdogduBCM2018,Tian2019BCM}. In the context of PGO, this is the
		case when every block consists only of a \emph{single} pose which
		naturally arises in CNL.
		We provide a quick sketch in the following. First, note that the reduced
		problem \eqref{eq:reduced_cost} is an unconstrained convex quadratic
		over the Euclidean component of $\Mcal_b = \Stiefel(d,r) \times
		\Real^r$ (i.e., the so-called \emph{lifted} translation component). We
		can thus first eliminate this component
		analytically by minimizing the reduced
		cost over the lifted translation vector, thereby further reducing \eqref{eq:reduced_cost} to an optimization problem
		over $\Stiefel(d,r)$. Interestingly, the resulting problem too
		admits a closed-form solution via the projection operator onto $\Stiefel(d,r)$ provided
		in \eqref{eq:projection_to_stiefel} (see also \cite[Sec.~2]{Tian2019BCM} for
		a similar approach). Finally, using this solution we can recover the
		optimal value for the Euclidean component via linear least
		squares (see, e.g., \cite{KasraTRO}).
		\label{rem:bcm}
\end{remark}

\subsection{Accelerated Riemannian Block-Coordinate Descent}
\label{sec:acceleration}

\begin{algorithm}[t]
	\caption{\textsc{\small Accelerated Riemannian Block-Coordinate Descent ($\ARBCD$)}}
	\label{alg:ARBCD}
	\edit{}{
		\begin{algorithmic}[1]
			\renewcommand{\algorithmicrequire}{\textbf{Input:}}
			\renewcommand{\algorithmicensure}{\textbf{Output:}}
			\Require
			\Statex - Global cost function $f: \Mcal \triangleq \Mcal_1 \times \hdots \times \Mcal_{N} \to \Real$.
			\Statex - Initial solution $X^0 \in \Mcal$.
			\Statex - Stopping condition on gradient norm $\epsilon$.
			\Statex - Restart constant $c_1 > 0$.
			\Ensure
			\Statex - First-order critical point $X^\star$.
			\vspace{0.1cm}
			\State $k\leftarrow 0, V^0 \leftarrow X^0, \gamma_{-1} \leftarrow 0$.
			\While{$\norm{\rgrad f(X^k)} > \epsilon$}
			\State 
			$\gamma_k \leftarrow (1 + \sqrt{1 + 4N^2{\gamma_{k-1}}^2})/{2N}, \; \alpha_k \leftarrow 1/{\gamma_k N}$.
			\label{alg:abcd_scalar_updates}
			\State {\color{gray} // Y update}
			\State $Y^k \leftarrow \proj_{\Mcal}((1-\alpha_k) X^k + \alpha_k V^k)$.
			\label{alg:abcd_y_update}
			\State {\color{gray} // X update}
			\State Select next block $b_k \in [N]$.
			\State Update the selected block $X^{k+1}_{b_k} \leftarrow \textsc{BlockUpdate}(f_{b_k}, Y_{b_k}^k)$ \label{alg:abcd_x_update}.
			\State Carry over all other blocks $X^{k+1}_{{b'}} \leftarrow Y^{k}_{{b'}}, \, \forall b' \neq b_k$.
			\State {\color{gray} // V update}
			\State $V^{k+1} \leftarrow \proj_{\Mcal} ( V^k + \gamma_k (X^{k+1} - Y^k))$.
			\label{alg:abcd_v_update}
			\State {\color{gray} // Adaptive restart}
			\If {$f(X^{k}) - f(X^{k+1}) < c_1 \norm{\rgrad_{b_k} f(X^k)}^2 $} \label{alg:restart_start}
					\State {\color{gray} // Use default block update}
					\label{alg:actual_restart_start}
				\State $X^{k+1}_{b_k} \leftarrow \textsc{BlockUpdate}(f_{b_k}, X_{b_k}^k)$.
				\State Carry over all other blocks $X^{t+1}_{{b'}} \leftarrow X^{k}_{{b'}}, \, \forall b' \neq b_k$.
				\State {\color{gray} // Reset Nesterov's acceleration}
				\State $V^{k+1} \leftarrow X^{k+1}$.
				\State $\gamma_k = 0$.
					\label{alg:actual_restart_end}
			\EndIf \label{alg:restart_end}
			\State $k \leftarrow k+1$.
			\EndWhile
			\State \Return $X^\star = X^k$.
		\end{algorithmic}
	}
\end{algorithm}

In practice, many PGO problems are poorly conditioned. 
Critically, this means that a generic first-order algorithm can suffer from slow convergence as the iterates approach a first-order critical point. Such slow convergence is also manifested by the typical \emph{sublinear} convergence rate, e.g., for Riemannian gradient descent as shown in \cite{Boumal2018Convergence}.
To address this issue, Fan and Murphy \cite{Fan2019ISRR,Fan2020Arxiv} recently developed a majorization-minimization algorithm for PGO. 
Crucially, their approach can be augmented with a generalized version of Nesterov's acceleration that significantly speeds up empirical convergence.

Following the same vein of ideas, we show that it is possible to significantly speed up $\RBCD$ by adapting the celebrated accelerated coordinate-descent method (ACDM), originally developed by Nesterov \cite{Nesterov2012ACD} to solve smooth convex optimization problems.
Compared to the standard randomized coordinate descent method, ACDM enjoys an accelerated convergence rate of $\mathcal{O}(1/k^2)$.
Let $N$ denote the dimension (number of coordinates) in the problem. 
ACDM updates two scalar sequences $\{\gamma_k\}, \{\alpha_k\}$ and
three sequences of iterates $\{x^k\}, \{y^k\}, \{v^k\} \in \Real^N$.
\begin{align}
	\gamma_k &= (1 + \sqrt{1 + 4N^2{\gamma_{k-1}}^2})/{2N}, \label{eq:gamma_update}\\
	\alpha_k &= 1/(\gamma_k N), \\
	y^k &= (1-\alpha_k)x^k + \alpha_k v^k, \label{eq:y_update} \\
	x^{k+1} &= y^k - 1/L_{b_k} \nabla_{b_k} f(y^k), \label{eq:x_update} \\
	v^{k+1} &= v^k + \gamma_k (x^{k+1} - y^k). \label{eq:v_update}
\end{align}
In \eqref{eq:x_update}, $L_{b_k}$ is the Lipschitz constant of the gradient that corresponds to coordinate $b_k$. 
Note that compared to standard references (e.g., \cite{Nesterov2012ACD,Wright2015Survey}), we have slightly changed the presentation of ACDM, so that later it can be extended to our Riemannian setting in a more straightforward manner.  
Still, it can be readily verified that \eqref{eq:gamma_update}-\eqref{eq:v_update} are equivalent to the original algorithm.\footnote{For example, 
we can recover \eqref{eq:gamma_update}-\eqref{eq:v_update} from \cite[Algorithm
4]{Wright2015Survey}, by setting the strong convexity parameter $\sigma$ to zero.}

In Algorithm~\ref{alg:ARBCD}, we adapt the ACDM iterations to design an accelerated variant of $\RBCD$, which we call $\ARBCD$.
We leverage the fact that
our manifolds of interest are naturally embedded within some linear space.
This allows us to first perform the additions and subtractions as stated in  \eqref{eq:gamma_update}-\eqref{eq:v_update} in the linear ambient space, and subsequently \emph{project} the result back to the manifold. 
For our main manifold of interest $\Manifold(r,n)$, the projection operation only requires computing the SVD for each Stiefel component, as shown in \eqref{eq:projection_to_stiefel}. 
Note that the original ACDM method performs a coordinate descent step \eqref{eq:x_update} at each iteration. 
	In $\ARBCD$, we generalize \eqref{eq:x_update} by employing the $\BlockUpdate$ procedure (Algorithm~\ref{alg:BlockUpdate}) to perform a descent step on a \emph{block coordinate} $Y_b \in \Mcal_b$ (Algorithm~\ref{alg:ARBCD}, line~\ref{alg:abcd_x_update}).

Unlike the convex case, it is unclear how to prove convergence of the above acceleration scheme subject to the non-convex manifold constraints.
	Fortunately, convergence can be guaranteed by adding \emph{adaptive restart} \cite{Odonoghue2012}, which has also been employed in recent works \cite{Fan2019ISRR,fan2020majorization}.
The underlying idea is to ensure that each $\ARBCD$ update (specifically on the
$\{X^t\}$ variables) yields a \emph{sufficient reduction} of the overall cost function.
This is quantified by comparing the descent with the squared
gradient norm at the selected block (Algorithm~\ref{alg:ARBCD}, line~\ref{alg:restart_start}), where the constant $c_1 > 0$ specifies the minimum amount of descent enforced at each iteration.
If this criterion is met, the algorithm simply continues to the next iteration.
If not, the algorithm switches to the default block update method (same as $\RBCD$), and restarts the acceleration scheme from scratch. 
Empirically, we observe that setting $c_1$ close to zero (corresponding to a permissive acceptance criterion) gives the best performance.

\begin{remark}[Adaptive vs. fixed restart schemes]
	\label{rem:restart}
	Our	adaptive restart scheme requires aggregating information from
	all robots to evaluate the cost function and gradient norm
	(Algorithm~\ref{alg:ARBCD}, line~\ref{alg:restart_start}). This step 
	may become the communication bottleneck of the whole algorithm. 
	While in theory we need adaptive restart to guarantee convergence (see
	Section~\ref{sec:convergence}), a practical remedy is to employ a
	\emph{fixed restart} scheme \cite{Odonoghue2012} whereby we simply restart acceleration
	(Algorithm~\ref{alg:ARBCD},
	lines~\ref{alg:actual_restart_start}-\ref{alg:actual_restart_end}) periodically in fixed
	intervals.  Our empirical results in Section~\ref{sec:experiments} show that the fixed restart
	scheme also achieves significant acceleration, although is 
	inferior to adaptive restart scheme.
\end{remark}

\begin{remark}[Communication requirements of $\ARBCD$]
	With fixed restart, the communication pattern of $\ARBCD$ is identical to $\RBCD$.
	In particular, with synchronized clocks, robots can update the scalars $\gamma_k$ and $\alpha_k$ (line~\ref{alg:abcd_scalar_updates}) locally in parallel.
	Similarly, the ``$Y$ update'' (line~\ref{alg:abcd_y_update}) and ``$V$ update'' steps (line~\ref{alg:abcd_v_update}) do not require communication, since both only involve local linear combinations and projections to manifold.
	The main communication happens before the ``$X$ update'' (line~\ref{alg:abcd_x_update}), where each robot communicates the public components of their $Y$ variables with neighbors in the global pose graph.
	Finally, if adaptive restart is used, robots need to communicate and aggregate global cost and gradient norms to evaluate the restart condition (line~\ref{alg:restart_start}).
\end{remark}

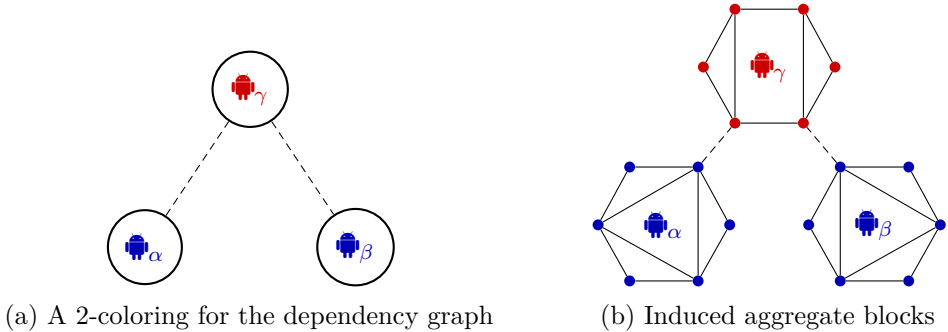
\begin{figure}[t]
	\centering
	\begin{subfigure}[t]{0.35\textwidth}
		\centering
		\begin{tikzpicture}[scale=0.7]
		\tikzstyle{robot}=[circle,scale=1,draw,thick]
		\tikzstyle{special vertex}=[circle,fill=red,scale=0.4]
		\tikzstyle{square vertex}=[rectangle,fill,scale=0.5,draw]
		\tikzstyle{diamond vertex}=[regular polygon,regular polygon
		sides=3,rotate=45,fill,scale=0.3,draw]
		\node[robot] (r1) at (1.5,1.5) [] {\color{blue!70!black}\faAndroid$_{\alpha}$};
		\node[robot] (r2) at (5.5,1.5) [] {\color{blue!70!black}\faAndroid$_{\beta}$};
		\node[robot] (r3) at (3.5, 4.5) [] {\color{red!80!black}\faAndroid$_{\gamma}$};
		\draw[densely dashed](r3) -- (r2);
		\draw[densely dashed](r1) -- (r3);
		
		\end{tikzpicture}
		\captionsetup{justification=centering}
		\caption{A $2$-coloring for the dependency graph}
		\label{fig:parcslam}
	\end{subfigure}
	~
	\begin{subfigure}[t]{0.35\textwidth}
		\centering
		\begin{tikzpicture}[scale=0.7]
		\tikzstyle{vertex}=[circle,fill,scale=0.4]
		\tikzstyle{red vertex}=[circle,fill=red!80!black,scale=0.4]
		\tikzstyle{blue vertex}=[circle,fill=blue!70!black,scale=0.4]
		\tikzstyle{special3
			vertex}=[rectangle,fill=blue!70!black,scale=0.8]
		\tikzstyle{special vertex}=[star,line width=0.5mm,fill=blue!70!black,scale=0.6]
		\tikzstyle{special2 vertex}=[star,fill=red!80!black,line
		width=0.5mm,scale=0.6]
		\tikzstyle{square vertex}=[rectangle,fill,scale=0.5,draw]
		\tikzstyle{diamond vertex}=[regular polygon,regular polygon
		sides=5,rotate=45,fill,scale=0.5,draw]
		
		\node[blue vertex] at (0.85,2.598076) (Av1) {};
		\node[blue vertex] at (2.15,2.598076) (Av2) {};
		\node[blue vertex] at (2.75,1.499037) (Av3) {};
		\node[blue vertex] at (2.15, 0.433013) (Av4) {};
		\node[blue vertex] at (0.85, 0.433013) (Av5) {};
		\node[blue vertex] at (0.25,1.499037) (Av6) {};
		
		\node[blue vertex] at (4.85,2.598076) (Bv1) {};
		\node[blue vertex] at (6.15,2.598076) (Bv2) {};
		\node[blue vertex] at (6.75,1.499037) (Bv3) {};
		\node[blue vertex] at (6.15, 0.433013) (Bv4) {};
		\node[blue vertex] at (4.85, 0.433013) (Bv5) {};
		\node[blue vertex] at (4.25,1.499037) (Bv6) {};
		
		\node[red vertex] at (2.85,5.598076) (Cv1) {};
		\node[red vertex] at (4.15,5.598076) (Cv2) {};
		\node[red vertex] at (4.75,4.499037) (Cv3) {};
		\node[red vertex] at (4.15, 3.433013) (Cv4) {};
		\node[red vertex] at (2.85, 3.433013) (Cv5) {};
		\node[red vertex] at (2.25,4.499037) (Cv6) {};

		\node (r1) at (1.5,1.5) [] {\color{blue!70!black}\faAndroid$_{\alpha}$};
		\node (r2) at (5.5,1.5) [] {\color{blue!70!black}\faAndroid$_{\beta}$};
		\node (r3) at (3.5, 4.5) [] {\color{red!80!black}\faAndroid$_{\gamma}$};
		\draw[densely dashed](Av2) -- (Cv5);
		\draw[densely dashed](Bv1) -- (Cv4);
		
		\draw[](Av1) -- (Av2);
		\draw[](Av2) -- (Av3);
		\draw[](Av3) -- (Av4);
		\draw[](Av4) -- (Av5);
		\draw[](Av5) -- (Av6);
		\draw[](Av6) -- (Av1);
		\draw[](Av2) -- (Av4);
		\draw[](Av2) -- (Av6);
		\draw[](Av4) -- (Av6);
		
		\draw[](Bv1) -- (Bv2);
		\draw[](Bv2) -- (Bv3);
		\draw[](Bv3) -- (Bv4);
		\draw[](Bv4) -- (Bv5);
		\draw[](Bv5) -- (Bv6);
		\draw[](Bv6) -- (Bv1);
		\draw[](Bv1) -- (Bv5);
		\draw[](Bv1) -- (Bv3);
		\draw[](Bv5) -- (Bv3);
		
		\draw[](Cv1) -- (Cv2);
		\draw[](Cv2) -- (Cv3);
		\draw[](Cv3) -- (Cv4);
		\draw[](Cv4) -- (Cv5);
		\draw[](Cv5) -- (Cv6);
		\draw[](Cv6) -- (Cv1);
		%
		\draw[](Cv2) -- (Cv4);
		\draw[](Cv1) -- (Cv5);
		\end{tikzpicture}
		\captionsetup{justification=centering}
		\caption{Induced aggregate blocks}
		\label{fig:parcslam2}
	\end{subfigure}
	\caption
	{\small Parallel updates for 
		the collective pose graph shown in Figure~\ref{fig:diagram}:
		(a) First, we find a coloring for the corresponding dependency graph
		such that adjacent robots have different colors; (b) The
		$2$-coloring induces two ``aggregate blocks'' $
		{\Acal_{\color{blue!70!black}1}},
		{\Acal_{\color{red!80!black}2}}$ where
		$\Acal_{\color{blue!70!black}1}$ and $\Acal_{\color{red!80!black}2}$
		consist of all blue and red vertices, respectively. In each iteration,
		we select a color and update the corresponding variables. Note that 
		$\Acal_{\color{blue!70!black}1}$ cotains variables from both
		$\text{\faAndroid}_\alpha$ and
		$\text{\faAndroid}_\beta$, and therefore these robots can update their variables in
		parallel when blue is selected.
	}
	\label{fig:pardiagram}
\end{figure}

{ 
	\subsection{Parallel Riemannian Block-Coordinate Descent}
	\label{sec:parallel_execution}
	Thus far in each round of $\RBCD$ and $\ARBCD$ (Algorithms~\ref{alg:BCD} and \ref{alg:ARBCD}),
	exactly one robot performs \textsc{BlockUpdate}
	(Algorithm~\ref{alg:BlockUpdate}). However, after a slight modification
	of our blocking scheme, \emph{multiple} robots may update their variables
	\emph{in parallel} as long as they are \emph{not} neighbours in the dependency graph
	(i.e., do not share an inter-robot
	loop closure; see Figure~\ref{fig:dependency_graph}).
	This is achieved by leveraging the natural graphical decomposition of objectives
	in Problem \ref{prob:se_sync_riemannian} (inherited from Problem
	\ref{prob:se_sync}). Updating variables in parallel can significantly speed up the local
	search.
	
	Gauss-Seidel-type updates can be executed in parallel using a classical
	technique known as red-black coloring (or, more generally, multicoloring
	schemes) \citep{bertsekas1989parallel}. 
	We apply this technique to PGO (Figure~\ref{fig:pardiagram}):
	
	\begin{enumerate}
		\item   First, we find a coloring for the set of
		\emph{robots}
		such that \emph{adjacent robots} in the dependency graph have
		different colors (Figure~\ref{fig:parcslam}). Although finding a vertex coloring with the \emph{smallest} number of
		colors is NP-hard, simple greedy approximation algorithms
		can produce a $(\Delta + 1)$-coloring, where $\Delta$ is the
		maximum degree of the dependency graph; see \cite{barenboim2009distributed,schneider2010new} and the
		references therein for distributed algorithms.
		Note that $\Delta$ is often bounded by a small constant due to
		the sparsity of the CSLAM dependency graph. 
		\item  In each iteration, we select a color (instead of a single
		robot) by adapting the block selection
		rules presented in Section~\ref{sec:block_selection_rule}.
		The robots that have the selected color then update their
		variables in parallel.
	\end{enumerate}
	
	Implementing the (generalized) importance sampling and
	greedy rules (Section~\ref{sec:block_selection_rule}) with coloring
	requires additional coordination between the robots. In
	particular, the greedy rule requires computing the sum of squared gradient
	norms for each color at the beginning of each iteration. Similar to
	Section~\ref{sec:block_selection_rule}, a na\"{i}ve approach would
	be to flood the network with the current squared gradient norms such that after
	a sufficient number of rounds (specifically, the diameter of the dependency
	graph), every robot aggregates all squared gradient norm information for every
	color. Robots can then independently compute the sum of squared gradient norms
	for every color and update their block only if their color has the largest gradient
	norm among all colors. We conclude this part by noting that it is also possible to
	allow \emph{all} robots to update their \emph{private} variables in \emph{all}
	iterations (irrespective of the selected color)
	because private variables are separated from each other by public variables.
}

\section{Convergence Analysis for Riemannian Block-Coordinate Descent}
\label{sec:convergence}

In this section, we formally establish first-order convergence guarantees for RBCD (Algorithm~\ref{alg:BCD}) and its accelerated variant $\ARBCD$ (Algorithm~\ref{alg:ARBCD}), for generic optimization problems of the form 
\eqref{eq:manopt_problem} defined on the Cartesian product of smooth matrix submanifolds.
Then, at the end of this section, we remark on how our general convergence guarantees apply when solving the rank-restricted SDPs of PGO (Remark~\ref{rem:convergence_pgo}), with detailed proofs and discussions in Appendix~\ref{sec:convergence_pgo}. Our global convergence proofs extend the recent work of Boumal et al.~\citep{Boumal2018Convergence}. 
Similar to \citep{Boumal2018Convergence}, we begin by listing and discussing several mild technical assumptions.

\begin{assumption}[Lipschitz-type gradient for pullbacks]
	In the optimization problem \eqref{eq:manopt_problem}, consider the reduced cost $f_b$ and its pullback $\fhat_b$ for an arbitrary block $b \in [N]$.
	There exists a constant $c_b \geq 0$ such that 
	at any iterate $X^k$ generated by a specified algorithm, 
	the following inequality holds for all $\eta_b \in T_{X^k_b} \Mcal_b$,
	\begin{equation}
	\big|\fhat_b(\eta_b) - [f_b(X^k_b) + \langle \eta_b, \rgrad f_b(X^k_b) \rangle]\big| \leq \frac{c_b}{2} \norm{\eta_b}^2_2.
	\label{eq:lipschitz_gradient_pullback}
	\end{equation}
	\label{as:lipschitz_gradient_pullback}
\end{assumption}

In \cite{Boumal2018Convergence}, Boumal et al. use  \eqref{eq:lipschitz_gradient_pullback} as a convenient generalization of the classical property of Lipschitz continuous gradient. 
With this property, the authors show that it is straightforward to establish global first-order convergence guarantees, e.g., for the well-known Riemannian gradient descent algorithm.
Furthermore, it is shown that \eqref{eq:lipschitz_gradient_pullback} holds under mild conditions in practice, e.g., whenever the cost function has Lipschitz continuous gradients in the ambient Euclidean space and the underlying submanifold is compact \citep[Lemma~2.7]{Boumal2018Convergence}. 
In Appendix~\ref{sec:convergence_pgo}, we will generalize this result to show that Assumption~\ref{as:lipschitz_gradient_pullback} holds when running \RBCD\ and \ARBCD\ in the context of PGO.  

\begin{assumption}[Global radial linearity of $H$]
	In \BlockUpdate\ (Algorithm~\ref{alg:BlockUpdate}), 
	the user-specified map $H$ \eqref{eq:reduced_model} is globally radially linear, i.e., for any block $b \in [N]$,
	\begin{equation}
	H[c \eta_b] = c H[\eta_b], \, \text{ for all } \eta_b \in T_{X_b} \Mcal_b \text{ and } c \geq 0.
	\label{eq:global_radial_linearity}
	\end{equation}
	\label{as:global_radial_linearity}
\end{assumption}

\begin{assumption}[Boundedness of $H$]
	In \BlockUpdate\ (Algorithm~\ref{alg:BlockUpdate}), 
	the user-specified map $H$ \eqref{eq:reduced_model} is bounded along the iterates of local search,
	i.e., 
	there exists $c_0 \geq 0$ such that for at any iterate $X^k$ generated by a specified algorithm, the following inequality holds for any block $b \in [N]$, 
	\begin{equation}
	\max_{\eta_b \in T_{X^k_b} \Mcal_b, \norm{\eta_b} \leq 1} |\langle \eta_b, H[\eta_b] \rangle| \leq c_0.
	\label{eq:H_bounded}
	\end{equation}
	\label{as:H_bounded}
\end{assumption}

\begin{assumption}[Lower bound on initial trust-region radius]
	In \BlockUpdate\ (Algorithm~\ref{alg:BlockUpdate}), 
	the initial trust-region radius $\Delta_0$ is bounded away from zero by,
	\begin{equation}
	\Delta_0 \geq \lambda_{b} \norm{\rgrad_{X_b} f_b},
	\end{equation}
	where $X_b$ is the input block estimate to Algorithm~\ref{alg:BlockUpdate}
	and $\lambda_b$ is a block-specific constant defined as, 
	\begin{equation}
	\lambda_b \triangleq \frac{1}{8(c_b + c_0)}. 
	\label{eq:lambda_b}
	\end{equation}
	\label{as:trust_region_radius_bound}
\end{assumption}

Assumptions~\ref{as:global_radial_linearity}-\ref{as:trust_region_radius_bound} concern the execution of \BlockUpdate\ (Algorithm~\ref{alg:BlockUpdate}), and are once again fairly lax in practice. 
In particular, the simplest choice of $H$ that satisfies 
radial linearity \eqref{eq:global_radial_linearity} and boundedness \eqref{eq:H_bounded} is the identity mapping. 
In our PGO application, 
we use the Riemannian Hessian $H = \Hess f_b(X_b)$ to leverage local second-order information for faster convergence.
In this case, $H$ is still radially linear, and in Appendix~\ref{sec:convergence_pgo} we show that $H$ is also bounded along the sequence of iterates generated by \RBCD\ or \ARBCD. 
Finally, Assumption~\ref{as:trust_region_radius_bound} can be easily satisfied by using a sufficiently large initial trust-region radius. 
We are now ready to establish an important theoretical result, which states that each block update (Algorithm~\ref{alg:BlockUpdate}) yields \emph{sufficient decrease} on the corresponding reduced cost function.

\begin{lemma}[Sufficient descent property of Algorithm~\ref{alg:BlockUpdate}]
	Under Assumptions~\ref{as:lipschitz_gradient_pullback}-\ref{as:trust_region_radius_bound}, applying \textsc{BlockUpdate} (Algorithm~\ref{alg:BlockUpdate}) on a block $b \in [N]$ with an input value $X_b^k$ decreases the reduced cost by at least,
	\begin{equation}
	f_{b}(X^{k}_{b}) - f_{b}(X^{k+1}_{b}) \geq \frac{1}{4} \lambda_{b} \norm{\rgrad f_{b}(X_{b}^k)}^2,
	\end{equation}
	where $\lambda_{b}$ is the block-specific constant corresponding to the selected block defined in \eqref{eq:lambda_b}.
	\label{lem:sufficient_decrease}
\end{lemma}

Lemma~\ref{lem:sufficient_decrease} states that after each iteration of $\RBCD$, we are guaranteed to decrease the corresponding reduced cost function. 
Furthermore, the amount of reduction is \emph{lower bounded} by some constant times the squared gradient norm at the selected block.  
Thus, if we execute $\RBCD$ for long enough, intuitively we should expect the iterates to converge to a solution where the gradient norm is zero (i.e., a first-order critical point). The following theorem formalizes this result.

\begin{theorem}[Global convergence rate of $\RBCD$]
	Let  $f^\star$ denote the global minimum of the optimization problem
	\eqref{eq:manopt_problem}.
	Denote the iterates of $\RBCD$ (Algorithm~\ref{alg:BCD}) as 
	$X^0, X^1, \hdots, X^{K-1}$, and the corresponding block selected at each iteration as $b_0, \hdots, b_{K-1}$.
	Under Assumptions~\ref{as:lipschitz_gradient_pullback}-\ref{as:trust_region_radius_bound}, $\RBCD$ with \emph{uniform sampling} or \emph{importance sampling}
	have the following guarantees,
	\begin{equation}
	\min_{0 \leq k \leq K-1} \mathbb{E}_{b_{0:k-1}}\norm{\rgrad f(X^k)}^2 \leq \frac{4N (f(X^0) - f^\star)}{K \cdot \min_{b \in [N]}\lambda_b }.
	\label{eq:general_convergence_uniform}
	\end{equation}
	In addition, $\RBCD$ with \emph{greedy selection}
	yields the following \emph{deterministic} guarantee,
	\begin{equation}
	\min_{0 \leq k \leq K-1} \norm{\rgrad f(X^k)}^2 \leq \frac{4N(f(X^0) -
		f^\star)}{K \cdot \min_{b \in [N]}\lambda_b }.
	\label{eq:general_convergence_greedy}
	\end{equation}
	\label{thm:general_convergence}
\end{theorem}

Theorem~\ref{thm:general_convergence} establishes a \emph{global} sublinear convergence rate for $\RBCD$.\footnote{
	We note that in our current analysis, uniform sampling and importance sampling share the same convergence rate estimate. 
	In practice, however, it is usually the case that importance sampling yields much faster empirical convergence (see Section~\ref{sec:experiments}). 
	This result suggests that it is possible to further improve the convergence guarantees for importance sampling. We leave this for future work.}
Specifically, as the number of iterations $K$ increases, the squared gradient norm decreases at the rate of $\mathcal{O}(1/K)$.
Using the same proof technique, we can establish a similar convergence guarantee for the accelerated version $\ARBCD$. 

\begin{theorem}[Global convergence rate of $\ARBCD$]
	Let  $f^\star$ denote the global minimum of the optimization problem
	\eqref{eq:manopt_problem}.
	Denote the iterates of $\ARBCD$ (Algorithm~\ref{alg:ARBCD}) as 
	$X^0, X^1, \hdots, X^{K-1}$, and the corresponding block selected at each iteration as $b_0, \hdots, b_{K-1}$.
	Define the constant, 
	\begin{equation}
	C \triangleq \min\left(c_1, \min_{b \in [N]} \lambda_b/4\right).
	\end{equation}
	Under Assumptions~\ref{as:lipschitz_gradient_pullback}-\ref{as:trust_region_radius_bound}, $\ARBCD$ with \emph{uniform sampling} or \emph{importance sampling}
	have the following guarantees,
	\begin{equation}
	\min_{0 \leq k \leq K-1} \mathbb{E}_{b_{0:k-1}}\norm{\rgrad f(X^k)}^2 \leq {N (f(X^0) - f^\star)}/{CK}.
	\label{eq:accelerated_convergence_uniform}
	\end{equation}
	In addition, $\ARBCD$ with \emph{greedy selection}
	yields the following \emph{deterministic} guarantee,
	\begin{equation}
	\min_{0 \leq k \leq K-1} \norm{\rgrad f(X^k)}^2 \leq {N(f(X^0) - f^\star)}/{CK}.
	\label{eq:accelerated_convergence_greedy}
	\end{equation}
	\label{thm:accelerated_convergence}
\end{theorem}

\begin{remark}[Convergence on Problem~\ref{prob:Rank_restricted_SDP}]
	\label{rem:convergence_pgo}
	So far, we have shown that under mild conditions, \RBCD\ and \ARBCD\ are guaranteed to converge when solving general optimization problems defined on the Cartesian product of smooth matrix submanifolds.
	Recall that in the specific application of PGO, 
	we are using \RBCD\ and \ARBCD\ to solve a sequence of rank-restricted semidefinite relaxations (Problem~\ref{prob:Rank_restricted_SDP}).
	In Appendix~\ref{sec:convergence_pgo}, we show that Assumptions~\ref{as:lipschitz_gradient_pullback}-\ref{as:trust_region_radius_bound} are satisfied in this case, 
	and hence \RBCD\ and \ARBCD\ retain their respective convergence guarantees.
	In particular, we show that although Assumption~\ref{as:lipschitz_gradient_pullback} (Lipschitz-type gradient) and Assumption~\ref{as:H_bounded} (bounded Hessian) do not hold globally (due to the non-compact translation search space), 
	they still hold within the \emph{sublevel set} of the cost function determined by the initial iterate $X^0$.
	Since \RBCD\ and \ARBCD\ are inherently \emph{descent} methods,
	it suffice to restrict our attention to this initial sublevel set as future iterates will not leave this set.
	The reader is referred to Appendix~\ref{sec:convergence_pgo} for 
	more details.
\end{remark}

\begin{remark}[Convergence under parallel executions]
	With the parallel iterations described in Section~\ref{sec:parallel_execution}, 
	we can further improve the global convergence rates of RBCD and $\ARBCD$
	in Theorems~\ref{thm:general_convergence} and
	\ref{thm:accelerated_convergence}.
	For example, the constant $N$ that appears in the rate estimates of uniform sampling and importance sampling \eqref{eq:general_convergence_uniform} can be replaced by the number of aggregate blocks (colors) in the dependency graph.
	Recall that in practice, this is typically bounded by a small constant,
	e.g., the maximum degree in the sparse robot-level dependency graph (see
	Section~\ref{sec:parallel_execution}).
	\label{rem:parallel_convergence}
\end{remark}

\section{Distributed Verification}
\label{sec:verification}
In this section we address the problem of \emph{solution verification} \cite{Carlone2015Verification} in the distributed setting.  Concretely, we propose distributed solution verification and saddle escape algorithms to certify the optimality of a first-order critical point $X$ of the rank-restricted relaxation \eqref{eq:se_sync_riemannian} as a global minimizer $Z = X^\top X$ of problem \eqref{eq:se_sdp}, and for \emph{escaping} from suboptimal stationary points after ascending to the next level of the Riemannian Staircase (Algorithm \ref{alg:riemannian_staircase}).  {To the best of our knowledge, these are the first distributed solution verification algorithms to appear in the literature.}

Our approach is based upon the following simple theorem of the alternative, which is a specialization of  \cite[Theorem 4]{Rosen2020Scalable} to problems \eqref{eq:se_sdp} and \eqref{eq:se_sync_riemannian}:
\begin{theorem}[Solution verification and saddle escape]
\label{thm:verification}
Let $X \in \Manifold(r,n)$ be a first-order critical point of the rank-restricted semidefinite relaxation \eqref{eq:se_sync_riemannian}, and define:
\begin{subequations}
\label{verification_equations}
\begin{equation}
\label{Lagrange_multipliers}
 \Lambda(X) \triangleq \SymBlockDiag_d^{+}(X^\top X \ConLapT),
\end{equation}
\begin{equation}
\label{certificate_matrix}
 S(X) \triangleq \ConLapT - \Lambda(X).
\end{equation}
\end{subequations}
Then exactly one of the following two cases holds:
\begin{enumerate}
 \item [$(a)$] $S(X) \succeq 0$ and $Z = X^\top X$ is a global minimizer of \eqref{eq:se_sdp}.
 \item [$(b)$] There exists $v \in \Real^{(d+1)n}$ such that $v^\top S(X) v <
		 0$, and in that case:
 \begin{equation}
 \label{lifted_suboptimal_critical_point}
  X_{+} \triangleq \begin{bmatrix}
           X \\
           0
          \end{bmatrix} \in \Manifold(r+1, n)
 \end{equation}
is a first-order critical point of \eqref{eq:se_sync_riemannian} attaining the same objective value as $X$, and
\begin{equation}
\label{second_order_descent_direction}
 \dot{X}_{+} \triangleq 
 \begin{bmatrix} 0 \\
 v^\top 
 \end{bmatrix} \in T_{X_{+}}(\Manifold(r+1, n))
\end{equation}
is a second-order direction of descent from $X_{+}$. In particular, taking
$v$ to be the eigenvector corresponding to the smallest eigenvalue of $S(X)$
satisfies the above conditions.
\end{enumerate}
\end{theorem}

\begin{remark}[Interpretation of Theorem \ref{thm:verification}]  Let us provide a bit of intuition for what Theorem \ref{thm:verification} conveys. Part (a)  is simply the standard (necessary and sufficient) conditions for $Z = X^\top X$ to be the solution of the (convex) semidefinite program \eqref{eq:se_sdp} \cite{Vandenberghe1996Semidefinite}.  In the event that these conditions are \emph{not} satisfied (and therefore $Z$ is \emph{not} optimal in \eqref{eq:se_sdp}), there must exist a direction of descent $\dot{Z} \in \Sym^{(d+1)n}$ from $Z$ that is \emph{not} captured in the low-rank factorization \eqref{eq:se_sync_riemannian}, \emph{at least to first order} (since $X$ is  stationary).  This could be because $X$ is a saddle point of the nonconvex problem \eqref{eq:se_sync_riemannian} (in which case there may exist a \emph{second}-order direction of descent from $X$), or because the descent direction $\dot{Z}$ is towards a set of higher-rank matrices than the rank-$r$ factorization  used in \eqref{eq:se_sync_riemannian} is able to capture.  Part (b) of Theorem \ref{thm:verification} provides an approach that enables us to address \emph{both} of these potential obstacles simultaneously, by using a negative eigenvector of the certificate matrix $S(X)$ to construct a \emph{second}-order direction of descent $\dot{X}_{+}$ from $X_{+}$, the lifting of $X$ to the next (higher-rank) ``step'' of the Riemannian Staircase.  Geometrically, this construction is based upon the (easily verified) fact that $S(X)$ is the Hessian of the Lagrangian $\nabla_{X}^2 \mathcal{L}$ of the extrinsic (constrained) form of \eqref{eq:se_sync_riemannian}, and therefore $\langle \dot{X}_{+}, \nabla_X^2 \mathcal{L} \: \dot{X}_{+} \rangle = \langle v, S(X)v \rangle < 0$, so that $\dot{X}_{+}$ is indeed a direction of second-order descent from the lifted stationary point $X_{+}$  \cite{Rosen2020Scalable, Boumal2016NIPS,Journee2010LowRank}.
\label{saddle_escape_remark}
 \end{remark}

 In summary, Theorem \ref{thm:verification} enables us to determine whether a first-order critical point $X$ of \eqref{eq:se_sync_riemannian} corresponds to a minimizer $Z = X^\top X$ of \eqref{eq:se_sdp}, and to \emph{descend} from $X$ if necessary, by computing the minimum eigenpair $(\lambda, v)$ of the certificate matrix $S(X)$ defined in \eqref{verification_equations}.  In the original SE-Sync algorithm, the corresponding minimum-eigenvalue computation is performed by means of spectrally-shifted Lanczos iterations \cite{Rosen19IJRR, Rosen17IROS}; while this works well for the centralized SE-Sync method, adopting the Lanczos algorithm in the distributed setting would require an excessive degree of communication among the agents.  Therefore, in the next subsection, we investigate alternative strategies for computing the minimum eigenpair that are more amenable to a distributed implementation.
  
 \subsection{Distributed Minimum-eigenvalue Computation}
 \label{sec:distributed_minimum_eigenpair_computation}
In this subsection we describe an efficient distributed algorithm for computing the minimum eigenpair of the certificate matrix $S(X)$ required in Theorem \ref{thm:verification}.  We begin with a brief review of eigenvalue methods.

 In general, the Lanczos procedure is the preferred technique for computing a small number of extremal (maximal or minimal) eigenpairs of a symmetric matrix $A \in \Sym^n$ \cite[Chp.\ 9]{Golub1996Matrix}.  In brief, this method proceeds by approximating $A$ using its orthogonal projection $P_k \triangleq V_k^\top A V_k$ onto the \emph{Krylov subspace}:
\begin{equation}
 \label{Krylov_subspace_definition}
 \mathcal{K}(A, x_0, k) \triangleq \Span \left \lbrace x_0, Ax_0, \dotsc, A^{k-1} x_0 \right \rbrace = \image V_k,
\end{equation}
where $x_0 \in \Real^n$ is an initial vector and $V_k \in \Stiefel(k, n)$ is a matrix whose columns (called \emph{Lanczos vectors}) provide an orthonormal basis for $\mathcal{K}(A, x_0, k)$.  Eigenvalues $\theta_i$ of the (low-dimensional) approximation $P_k$, called \emph{Ritz values}, may then be taken as approximations for eigenvalues of $A$.  The $k$-dimensional Krylov subspace $\mathcal{K}(A, x_0, k)$ in \eqref{Krylov_subspace_definition} is iteratively expanded  as the algorithm runs (by computing additional matrix-vector products), thereby providing increasingly better approximations of $A$'s extremal eigenvalues (in accordance with the Courant-Fischer variational characterization of eigenvalues \cite[Thm.\ 8.1.2]{Golub1996Matrix}). The Lanczos procedure thus provides an efficient means of estimating a subset of $A$'s spectrum (particularly its extremal eigenvalues) to high accuracy at the cost of only a relatively small number (compared to $A$'s dimension $n$) of matrix-vector products, especially if the initial vector $x_0$ lies close to an eigenvector of $A$.  In particular, if $\lambda(A) = \lambda_1 > \lambda_2 > \dotsb > \lambda_n$ and $\phi_1$ is the eigenvector associated to $\lambda_1$, it is well-known that the error in the eigenvector estimate $y_1$ from the maximal Ritz pair $(\theta_1, y_1)$ decays asymptotically according to \cite[eq.\ (2.15)]{Saad1980Rates}:
\begin{equation}
\label{Lanczos_convergence_rate}
\sin(\phi_1, y_1) \sim \tau_1^{-k},
\end{equation}
where
\begin{subequations}
\label{Lanczos_convergence_rate_parameters}
\begin{equation}
\tau_1 = \rho_1 + \sqrt{\rho_1^2 - 1},
\end{equation}
\begin{equation}
\rho_1 = 1 + 2\frac{\lambda_1 - \lambda_2}{\lambda_2 - \lambda_n}.
\end{equation}
\end{subequations}

However, while the Lanczos procedure is the method of choice for computing a few extremal eigenpairs in the \emph{centralized} setting, it is unfortunately not well-suited to \emph{distributed} computations when inter-node communication is a bottleneck.  This is because the Lanczos vectors $V_k$ must be periodically re-orthonormalized in order to preserve the accuracy of the estimated eigenpairs $(\theta_i, v_i)$.  While several strategies have been proposed for performing this reorthonormalization, all of them essentially involve computing a QR decomposition of (possibly a subset of columns from) $V_k$ (see \cite[Sec.\ 9.2]{Golub1996Matrix} and the references therein).  Constructing this decomposition in the distributed setting would require frequent synchronized all-to-all message passing, which is impractical when inter-node communication is expensive or unreliable.

We are therefore interested in exploring alternatives to the Lanczos method that require less coordination in the distributed setting.  Many commonly-used eigenvalue methods can  be viewed as attempts to simplify the ``gold-standard'' Lanczos procedure, achieving a reduction in storage and per-iteration computation at the cost of a slower convergence rate.  An extreme example of this is the well-known \emph{power method} \cite[Sec.\ 8.2]{Golub1996Matrix}, which can be viewed as a simplification that retains only the final generator $A^{k-1} x_0$ of the Krylov subspace $\mathcal{K}(A, x_0, k)$ in \eqref{Krylov_subspace_definition} at each iteration.  This leads to a very simple iteration scheme, requiring only matrix-vector products:\footnote{Note that while the power method iteration is commonly written in the normalized form $x_{k+1} = Ax_k / \lVert A x_k \rVert$, normalization is only actually required to compute the Ritz value $\theta_k = x_k^\top A x_k / \lVert x_k \rVert^2$ associated with $x_k$. }
\begin{equation}
\label{power_method_iteration}
 x_{k+1} = A x_k.
\end{equation}
Note that with $A = S(X)$, the matrix-vector product in \eqref{power_method_iteration} can be computed using the same inter-agent communication pattern already employed in each iteration of the RBCD method developed in Section \ref{sec:bcd}, and so is well-suited to a distributed implementation.  However, the power method's simplicity comes at the expense of a reduced convergence rate.  In particular, if $\lambda_1$ and $\lambda_2$ are the two largest-magnitude eigenvalues of $A$ (with $\lvert \lambda_1 \rvert > \lvert \lambda_2 \rvert$), then the vector $y_1$ in the Ritz estimate $(\theta_1, y_1)$ computed via the power method converges to the dominant eigenvector $\phi_1$ of $A$ according to \cite[Thm.\ 8.2.1]{Golub1996Matrix}:
\begin{equation}
\label{power_method_convergence_rate}
\sin(\phi_1, y_1) \sim \left\lvert \frac{\lambda_1}{\lambda_2} \right \rvert^{-k}.
\end{equation}

Let us compare the rates \eqref{Lanczos_convergence_rate} and \eqref{power_method_convergence_rate} for the case in which the eigengap $\gamma \triangleq \lambda_1 - \lambda_2$ is small relative to the diameter $D$ of the spectrum of $A$ (which is of the same order order as $\lambda_2 - \lambda_n$ and $\lambda_2$ for the Lanczos and power methods, respectively); intuitively, this is the regime in which the problem is hard, since it is difficult to distinguish $\lambda_1$ and $\lambda_2$ (and consequently their associated eigenspaces).  In particular, let us compute the number of iterations $k(\epsilon)$ required to reduce the angular error in the dominant Ritz vector $y_1$ by the factor $\epsilon \in (0,1)$.  For the Lanczos method, using \eqref{Lanczos_convergence_rate} and \eqref{Lanczos_convergence_rate_parameters} and assuming $\gamma \ll D$, we estimate:
\begin{equation}
k_l(\epsilon) = - \frac{\log \epsilon}{\log \tau_1} \approx -\frac{\log \epsilon}{\log \left(1 + 2\frac{\gamma}{D} + \sqrt{4 \frac{\gamma}{D} + 4 \frac{\gamma^2}{D^2} } \right)} \approx -\frac{\log \epsilon}{\log \left(1 + 2 \sqrt{\frac{\gamma}{D}} \right)} \approx -\frac{\log \epsilon}{2 \sqrt{\frac{\gamma}{D}}}.
\end{equation}
Similarly, using \eqref{power_method_convergence_rate}, the analogous estimate for the power method is:
\begin{equation}
k_p(\epsilon) = - \frac{\log \epsilon}{\log \left(\frac{\lambda_1}{\lambda_2}\right)} =  -\frac{\log \epsilon}{\log \left(1 + \frac{\gamma}{\lambda_2}\right)} \approx - \frac{\log \epsilon}{\gamma / D}.
\end{equation}
In our target application (certifying the optimality of a first order-critical
point $X$), the minimum eigenvalue we must compute will \emph{always} belong to
a tight cluster whenever $X$ is a global minimizer,\footnote{This is an
immediate consequence of the (extrinsic) first-order criticality condition for
\eqref{eq:se_sync_riemannian}, which requires $S(X)X^\top = 0$, i.e., that
\emph{each row} of $X$ be an eigenvector of $S(X)$ with eigenvalue $0$
\cite[Sec.\ III-C]{Rosen17IROS}.} so the power method's $\Ocal(1/\gamma)$
dependence upon the eigengap $\gamma$ translates to a substantial reduction in
performance versus the Lanczos method's $\Ocal(1/\sqrt{\gamma})$ rate.

In light of these considerations, we propose to adopt the recently-developed \emph{accelerated power method} \cite{deSa2018Accelerated} as our distributed eigenvalue algorithm of choice.  In brief, this method modifies the standard power iteration \eqref{power_method_iteration} by adding a Polyak momentum term, producing the iteration:
\begin{equation}
\label{accelerated_power_iteration}
x_{k+1} = Ax_k - \beta x_{k-1},
\end{equation}
where $\beta \in \Real_{+}$ is a \emph{fixed} constant.  We note that because
$\beta$ is constant, the iteration \eqref{accelerated_power_iteration} has the
same communication pattern as the standard power method
\eqref{power_method_iteration}, and so is well-suited to implementation in the
distributed setting.  Furthermore, despite the simplicity of the modification
\eqref{accelerated_power_iteration} versus \eqref{power_method_iteration}, the
addition of momentum  actually allows the accelerated power method to
\emph{match} the $\Ocal(1/\sqrt{\gamma})$ dependence of the Lanczos method on the dominant eigengap for a well-chosen parameter $\beta$.  More precisely, we have the following result:

\begin{theorem}[Theorem 8 of \cite{deSa2018Accelerated}]
\label{accelerated_power_method_convergence_theorem}
Let $A \in \PSD^n$  with eigenvalues $\lambda_1 > \lambda_2 \ge \dotsb \ge \lambda_n \ge 0$, and $\phi_1$ be the eigenvector associated with the maximum eigenvalue $\lambda_1$.   Given $\beta \in \Real_{+}$ satifying $\beta < \lambda_1^2 / 4$ and an initial Ritz vector $y_0 \in \Real^n$, after $k$ accelerated power iterations \eqref{accelerated_power_iteration} the angular error in the Ritz vector $y_k$ satisfies:
\begin{equation}
\label{accelerated_power_iteration_error_bound}
 \sin(\phi_1, y_k) \le \frac{\sqrt{1 - (\phi_1^\top y_0)^2}}{\lvert \phi_1^\top y_0 \rvert}  \cdot 
 \begin{cases}
  2 \left( \frac{2\sqrt{\beta}}{\lambda_1 + \sqrt{\lambda_1^2 - 4 \beta}}
  \right)^{k} & \beta > {\lambda_2^2}/{4} \\[0.5cm]
  \left(\frac{\lambda_2 + \sqrt{\lambda_1^2 - 4\beta}}{\lambda_1 +
  \sqrt{\lambda_1^2 - 4\beta}}\right)^{k} & \beta \le {\lambda_2^2}/{4}
  \end{cases}
  .
\end{equation}
\end{theorem}
\begin{remark}[Selection of $\beta$]
\label{rem:select_beta}
While the hypotheses of Theorem \ref{accelerated_power_method_convergence_theorem} require that $\beta < \lambda_1^2 / 4$, note that lower bounds on $\lambda_1$ (which provide admissible values of $\beta$) are easy to obtain; indeed, the Courant-Fischer theorem \cite[Thm. 8.1.2]{Golub1996Matrix} implies that $\lambda_1 \ge y^\top A y$ for \emph{any} unit vector $y \in \Real^n$.  We also observe that the bound on the right-hand side of \eqref{accelerated_power_iteration_error_bound} is an \emph{increasing} function of $\beta$ for $\beta > \lambda_2^2 / 4$ and a \emph{decreasing} function for $\beta \le \lambda_2^2/4$, with
\begin{equation}
		\lim_{\beta \to \left({\lambda_2^2}/{4}\right)^{+}}
		\frac{2\sqrt{\beta}}{\lambda_1 + \sqrt{\lambda_1^2 - 4\beta}} =
		\frac{\lambda_2}{\lambda_1 + \sqrt{\lambda_1^2 - \lambda_2^2}} <
		\frac{\lambda_2 + \sqrt{\lambda_1^2 - \lambda_2^2}}{\lambda_1 +
		\sqrt{\lambda_1^2 - \lambda_2^2}} = \lim_{\beta \to \left({\lambda_2^2}/{4}\right)^{-}}  \frac{\lambda_2 + \sqrt{\lambda_1^2 - 4\beta}}{\lambda_1 + \sqrt{\lambda_1^2 - 4 \beta}};
\end{equation}
consequently, the optimal rate is achieved in the limit $\beta \rightarrow (\lambda_2^2 / 4)^{+}$.
\end{remark}

\begin{algorithm}[t]
	\caption{\textsc{\small Minimum eigenpair (MinEig)} }
	\label{alg:MinEig}
	\edit{}{
	\begin{algorithmic}[1]
		\renewcommand{\algorithmicrequire}{\textbf{Input:}}
		\renewcommand{\algorithmicensure}{\textbf{Output:}}
		\Require Certificate matrix $S = S(X)$ from \eqref{certificate_matrix}.
		\Ensure Minimum eigenpair $(\lambda_{\textnormal{min}}, v_{\textnormal{min}})$ of $S$.  
		\vspace{0.1cm}
		\State Compute dominant (maximum-magnitude) eigenpair $(\lambda_{\textnormal{dom}}, v_{\textnormal{dom}})$ of $S$ using power iteration \eqref{power_method_iteration}.  \label{Min_eig_alg_power_method_iteration}
		\If{$\lambda_{\textnormal{dom}} < 0$}
		\State \Return $(\lambda_{\textnormal{dom}}, v_{\textnormal{dom}})$
		\EndIf
		\State Compute maximum eigenpair $(\theta, v)$ of $C \triangleq \lambda_{\textnormal{dom}}I - S$ using accelerated power iteration \eqref{accelerated_power_iteration}.  \label{Min_eig_accelerated_power_method_step}
		\State \Return $(\lambda_{\textnormal{dom}} - \theta, v)$
	\end{algorithmic}
	}
\end{algorithm}

Combining the power and accelerated power methods with the spectral shifting strategy proposed in \cite[Sec.\ III-C]{Rosen17IROS} produces our distributed minimum-eigenvalue method (Algorithm \ref{alg:MinEig}).  In brief, the main idea is to construct a spectrally-shifted version $C$ of the certificate matrix $S(X)$ such that (i) a maximum eigenvector $v$ of $C$ coincides with a \emph{minimum} eigenvector of $S$, and (ii) $C \succeq 0$, so that we can recover the maximum eigenpair $(\theta, v)$ of $C$ using  accelerated power iterations \eqref{accelerated_power_iteration}.  Algorithm \ref{alg:MinEig} accomplishes this by first applying the (basic) power method \eqref{power_method_iteration} to estimate the dominant eigenpair $(\lambda_{\textnormal{dom}}, v_{\textnormal{dom}})$ of $S$ in line \ref{Min_eig_alg_power_method_iteration} (which does \emph{not} require $S \succeq 0$), and then applying the accelerated power method to compute the maximum eigenpair ($\theta, v)$ of $C = \lambda_{\textnormal{dom}}I - S \succeq 0$ in line \ref{Min_eig_accelerated_power_method_step}.  Note that while the minimum eigenvalue of $S(X)$ belongs to a tight cluster whenever $X$ is optimal for \eqref{eq:se_sync_riemannian} (necessitating our use of \emph{accelerated} power iterations in line \ref{Min_eig_accelerated_power_method_step}), the \emph{dominant} eigenvalue of $S$ is typically well-separated, and therefore can be computed to high precision using only a small number of power iterations in line \ref{Min_eig_alg_power_method_iteration}.  

\begin{remark}[Communication requirements of Algorithm~\ref{alg:MinEig}]
	{The bulk of the work in Algorithm~\ref{alg:MinEig} lies in updating the eigenvector estimate via the matrix-vector products 
	\eqref{power_method_iteration} and \eqref{accelerated_power_iteration}.}
	In the distributed regime, these can be implemented by having each robot {estimate the block of the eigenvector that corresponds to its own poses.}
	The communication pattern of this process is determined by the sparsity structure of the underlying matrix, which for our application are the dual certificate $S$ and its spectrally-shifted version $C$. 
	Fortunately, both $S$ and $C$ inherit the sparsity of the connection Laplacian $\ConLapT$.
	{This means that at each iteration of \eqref{power_method_iteration} or \eqref{accelerated_power_iteration}, each robot only needs to communicate with its neighbors in the global pose graph.} 
	Therefore, Algorithm \ref{alg:MinEig} provides an efficient way (in terms of \emph{both} computation \emph{and} communication) to compute a minimum eigenpair of $S$ in the distributed setting.
\end{remark}

\subsection{Descent from Suboptimal Critical Points}
In this subsection we describe a simple procedure for \emph{descending} from a first-order critical point $X \in \Manifold(r,n)$ of Problem \ref{prob:Rank_restricted_SDP} and restarting local optimization in the event that $\ZT = X^\top X$ is \emph{not} a minimizer of Problem \ref{prob:SDP_relaxation_for_PGO} (as determined by $\lambda < 0$, where $(\lambda, v)$ is the minimum eigenpair of the certificate matrix $S(X)$ in Theorem \ref{thm:verification}).  

In this setting, Theorem \ref{thm:verification}(b) shows how to use the minimum
eigenvector $v$ of $S(X)$ to construct a \emph{second}-order direction of
descent $\dot{X}_{+}$ from the lifting $X_{+}$ of $X$ to the next level of the
Riemannian Staircase.  Therefore, we can descend from $X_{+}$ by performing a
simple backtracking line-search along $\dot{X}_{+}$; we summarize this procedure
as Algorithm \ref{alg:EscapeSaddle}.  Note that since $\rgrad f(X_{+}) = 0$ and
$\langle \dot{X}_{+}, \Hess f(X_{+})[\dot{X}_{+}] \rangle < 0$ by Theorem
\ref{thm:verification}(b), letting $X(\alpha) \triangleq \Retr_{X_{+}}(\alpha
\dot{X}_{+})$, there exists a stepsize $\delta > 0$ such that $f(X(\alpha)) <
f(X_{+})$ and $\lVert \rgrad f(X(\alpha)) \rVert > 0$ for all $0 < \alpha <
\delta$, and therefore the loop in line
\ref{backtracking_line_search_termination_criteria} is guaranteed to terminate
after finitely many iterations.  Algorithm \ref{alg:EscapeSaddle} is thus
well-defined.  Moreover, since $\alpha$ decreases at an exponential rate (line
\ref{backtracking_line_search_stepsize_update}), in practice only a handful of
iterations are typically required to identify an acceptable stepsize.
Therefore, even though Algorithm \ref{alg:EscapeSaddle} requires coordination
among all of the agents (to evaluate the objective $f(X(\alpha))$ and gradient
norm $\lVert \rgrad f(X(\alpha)) \rVert$ each trial point $X(\alpha)$, and to
distribute the trial stepsize $\alpha$), it requires a sufficiently small number
of (very lightweight) globally-synchronized messages to remain tractable in the
distributed setting.  Finally, since the point returned by Algorithm
\ref{alg:EscapeSaddle} has nonzero gradient, it provides a nonstationary
initialization for local search at the next level $r+1$ of the Riemannian
Staircase (Algorithm~\ref{alg:riemannian_staircase}), thereby enabling us to continue the search for a low-rank factor in Problem \ref{prob:se_sync_riemannian} corresponding to a \emph{global} minimizer of the SDP relaxation Problem \ref{prob:SDP_relaxation_for_PGO}.

\begin{algorithm}[t]
	\caption{\textsc{\small Descent from a suboptimal critical point $X_{+}$ (EscapeSaddle)} }
	\label{alg:EscapeSaddle}
	\edit{}{
	\begin{algorithmic}[1]
		\renewcommand{\algorithmicrequire}{\textbf{Input:}}
		\renewcommand{\algorithmicensure}{\textbf{Output:}}
		\Require 
		\Statex - Lifted suboptimal critical point $X_{+}$ as defined in \eqref{lifted_suboptimal_critical_point}.
		\Statex - Second-order descent direction $\dot{X}_{+}$ as defined in \eqref{second_order_descent_direction}.
		\Ensure Feasible point $X \in \Manifold(r+1, n)$ satisfying $f(X) < f(X_{+})$, $\lVert \rgrad f(X) \rVert > 0$.
		\vspace{0.1cm}
		\State Set initial stepsize: $\alpha = 1$.
		\State Set initial trial point: $X \leftarrow \Retr_{X_{+}}(\alpha \dot{X}_{+})$
		\While{ $f(X) \ge f(X_{+})$ or $\lVert \rgrad f(X) \rVert = 0$}  \label{backtracking_line_search_termination_criteria}
		\State Halve steplength: $\alpha \leftarrow \alpha / 2$.  \label{backtracking_line_search_stepsize_update}
		\State Update trial point: $X \leftarrow \Retr_{X_{+}}(\alpha \dot{X}_{+})$.  \label{trial_point_computation_in_saddle_escape}
		\EndWhile
		\State \Return $X$.
	\end{algorithmic}
	}
\end{algorithm}

\section{Distributed Initialization and Rounding}
\label{sec:initialization_and_rounding}

\subsection{Distributed Initialization}
\label{sec:initialization}

A distinguishing feature of our approach versus prior distributed PGO methods is that it enables the direct computation of \emph{globally} optimal solutions of the PGO problem \eqref{eq:se_sync} via (convex) semidefinite programming, and therefore does \emph{not} depend upon a high-quality initialization in order to recover a good solution.  Nevertheless, it can still benefit (in terms of reduced computation time) from being supplied with a high-quality initial estimate whenever one is available.
%

Arguably the simplest method of constructing such an initial estimate is
\emph{spanning tree initialization} \cite{konolige2010sparse}.
As the name suggests, we compute the initial pose estimates by propagating the noisy pairwise measurements along an arbitrary spanning tree of the global pose graph.
In the distributed scenario, this technique incurs minimal computation and communication costs, as robots only need to exchange few public poses with their neighbors.

While efficient, the spanning tree initialization is heavily influenced by the noise of selected edges in the pose graph.
A more resilient but also more heavyweight method is \emph{chordal initialization}, originally developed to initialize rotation synchronization \cite{Martinec2007Robust,Carlone2015Initialization}. 
With this technique, one first relaxes rotation synchronization into a linear least squares problem, and subsequently projects the solution back to the rotation group.
For distributed computation, \citet{Choudhary17IJRR} propose to solve the 
resulting linear least squares problem via distributed iterative techniques such as the Jacobi and Gauss-Seidel methods \citep{bertsekas1989parallel}.
One detail to note is that in \citep{Choudhary17IJRR}, the translation estimates are not explicitly initialized but are instead directly optimized during a single distributed Gauss-Newton iteration. 
However, we find that this approach leads to poor convergence for \cite{Choudhary17IJRR} on some real-world datasets. 
To resolve this, in this work we also explicitly initialize the translations by fixing the rotation initialization and using distributed Gauss-Seidel to solve the reduced linear least squares over translations.
On the other hand, to prevent significant communication usage in the initialization stage, we limit the number of Gauss-Seidel iterations to 50 for both rotation and translation initialization. 

Lastly, we note that both initialization approaches return an initial solution $T \in \SE(d)^n$ on the original search space of PGO. Nevertheless, recall from Algorithm~\ref{alg:riemannian_staircase} that the Riemannian Staircase requires a initial point $X \in \Manifold(r_0, n)$ on the search space of the rank-restricted SDP (Problem~\ref{prob:se_sync_riemannian}). We thus need a mechanism that lifts the initial solution from $\SE(d)^n$ to the higher dimensional $\Manifold(r_0, n)$. 
This can be achieved by 
sampling a random point $Y_\text{rand} \in \Stiefel(d,r)$, and subsequently setting $X = Y_\text{rand} T$.

\subsection{Distributed Rounding}
\label{sec:rounding}

After solving the SDP relaxation, we need to 
``round'' the low-rank factor $X^\star \in \Manifold(r,n)$ returned by the Riemannian Staircase to a 
feasible solution to the original PGO problem (see line~\ref{alg:rounding} in Algorithm~\ref{alg:dpgo}). 
In this section, we describe a distributed rounding procedure that incurs minimal computation and communication costs, and furthermore is guaranteed to return a global minimizer to the original PGO (Problem~\ref{prob:se_sync}) provided that the SDP relaxation is \emph{exact}.

Given the output $X^\star$ from the Riemannian Staircase, consider its individual components that correspond to the ``lifted'' rotation and translation variables,
\begin{equation}
	X^\star = \begin{bmatrix}
	Y_1^\star & p_1^\star & \hdots & Y_n^\star & p_n^\star
	\end{bmatrix} \in (\Stiefel(d,r) \times \Real^r)^n.
\end{equation}
In Theorem~\ref{thm:tightness_informal}, we have shown that if the SDP relaxation is exact, then the first block-row of the corresponding SDP solution, which can be written as 
$T^\star \triangleq (Y^\star_1)^\top X^\star$, gives a global minimizer to PGO (Problem~\ref{prob:se_sync}). 
Looking at the rotation and translation of each pose in $T^\star$ separately, 
\begin{equation}
	R_i^\star = (Y_1^\star)^\top Y_i, \; t_i^\star = (Y_1^\star)^\top p_i.
	\label{eq:rounding_exact}
\end{equation}
Equation \eqref{eq:rounding_exact} thus recovers globally optimal rotation and translation estimates.
If the SDP relaxation is not exact, the $R_i^\star$ as computed in \eqref{eq:rounding_exact} may not be a valid rotation. To ensure feasibility, we additionally project it to $\SOd(d)$, 
\begin{equation}
	R_i = \proj_{\SOd(d)} ({Y^\star_1}^\top Y^\star_i).
	\label{eq:rounding_approximate}
\end{equation}
In \eqref{eq:rounding_approximate}, the projection can be carried out by computing the SVD. 

\begin{remark}[Communication requirements of distributed initialization and rounding]
	The distributed chordal initialization \citep{Choudhary17IJRR}
	has a similar communication pattern as distributed local search, where at each iteration robots exchange messages corresponding to their public poses. 
	On the other hand, distributed rounding incurs minimal communication, since agents only need to relay the small $r$-by-$d$ matrix $Y_1^\star$ over the network.
\end{remark}

\section{Experiments}
\label{sec:experiments}
We perform extensive evaluations of the proposed \AlgName\ algorithm on both simulations and benchmark CSLAM datatsets. 
Our simulation consists of multiple robots moving next to each other in a 3D grid with lawn mower trajectories. 
With a given probability (default $0.3$), 
loop closures are added to connect neighboring poses.
For all relative measurements, 
we simulate isotropic Langevin rotation noise according to \eqref{eq:rotation_noise_model} with mode $I_d$ and concentration parameter $\kappa$. To make the process of setting $\kappa$ more intuitive, we first set a desired standard deviation $\sigma_R$ for the \emph{rotation angle} of the rotational noise, and then use the asymptotic approximation shown in SE-Sync \cite[Appendix~A]{Rosen19IJRR} to compute the corresponding concentration parameter $\kappa$.
We also simulate Gaussian translation noise according to \eqref{eq:translation_noise_model} with zero mean and standard deviation $\sigma_t$.
The default noise parameters are $\sigma_R = 3^\circ, \sigma_t = 0.05m$.
See Figure~\ref{fig:example_simulation} for an example simulation together with the certified global minimizer found by \AlgName\ (Algorithm~\ref{alg:dpgo}).
All implementations are written in MATLAB.
All experiments are carried out on a laptop with an Intel i7 CPU and 8 GB RAM.

\begin{figure}[t]
	\centering
	\begin{subfigure}[t]{0.33\textwidth}
		\centering
		\includegraphics[width=\textwidth]{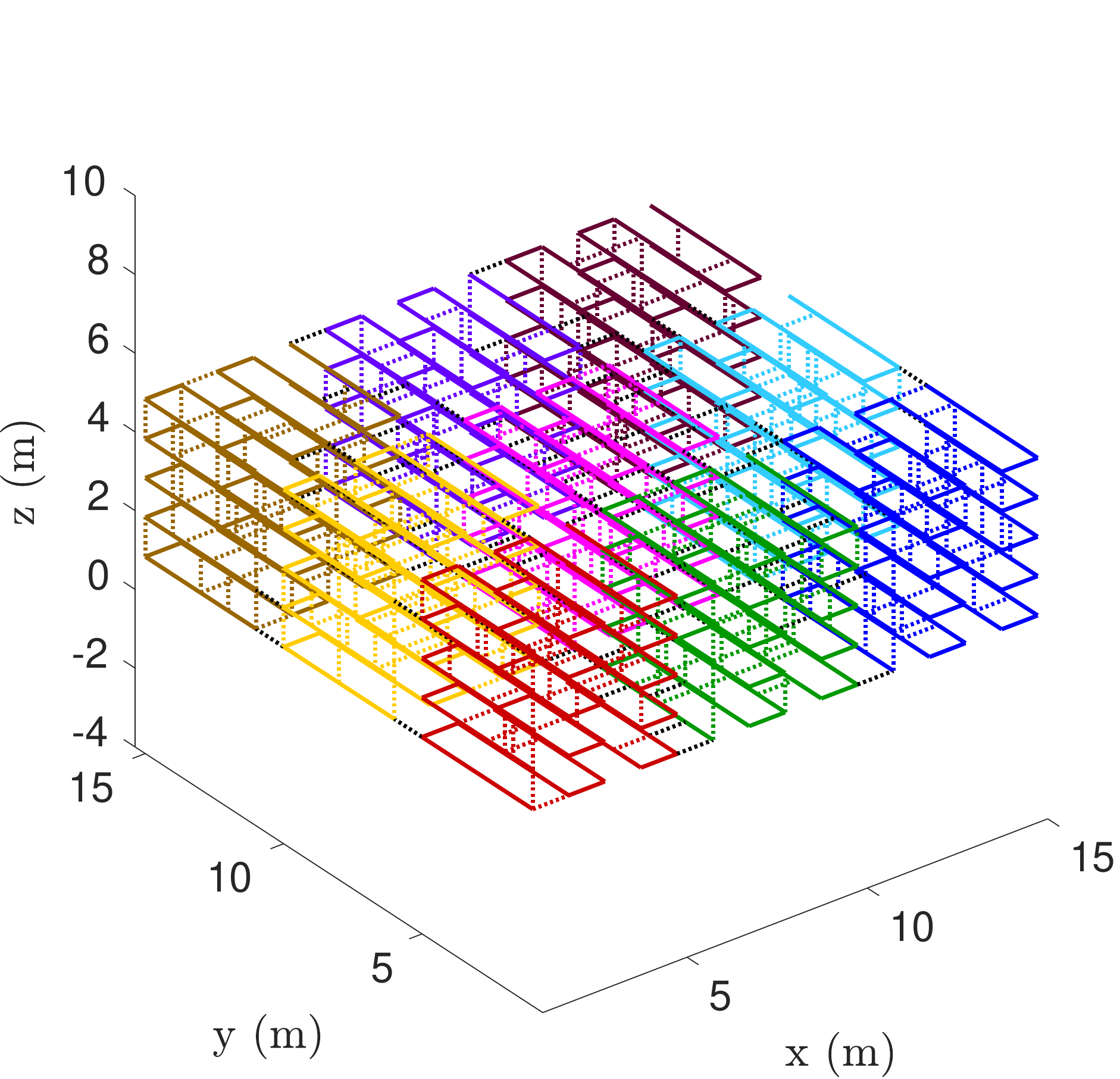}
		\caption{\small Ground truth}
	\end{subfigure}
	\hspace{1cm}
	\begin{subfigure}[t]{0.33\textwidth}
		\centering
		\includegraphics[width=\textwidth]{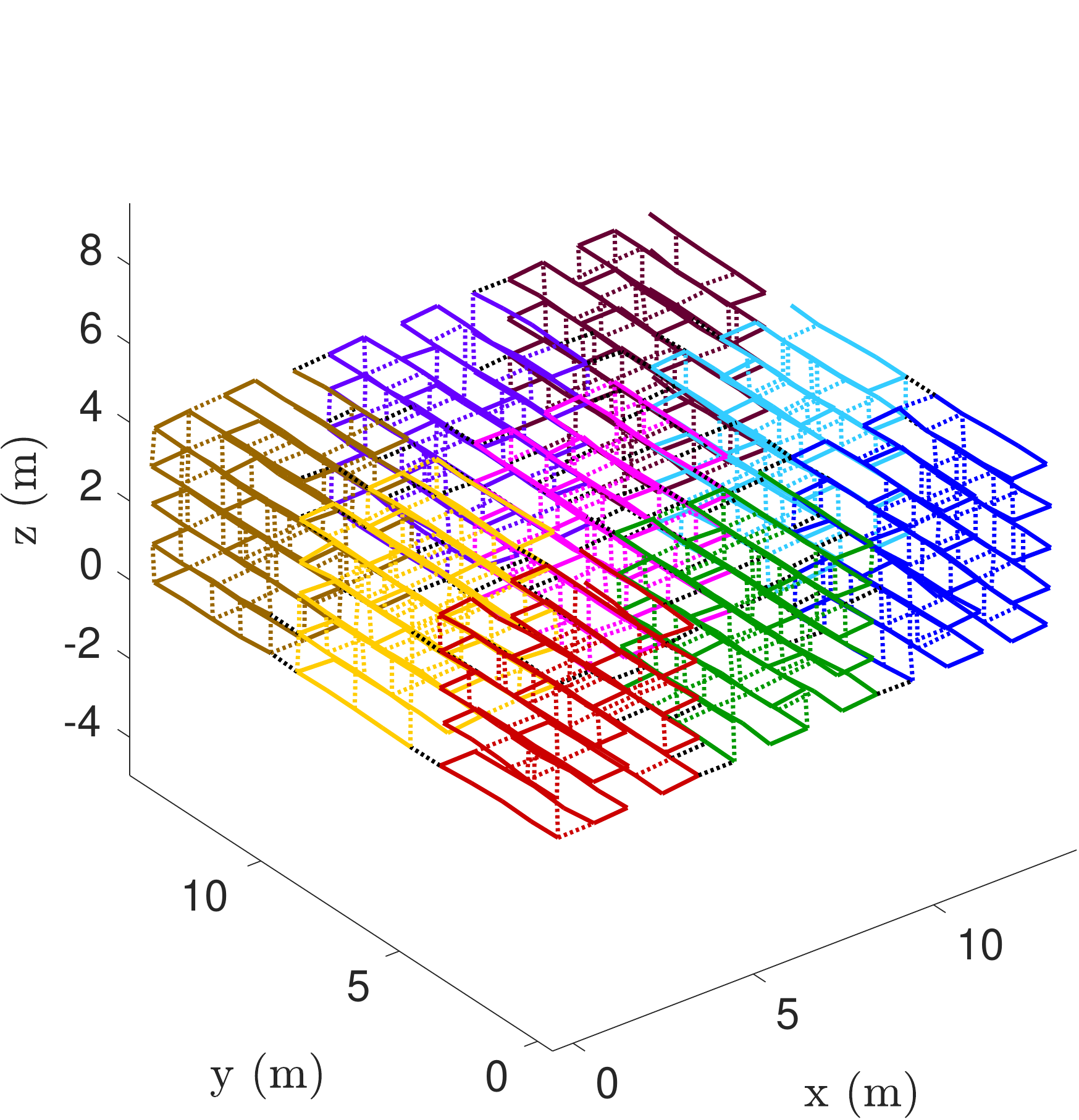}
		\caption{\small Certified solution from Algorithm~\ref{alg:dpgo}}
	\end{subfigure}
	\caption{\small  
		Example simulation consisting of 9 robots (trajectories shown in different colors) where each robot has $125$ poses. 
		Loop closures are drawn as dotted lines.
		(a) Ground truth. (b) Certified global minimizer returned by \AlgName\ (Algorithm~\ref{alg:dpgo}).
	}
	\label{fig:example_simulation}
\end{figure} 

When evaluating our approach and baseline methods, we use the following performance metrics. First, we compute the optimality gap $f - f^\star_\text{SDP}$, where $f^\star_\text{SDP}$ is the optimal value of the centralized SDP relaxation computed using {SE-Sync \cite{Rosen19IJRR}}.
In the (typical) case that the SDP relaxation is exact, the rounded PGO estimates returned by both our approach and SE-Sync will achieve a zero optimality gap (to within numerical tolerances).
Additionally, when evaluating convergence rates of local search methods, we compute the evolution of the Riemannian gradient norm, which quantifies how fast the iterates are converging to a first-order critical point.
Lastly, in some experiments, we also compute the translational and rotational {root mean square errors (RMSE)} with respect to the solution of SE-Sync. 
Given two sets of translations $t, t' \in \Real^{d \times n}$, the translational RMSE is based on the standard $\ell_2$ distance after aligning $t,t'$ in the same frame.
Given two sets of rotations $R, R' \in \SOd(d)^n $, we define the rotational RMSE to be,
\begin{equation}
	\text{RMSE}(R, R') = \sqrt{d_S(R,R')^2/n},
	\label{eq:rotation_RMSE_def}
\end{equation}
where $d_S(R,R')$ is the \emph{orbit distance} for $\SOd(d)^n$ defined in SE-Sync \cite[Appendix C.1]{Rosen19IJRR}. 
Although we measure rotational and translational errors separately, both still give meaningful metrics of the overall PGO solution. 
This is especially true for rotational RMSE, since we know that given rotation estimates, the corresponding optimal translations can be computed in closed-form \cite{Rosen19IJRR}.

The rest of this section is organized as follows. 
In Section~\ref{sec:local_search_experiments}, we evaluate our distributed local search methods, specifically 
$\RBCD$ and its accelerated version $\ARBCD$, when solving the rank-restricted relaxations of PGO. 
Then, in Section~\ref{sec:verification_experiments}, we evaluate the proposed distributed verification scheme.
Lastly, in Section~\ref{sec:full_experiments}, we evaluate our complete distributed certifiably correct PGO algorithm (Algorithm~\ref{alg:dpgo}).

\subsection{Evaluations of Distributed Local Search}
\label{sec:local_search_experiments}

We first evaluate the performance of the proposed $\RBCD$ and $\ARBCD$ algorithms when solving the rank-restricted relaxations (Problem~\ref{prob:se_sync_riemannian}). 
Recall that this serves as the central local search step in our overall Riemannian Staircase algorithm.
By default, we set the relaxation rank to $r = 5$, and enable parallel execution as described in Section~\ref{sec:parallel_execution}.

\begin{figure}[t]
	\centering
	\begin{subfigure}[t]{0.23\textwidth}
		\centering
		\includegraphics[trim=5 10 40 20, clip, width=\textwidth]
		{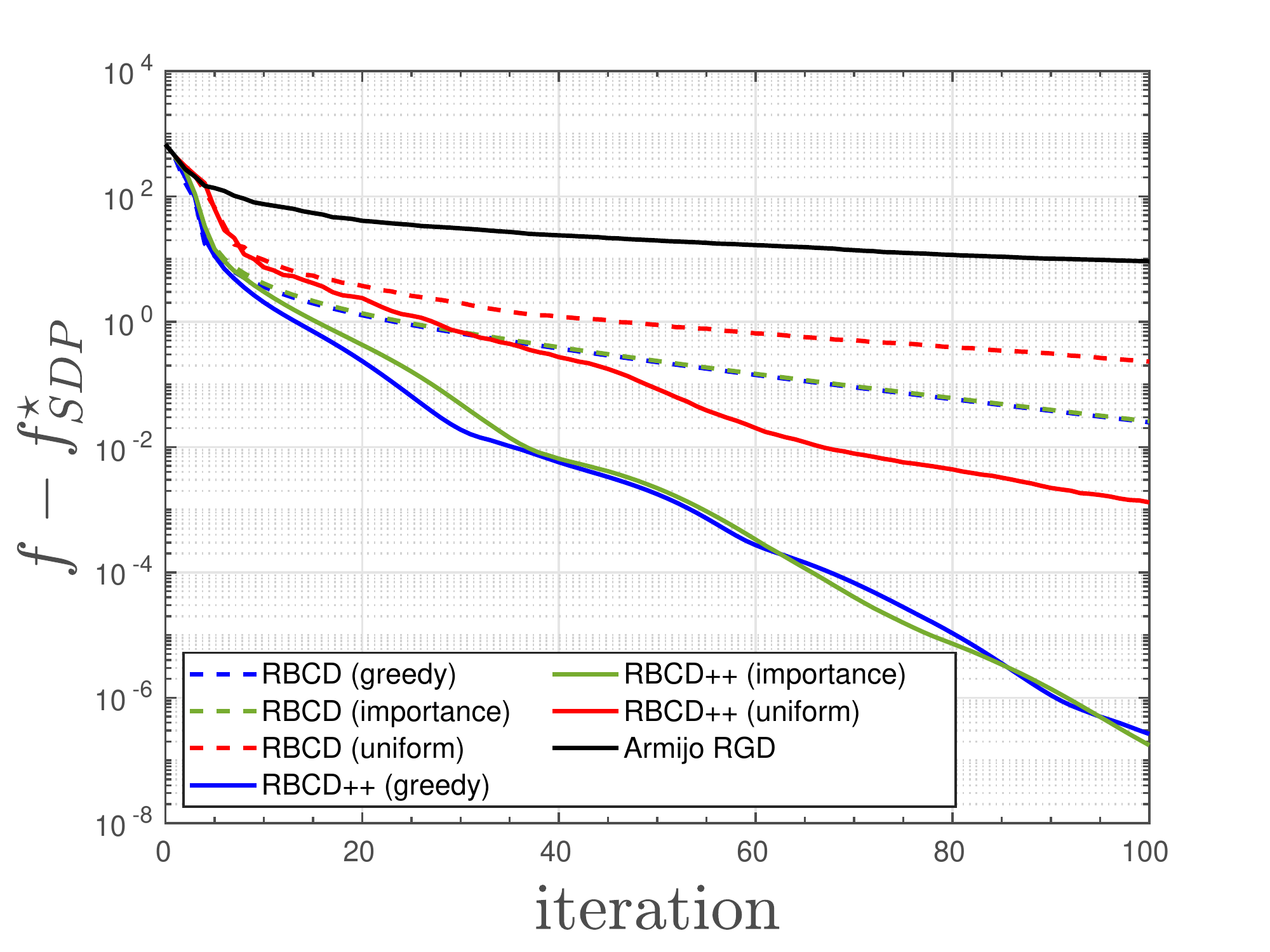}
		\caption{\small Average optimality gap}
		\label{fig:local_search_optgap}
	\end{subfigure}
	~
	\begin{subfigure}[t]{0.23\textwidth}
		\centering
		\includegraphics[trim=5 10 40 20, clip, width=\textwidth]
		{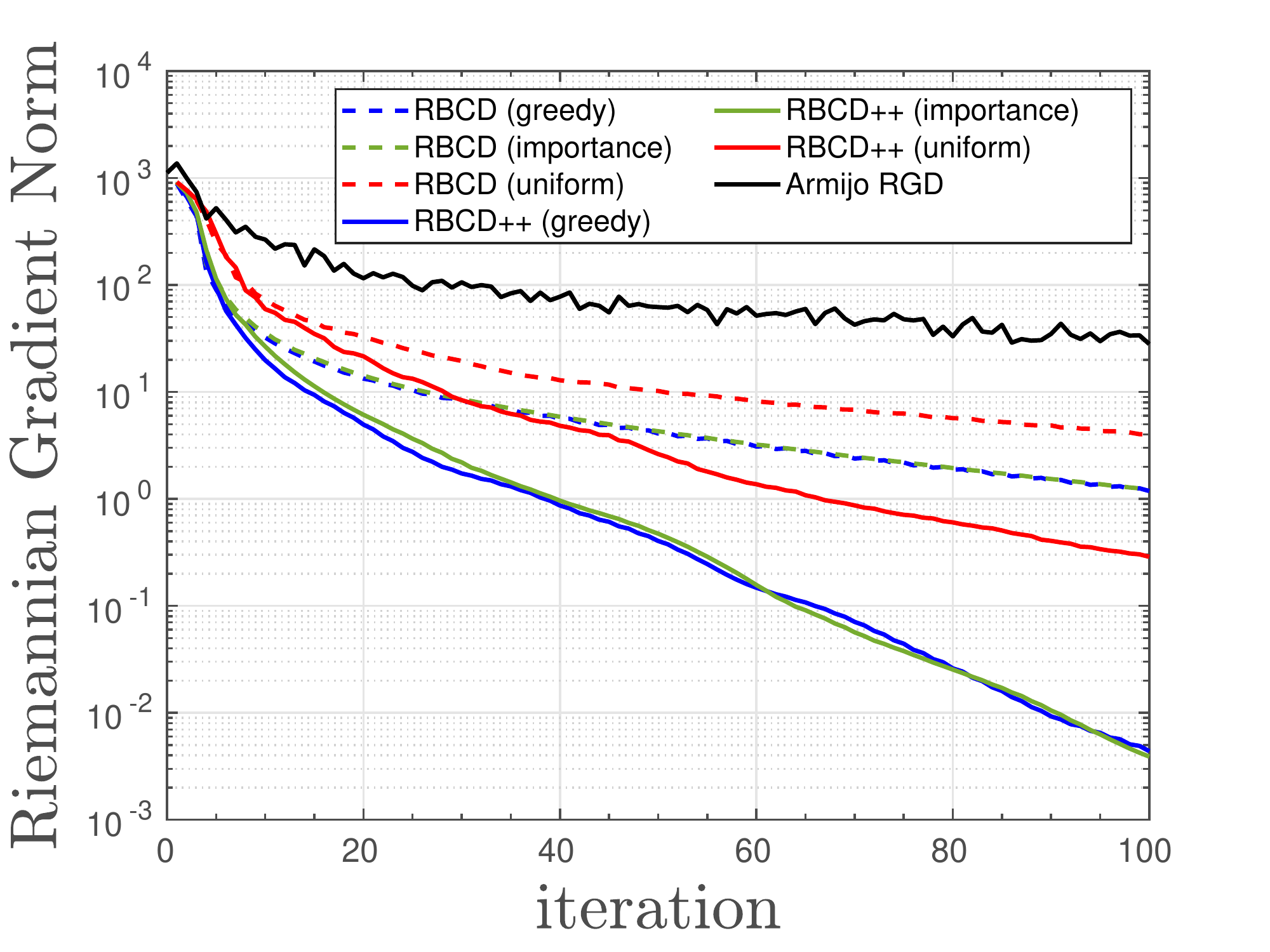}
		\caption{\small Average gradient norm}
		\label{fig:local_search_gradnorm}
	\end{subfigure}
	~
	\begin{subfigure}[t]{0.23\textwidth}
		\centering
		\includegraphics[width=\textwidth]{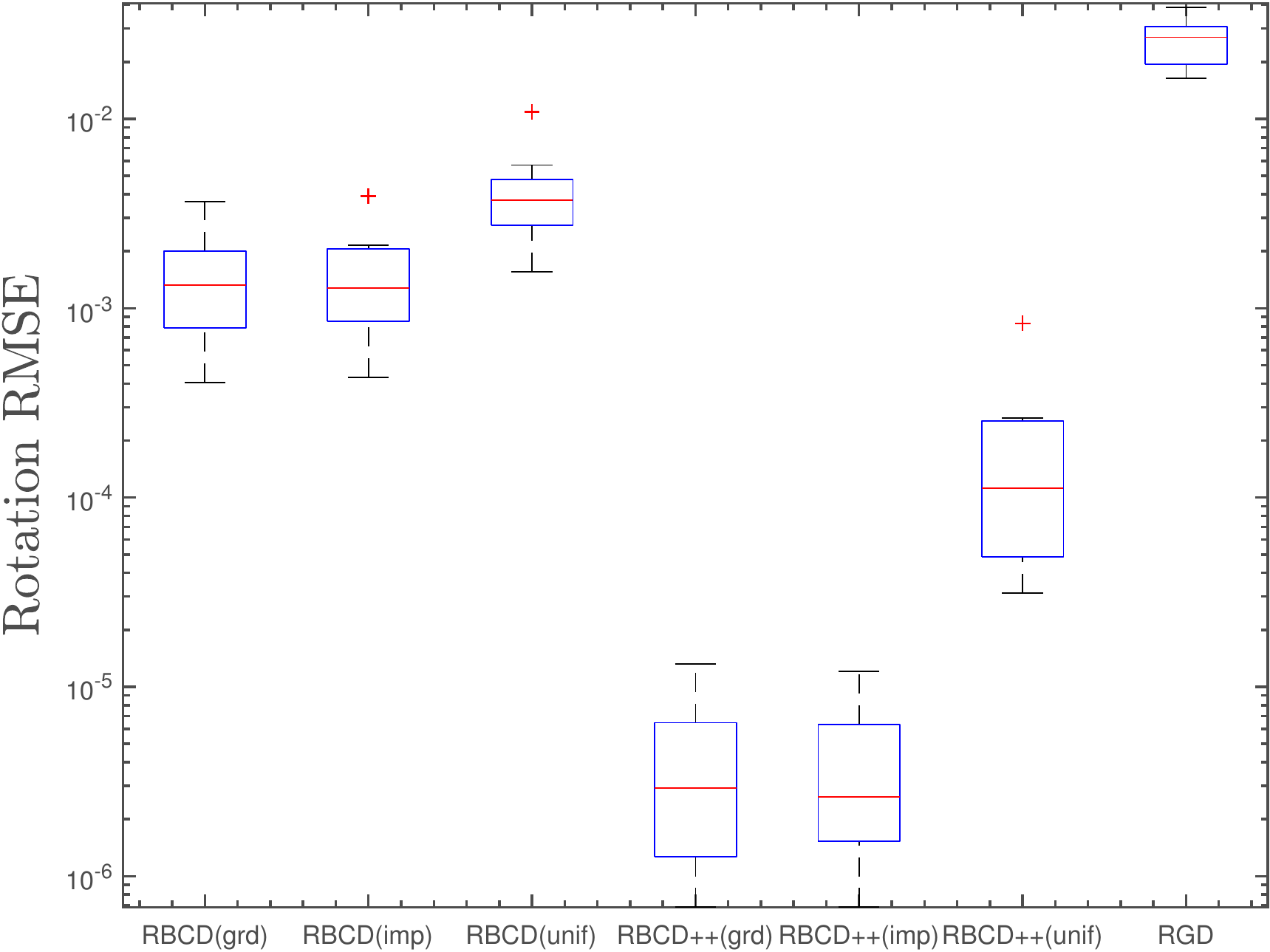}
		\caption{\small rotation RMSE}
		\label{fig:local_search_rotRMSE}
	\end{subfigure}
	~
	\begin{subfigure}[t]{0.24\textwidth}
		\centering
		\includegraphics[width=\textwidth]{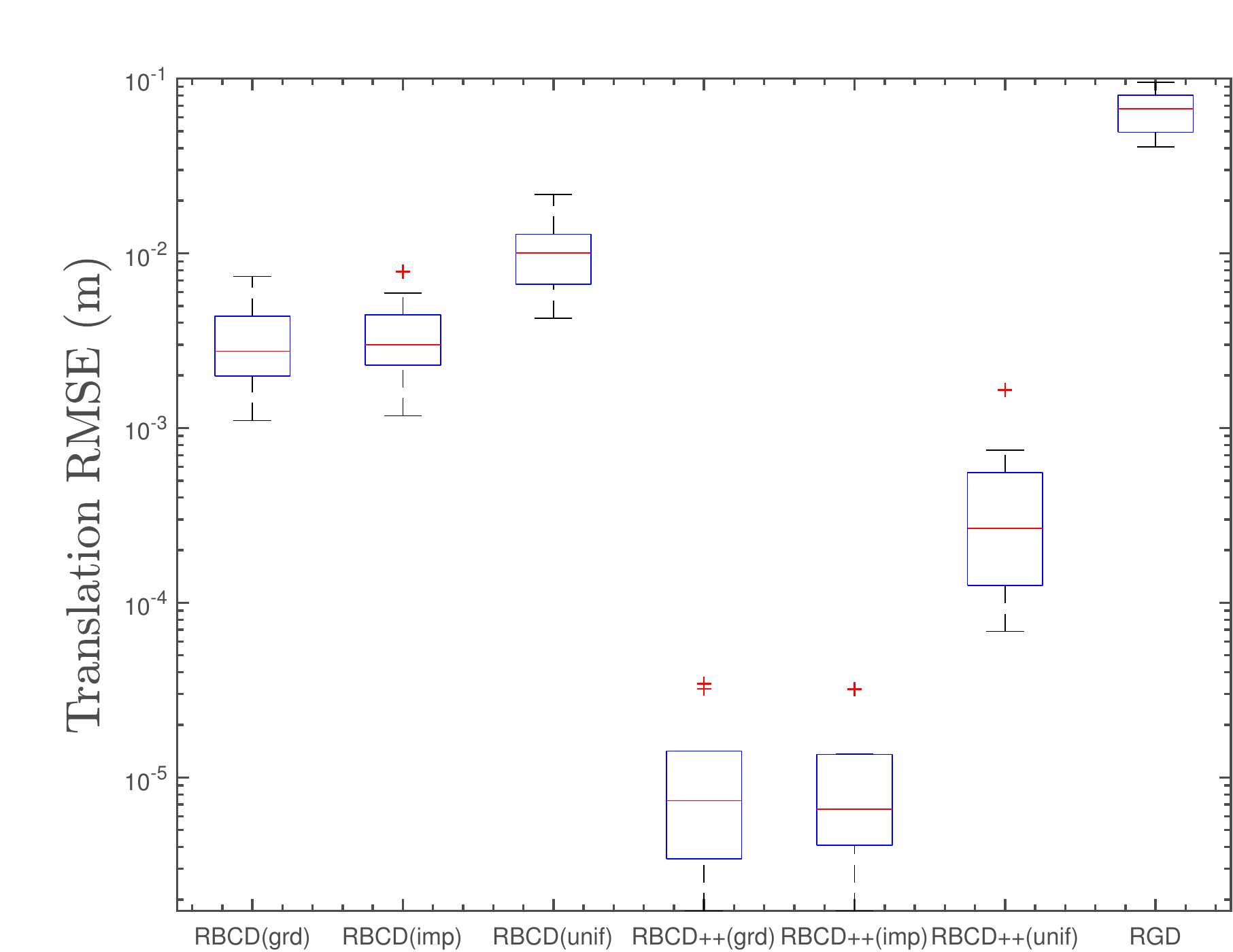}
		\caption{\small translation RMSE}
		\label{fig:local_search_transRMSE}
	\end{subfigure}
	\caption{\small  
		Convergence rates and final estimation errors of $\RBCD$ and $\ARBCD$ with uniform, importance, or greedy selection rules, on 
		the 9 robot simulation shown in Figure~\ref{fig:example_simulation}.
		(a) Evolution of optimality gap averaged averaged over 10 random runs.
		(b) Evolution of Riemannian gradient norm averaged over 10 random runs.
		(c) Boxplot of final rotation RMSE (after rounding) with respect to the global minimizer.
		(d) Boxplot of final translation RMSE (after rounding) with respect to the global minimizer.
	}
	\label{fig:local_search_convergence}
\end{figure}

Figure~\ref{fig:local_search_convergence} shows the performance on our 9 robot scenario shown in Figure~\ref{fig:example_simulation}. 
We report the performance of our proposed methods using all three block selection rules proposed in Section~\ref{sec:block_selection_rule}: uniform sampling, importance sampling, and greedy selection.
For reference, we also compare our performance against the Riemannian gradient descent (RGD) algorithm with Armijo's backtracking line search implemented in Manopt \cite{manopt}.
As the results demonstrate, both $\RBCD$ and $\ARBCD$ dominate the baseline RGD algorithm in terms of convergence speed and solution quality. 
As expected, importance sampling and greedy block selection also lead to faster convergence compared to uniform sampling.
Furthermore, $\ARBCD$ shows significant empirical acceleration and is able to converge to the global minimum with high precision using only 100 iterations. 
Note that in this experiment, we choose to report convergence speed with respect to iteration number, because it is directly linked to the number of communication rounds required during distributed optimization. 
For completeness, we also note that the average runtime of \textsc{BlockUpdate} (implemented based on a modified version of the trust-region solver in Manopt \cite{manopt}) is $0.023$~s.

\begin{figure}[t]
	\centering
	\begin{subfigure}[t]{0.33\textwidth}
		\centering
		\includegraphics[trim=5 10 40 20, clip, width=\textwidth]
		{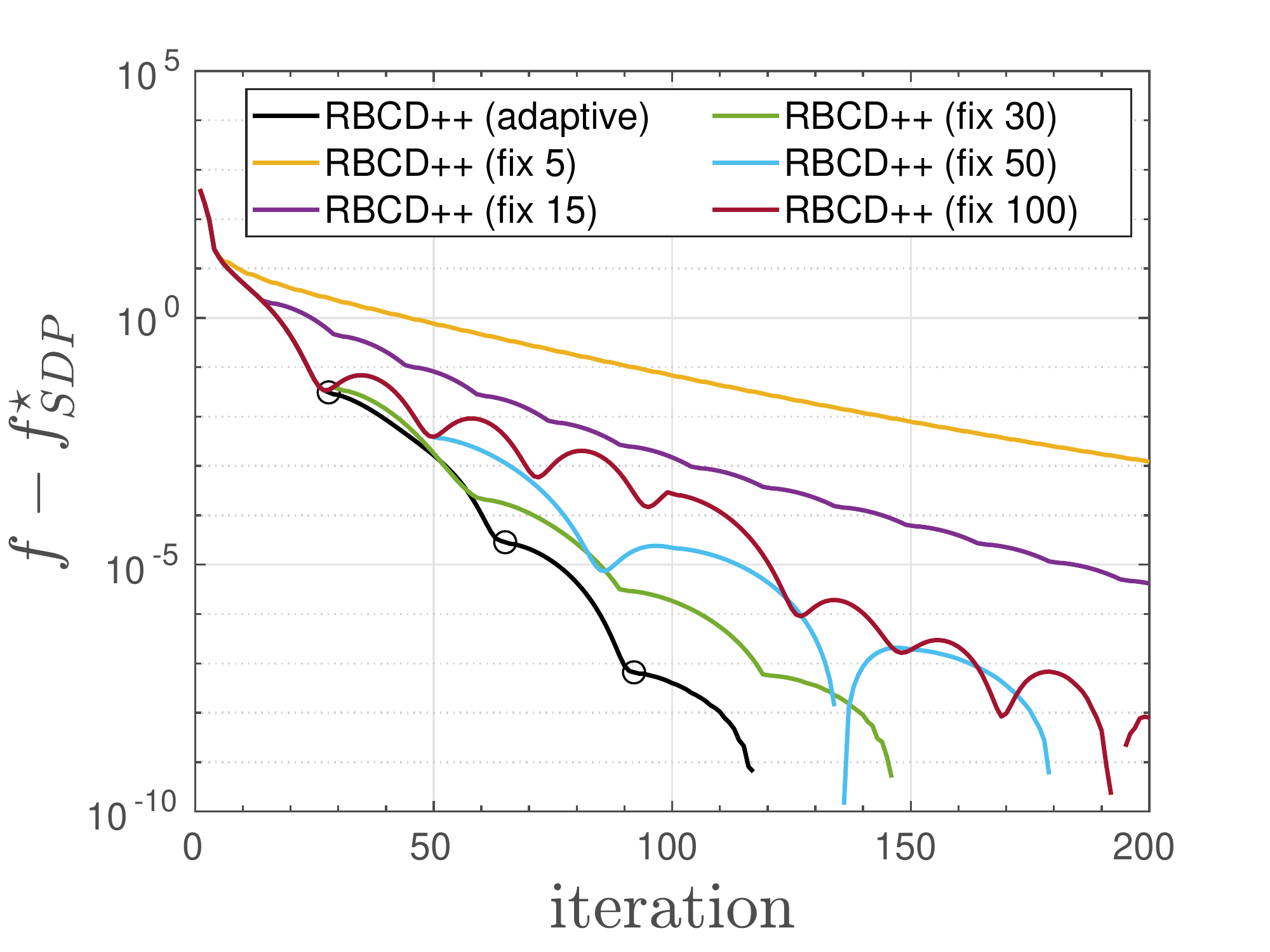}
		\caption{\small Optimality Gap}
		\label{fig:local_search_restart_optgap}
	\end{subfigure}
	\hspace{1cm}
	\begin{subfigure}[t]{0.33\textwidth}
		\centering
		\includegraphics[trim=5 10 40 20, clip, width=\textwidth]
		{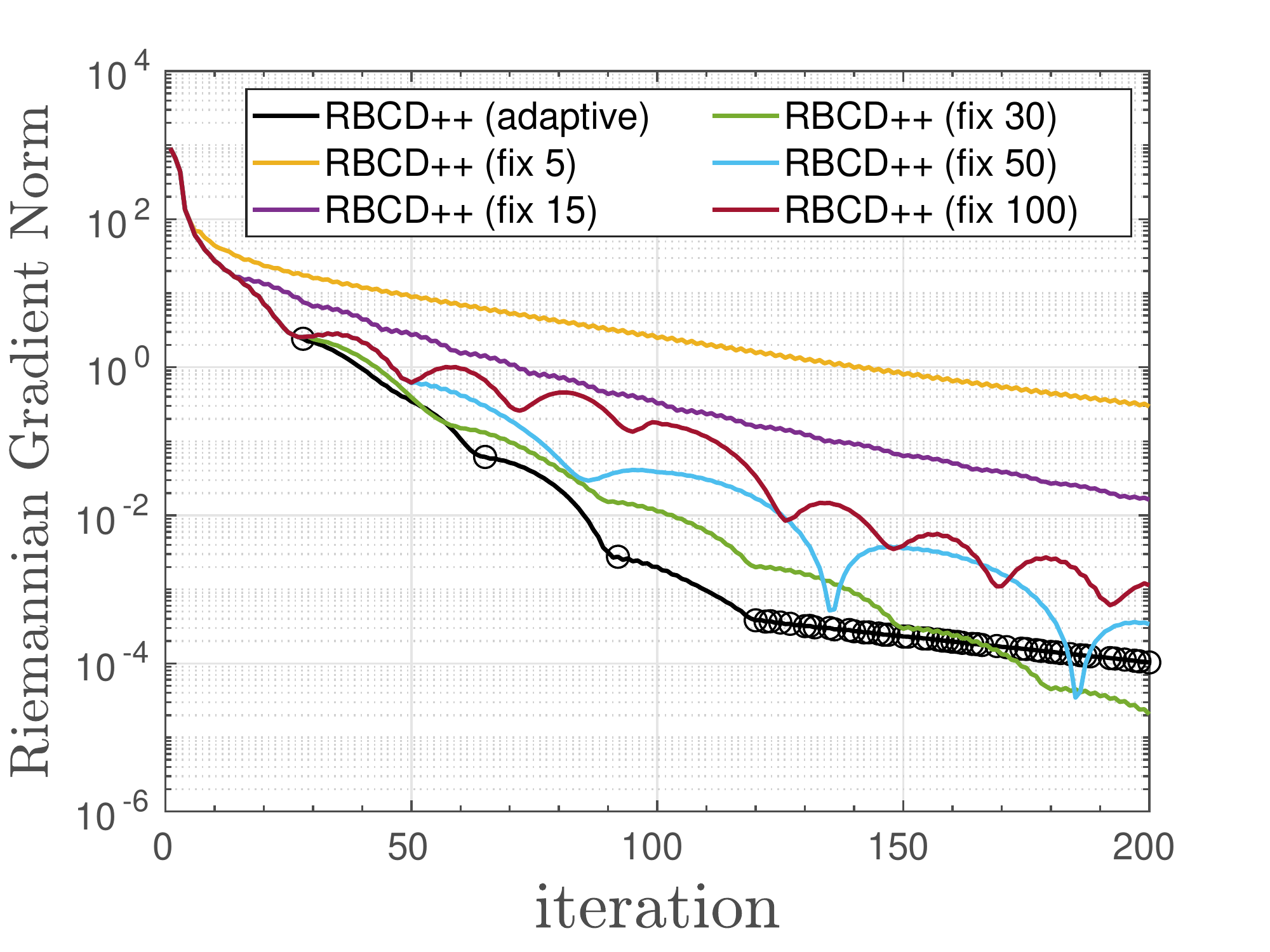}
		\caption{\small  Gradient norm}
		\label{fig:local_search_restart_gradnorm}
	\end{subfigure}
	\caption{\small  
		Adaptive restart vs. fixed restart for $\ARBCD$ (greedy selection) on a random instance of our simulation. 
		For fixed restart, we use restart frequency ranging from every 5 to every 100 iterations. 
		For adaptive restart (black curves), we also highlight iterations where restart is triggered with circle markers.
		(a) Evolution of optimality gaps.
		(b) Evolution of Riemannian gradient norm.
	}
	\label{fig:local_search_restart}
\end{figure} 

In Figure~\ref{fig:local_search_convergence}, we report the performance of $\ARBCD$ with the default adaptive restart scheme (see Algorithm~\ref{alg:ARBCD}). 
As we have discussed in Remark~\ref{rem:restart}, a less expensive and hence more practical restart scheme is simply to reset Nesterov's acceleration after a fixed number of iterations; this scheme is typically referred to as \emph{fixed restart} in the literature.
In Figure~\ref{fig:local_search_restart}, we compare adaptive restart with fixed restart on a random instance of our simulation.
For fixed restart, we use different restart frequency ranging from every 5 to every 100 iterations. 
We observe that with a short restart period (e.g., 5), convergence of $\ARBCD$ is significantly slowed down.
This result is expected, as frequent restarting essentially removes the effect of acceleration from the iterations of $\ARBCD$.
In the extreme case of restarting at every iteration, the algorithm essentially reduces to $\RBCD$.
On the other hand, long restart period (e.g., 100) also has a negative impact, and we observe that the overall convergence displays undesirable oscillations.
Finally, we find that a suitably chosen restart period (e.g., 30) demonstrates a superior convergence rate that is similar to the adaptive restart scheme.

\begin{figure}[t]
	\centering
	\begin{subfigure}[t]{0.23\textwidth}
		\centering
		\includegraphics[trim=5 10 30 20, clip, width=\textwidth]
		{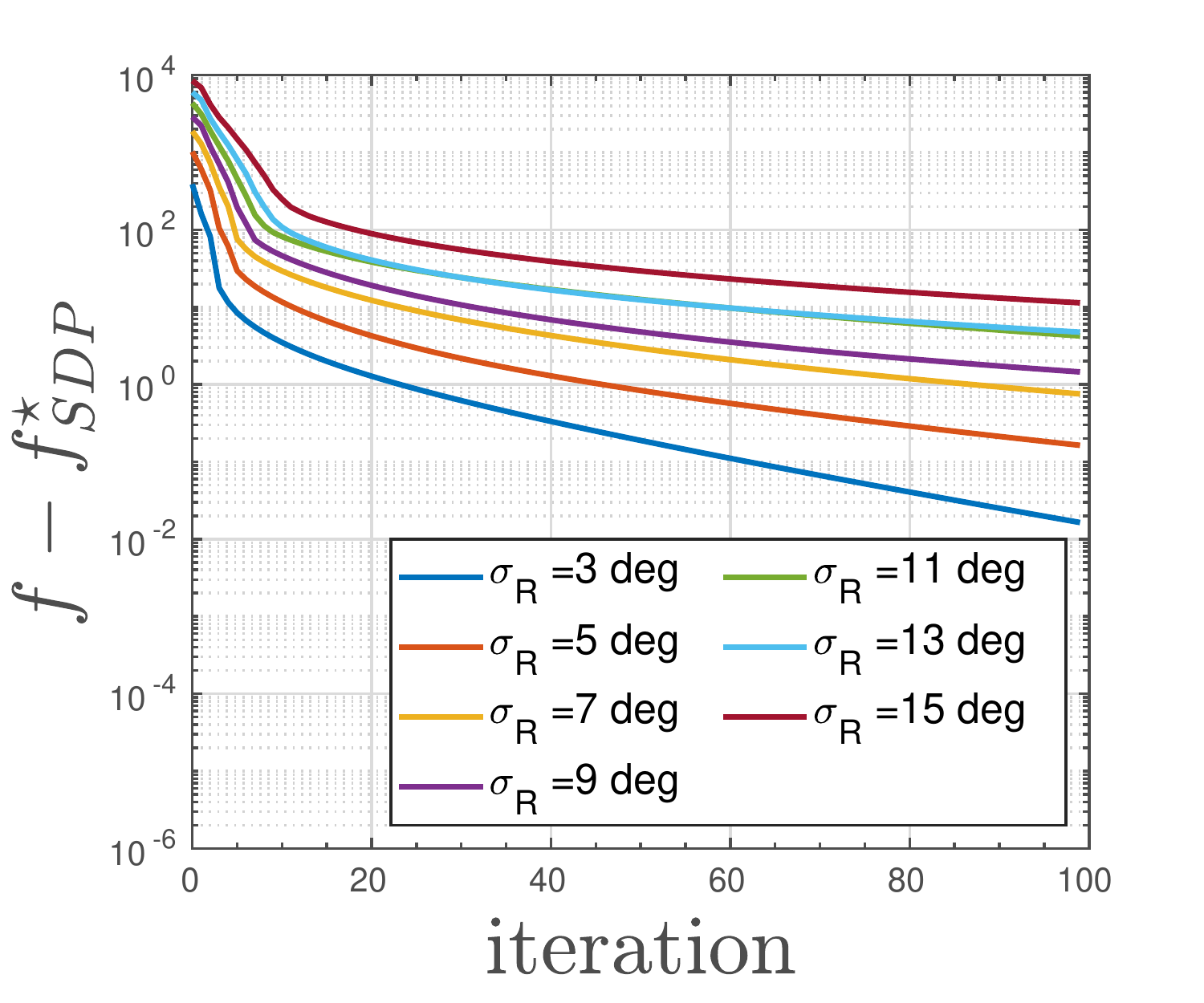}
		\caption{\small $\RBCD$ under increasing rotation noise}
		\label{fig:vary_rotation_noise_BCD}
	\end{subfigure}
	~
	\begin{subfigure}[t]{0.23\textwidth}
		\centering
		\includegraphics[trim=5 10 30 20, clip, width=\textwidth]
		{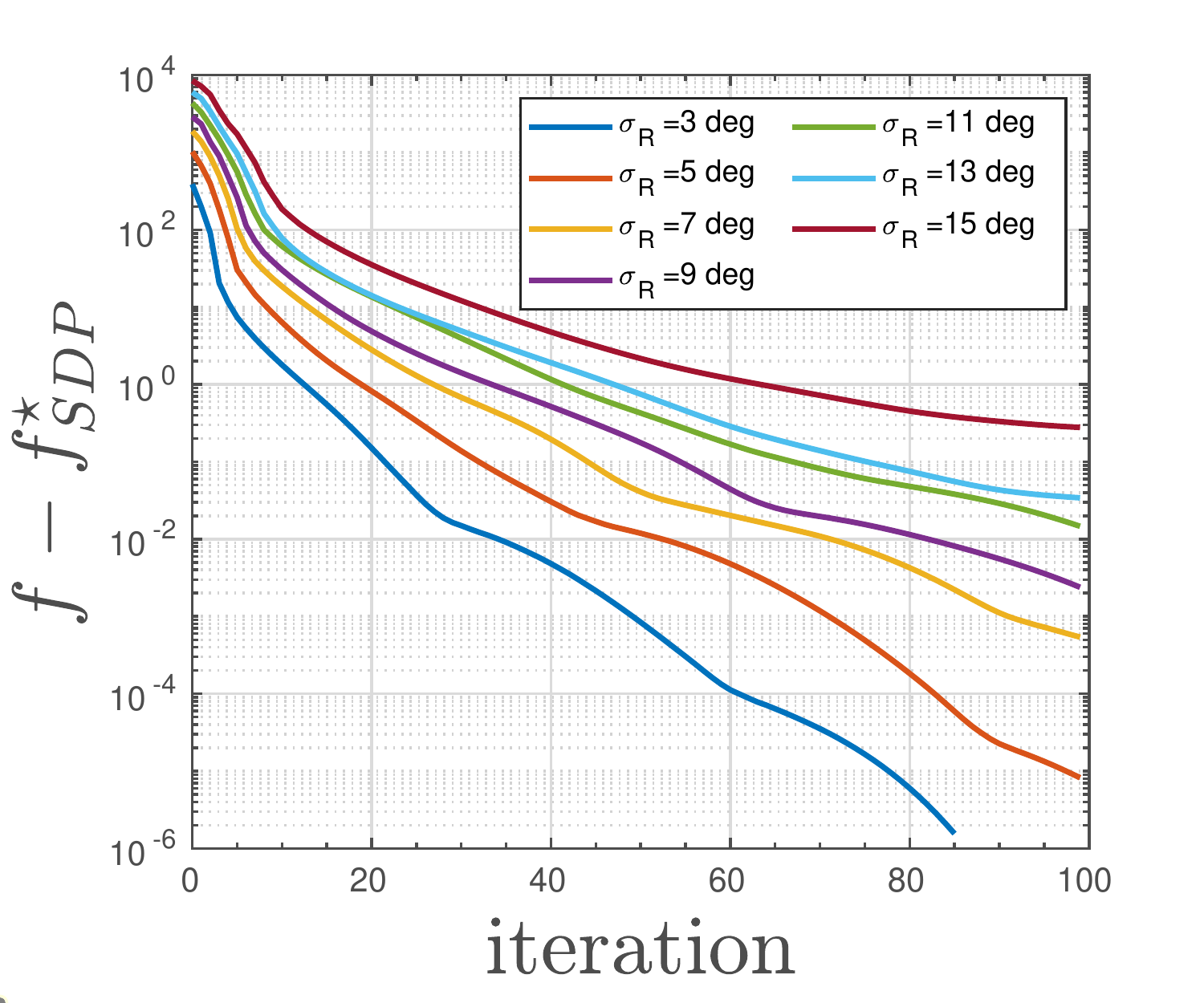}
		\caption{\small  $\ARBCD$ under increasing rotation noise}
		\label{fig:vary_rotation_noise_ABCD}
	\end{subfigure}
	~
	\begin{subfigure}[t]{0.23\textwidth}
		\centering
		\includegraphics[trim=5 10 30 20, clip, width=\textwidth]
		{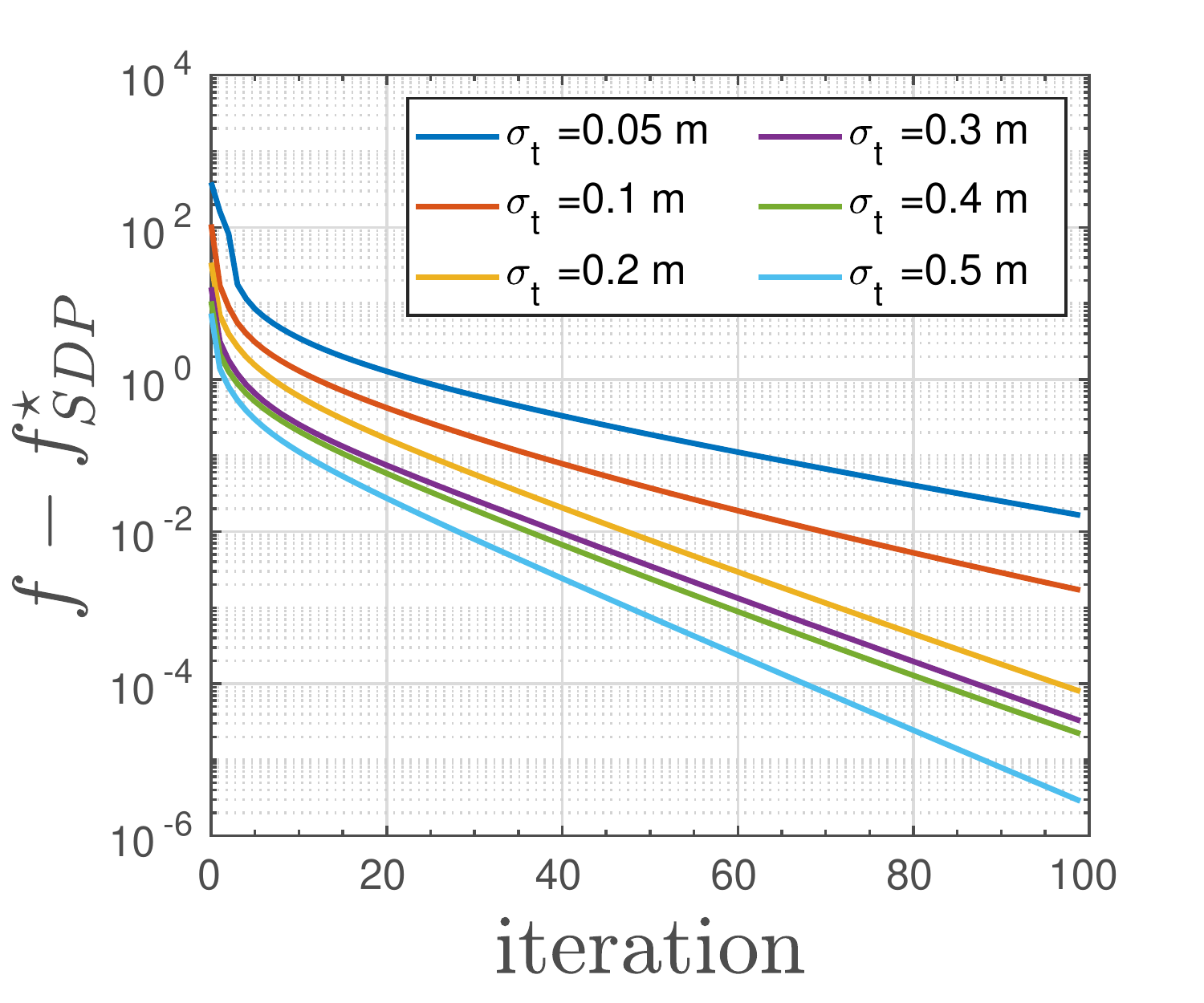}
		\caption{\small  $\RBCD$ under increasing translation noise}
		\label{fig:vary_translation_noise_BCD}
	\end{subfigure}
	~
	\begin{subfigure}[t]{0.23\textwidth}
		\centering
		\includegraphics[trim=5 10 30 20, clip, width=\textwidth]
		{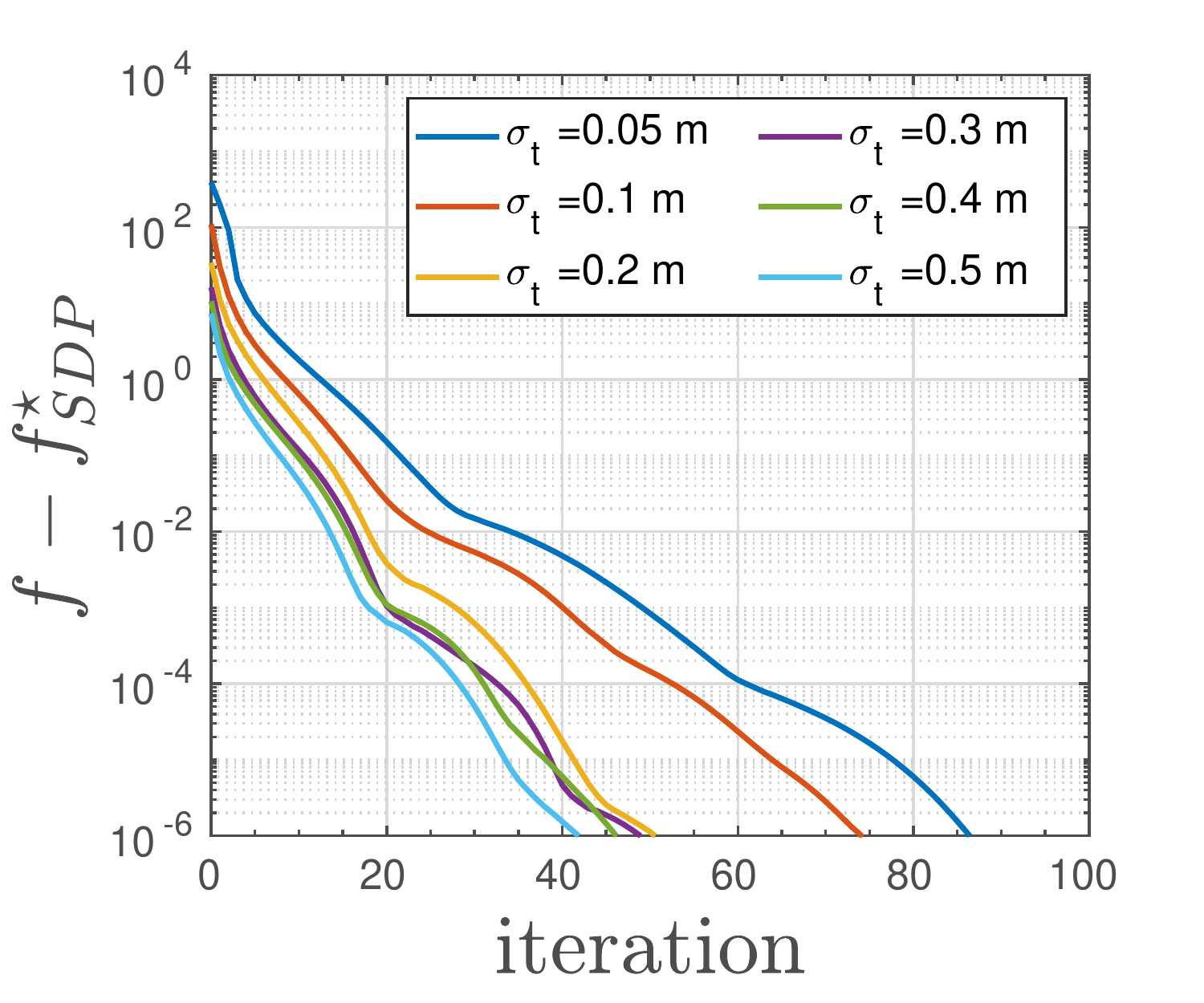}
		\caption{\small  $\ARBCD$ under increasing translation noise}
		\label{fig:vary_translation_noise_ABCD}
	\end{subfigure}
	\caption{\small  
		Convergence of $\RBCD$ and $\ARBCD$ with greedy selection under varying rotation and translation measurement noise. 
		(a)-(b) Convergence of $\RBCD$ and $\ARBCD$ under increasing rotation noise and fixed translation noise of $\sigma_t = 0.05$~m.
		(c)-(d) Convergence of $\RBCD$ and $\ARBCD$ under increasing translation noise and fixed rotation noise of $\sigma_t = 3^\circ$.
		All results are averaged across 10 random runs.
	}
	\label{fig:local_search_noise}
\end{figure}

On the same 9 robot scenario, we also report the convergence of $\RBCD$ and $\ARBCD$ (using greedy selection) under increasing measurement noise, shown in  
Figure~\ref{fig:local_search_noise}.
As expected, as rotation noise increases, the convergence rates of both $\RBCD$ and $\ARBCD$ are negatively impacted. 
On the other hand, we observe that increasing translation noise actually leads to better convergence behavior, as shown in Figure~\ref{fig:vary_translation_noise_BCD}-\ref{fig:vary_translation_noise_ABCD}. 
To explain these observations, 
we conjecture that increasing rotation noise and decreasing translation noise magnify the \emph{ill conditioning} of the optimization problem.  Qualitatively similar results were also reported in \cite{Rosen19IJRR} (decreasing translational noise was observed to \emph{increase} SE-Sync's wall-clock time). 

\begin{figure}[t]
	\centering
	\begin{subfigure}[t]{0.23\textwidth}
		\centering
		\includegraphics[trim=5 10 30 20, clip, width=\textwidth]
		{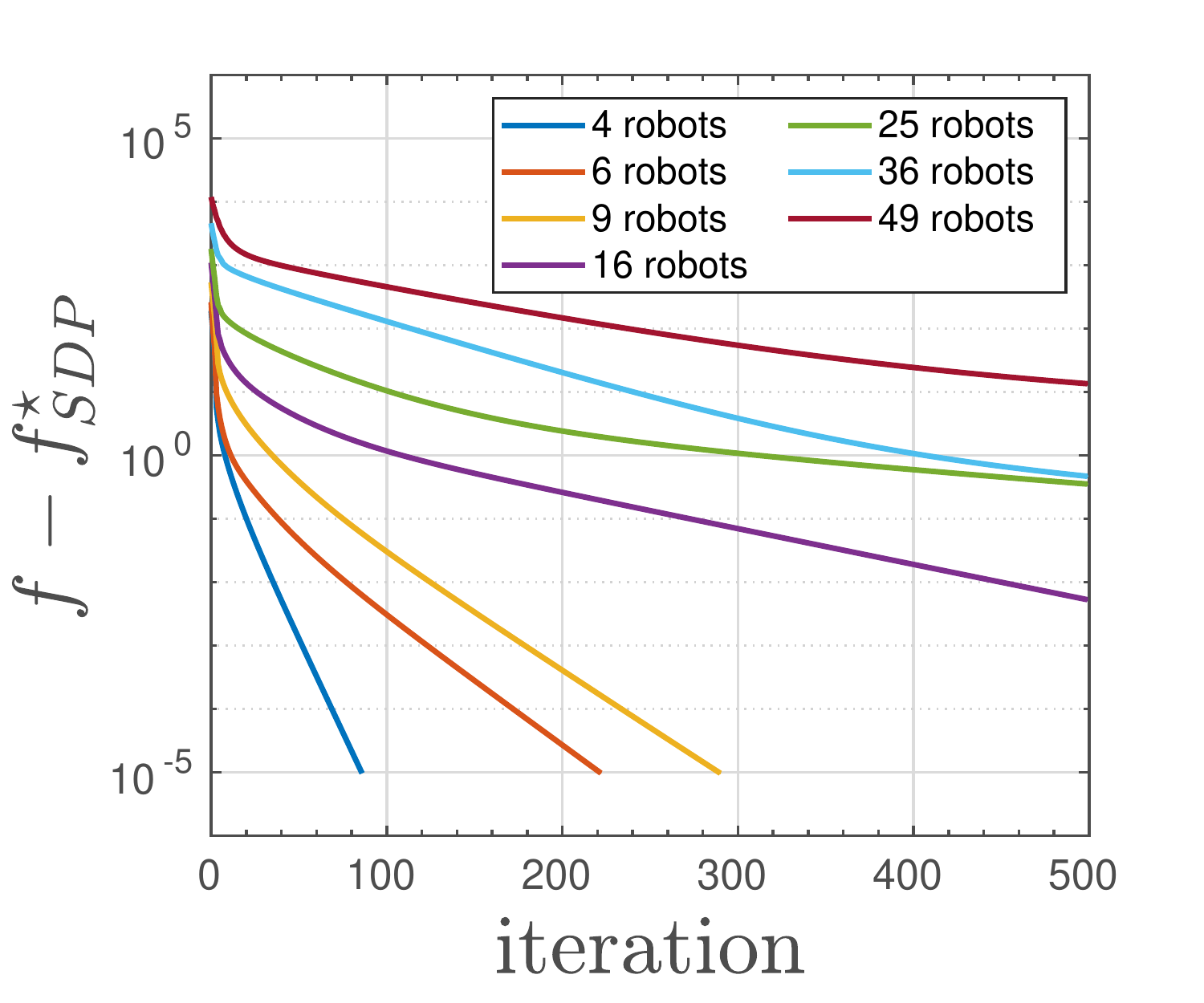}
		\caption{\small $\RBCD$ optimality gap}
		\label{fig:number_robot_BCD_optgap}
	\end{subfigure}
	~
	\begin{subfigure}[t]{0.23\textwidth}
		\centering
		\includegraphics[trim=5 10 30 20, clip, width=\textwidth]
		{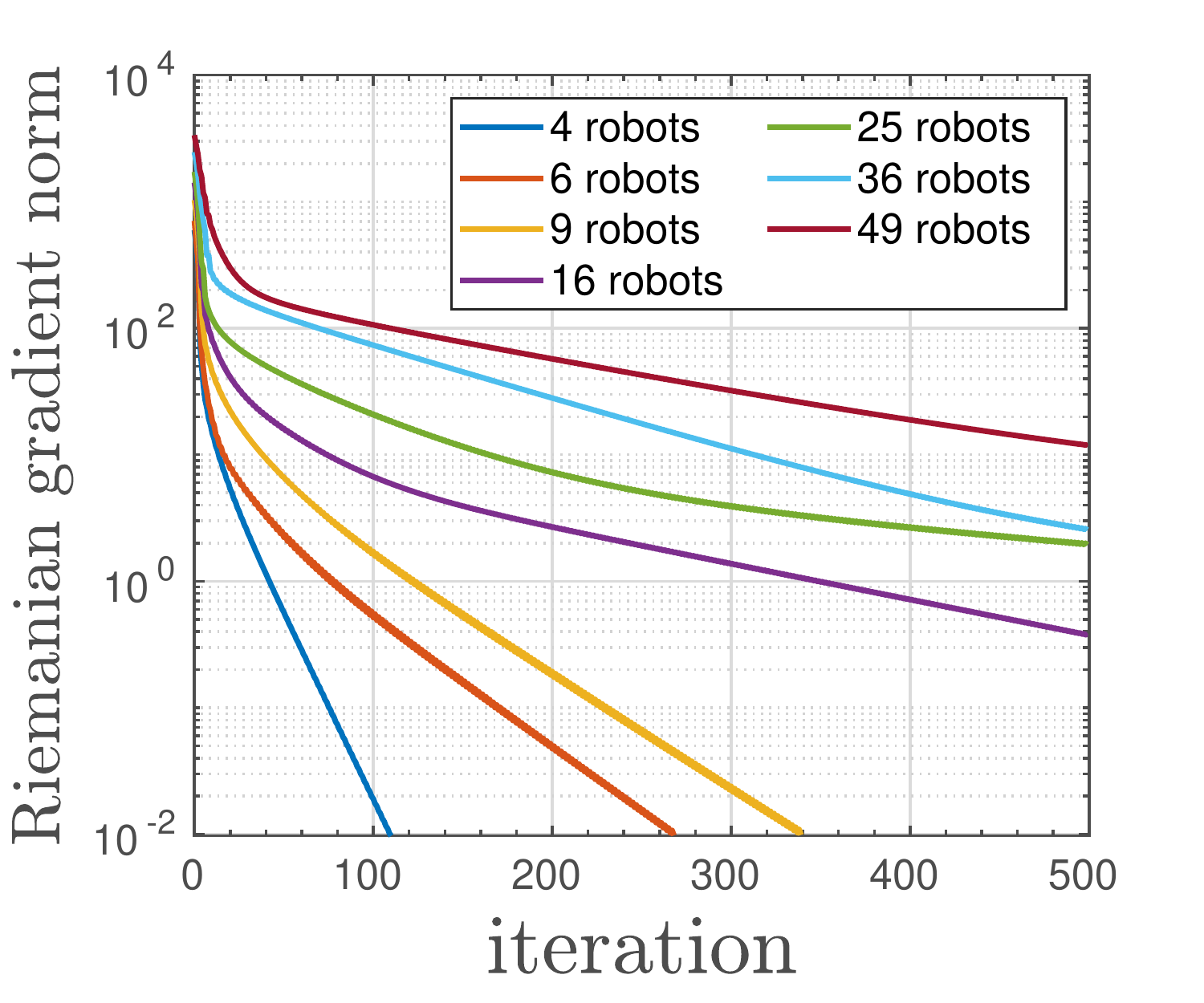}
		\caption{\small $\RBCD$ gradient norm}
		\label{fig:number_robot_BCD_gradnorm}
	\end{subfigure}
	~
	\begin{subfigure}[t]{0.23\textwidth}
		\centering
		\includegraphics[trim=5 10 30 20, clip, width=\textwidth]
		{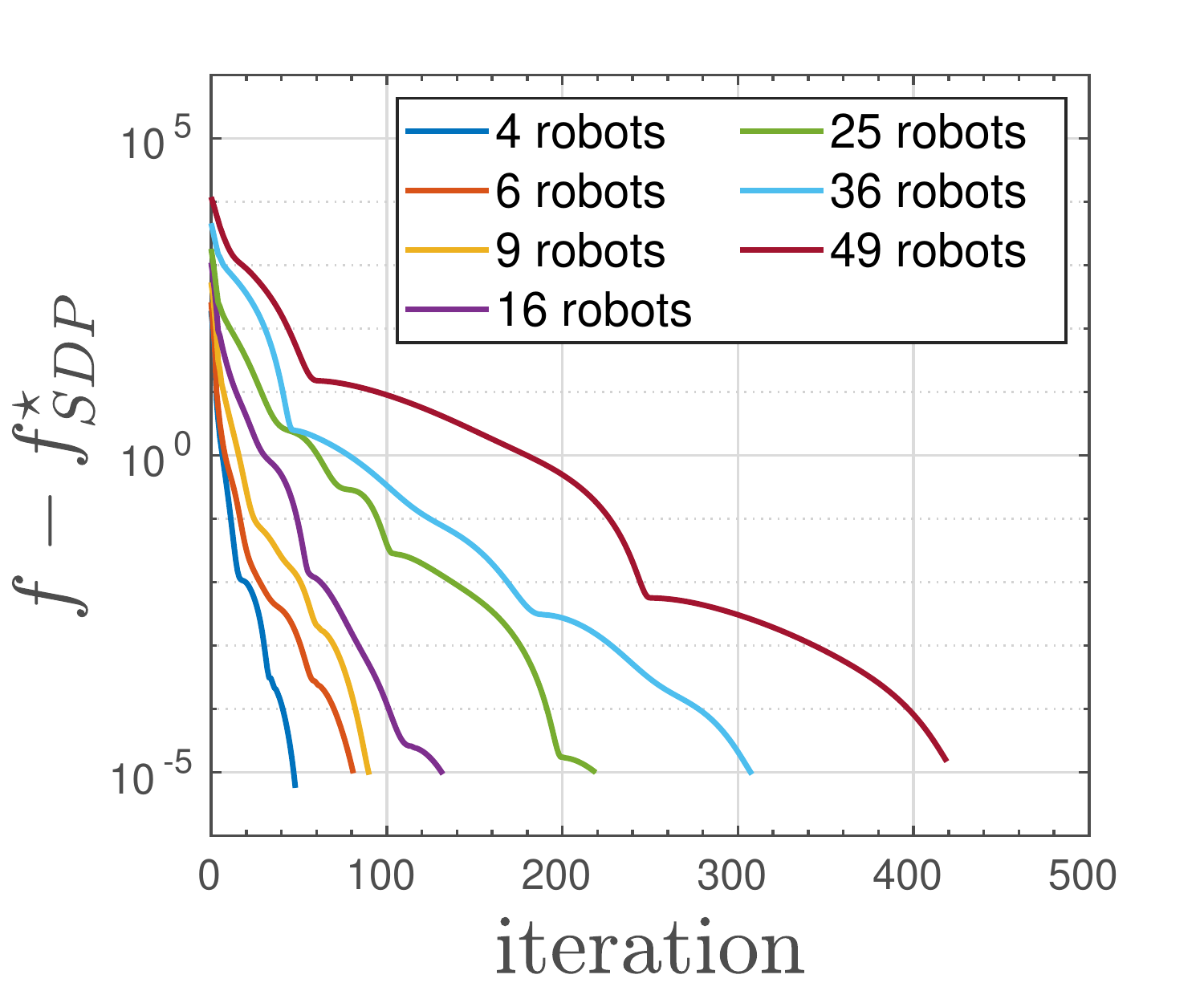}
		\caption{\small $\ARBCD$ optimality gap}
		\label{fig:number_robot_ABCD_optgap}
	\end{subfigure}
	~
	\begin{subfigure}[t]{0.23\textwidth}
		\centering
		\includegraphics[trim=5 10 30 20, clip, width=\textwidth]
		{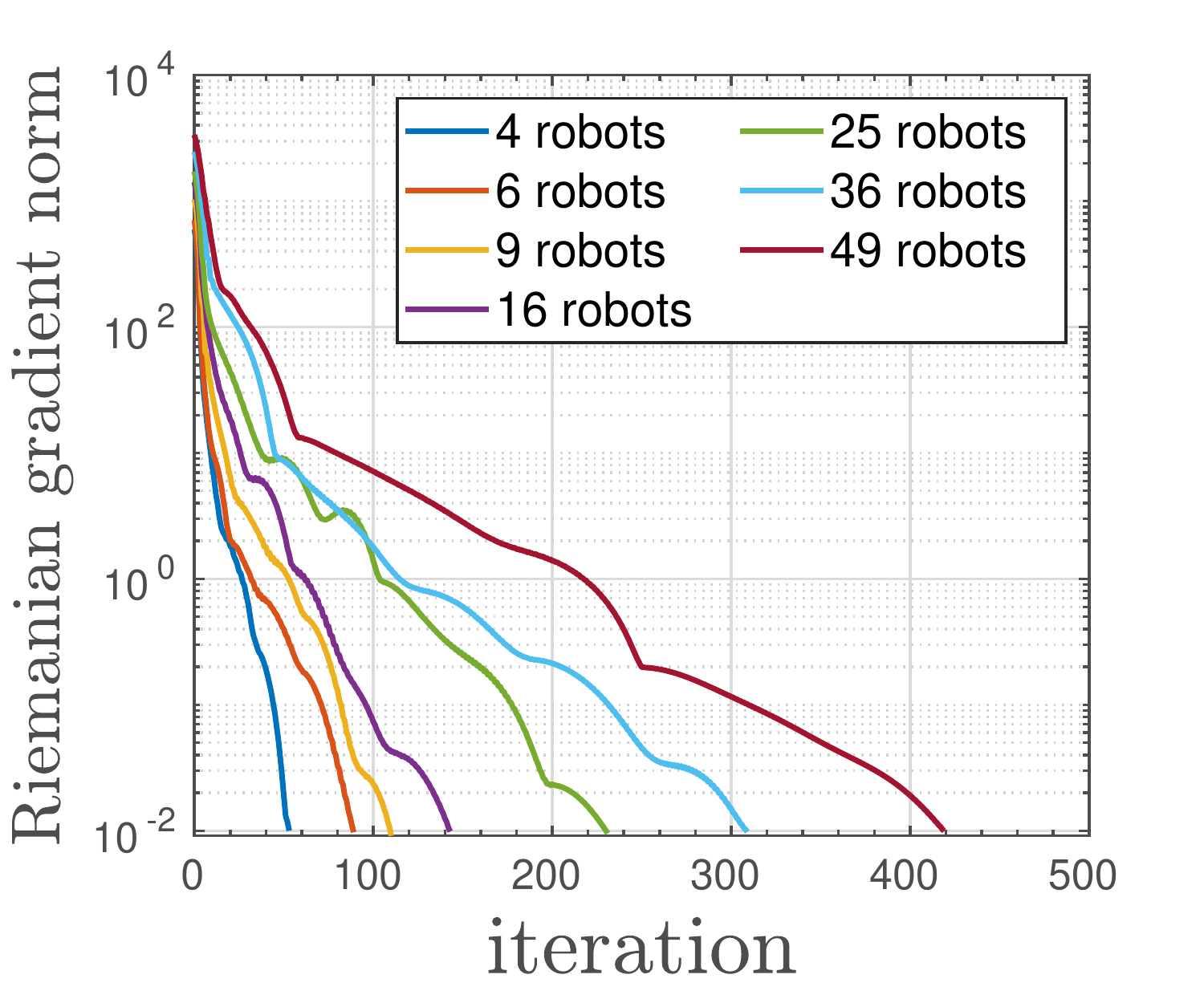}
		\caption{\small $\ARBCD$ gradient norm}
		\label{fig:number_robot_ABCD_gradnorm}
	\end{subfigure}
	\caption{\small  
		Scalability of $\RBCD$ and $\ARBCD$ with greedy selection
		as the number of robots increases.
		Each robot has 125 poses. Convergence speed is measured in terms of the Riemannian gradient norm.
	}
	\label{fig:local_search_scalability}
\end{figure} 

Lastly, we evaluate the scalability of $\RBCD$ and $\ARBCD$ (both with greedy block selection) as the number of robots increases from 4 to 49 in the simulation. 
As each robot has 125 poses, the maximum size of the global PGO problem is 6125.
Figure~\ref{fig:local_search_scalability} reports the convergence speeds measured in Riemannian gradient norm.
Both $\RBCD$ and $\ARBCD$ are reasonably fast for small number of robots. 
Nevertheless, the non-accelerated $\RBCD$ algorithm begins to show slow convergence as the number of robots exceeds 16. 
In comparison, our accelerated $\ARBCD$ algorithm shows superior empirical convergence speed, even in the case of 49 robots.
We note that in this case although $\ARBCD$ uses 400 iterations to achieve a Riemannian gradient norm of $10^{-2}$, the actual optimality gap (Figure~\ref{fig:number_robot_ABCD_optgap}) decreases much more rapidly to $10^{-5}$, which indicates that our solution is very close to the global minimum.

\subsection{Evaluations of Distributed Verification}
\label{sec:verification_experiments}

\begin{figure}[t]
	\centering
	\begin{subfigure}[t]{0.23\textwidth}
		\centering
		\includegraphics[trim=5 10 40 5, clip, width=\textwidth]
		{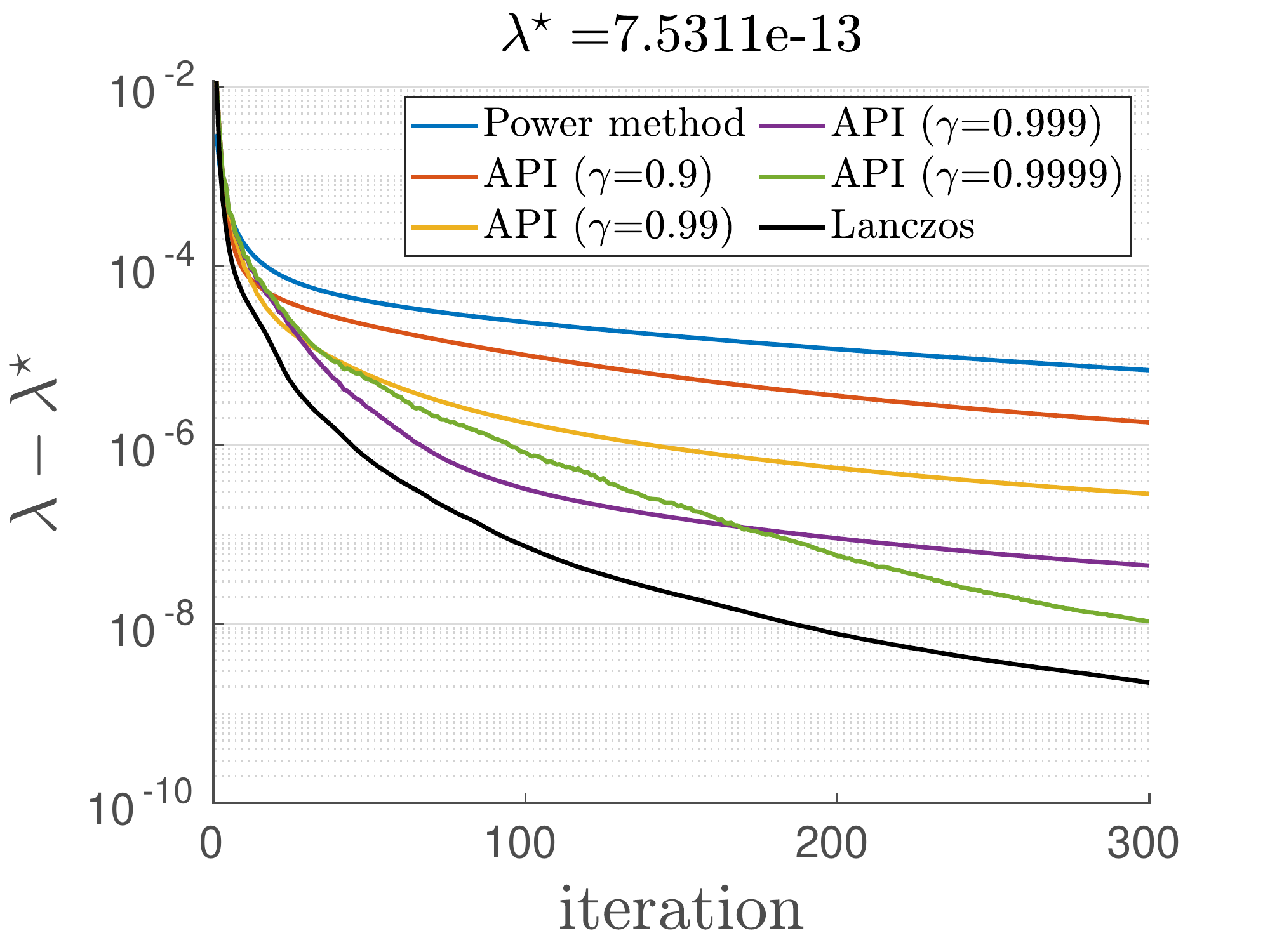}
		\caption{\small Distance to minimum eigenvalue when verifying global minimizer}
		\label{fig:verification_global_min_eigenvalue}
	\end{subfigure}
	~
	\begin{subfigure}[t]{0.23\textwidth}
		\centering
		\includegraphics[trim=5 10 40 20, clip, width=\textwidth]
		{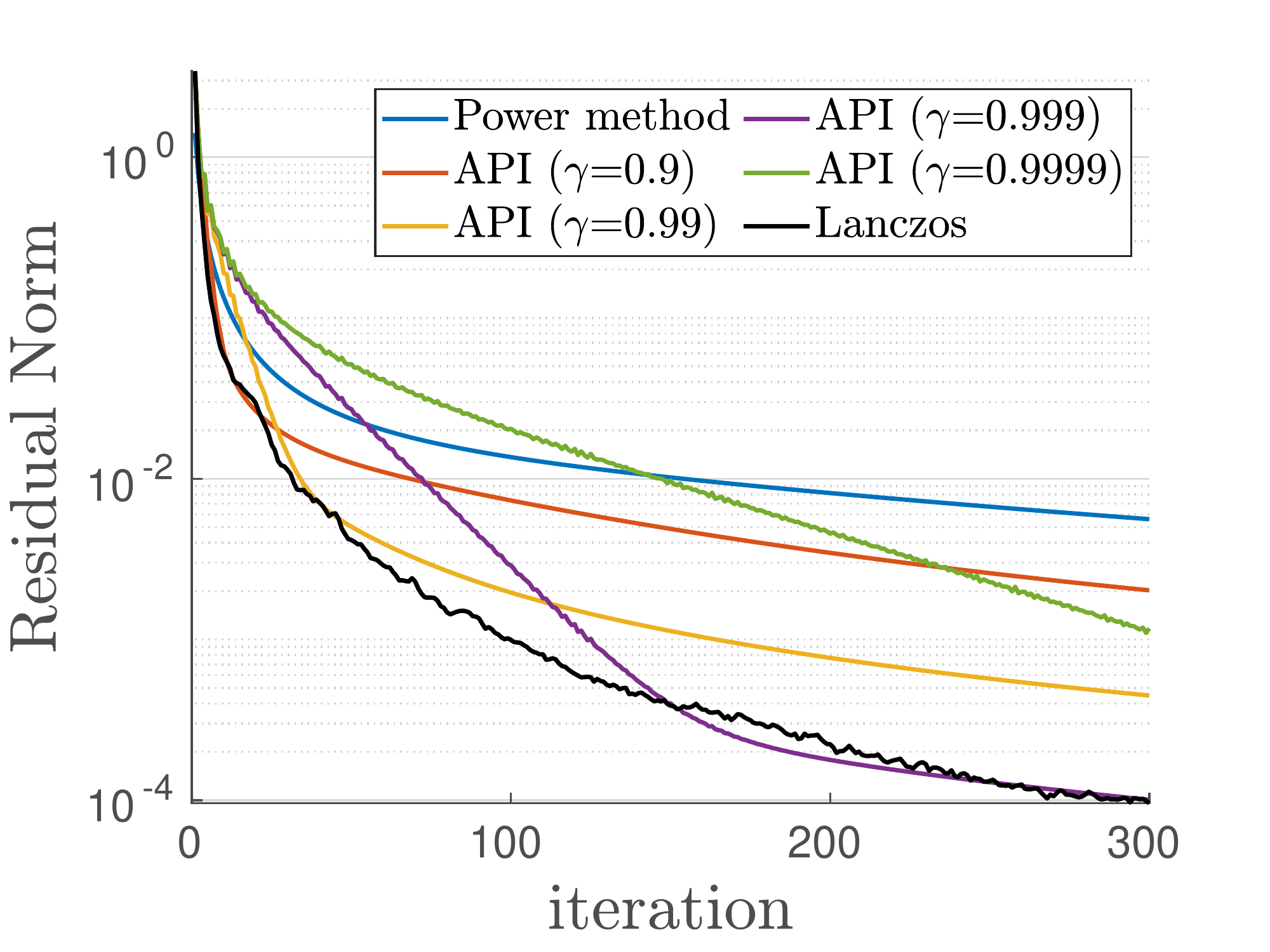}
		\caption{\small Eigenvector residual norm when verifying global minimizer}
		\label{fig:verification_global_min_residual}
	\end{subfigure}
	~
	\begin{subfigure}[t]{0.23\textwidth}
		\centering
		\includegraphics[trim=5 10 40 5, clip, width=\textwidth]
		{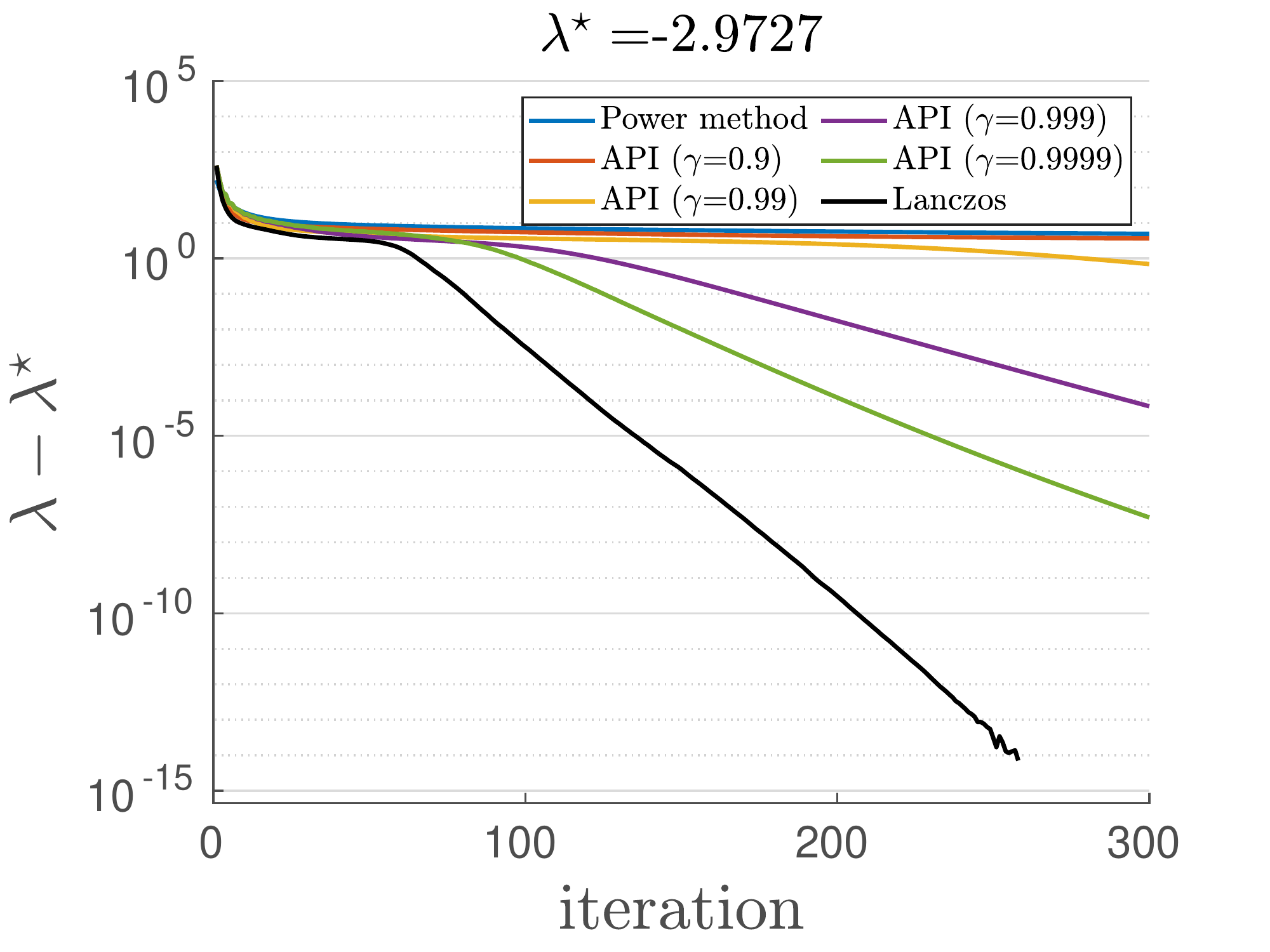}
		\caption{\small Distance to minimum eigenvalue when verifying suboptimal local minimizer}
		\label{fig:verification_local_min_eigenvalue}
	\end{subfigure}
	~
	\begin{subfigure}[t]{0.23\textwidth}
		\centering
		\includegraphics[trim=5 10 40 20, clip, width=\textwidth]
		{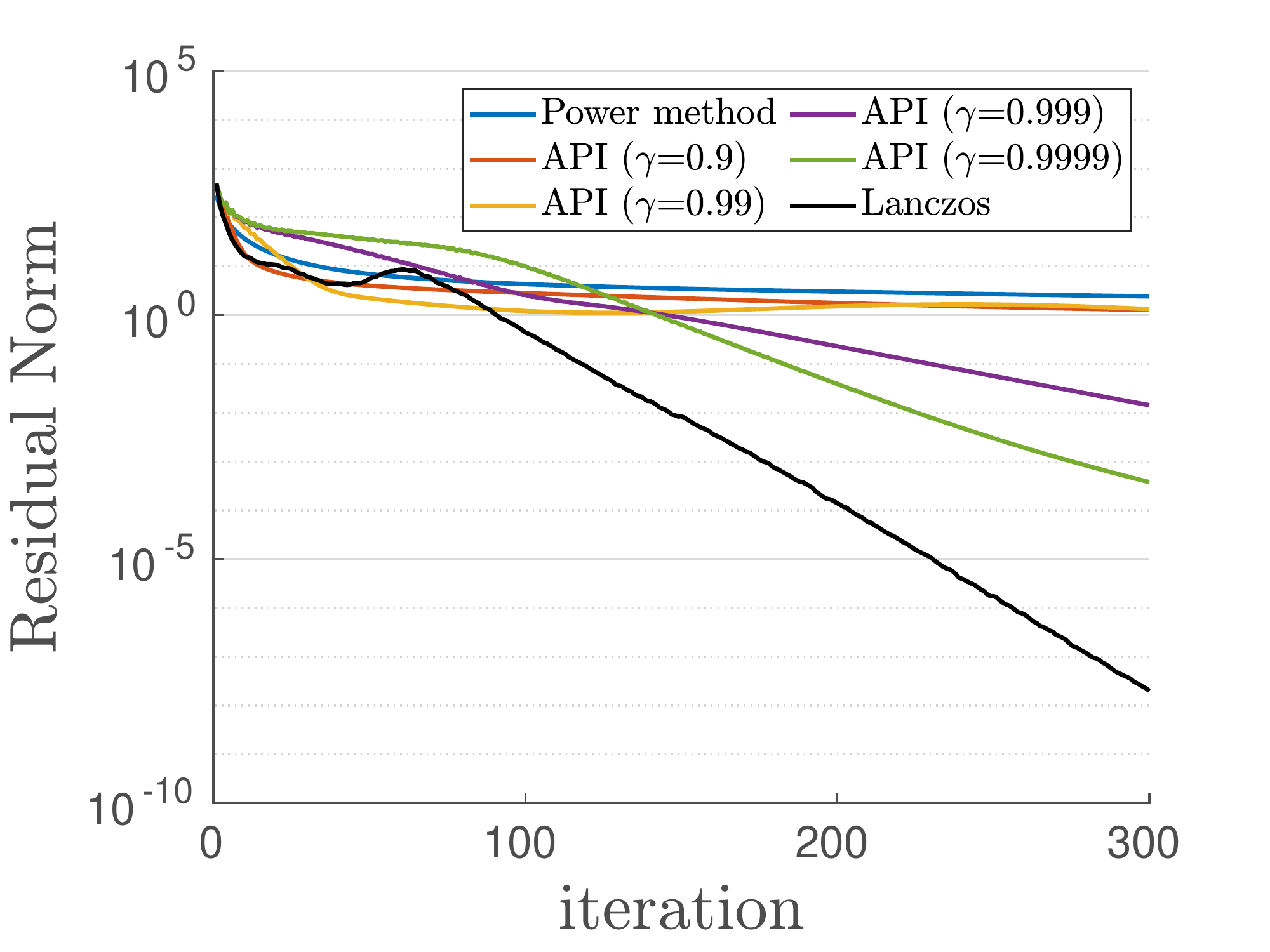}
		\caption{\small Eigenvector residual norm when verifying suboptimal local minimizer}
		\label{fig:verification_local_min_residual}
	\end{subfigure}
	\caption{\small  
		Performance of accelerated power iteration (API) on the \texttt{Killian court} dataset, using different values of $\gamma$ to set the momentum term according to \eqref{eq:select_beta}. 
		(a)-(b) Verification of a global minimizer computed by SE-Sync.
		(c)-(d) Verification of a suboptimal first-order critical point.
		In both cases, 
		the minimum eigenvalue of the dual certificate matrix (denoted as $\lambda^\star$)  is computed using the \texttt{eigs} function in MATLAB. 
	}
	\label{fig:verification_experiment}
\end{figure} 

In this section we evaluate our proposed distributed verification method.
Recall from Section~\ref{sec:verification} that the bulk of work happens when using accelerated power iteration to compute the dominant eigenpair of the spectrally-shifted dual certificate matrix $C \triangleq \lambda_\text{dom} I - S(X)$. 
We thus examine the efficiency of this process, and compare the performance of accelerated power iteration against standard power method and the centralized Lanczos procedure. 
We note that since $C$ and $S(X)$ share the same set of eigenvectors, in our experiment we still report results based on the estimated eigenvalues of $S(X)$.

From Remark~\ref{rem:select_beta}, the accelerated power iteration achieves the theoretical optimal rate when the employed momentum term satisfies $\beta \approx \lambda^2_2/4$, where $\lambda_2$ is the second dominant eigenvalue of $C$. 
Since we know that $\lambda_{\text{dom}}$ belongs to a tight cluster whenever $X$ is globally optimal, we expect that typically $\lambda_2 \approx  \lambda_\text{dom}$.
Using this insight, in our experiment we first estimate $\lambda_2$ by multiplying $\lambda_\text{dom}$ with a factor $\gamma < 1$ that is close to one, i.e., $\widehat{\lambda}_2 = \gamma \lambda_\text{dom}$. 
Subsequently we use this estimated value to set the momentum term, 
\begin{equation}
	\begin{aligned}
	\beta = \widehat{\lambda}^2_2/4 = \gamma^2 \lambda^2_\text{dom} / 4.
	\end{aligned}
	\label{eq:select_beta}
\end{equation}

We design two test cases using the \texttt{Killian court} dataset.
In the first case, we verify the global minimizer computed by SE-Sync \cite{Rosen19IJRR}. 
By Theorem~\ref{thm:verification}, the dual certificate matrix $S(X)$ must be positive semidefinite.
Furthermore, since $S(X)$ always has a nontrivial nullspace spanned by the rows of the corresponding primal solution, we expect the minimum eigenvalue of $S(X)$ to be zero in this case. 
Indeed, when computing this using the \texttt{eigs} function in MATLAB, the final value (denoted as $\lambda^\star$ in Figure~\ref{fig:verification_global_min_eigenvalue}) is close to zero to machine precision.
Figure~\ref{fig:verification_global_min_eigenvalue} shows how fast each method converges to $\lambda^\star$, where we use an initial eigenvector estimate obtained by (slightly) randomly perturbing a row of the global minimizer \cite{Rosen17IROS}. 
Figure~\ref{fig:verification_global_min_residual} shows the corresponding Ritz residual for the estimated eigenvector $v$. Assuming $v$ is normalized, this is given by,
\begin{equation}
	\text{ResidualNorm}(v) = \norm{S(X)v - (v^\top S(X) v) v}_2.
	\label{eq:eigenvector_residual}
\end{equation}
As the results suggest, with a suitable choice of $\gamma$, accelerated power iteration is significantly faster than the standard power method. 
Furthermore, in this case convergence speed is close to the Lanczos procedure. 

In the second case, we verify a suboptimal first-order critical point obtained by running $\ARBCD$ with $r = d$ using random initialization. 
We verify that the minimum eigenvalue of $S(X)$ is negative ($\approx -2.97$), which is consistent with the prediction of Theorem~\ref{thm:verification}.
In this case, we observe that using a \emph{random} initial eigenvector estimate leads to better convergence compared to obtaining the initial estimate from a perturbed row of the primal solution. Intuitively, using the perturbed initial guess would cause the iterates of power method to be ``trapped'' for longer period of time near the zero eigenspace spanned by the rows of $X$.
Figure~\ref{fig:verification_local_min_eigenvalue}-\ref{fig:verification_local_min_residual} shows results generated with random initial eigenvector estimate. 
Note that there is a bigger performance gap between accelerated power iteration and the Lanczos algorithm.
However, we also note that in reality, full convergence is actually not needed in this case.
Indeed, from Theorem~\ref{thm:verification}, we need only identify \emph{some} direction that satisfies $v^\top S(X) v < 0$ in order to escape the current suboptimal solution.

\subsection{Evaluations of Complete Algorithm (Algorithm~\ref{alg:dpgo})}
\label{sec:full_experiments}

So far, we have separately evaluated the proposed local search and verification techniques.
In this section, we evaluate the performance of the complete \AlgName\ algorithm (Algorithm~\ref{alg:dpgo}) that uses  distributed Riemannian Staircase (Algorithm~\ref{alg:riemannian_staircase}) to solve the SDP relaxation of PGO. 
By default, at each level of the Staircase we use $\ARBCD$ with greedy selection to solve the rank-restricted relaxation until the Riemannian gradient norm reaches $10^{-1}$.
Then, we use the accelerated power method to verify the obtained solution. 
To set the momentum term, we employ the same method introduced in the last section with $\gamma = 0.999$ in \eqref{eq:select_beta}. The accelerated power iteration is deemed converged when the eigenvector residual defined in \eqref{eq:eigenvector_residual} reaches $10^{-2}$.

\begin{figure}[t]
	\centering
	\begin{subfigure}[t]{0.25\textwidth}
		\centering
		\includegraphics[trim=5 0 40 10, clip, width=\textwidth]
		{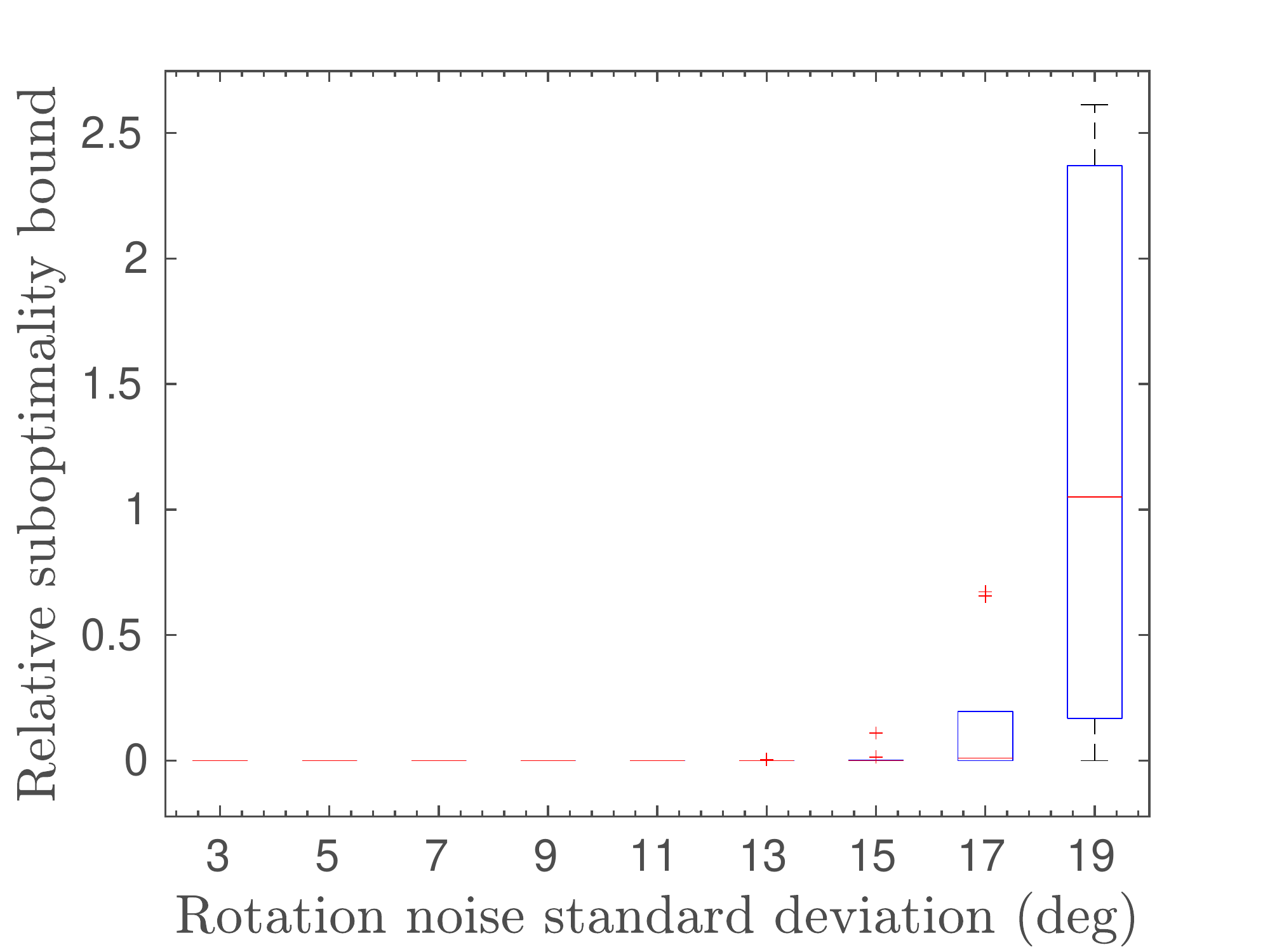}
		\caption{\small  Relative suboptimality bound under increasing rotation noise}
		\label{fig:DRS_vary_rotation_noise_subopt_gap}
	\end{subfigure}
	\hspace{1cm}
	\begin{subfigure}[t]{0.25\textwidth}
		\centering
		\includegraphics[trim=5 0 40 20, clip, width=\textwidth]
		{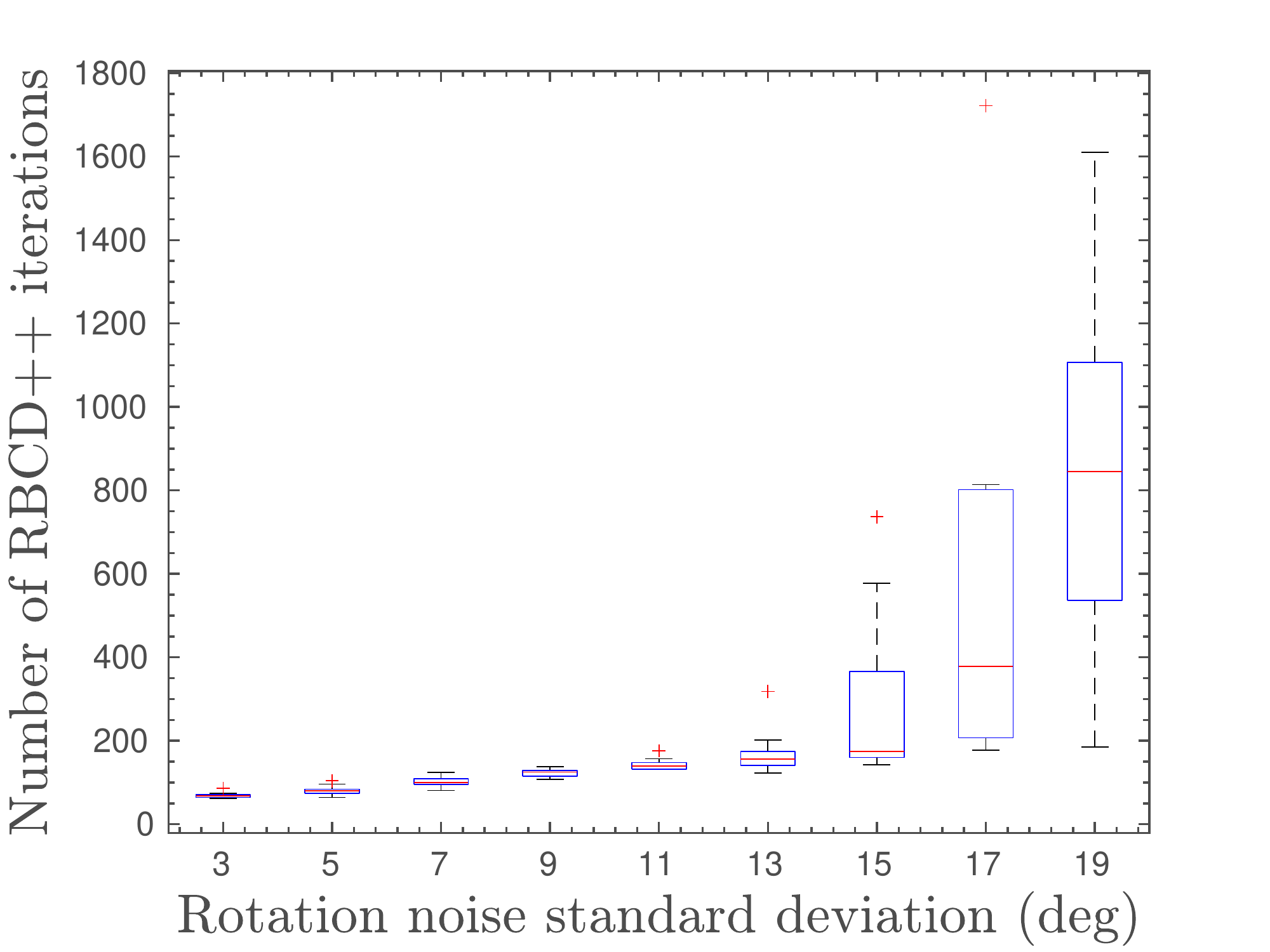}
		\caption{\small Total number of $\ARBCD$ iterations under increasing rotation noise}
		\label{fig:DRS_vary_rotation_noise_local_search_iters}
	\end{subfigure}
	\hspace{1cm}
	\begin{subfigure}[t]{0.25\textwidth}
		\centering
		\includegraphics[trim=5 0 40 20, clip, width=\textwidth]
		{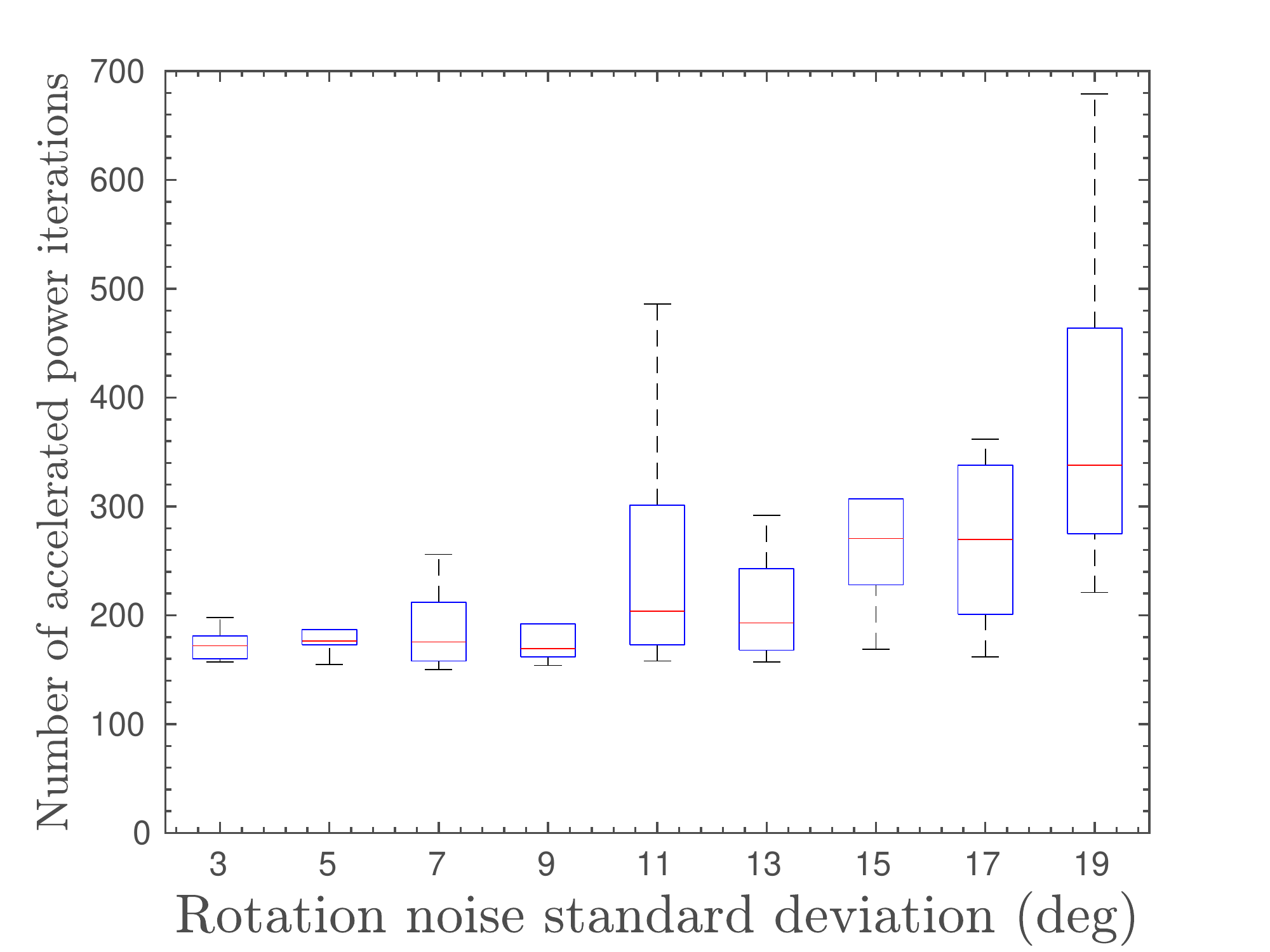}
		\caption{\small Total number of accelerated power iterations under increasing rotation noise}
		\label{fig:DRS_vary_rotation_noise_verification_iters}
	\end{subfigure}
	\\
	\begin{subfigure}[t]{0.25\textwidth}
		\centering
		\includegraphics[trim=5 0 40 5, clip, width=\textwidth]
		{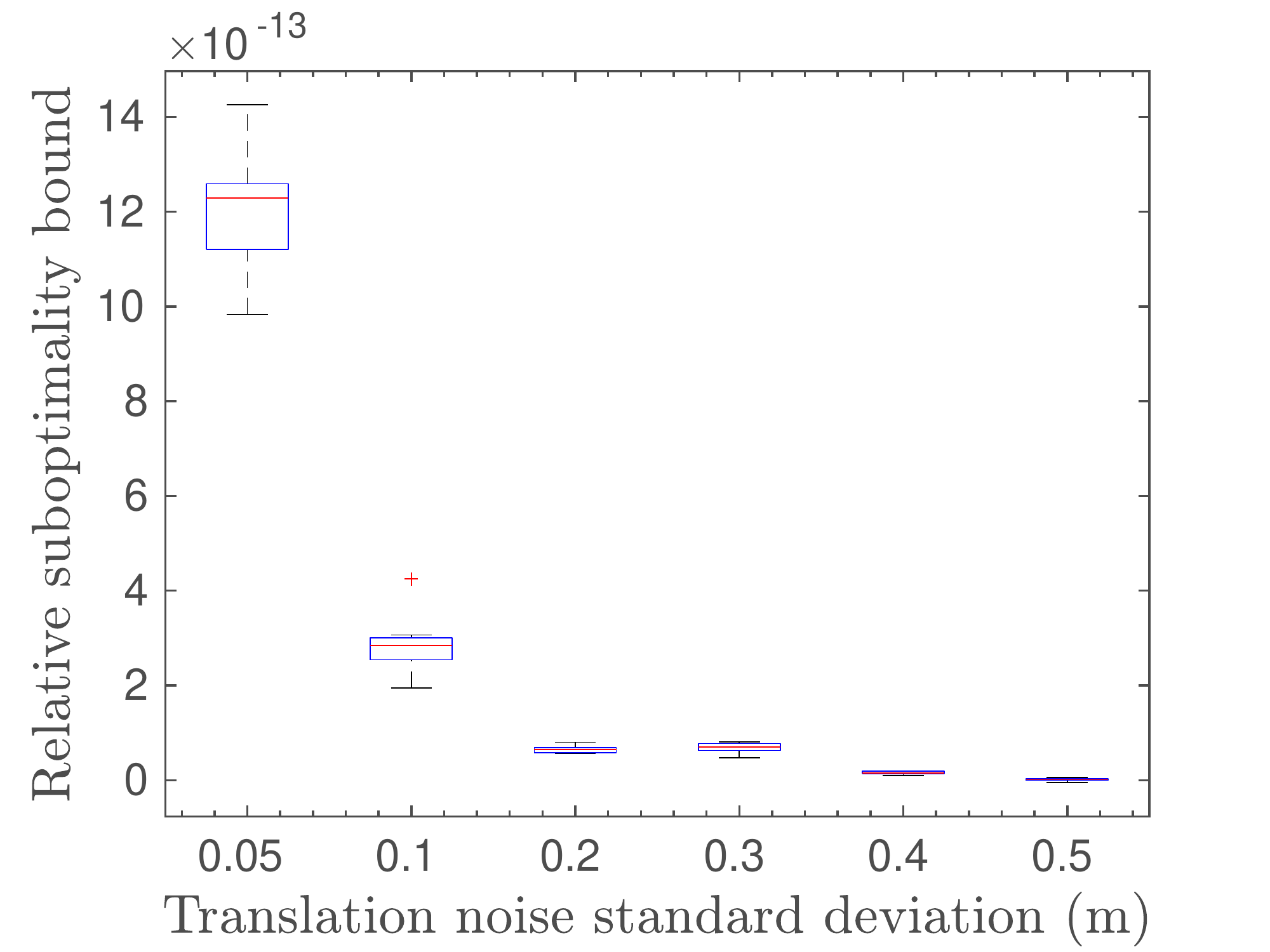}
		\caption{\small Relative suboptimality bound under increasing translation noise}
		\label{fig:DRS_vary_translation_noise_subopt_gap}
	\end{subfigure}
	\hspace{1cm}
	\begin{subfigure}[t]{0.25\textwidth}
		\centering
		\includegraphics[trim=5 0 40 20, clip, width=\textwidth]
		{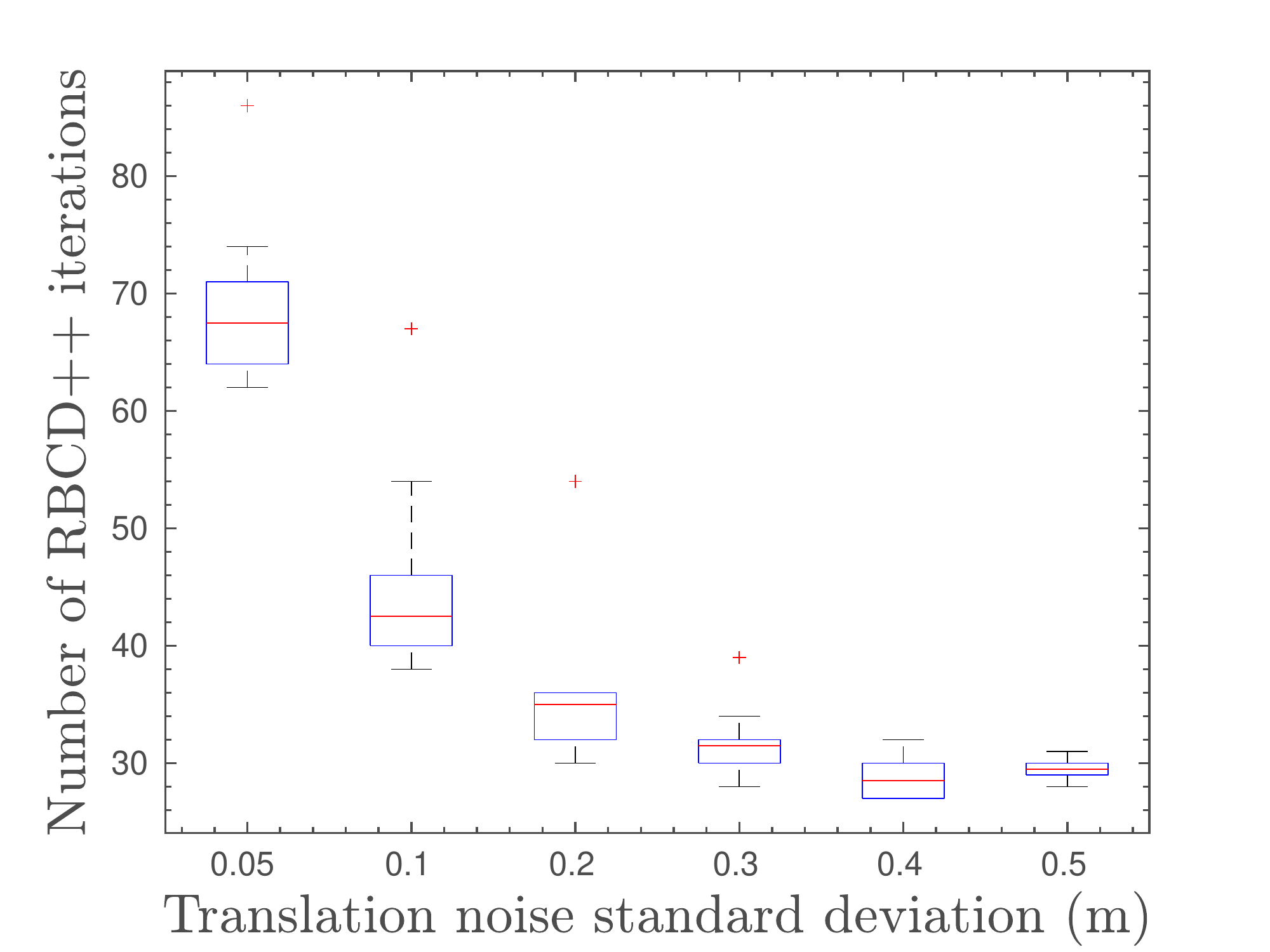}
		\caption{\small Total number of $\ARBCD$ iterations under increasing translation noise}
		\label{fig:DRS_vary_translation_noise_local_search_iters}
	\end{subfigure}
	\hspace{1cm}
	\begin{subfigure}[t]{0.25\textwidth}
		\centering
		\includegraphics[trim=5 0 40 20, clip, width=\textwidth]
		{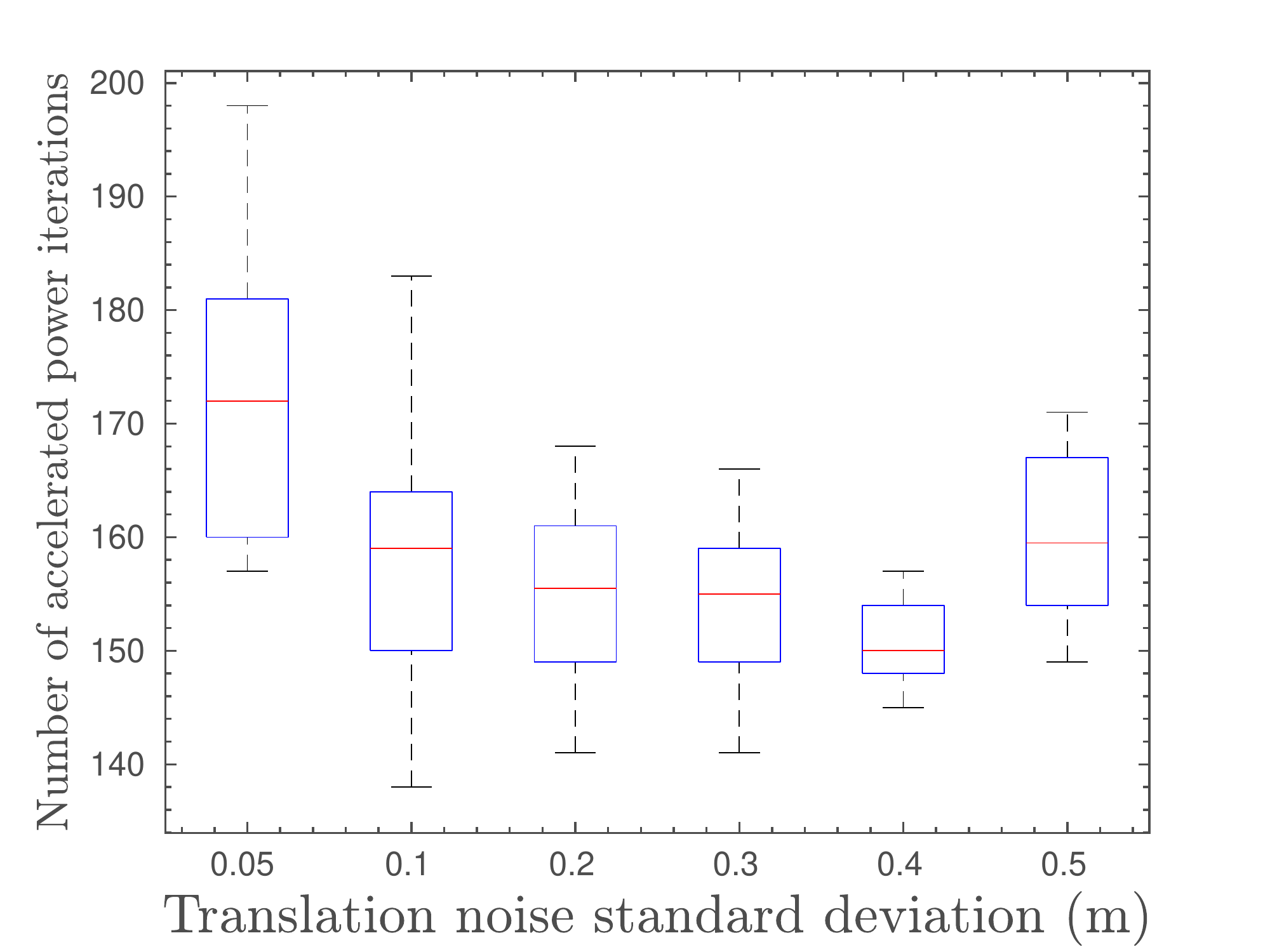}
		\caption{\small  Total number of accelerated power iterations under increasing translation noise}
		\label{fig:DRS_vary_translation_noise_verification_iters}
	\end{subfigure}
	\caption{\small  
		Evaluation of the proposed \AlgName\ (Algorithm~\ref{alg:dpgo}) under increasing measurement noise. 
		At each noise level, we simulate 10 random realizations of the 9-robot scenario shown in Figure~\ref{fig:example_simulation}.
		Left column shows boxplot of relative suboptimality bound $(f(T)-f^\star_\text{SDP})/f^\star_\text{SDP}$.
		Middle and right columns show boxplots of total number of iterations used by $\ARBCD$ and accelerated power method.
		The top row shows results under increasing rotation noise $\sigma_R \in [3,19]$~deg and fixed translation noise $\sigma_t = 0.05$~m.
		The bottom row shows results under increasing translation noise $\sigma_t \in [0.05, 0.5]$~m and fixed rotation noise $\sigma_R = 3$ deg.
	}
	\label{fig:DRS_noise_experiment}
\end{figure} 

We first examine the exactness of the SDP relaxation in the 9-robot scenario shown in Figure~\ref{fig:example_simulation} under increasing measurement noise.
Recall that \AlgName\ returns both a rounded feasible solution $T \in \SE(d)^n$ as well as the optimal value of the SDP relaxation $f^\star_\text{SDP}$.
To evaluate exactness, we record the upper bound on the relative suboptimality of $T$, defined as $(f(T)-f^\star_\text{SDP})/f^\star_\text{SDP}$.
We note that a zero suboptimality bound means that the SDP relaxation is exact and the solution $T$ is a global minimizer. 
As shown in the first column of Figure~\ref{fig:DRS_noise_experiment}, \AlgName\ is capable of finding global minimizers for all translation noise considered in our experiments, and for rotation noise up to 11 degree, which is still much larger than noise magnitude typically encountered in SLAM. 
The middle and right columns of Figure~\ref{fig:DRS_noise_experiment} show the total number of iterations used by $\ARBCD$ and accelerated power method, across all levels of the staircase. 
Interestingly, we observe that changing measurement noise has a greater impact for distributed local search compared to distributed verification. 

\begin{table*}[t]
	\setlength{\tabcolsep}{1.5pt}
	\renewcommand{\arraystretch}{1.2}
	\centering
	\caption{ \small Evaluation on benchmark PGO datasets. Each dataset simulates a CSLAM scenario with five robots.
		On each dataset, we report the objective value achieved by initialization,  centralized SE-Sync \cite{Rosen19IJRR}, the distributed Gauss-Seidel (DGS) with 
		SOR parameter 1.0 as recommended in \cite{Choudhary17IJRR},
		and the proposed \AlgName\ algorithm. 
		For the latter two distributed algorithms, we also report the total number of local search iterations. 
		On each dataset, we highlight the distributed algorithm that (i) achieves lower objective and (ii) uses less local search iterations.
		On all datasets, our approach is able to verify its solution as the global minimizer.
		We note that the numerical difference with SE-Sync on some datasets is due to the looser convergence condition during distributed local search.
	}
	{\small
		\begin{tabular}{|c|c|c||c|c||c|c||c|c|}
			\hline
			\multirow{2}{*}{Dataset} & \multirow{2}{*}{\# Vertices} & \multirow{2}{*}{\# Edges} & \multicolumn{4}{c||}{Objective} & \multicolumn{2}{c|}{Local Search Iterations} \\ \cline{4-9} 
			&       &       & Init.  & SE-Sync \cite{Rosen19IJRR} & DGS \cite{Choudhary17IJRR}  & \AlgName\ & DGS \cite{Choudhary17IJRR} & \AlgName\  \\ 
			\hline \hline
			Killian Court (2D)      & 808   & 827   & 229.0 & 61.15   & {63.52} & \textbf{61.22}      & {27105}  & \textbf{189}        \\ \hline
			CSAIL (2D)              & 1045  & 1171  & 31.50 & 31.47   & {31.49} & \textbf{31.47}      & {\textbf{10}}  & {197}        \\ \hline
			Intel Research Lab (2D) & 1228  & 1483  & 396.6 & 393.7   & {428.89} &  \textbf{393.7}      & {\textbf{10}}  & {187}        \\ \hline
			Manhattan (2D)          & 3500  & 5453  & 369.0 & 193.9   & {242.05} &  \textbf{194.0}      & {1585} & \textbf{785}        \\ \hline
			KITTI 00 (2D)           & 4541  & 4676  & 1194  & 125.7   & {269.87} & \textbf{125.7}      & {\textbf{1485}} & {2750}       \\ \hline
			City10000 (2D)            & 10000 & 20687 & 5395  & 638.6   & {2975.2} & \textbf{638.7}      & {2465} & \textbf{1646}       \\ \hline
			Parking Garage (3D)     & 1661  & 6275  & 1.64  & 1.263   & {1.33}  & \textbf{1.311}      & {\textbf{25}} & {47}         \\ \hline
			Sphere (3D)             & 2500  & 4949  & 1892  & 1687    & {1689}  & \textbf{1687}       & {300}  & \textbf{53}         \\ \hline
			Torus (3D)              & 5000  & 9048  & 24617 & 24227   & {24246} & \textbf{24227}      & {100}  & \textbf{88}         \\ \hline
			Cubicle (3D)            & 5750  & 16869 & 786.0 & 717.1   & {726.69} & \textbf{717.1}      & {\textbf{45}}  & 556        \\ \hline
			Rim (3D)                & 10195 & 29743 & 8177 & 5461    & {5960.4}  & \textbf{5461}       & {\textbf{515}} & 1563       \\ \hline
	\end{tabular}}
	\label{tab:slam_benchmarks}
\end{table*}

Lastly, we evaluate \AlgName\ on benchmark datasets. Figure~\ref{fig:datasets} shows the globally optimal solutions returned by our algorithm.
In Table~\ref{tab:slam_benchmarks},
we compare the performance of \AlgName\ against the centralized certifiable SE-Sync algorithm \cite{Rosen19IJRR}, as well as the state-of-the-art distributed Gauss-Seidel (DGS) algorithm by Choudhary et al. \cite{Choudhary17IJRR}.
For DGS, we set the SOR parameter to 1.0 as recommended by the authors, and for which we also observe stable performance in general.
On all datasets, \AlgName\ is able to verify its solution as the global minimizer. 
We note that on some datasets, the final objective value is slightly higher than SE-Sync.
This is due to the looser convergence condition used in our distributed local search:
for $\ARBCD$ we set the gradient norm threshold to $10^{-1}$, while for SE-Sync we set the threshold to $10^{-6}$ in order to obtain a high-precision reference solution.\footnote{In general it is not reasonable to expect $\RBCD$ or $\ARBCD$ to produce solutions that are as precise as those achievable by SE-Sync in tractable time, since the former are \emph{first-order} methods, while the latter is \emph{second-order}.}
On the other hand, \AlgName\ is clearly more advantageous compared to DGS, 
as it returns a global minimum, often with fewer iterations. 
The performance of DGS is more sensitive to the quality of initialization, as manifested on the \texttt{Killian Court} and \texttt{City10000} datasets. 
In addition, while our local search methods are guaranteed to reduce the objective value at each iteration, DGS does \emph{not} have this guarantee as it is operating on a linearized approximation of the PGO problem. 

\begin{figure}[t]
	\centering
	\begin{subfigure}[t]{0.30\textwidth}
		\centering
		\includegraphics[width=\textwidth]{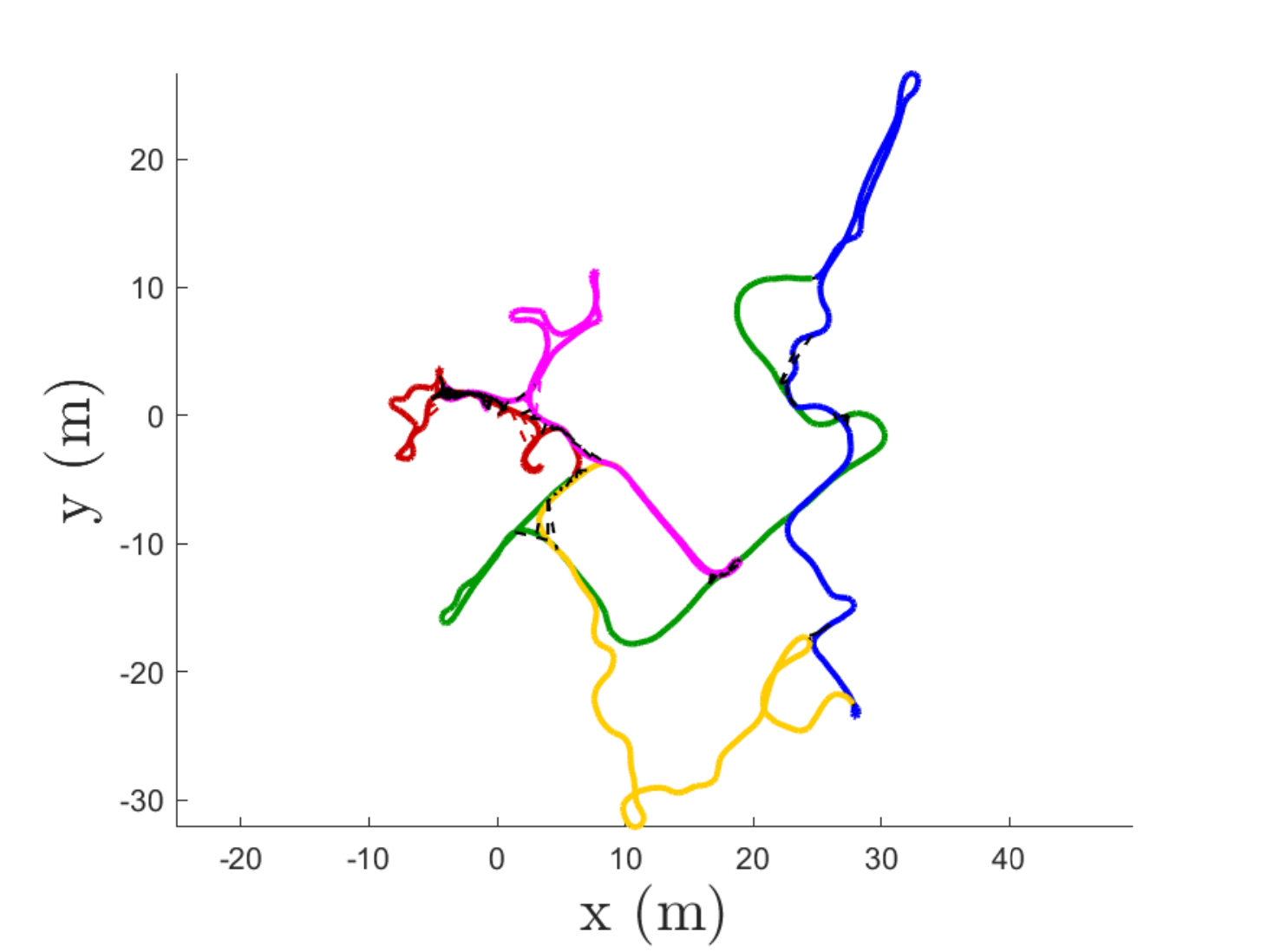}
		\caption{\small CSAIL}
	\end{subfigure}
	~
	\begin{subfigure}[t]{0.30\textwidth}
		\centering
		\includegraphics[width=\textwidth]{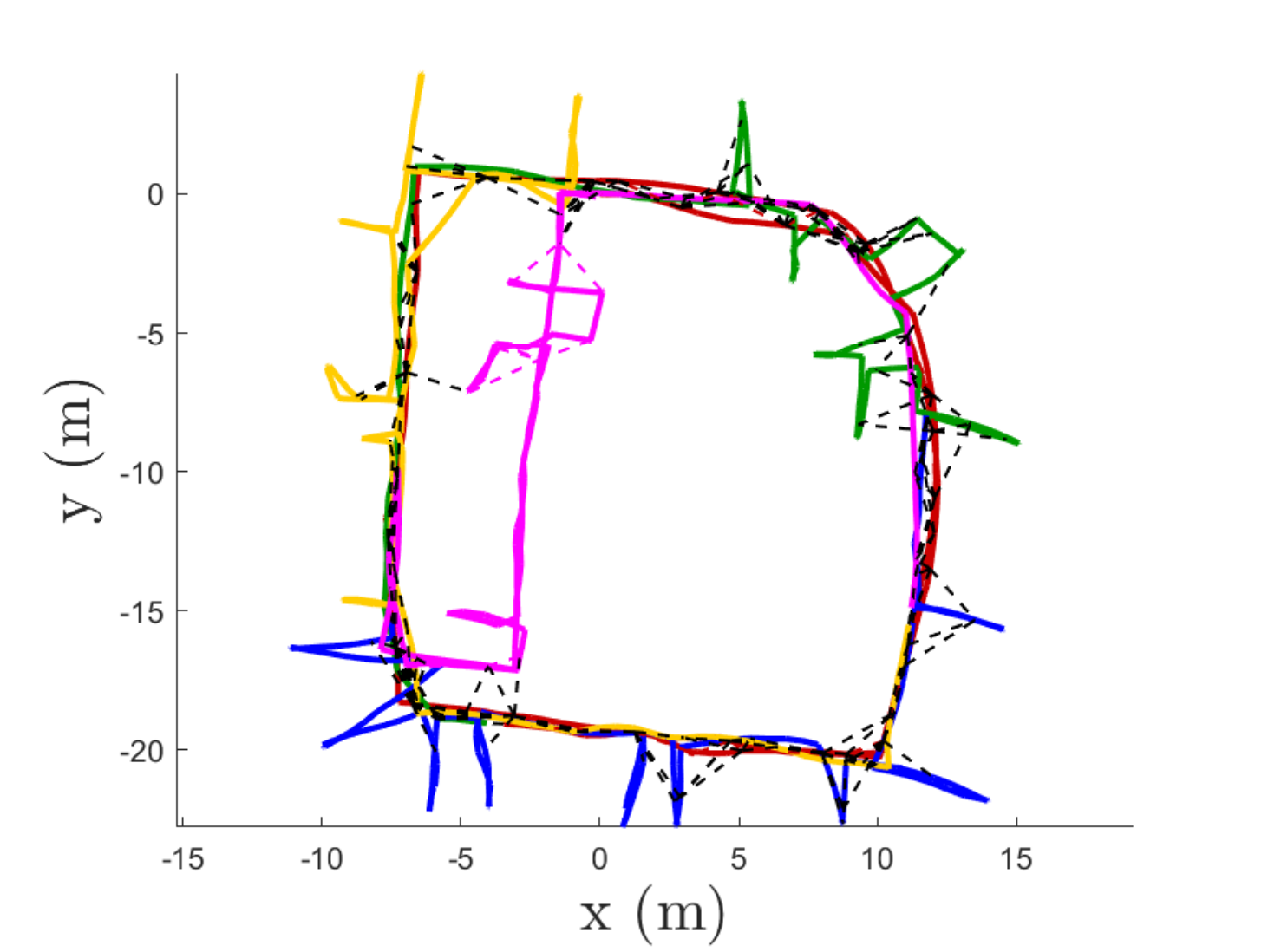}
		\caption{\small Intel Research Lab}
	\end{subfigure}
	~
	\begin{subfigure}[t]{0.30\textwidth}
		\centering
		\includegraphics[width=\textwidth]{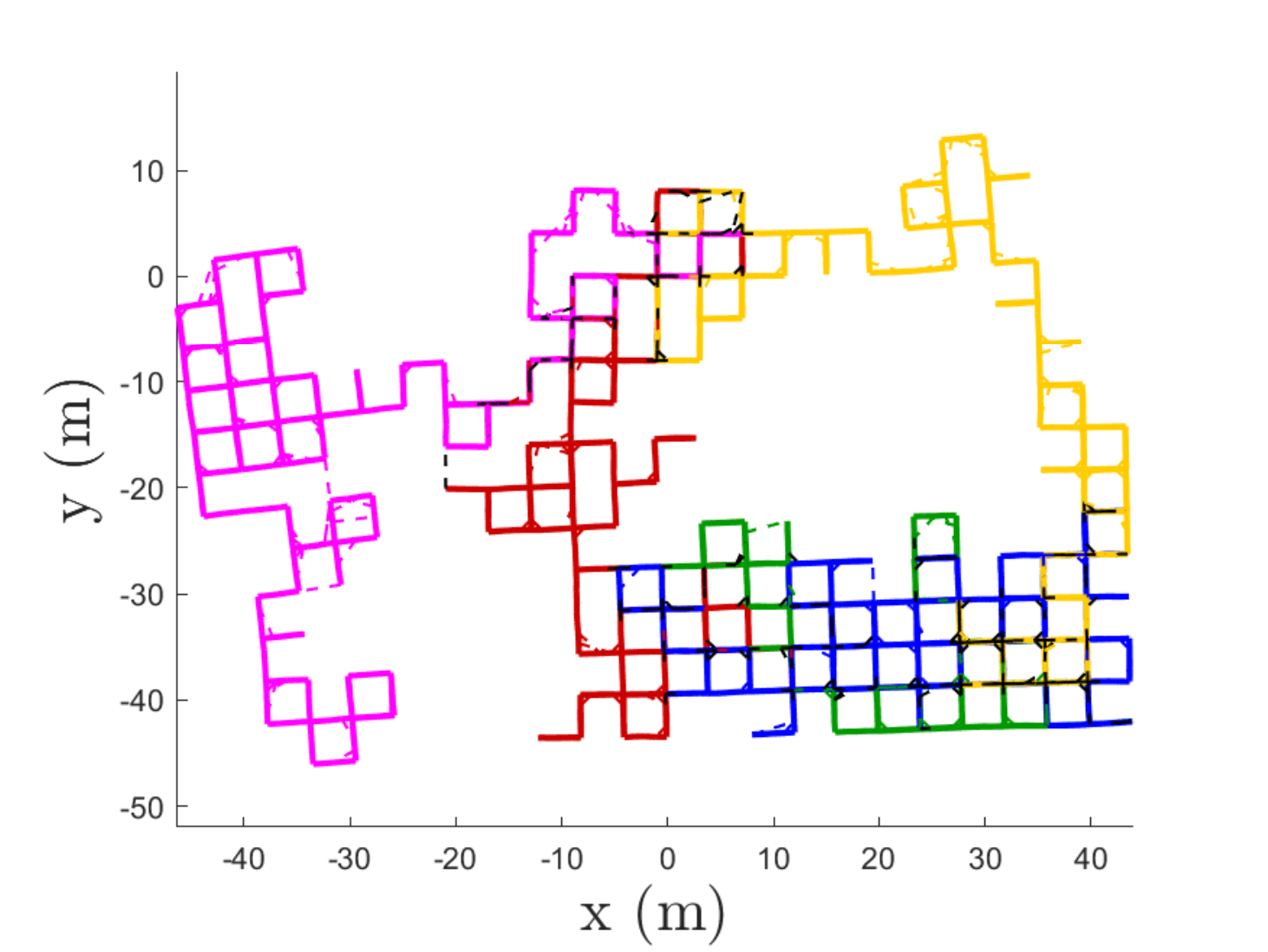}
		\caption{\small Manhattan}
	\end{subfigure}
	\\
	\begin{subfigure}[t]{0.30\textwidth}
		\centering
		\includegraphics[width=\textwidth]{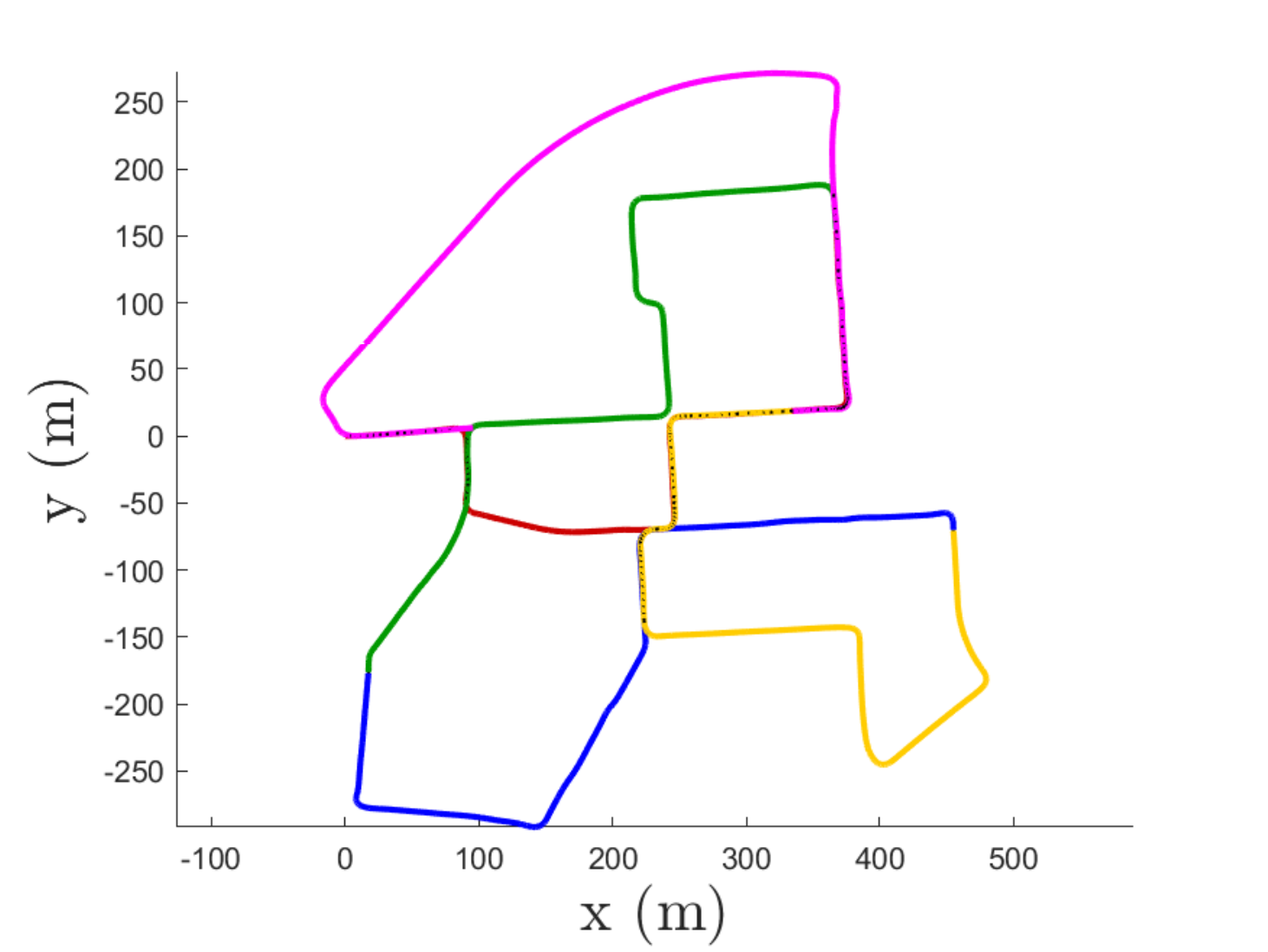}
		\caption{\small KITTI 00}
	\end{subfigure}
	~
	\begin{subfigure}[t]{0.30\textwidth}
		\centering
		\includegraphics[width=\textwidth]{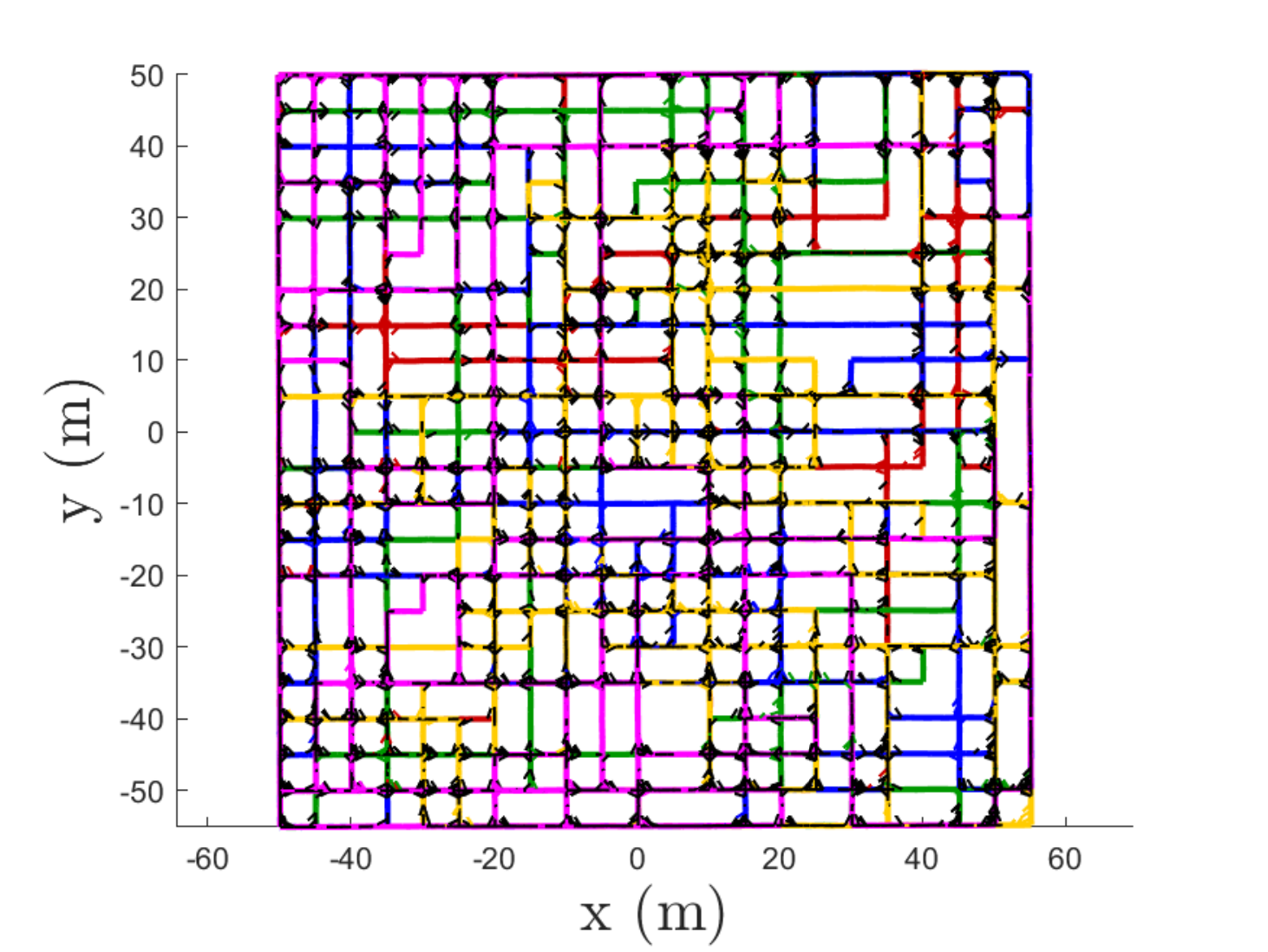}
		\caption{\small City10000}
	\end{subfigure}
	~
	\begin{subfigure}[t]{0.30\textwidth}
		\centering
		\includegraphics[width=\textwidth]{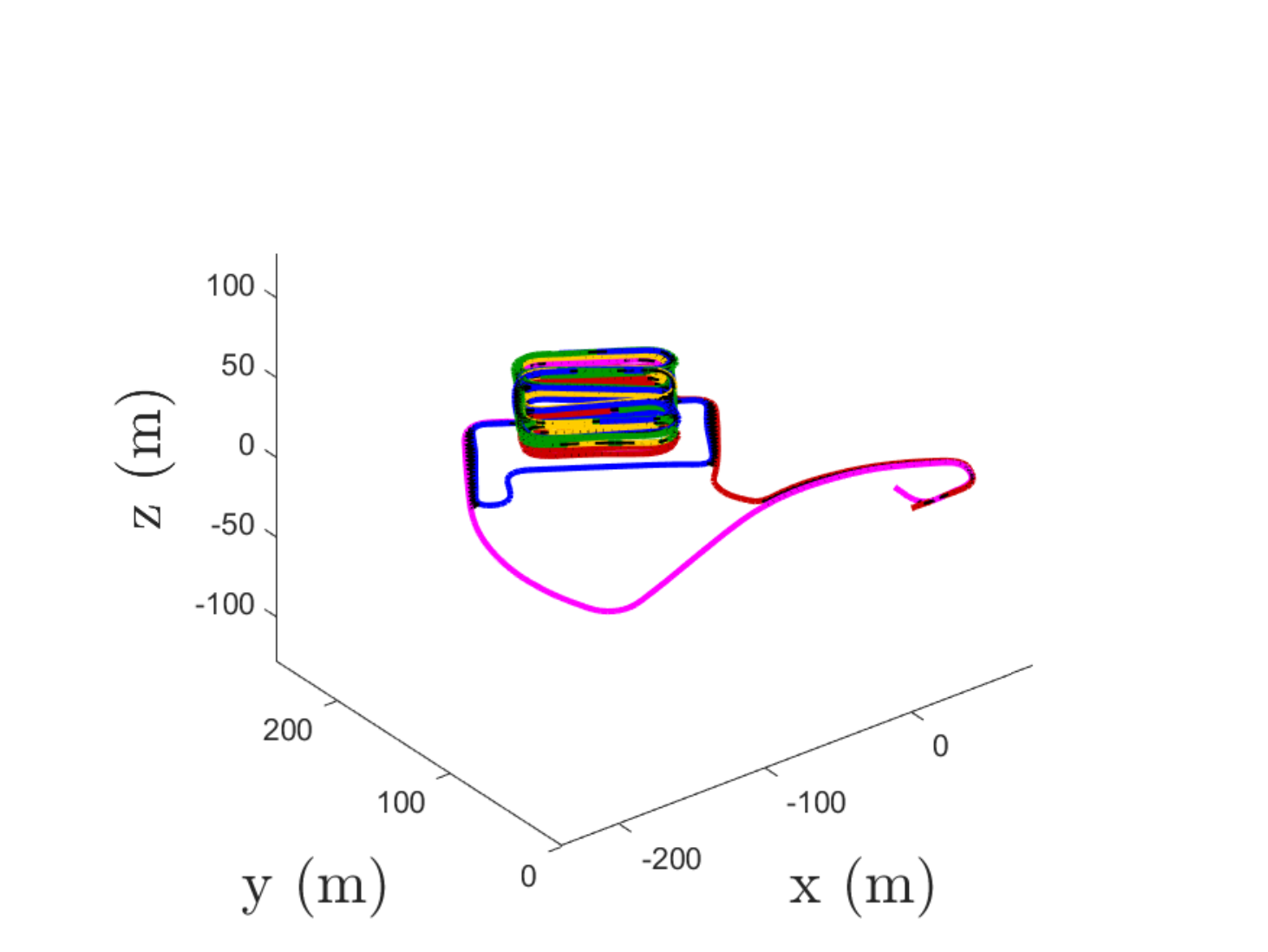}
		\caption{\small Parking Garage}
	\end{subfigure}
	\\
	\begin{subfigure}[t]{0.30\textwidth}
		\centering
		\includegraphics[width=\textwidth]{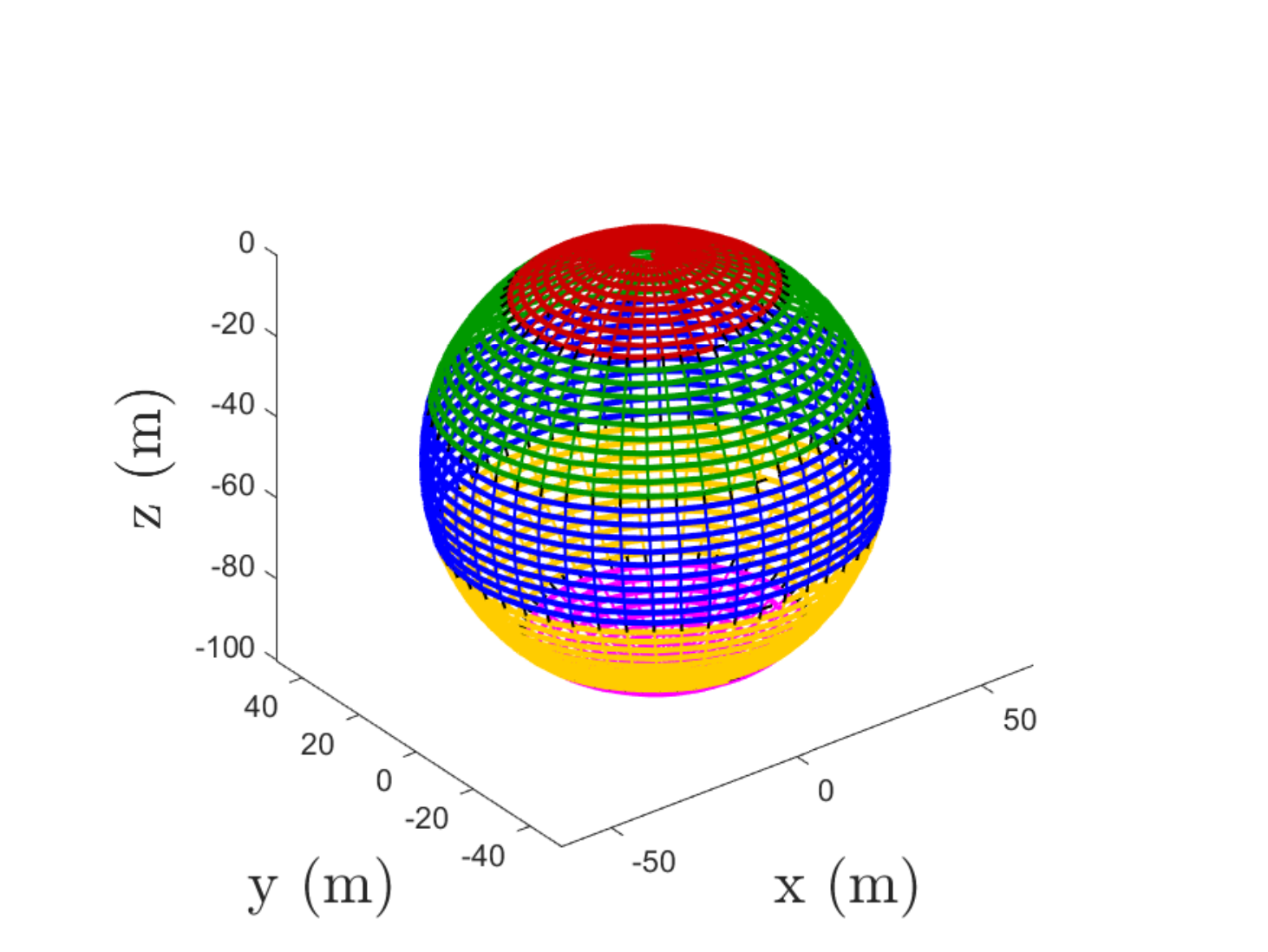}
		\caption{\small Sphere}
	\end{subfigure}
	~
	\begin{subfigure}[t]{0.30\textwidth}
		\centering
		\includegraphics[width=\textwidth]{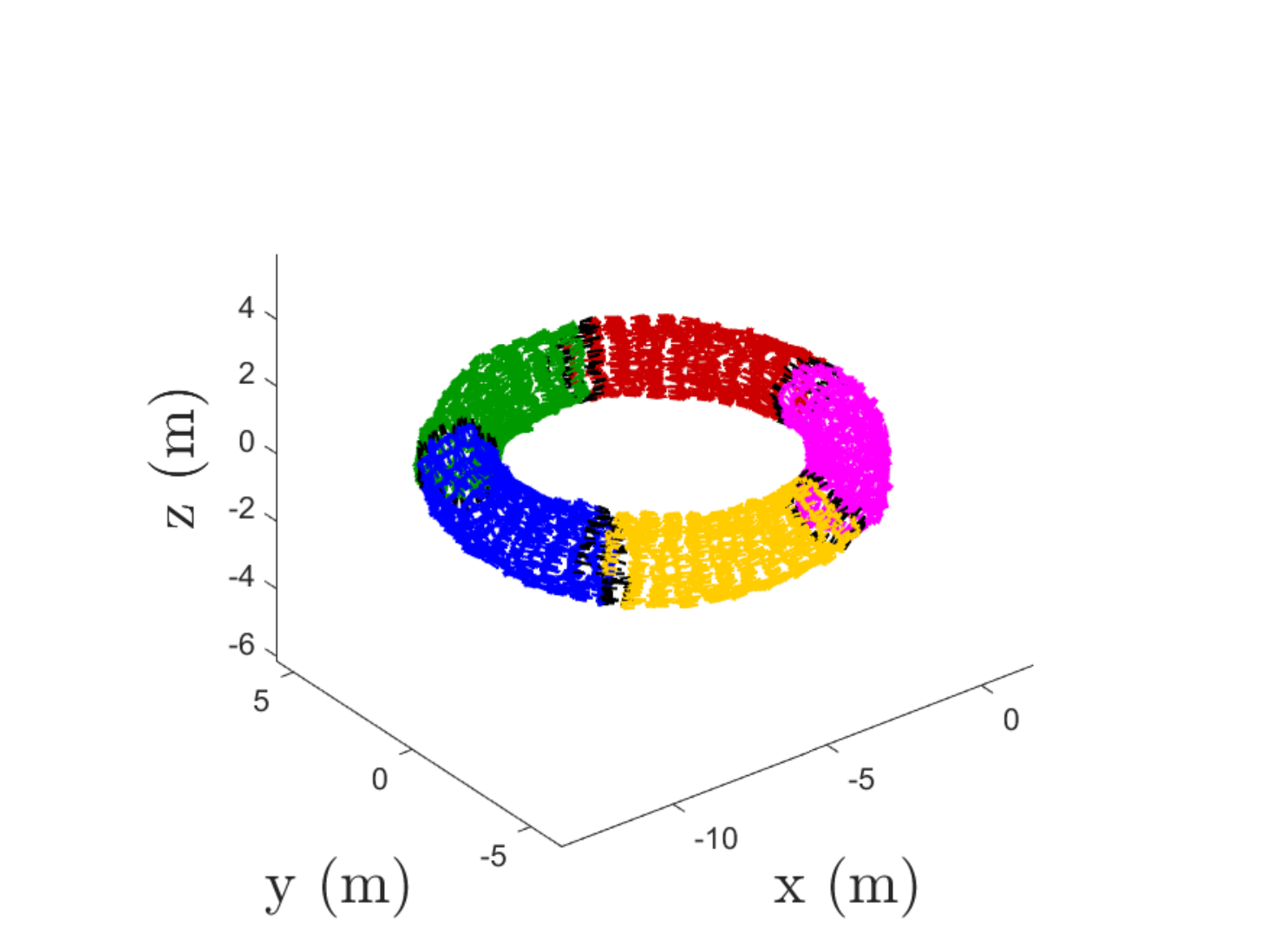}
		\caption{\small Torus}
	\end{subfigure}
	~
	\begin{subfigure}[t]{0.30\textwidth}
		\centering
		\includegraphics[width=\textwidth]{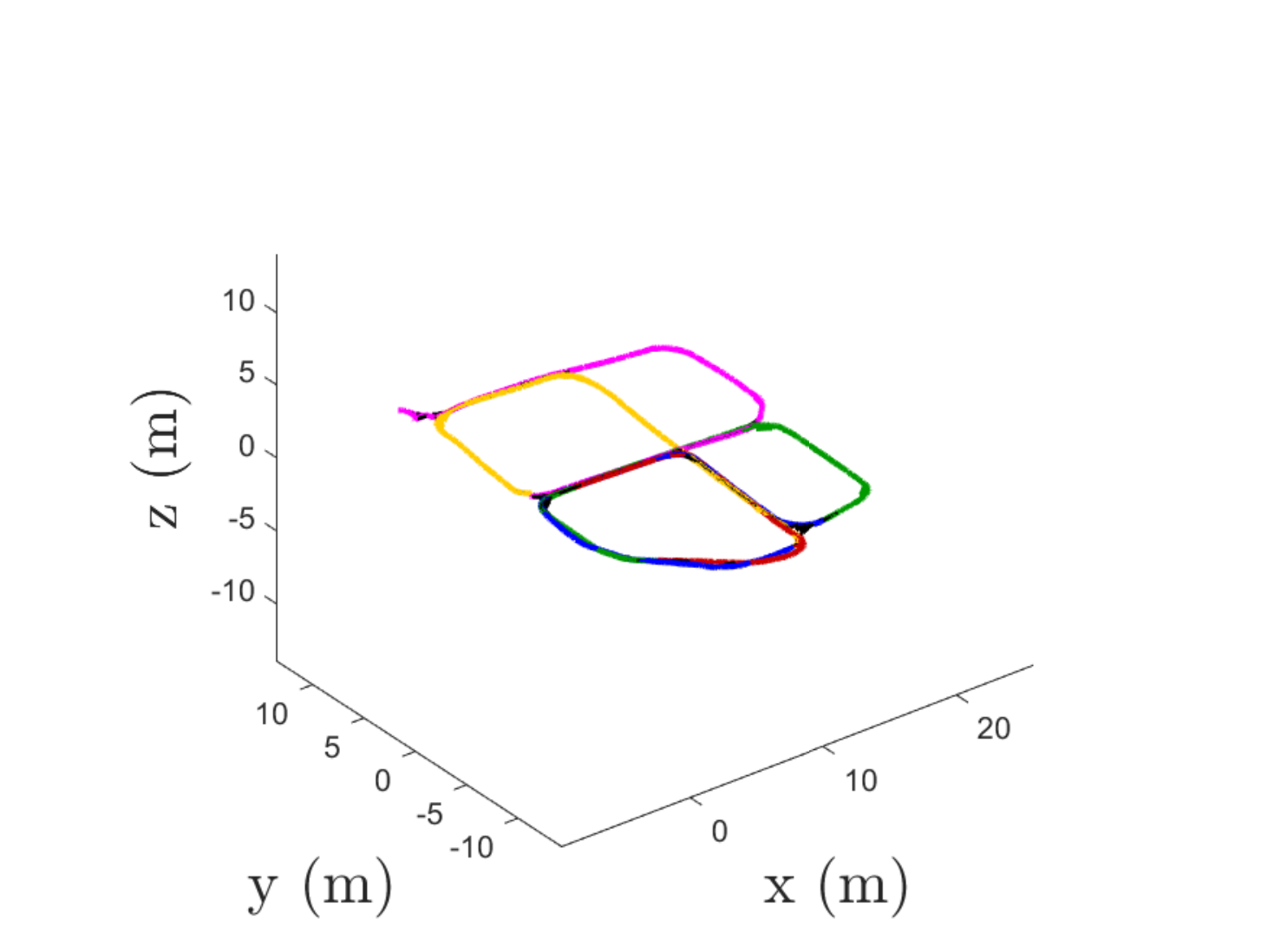}
		\caption{\small Cubicle}
	\end{subfigure}
	\caption{\small  
		Globally optimal estimates returned by \AlgName\ (Algorithm~\ref{alg:dpgo}) on benchmark datasets.
	}
	\label{fig:datasets}
\end{figure}

\begin{figure}[t]
	\centering
	\begin{subfigure}[t]{0.23\textwidth}
		\centering
		\includegraphics[trim=50 10 50 20, clip, width=\textwidth]
		{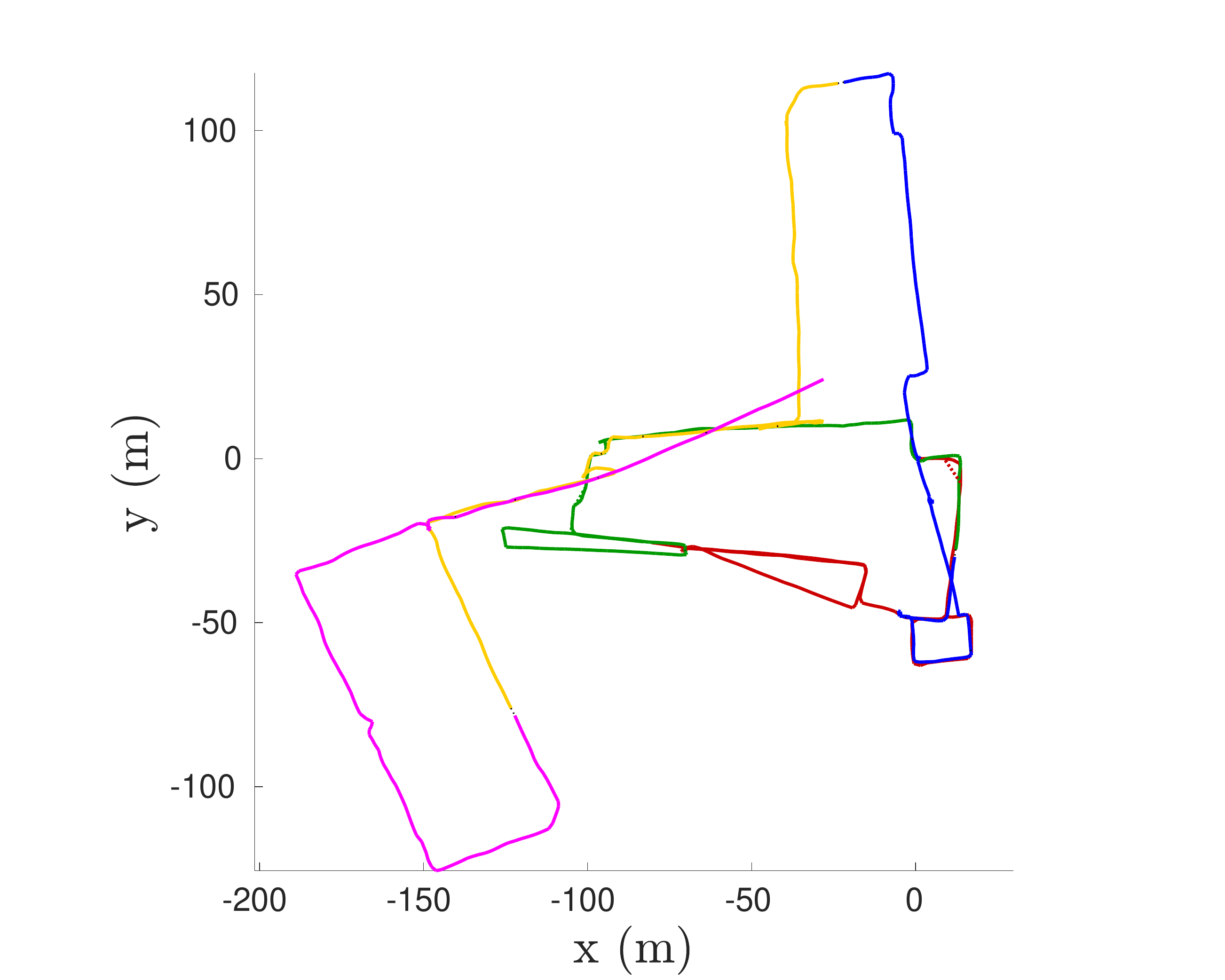}
		\caption{\small  Suboptimal critical point at rank  $r=3$}
		\label{fig:escape_local_min}
	\end{subfigure}
	~
	\begin{subfigure}[t]{0.23\textwidth}
		\centering
		\includegraphics[trim=50 10 50 20, clip, width=\textwidth]
		{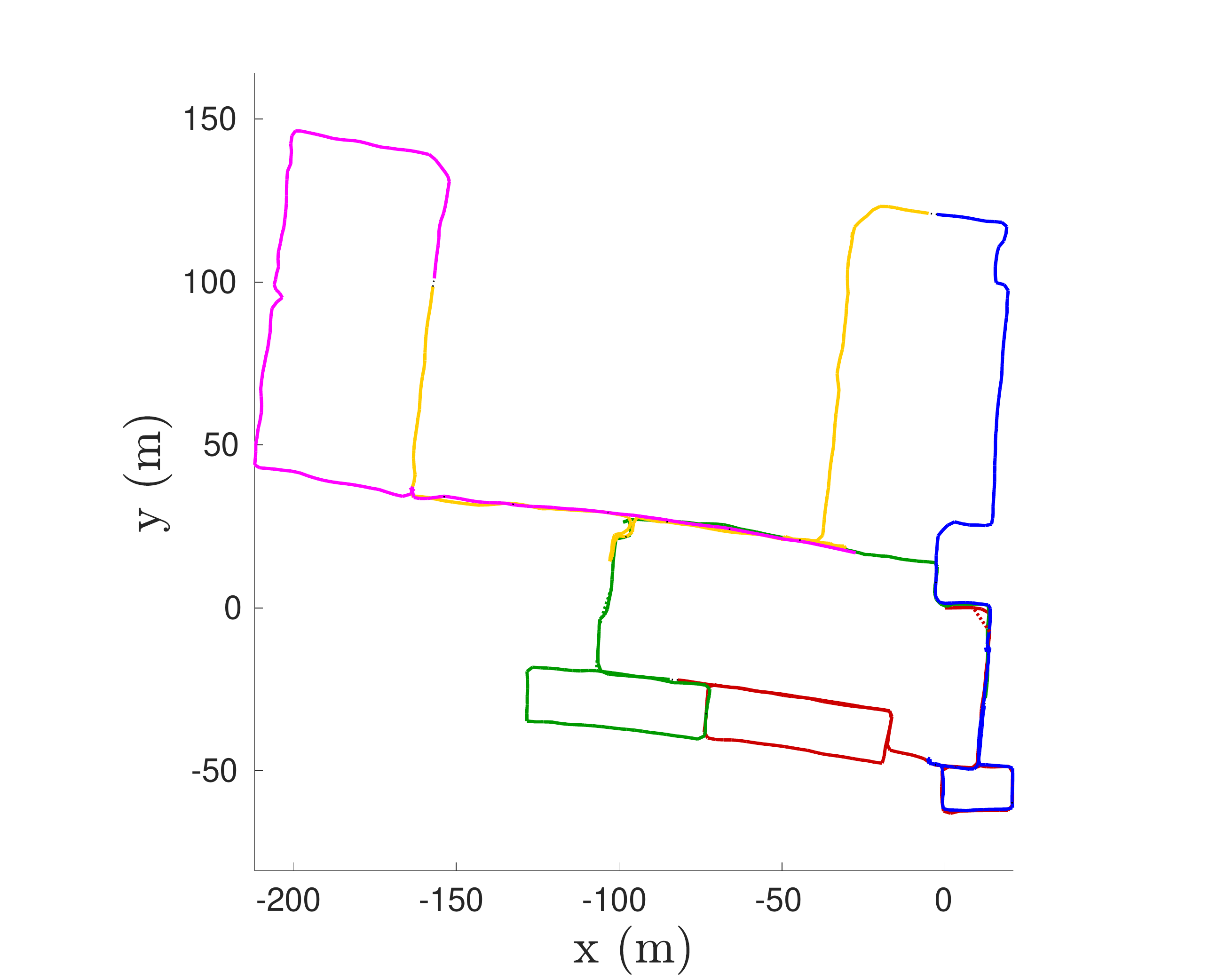}
		\caption{\small Certified global miminizer at rank $r = 4$}
		\label{fig:escape_global_min}
	\end{subfigure}
	~
	\begin{subfigure}[t]{0.23\textwidth}
		\centering
		\includegraphics[trim=5 10 40 20, clip, width=\textwidth]
		{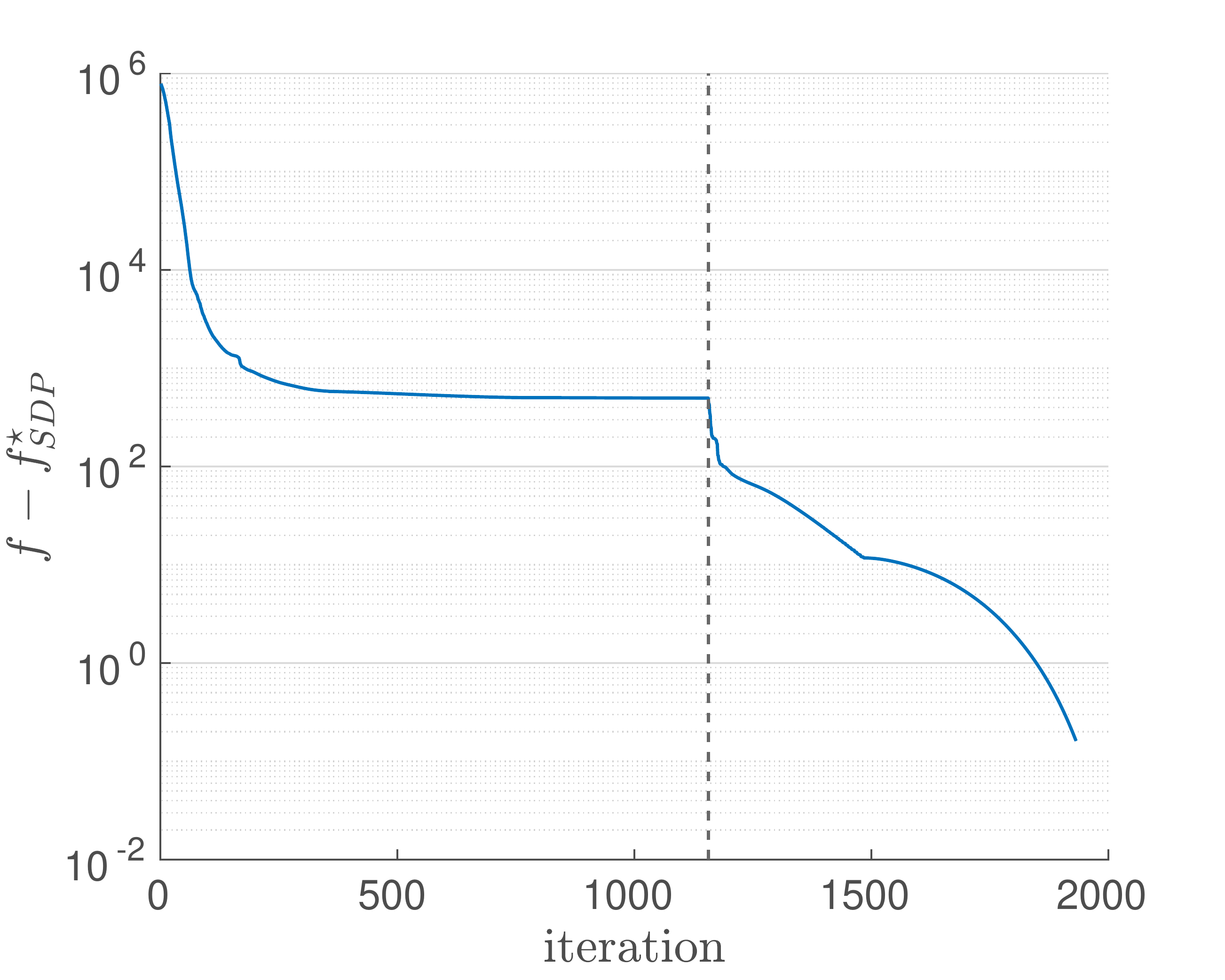}
		\caption{\small Combined optimality gap ($r=3$ and $r=4$) during local search}
		\label{fig:escape_optgap}
	\end{subfigure}
	~
	\begin{subfigure}[t]{0.23\textwidth}
		\centering
		\includegraphics[trim=5 10 40 20, clip, width=\textwidth]
		{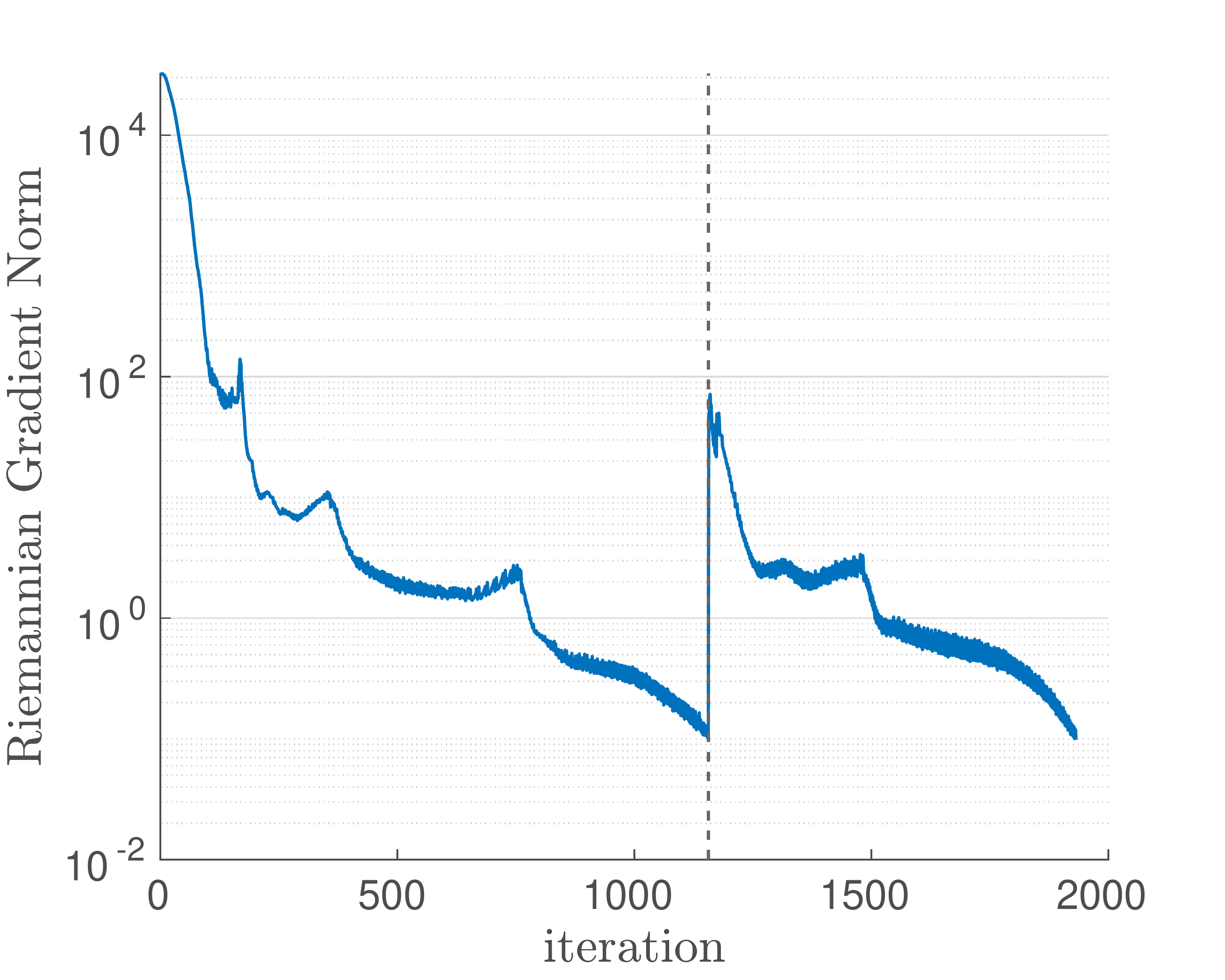}
		\caption{\small  Combined Riemannian gradient norm ($r=3$ and $r=4$) during local search}
		\label{fig:escape_rgradnorm}
	\end{subfigure}
	\caption{\small  
		\AlgName\ (Algorithm~\ref{alg:dpgo}) returns the global minimizer of the \texttt{Killian court} dataset even from random initialization.  
		(a) From random initialization, $\ARBCD$ converges to a suboptimal critical point at rank $r = 3$. 
		(b) Via distributed verification and saddle point escaping, our algorithm is able to escape the suboptimal solution and converges to the global minimizer at rank $r = 4$. 
		(c) Evolution of optimality gap across $r = 3$ and $r = 4$. 
		(d) Evolution of Riemannian gradient norm across $r = 3$ and $r=4$.
		The vertical dashed line indicates the transition from $r=3$ to $r=4$.
	}
	\label{fig:escape_experiment}
\end{figure} 

To further demonstrate the uniqueness of our algorithm as a \emph{global} solver, we show that it is able to converge to the global minimum even from random initialization. 
This is illustrated using the \texttt{Killian court} dataset in Figure~\ref{fig:escape_experiment}. 
Due to the random initialization, the first round of distributed local search at rank $r=3$ converges to a suboptimal critical point. 
This can be seen in Figure~\ref{fig:escape_optgap}, the the optimality gap at $r=3$
(first 1068 iterations) converges to a non-zero value. 
From distributed verification, our algorithm detects the solution as a saddle point and is able to escape and converges to the correct global minimizer at rank $r=4$. 
This can be clearly seen in Figure~\ref{fig:escape_rgradnorm}, where escaping successfully moves the iterate to a position with large gradient norm, from where local search can successfully descend to the global minimizer.


\section{Conclusion}
In this work we proposed the first \emph{certifiably correct} algorithm for \emph{distributed} pose-graph optimization.  Our method is based upon a sparse semidefinite relaxation of the pose-graph optimization problem that we prove enjoys the same exactness guarantees as current state-of-the-art centralized methods \cite{Rosen19IJRR}: namely, that its minimizers are \emph{low-rank} and provide \emph{globally optimal solutions} of the original PGO problem under moderate noise.  To solve large-scale instances of this relaxation in the distributed setting, we leveraged the existence of low-rank solutions to propose a distributed Riemannian Staircase framework, employing Riemannian block coordinate descent as the core distributed optimization method.  We proved that RBCD enjoys a \emph{global} sublinear convergence rate under standard (mild) conditions, and can be significantly accelerated using Nestorov's  scheme.  We also developed the first distributed solution verification and saddle escape algorithms to \emph{certify} the optimality of critical points recovered via RBCD, and to descend from suboptimal critical points if necessary. 
Finally, we provided extensive numerical evaluations, demonstrating that the proposed approach correctly recovers globally optimal solutions 
under moderate noise, and outperforms alternative distributed methods in terms of estimation quality 
and convergence speed.

In the future, we plan to study approaches to further improve iteration complexity on hard problem instances.
Another interesting direction is the design of certifiably correct distributed algorithm that can handle \emph{outlier} measurements in CSLAM and CNL.
While robustness has gained increasing attention in multi-robot SLAM \cite{Mangleson2018PCM,lajoie2019door}, recent \emph{certifiably robust} approaches such as \cite{Lajoie2019RAL} are still restricted to centralized setting due to high computational requirements.

Lastly, similar to other distributed algorithms,
the performance of \AlgName\ is expected to degrade as network conditions deteriorate (e.g., as network topology becomes sparser or delay increases).
To this end, our recent work \cite{Tian20ASAPP} studied convergence of distributed PGO under communication delay.
In general, however, designing distributed SLAM systems that better handle real-world communication challenges remains an important direction for future research.

\subsection*{Acknowledgments}
The authors would like to thank Luca Carlone for fruitful discussions 
that led to this work.
This work was supported in part by the NASA Convergent
Aeronautics Solutions project Design Environment for Novel
Vertical Lift Vehicles (DELIVER), by ONR under BRC award N000141712072, and by
ARL DCIST under Cooperative Agreement Number W911NF-17-2-0181.

\bibliographystyle{IEEEtranN}
\bibliography{slam,optimization,slamLoopClosure}

\clearpage

\begin{appendices}

\section{Exactness of SDP Relaxation}
\label{SDP_equivalence_appendix}

\begin{table*}[t]
	\centering
	\caption{ \small List of problems considered in this work. Here $i \in
			[n]$ where $n$ denotes the total number of poses in the (collective)
			pose graph. The dimension of the problem in denoted
		by $d \in \{2,3\}$. We note that Problem~\ref{prob:se_sdp_marg} and
\ref{prob:se_sync_riemannian_marg} are not directly used in the proposed
approach, but are nonetheless crucial for establishing the performance
guarantees of the SDP relaxation for pose synchronization (Theorem~\ref{thm:sdp_equivalence} and Theorem~\ref{thm:tightness_informal}).}
	\label{tab:problems}
	\setlength{\tabcolsep}{2pt}
	\renewcommand{\arraystretch}{2}
	\footnotesize
	\begin{tabular}{|c||l||c||c||c|}
		\hline
		\hspace{0.3cm}\textbf{\#}\hspace{0.3cm} & \multicolumn{1}{c||}{\textbf{Problem Description}} &
		\hspace{0.2cm}\textbf{Cost Function}\hspace{0.2cm} &
		\textbf{Domain} &
		\textbf{Constraints} \\ \hline \hline
		\ref{prob:se_sync}  & MLE for PGO    &  \eqref{eq:se_sync_cost}                  & $(R_i,t_i) \in
		\SOd(d) \times \Rset^d$  & --       \\ \hline
		\ref{prob:se_sdp}  & Full SDP Relaxation of PGO & $F(\ZT) \triangleq \langle
\ConLapT,\ZT\rangle$   & $\ZT \in \PSD^{n+dn}$     & ${\ZT}_{[i,i](1:d,1:d)} = I_d$           \\ \hline
		\ref{prob:se_sdp_marg}  & Rotation-only SDP 
		Relaxation for PGO & $ F_\text{R}(\ZR) \triangleq \langle\ConLapM,\ZR\rangle$           & $\ZR \in \PSD^{dn}$    &
		${\ZR}_{[i,i]} = I_d$ \\ \hline
		\ref{prob:se_sync_riemannian} & Rank-Restricted Full SDP for PGO
		& $f(X) \triangleq \langle \ConLapT, X^\top X\rangle$
		& $X \in (\Stiefel(d,r) \times \mathbb{R}^r)^n$
		& --           \\ \hline
		\ref{prob:se_sync_riemannian_marg} & 
		Rank-Restricted Rotation-only SDP for PGO & $f_\text{R}(Y) \triangleq \langle\ConLapM,Y^\top Y\rangle$           & $Y \in \Stiefel(d,r)^n$	&  --
		\\ \hline
	\end{tabular}
\end{table*}

In this section, we provide additional discussions of our SDP relaxation (Problem~\ref{prob:se_sdp}) and give a proof of its exactness under low noise. 
To facilitate the discussion, we summarize all problems that are considered in this work in Table~\ref{tab:problems}.
The core idea behind our proof is to establish certain \emph{equivalence} relations with the rotation-only SDP relaxation (Problem~\ref{prob:se_sdp_marg}), which is first developed by Rosen et al. in SE-Sync \cite{Rosen19IJRR}.
To begin, we first define the cost matrix $\ConLapM$ that appears in the rotation-only SDP; see also \cite[Equation~20(b)]{Rosen19IJRR}.
\begin{equation}
\ConLapM \triangleq \ConLap + \widetilde{\Sigma} - \widetilde{V}^\top L(W^\tau)^{\pinv} \widetilde{V},
\label{eq:ConLapM_def}
\end{equation}
In \eqref{eq:ConLapM_def}, $\ConLap \in \PSD^{dn}$ is the rotation connection Laplacian,
and $L(W^\tau) \in \PSD^n$ is the graph Laplacian of the pose graph with edges weighted by the translation measurement weights $\{\tau_{ij}\}$.
The remaining two matrices $\widetilde{V} \in \mathbb{R}^{n \times dn}$ and $\widetilde{\Sigma} \in \mathbb{R}^{dn \times dn}$ are formed using relative translation measurements. 
The exact expressions of these matrices are given in equations (13)-(16) in \cite{Rosen19IJRR} and are omitted here. 
Let us consider the rank-restricted version of the rotation-only SDP defined below.
\begin{problem}[Rotation-only Rank-restricted SDP for Pose Synchronization \cite{Rosen19IJRR}]
	\normalfont
	\begin{equation}
	\underset{Y \in \Stiefel(d,r)^n}{\text{minimize}}
	\quad \langle\ConLapM,Y^\top Y\rangle.
	\end{equation}
	\label{prob:se_sync_riemannian_marg}
\end{problem}

Problem~\ref{prob:se_sync_riemannian_marg} and the \emph{sparse} rank-restricted
relaxation we solve in this paper (Problem~\ref{prob:se_sync_riemannian}) are intimately connected. 
The following lemma precisely characterizes this connection, and also provides
an important tool for proving subsequent theorems in this section.

\begin{lemma}[Connections between Problems~\ref{prob:se_sync_riemannian} and \ref{prob:se_sync_riemannian_marg}]
	\normalfont
	Let 
	$X = \left[	Y_1 \,\, p_1 \,\, \hdots \,\, Y_n \,\, p_n \right] \in \Manifold(r,n)$ be a first-order critical point of  Problem~\ref{prob:se_sync_riemannian}.
	Let $ Y = \left[
			Y_1 \,\, \hdots \,\, Y_n
	\right] \in \Stiefel(d,r)^n$
	and $p = \left[	p_1 \,\, \hdots \,\, p_n \right] \in \Real^{r \times n}$
	be block matrices constructed from the Stiefel and Euclidean elements of $X$, respectively.  Then:
	\begin{enumerate}
		\item [(i)] The translations $p \in \Real^{r \times n}$ satisfy:
		\begin{equation}
		\label{translational_component_of_first_order_critical_point}
		p \in \left\{ -Y \widetilde{V}^\top L(W^\tau)^{\pinv} + c1_n^\top
				\middle\vert c
				\in \Real^r
		\right\}.
		\end{equation}
		\item [(ii)]  $Y$ is a first-order critical point of Problem \ref{prob:se_sync_riemannian_marg}, and $\langle\ConLapT,X^\top X\rangle = \langle\ConLapM,Y^\top
		Y\rangle.$
		\item [(iii)]  Let $\bar{\Lambda}(X)$ denote the symmetric $(d\times d)$-block-diagonal matrix constructed by extracting the nonzero $(d\times d)$ diagonal blocks from the Lagrange multiplier matrix $\Lambda(X)$ defined in \eqref{Lagrange_multipliers}:
		\begin{align}
		\bar{\Lambda}(X) & \in \SBD(d,n) \\
		\bar{\Lambda}(X)_{[i,i]} & \triangleq \Lambda(X)_{[i,i](1:d, 1:d)} \quad \forall i \in [n]
		\label{Lambda_bar_definition}
		\end{align}
		(see also \eqref{eq:SymBlockDiagPlus}).   Then the Lagrange multiplier
		matrix:
		\begin{equation}
		\Lambda_{\text{R}}(Y) \triangleq \SymBlockDiag(\ConLapM Y^\top Y)
		\end{equation}
		for the simplified (rotation-only) Problem \ref{prob:se_sync_riemannian_marg} (cf.\ \cite[eq.\ (107)]{Rosen19IJRR}) satisfies:
		\begin{equation}
		\label{relation_between_Lagrange_multipliers_of_SESync_fact_and_full_PGO_fact}
		\Lambda_\text{R}(Y) = \bar{\Lambda}(X).
		\end{equation}
		\item [(iv)] Let $\ST_{\text{R}}(Y) \triangleq \ConLapM - \Lambda_{\text{R}}(Y)$ denote the certificate matrix for the simplified (rotation-only) SE-Sync relaxation Problem \ref{prob:se_sync_riemannian_marg} \cite[Thm.\ 7]{Rosen19IJRR}.  Then $\ST_{\text{R}}(Y) \succeq 0$ if and only if $\ST(X) \succeq 0$.
		\item [(v)] $X$ is a global minimizer of Problem~\ref{prob:se_sync_riemannian} if and only if $Y$ is a global minimizer to Problem~\ref{prob:se_sync_riemannian_marg}. 
	\end{enumerate}
	\label{lem:rank_restricted_SDP_equivalence}
\end{lemma}

\begin{proof}

	Using the definition of the connection Laplacian matrix $\ConLapT$ in Problem~\ref{prob:se_sync_riemannian} (see \cite[Appendix~II]{Briales17CartanSync}), it can be shown that the cost function in Problem~\ref{prob:se_sync_riemannian} can be expanded into the following,
	\begin{equation}
	\langle\ConLapT,X^\top X\rangle = 
	\sum_{(i,j) \in \DirEdges} \kappa_{ij} \norm{Y_j - Y_i \Rtilde_{ij}}_F^2 
	+ \sum_{(i,j) \in \DirEdges} \tau_{ij} \norm{p_j - p_i - Y_i \ttilde_{ij}}_2^2. 
	\label{eq:expanded_cost}
	\end{equation}
	Using \eqref{eq:expanded_cost}, we may rewrite the cost function in a way that separates the Euclidean (translation) variables from the Stiefel (rotation) variables:
	\begin{equation}
		\langle\ConLapT,X^\top X\rangle 
		= 
		\begin{bmatrix}
		\vectorize(p) \\ 
		\vectorize(Y)
		\end{bmatrix}^\top 
		(M \otimes I_r)
		\begin{bmatrix}
		\vectorize(p) \\ 
		\vectorize(Y)
		\end{bmatrix}
		=
		\tr \left(
				\big[
						p \,\,\, Y
				\big]
		M
				\big[
						p \,\,\, Y
				\big]^\top
		\right).
		\label{eq:separated_cost_function}
	\end{equation}
	\begin{equation}
	M \triangleq \begin{bmatrix}
	L(W^\tau) & \widetilde{V} \\
	\widetilde{V}^\top & \ConLap + \widetilde{\Sigma}
	\end{bmatrix}.
	\label{eq:M_definition}
	\end{equation}
	Above, the $\vectorize(\cdot)$ operator concatenates columns of the input matrix into a single vector. 
	A detailed derivation for \eqref{eq:separated_cost_function} is already presented in \cite[Appendix~B]{Rosen19IJRR} for the case when $r = d$.
	For $r \geq d$, the derivation is largely identical with minor modifications to the dimensions of certain matrices, and thus is omitted. 
	We make an additional remark that the new data matrix $M$ \eqref{eq:M_definition} is related to the original connection Laplacian $Q$ via a permutation of the columns and rows. 
	
	Let us define $f(X) \triangleq \langle\ConLapT,X^\top X\rangle$ and $f_\text{R}(Y)
	\triangleq \langle\ConLapM,Y^\top Y\rangle$. 
	From \eqref{eq:separated_cost_function}, we derive the Euclidean gradients of $f(X)$ with respect to $p$ and $Y$, respectively. 
	\begin{align}
			\nabla_p f(X) &= 2 \left(p L(W^\tau) + Y \widetilde{V}^\top\right), \label{eq:p_gradient} \\
	\nabla_Y f(X) &= 2 \left(p \widetilde{V} + Y (\ConLap +
	\widetilde{\Sigma})\right) \label{eq:Y_gradient}.
	\end{align}
	Similarly, we also have:
	\begin{equation}
			\nabla_Y f_\text{R}(Y) = 2\,Y \ConLapM.
	\end{equation}
	
	\emph{Part (i):}  Since $X$ is a first-order critical point, the Euclidean gradient with respect to the translations must be zero. 
	In light of \eqref{eq:p_gradient}, we need to identify $p$ such that,
	\begin{equation}
		pL(W^\tau) + Y \widetilde{V}^\top = 0. 
	\end{equation}
	Using the general fact that $\vectorize(ABC) = (C^\top \otimes A) \vectorize(B)$, we may vectorize the above system of equations to,
	\begin{equation}
		(L(W^\tau) \otimes I_r) \vectorize(p) + \vectorize(Y \widetilde{V}^\top) = 0. 
		\label{eq:optimal_translation_system_of_equations}
	\end{equation}
	Define $A \triangleq L(W^\tau) \otimes I_r$ and $b \triangleq \vectorize(Y \widetilde{V}^\top)$.
	Since $A$ is the Kronecker product between the Laplacian of a connected graph and the identity matrix, it holds that $\rank(A) = rn-r$ and
	the kernel of $A$ is spanned by the columns of $U \triangleq 1_n \otimes I_r$. 
	We can equivalently express $\kernel(A)$ as (cf. \cite[Equation (76)]{Rosen19IJRR}),
	\begin{equation}
	\kernel(A) = \{Uc \,|  \,c \in \Real^r \} = \{\vectorize(c1_n^\top) \, | \, c \in \Real^r\}.
	\end{equation}
	Let $u = \vectorize(c1_n^\top) \in \kernel(A)$ be an arbitrary null vector of $A$.
	Consider the inner product between $u$ and $b$,
	\begin{equation}
	\langle b, u \rangle = \langle Y \widetilde{V}^\top, c1_n^\top \rangle = \tr(1_n^\top \widetilde{V} Y^\top c) = 0.
	\label{eq:b_orthogonal_to_u}
	\end{equation}
	The last equality is due to the fact that $1_n^\top \widetilde{V} = 0$ by the definition of $\widetilde{V}$; see equation (15) in \cite{Rosen19IJRR}.
	Therefore, we have proven that $b \perp \kernel(A)$. 
	Since $A$ is symmetric positive-semidefinite, it holds that $b \in \text{range}(A)$. 
	Thus, the system of linear equations \eqref{eq:optimal_translation_system_of_equations} admits infinitely many solutions, characterized by the following set,
	\begin{equation}
	\vectorize(p) \in \left\{ 
	-A^{\pinv} b + Uc \, \middle\vert \, c \in \Real^r
	\right\}.
	\label{eq:optimal_translation_vectorized}
	\end{equation}
	Recalling the definitions of $A$ and $b$, we can convert 
	\eqref{eq:optimal_translation_vectorized} back to matrix form (cf.
	\cite[Equation~(21)]{Rosen19IJRR}):
	\begin{equation}
	p \in \left\{-Y \widetilde{V}^\top L(W^\tau)^{\pinv} + c1_n^\top \,
	\middle\vert \, c \in \Real^r \right\}.
	\label{eq:optimal_translation}
	\end{equation}
	
	\emph{Part (ii):}  Substituting the translation expression
	\eqref{eq:optimal_translation} into the Euclidean gradient with respect to $Y$ \eqref{eq:Y_gradient}, we obtain:
	\begin{equation}
	\label{equality_of_Euclidean_gradients_with_respect_to_Y}
	\begin{split}
	\nabla_Y f(X) &= 2 \left((-Y \widetilde{V}^\top L(W^\tau)^{\pinv} +
	c1_n^\top) \widetilde{V} + Y (\ConLap + \widetilde{\Sigma})\right) \\
	&= 2\left(-Y \widetilde{V}^\top L(W^\tau)^{\pinv}\widetilde{V} + Y (\ConLap
	+ \widetilde{\Sigma})\right) \\
	&= 2Y \ConLapM \\
	&= \nabla f_\text{R}(Y),
	\end{split}
	\end{equation}
	where in the second line we again used the fact that the all-$1$s vector $1_n$ belongs to the null space of $\widetilde{V}$.
	Thus, we have shown that Problems \ref{prob:se_sync_riemannian} and \ref{prob:se_sync_riemannian_marg} have the same Euclidean gradient with respect to the Stiefel elements $Y$.
	Since $X$ is a first-order critical point, the Riemannian gradient of $f(X)$ with respect to $Y$ is zero, which implies that: 
	\begin{equation}
	\rgrad_Y f(X) = \proj_Y (\nabla_Y f(X)) = \proj_Y (\nabla f_\text{R}(Y)) = 
	\rgrad f_\text{R} (Y) = 0,
	\end{equation} 
	i.e., $Y$ is a first-order critical point of
	Problem~\ref{prob:se_sync_riemannian_marg}. Finally,  plugging the
	expression of $p$ into $f(X)$ shows that $	\langle\ConLapT,X^\top X\rangle = \langle\ConLapM,Y^\top
	Y\rangle$.
	
	\emph{Part (iii):} 
	The Lagrange multiplier matrices $\Lambda(X)$ and $\Lambda_{\text{R}}(Y)$ for Problem~\ref{prob:se_sync_riemannian} and Problem~\ref{prob:se_sync_riemannian_marg} are:
	\begin{align}
	\Lambda(X)  
	&= \SymBlockDiag_d^{+} \left (X^\top X \ConLapT \right)
	= \frac{1}{2} \SymBlockDiag_d^{+} \left (X^\top \nabla_X f(X)\right), \\
	\Lambda_{\text{R}}(Y) 
	&= \SymBlockDiag\left(Y^\top Y \ConLapM \right)
	= \frac{1}{2} \SymBlockDiag\left(Y^\top \nabla_Y f_{\text{R}}(Y) \right);
	\label{eq:Lagrange_multipliers_SE-Sync}
	\end{align}
	see \eqref{Lagrange_multipliers} and \cite[eq.\ (107)]{Rosen19IJRR}, respectively.  Extracting and aggregating the nonzero diagonal blocks $\bar{\Lambda}(X)$ of $\Lambda(X)$ (cf.\ \eqref{eq:SymBlockDiagPlus}), we obtain:
	\begin{equation}
	\bar{\Lambda}(X) = \frac{1}{2}\SymBlockDiag_d\left(Y^\top \nabla_Y f(X) \right)  = \frac{1}{2}\SymBlockDiag_d\left(Y^\top \nabla_Y f_{\text{R}}(Y) \right) = \Lambda_{\text{R}}(Y),
	\end{equation}
	where we have used \eqref{equality_of_Euclidean_gradients_with_respect_to_Y} for the middle equality.
	
	\emph{Part (iv):}  After permutation, the certificate matrix $\ST(X)$ defined in \eqref{certificate_matrix} can be written in the block form:
	\begin{equation}
	S(X) =  \begin{bmatrix}
	L(W^\tau) & \widetilde{V} \\
	\widetilde{V}^\top & \ConLap + \widetilde{\Sigma} - \bar{\Lambda}(X)
	\end{bmatrix}
	\end{equation}
	(cf.\ \eqref{eq:M_definition}).  Since $L(W^\tau) \succeq 0$ and $(I - L(W^\tau)L(W^\tau)^\dagger)\widetilde{V} = 0$, it follows from \cite[Thm.\ 4.3]{Gallier2010Schur} that $S(X) \succeq 0$ if only if the following generalized Schur complement of $S(X)$ with respect to $L(W^\tau)$ is positive semidefinite:
	\begin{equation}
	\left(\ConLap + \widetilde{\Sigma} - \bar{\Lambda}(X) \right) - \widetilde{V}^\top L(W^\tau)^\dagger \widetilde{V} = \ConLapM - \Lambda_\text{R}(Y) = S_\text{R}(Y),
	\end{equation}
	where we have used the substitutions \eqref{eq:ConLapM_def} and \eqref{relation_between_Lagrange_multipliers_of_SESync_fact_and_full_PGO_fact}.
	
	\emph{Part (v):} 
	We show that Problem~\ref{prob:se_sync_riemannian_marg} is obtained from Problem~\ref{prob:se_sync_riemannian} after analytically eliminating the translations $p$. 
	Consider the problem of minimizing the cost function
	\eqref{eq:separated_cost_function} with respect to translations $p$ only (as
	a function of $Y$). 
	Since this is an unconstrained convex quadratic problem, we can minimize
	this cost first with respect to $p$ by setting the corresponding gradient to zero.
	In part (i) we identified the set of all translations (for a fixed $Y$)
	satisfying this condition; see equation \eqref{eq:optimal_translation}.
	After replacing $p$ in the original cost function
	\eqref{eq:separated_cost_function} with any element from $\{-Y \widetilde{V}^\top L(W^\tau)^{\pinv} + c1_n^\top \,
	\vert \, c \in \Real^r \}$, we obtain the following rotation-only problem,
	\begin{equation}
	\normalfont
	\underset{Y \in \Stiefel(d,r)^n}{\text{minimize}}
	\quad 
	\tr(Y
	(\ConLap + \widetilde{\Sigma} - \widetilde{V}^\top L(W^\tau)^{\pinv} \widetilde{V})
	Y^\top) = 
	\tr(\ConLapM Y^\top Y),
	\end{equation}
	which is exactly Problem~\ref{prob:se_sync_riemannian_marg}. This concludes our proof. 
\end{proof}

\subsection{Proof of Theorem~\ref{thm:sdp_equivalence}}
\label{sec:equivalence_proof}

With Lemma~\ref{lem:rank_restricted_SDP_equivalence} in place, we are ready to prove the equivalence relations between our SDP relaxation (Problem~\ref{prob:se_sdp}) and its rotation-only version (Problem~\ref{prob:se_sdp_marg}), stated in Theorem~\ref{thm:sdp_equivalence} of the main paper. 

\begin{proof}[Proof of Theorem~\ref{thm:sdp_equivalence}]
	We give a constructive proof where we show that from a minimizer $\ZT^\star
	\in \PSD^{n+dn}$ to Problem~\ref{prob:se_sdp} (full SDP relaxation),
	we can recover a minimizer $\ZR^\star \in \PSD^{dn}$ to
	Problem~\ref{prob:se_sdp_marg} (rotation-only SDP relaxation) {with the same rank}, and vice versa. 
	Without loss of generality, let $r^\star = \rank(Z^\star) \geq d$. 
	Consider the rank-$r^\star$ factorization $Z^\star = (X^\star)^\top X^\star$. 
	Since $Z^\star$ is a feasible point for Problem~\ref{prob:se_sdp}, it can be readily verified that $X^\star$ is
	an element of the product manifold $\Manifold(r^\star,n)$
	\eqref{eq:my_manifold}. 
	Let $Y^\star \in \Stiefel(d,r^\star)^n$ be obtained by stacking all rotational components of $X^\star$. 
	We prove that $\ZR^\star = (Y^\star)^\top Y^\star$ is an optimal solution to Problem~\ref{prob:se_sdp_marg}. 
	To see this, first note that $X^\star$ is an optimal solution to the rank-restricted SDP Problem~\ref{prob:se_sync_riemannian}. 
	Therefore, by Lemma~\ref{lem:rank_restricted_SDP_equivalence}, it holds that,
	\begin{equation}
		\langle\ConLapT,\ZT^\star\rangle = \langle\ConLapT,{X^\star}^\top X^\star\rangle
		=\langle\ConLapM,{Y^\star}^\top Y^\star\rangle
		=\langle\ConLapM,\ZR^\star\rangle. 
		\label{eq:thm_sdp_eq1}
	\end{equation}
	Now, suppose $\ZR^\star$ is \emph{not} an optimal solution to Problem~\ref{prob:se_sync_riemannian_marg}. 
	Then there exists $\ZR^\ast$ such that $\langle \ConLapM, \ZR^\ast \rangle <\langle \ConLapM, \ZR^\star \rangle$. 
	Once again, without loss of generality, let $\rank(\ZR^\ast) = r^\ast$ and
	consider the rank-$r$ factorization $\ZR^\ast = {Y^\ast}^\top Y^\ast$ where
	$Y^\ast \in \Stiefel(d,r^\ast)^n$.
	Now suppose $p^\ast$ is an optimal
	value for translations given $Y^\ast$ (see \eqref{eq:optimal_translation}):
	\begin{equation}
			p^\ast \in \left\{-Y^\ast \widetilde{V}^\top L(W^\tau)^{\pinv} + c1_n^\top \,
	\middle\vert \, c \in \Real^r \right\}.
	\end{equation}
	Let $X^\ast \in \Manifold(r^\ast, n)$ be obtained by combining $Y^\ast$ and
	$p^\ast$ and define 
	$Z^\ast \triangleq {X^\ast}^\top X^\ast$. 
	Again by Lemma~\ref{lem:rank_restricted_SDP_equivalence}, it holds that,
	\begin{equation}
	\langle\ConLapT,\ZT^\ast \rangle  = \langle\ConLapT,{X^\ast}^\top X^\ast\rangle
	=  \langle \ConLapM, {Y^\ast}^\top Y^\ast\rangle
		=  \langle \ConLapM,\ZR^\ast \rangle. 
		\label{eq:thm_sdp_eq2}
	\end{equation}	
	The combination of \eqref{eq:thm_sdp_eq1} and \eqref{eq:thm_sdp_eq2} would imply that 
	$\langle \ConLapT, \ZT^\ast \rangle <\langle \ConLapT, \ZT^\star \rangle$, which contradicts the starting assumption that $\ZT^\star$ is an optimal solution. 
	Therefore $\ZR^\star$ must be an optimal solution with rank $r^\star$.
	To conclude the proof, note that using a similar argument, we can construct an optimal solution to
	Problem~\ref{prob:se_sdp} from an optimal solution $\ZR^\star$ to Problem~\ref{prob:se_sdp_marg} with the same rank.

\end{proof}

\subsection{Proof of Theorem~\ref{thm:tightness_informal}}
\label{sec:tightness_proof}

In this subsection, we formally prove the exactness guarantees of the SDP relaxation (Problem~\ref{prob:se_sdp}) used in this work, which is stated in Theorem~\ref{thm:tightness_informal} of the main paper. 

\begin{proof}[Proof of Theorem~\ref{thm:tightness_informal}]
	By \cite[Proposition~2]{Rosen19IJRR}, there exists a constant 
	as a function of the noiseless data matrix of the rotation-only SDP relaxation (Problem~\ref{prob:se_sdp_marg}), denoted as 
	$\beta \triangleq \beta(\underline{\ConLapM})$, such that if
	$\norm{\ConLapM - \underline{\ConLapM}}_2 < \beta$, 
	the rotation-only SDP relaxation (Problem~\ref{prob:se_sdp_marg})
	admits a {unique} solution $\ZR^\MLE = {R^\MLE}^\top R^\MLE$, where
	$R^\MLE \in \SOd(d)^n$ is a globally optimal rotation estimate to PGO (Problem~\ref{prob:se_sync}). Let $Z^\star$ be an arbitrary minimizer to Problem~\ref{prob:se_sdp}. By Theorem~\ref{thm:sdp_equivalence}, it holds that
	$\langle \ConLapT, Z^\star \rangle = \langle \ConLapM, \ZR^\MLE \rangle$. 
	Without loss of generality, let $\rank(Z^\star) = r$ where $r \geq d$. 
	Consider the rank-$r$ factorization $Z^\star = X^\top X$, where $X \in \Manifold(r,n)$.  
	Note that $X$ is a global minimizer to Problem~\ref{prob:se_sync_riemannian}, and hence by Lemma~\ref{lem:rank_restricted_SDP_equivalence}, it holds that, 
	\begin{equation}
		\langle \ConLapT, Z^\star \rangle 
		=
		\langle \ConLapT, X^\top X \rangle 
		= 
		\langle \ConLapM, Y^\top Y \rangle
		= 
		\langle \ConLapM, \ZR^\MLE \rangle. 
		\label{eq:objective_function_equality}
	\end{equation}
	Above, $Y \in \Stiefel(d,r)^n$ extracts the Stiefel elements from $X$. Since Problem~\ref{prob:se_sdp_marg} admits a unique minimizer, \eqref{eq:objective_function_equality}
	implies that $Y^\top Y$ is the same minimizer:
	\begin{equation}
		Y^\top Y = \ZR^\MLE = {R^\MLE}^\top R^\MLE.
		\label{eq:YtopY_is_minimizer}
	\end{equation}
	In addition, \eqref{eq:YtopY_is_minimizer} also implies that $\rank(Y) = \rank(\ZR^\MLE) = d$. We may thus consider the $d$-dimensional (thin) singular value decomposition
	$Y = U_d \Sigma_d V_d^\top$. 
	Let us define
	$\bar{Y} \triangleq \Sigma_d V_d^\top$. 
	Since $U_d \in \Stiefel(d,r)$, it holds that 
	${\bar{Y}}^\top \bar{Y} = {R^\MLE}^\top R^\MLE$,
	and therefore $\bar{Y}$ consists of $n$ orthogonal matrices $\bar{Y} \in \Orthogonal(d)^n$.
	By inspecting the
	first block row of this equality, we may further deduce that, 
	\begin{equation}
	{\bar{Y}_1}^\top \bar{Y}_i = {R^\MLE_1}^\top R^\MLE_i, \; \forall i \in [n]. 
	\label{eq:Ybar_i_expression}
	\end{equation}
	Multiplying both sides in \eqref{eq:Ybar_i_expression} from the left by $U_d \bar{Y}_1$,
	such that the left hand side simplifies to $Y_i$.
	\begin{equation}
	Y_i = U_d \bar{Y}_1 {R^\MLE_1}^\top R^\MLE_i. 
	\label{eq:Yi_expression}
	\end{equation}
	Let us define $A \triangleq  U_d \bar{Y}_1 {R^\MLE_1}^\top$. 
	Since $U_d \in \Stiefel(d,r)$ and $\bar{Y}_1, R_1^\MLE \in \Orthogonal(d)$, it holds that $A \in \Stiefel(d,r)$.
	Combining equality \eqref{eq:Yi_expression} for all $i$ yields the following compact equation. 
	\begin{equation}
	Y = AR^\MLE.
	\label{eq:Y_equals_AR}
	\end{equation}
	
	Let $p \in \Real^{r\times n}$ contains the translations in $X$. Since $X$ is a global minimizer, it is also a first-order critical point. Therefore, we can apply part (i) of Lemma~\ref{lem:rank_restricted_SDP_equivalence} to relate $p$ with $Y$:
	\begin{align}
	p &= -Y \widetilde{V}^\top L(W^\tau)^{\pinv} + c1_n^\top 
	\why{\text{Lemma~\ref{lem:rank_restricted_SDP_equivalence}}} \\
	&= -A R^\MLE \widetilde{V}^\top L(W^\tau)^{\pinv} + c1_n^\top 
	&&  \eqref{eq:Y_equals_AR}\\
	&= A t^\MLE + c1_n^T. 
	\label{eq:p_equals_At}
	\end{align}
	In the last equality \eqref{eq:p_equals_At}, we have defined 
	$t^\MLE \triangleq - R^\MLE \widetilde{V}^\top L(W^\tau)^{\pinv}$. 
	Notice that $t^\MLE$ corresponds to a set of globally optimal translations.
	Finally, we note that the first block-row of the SDP solution $Z^\star = X^\top X$ may be expressed as,
	\begin{align}
	{\ZT^\star}_{(1:d,:)} 
	&=
	(Y_1)^\top
	\begin{bmatrix}
	Y_1& p_1 & \hdots & Y_n & p_n
	\end{bmatrix} \\
	&= 
	(AR^\MLE_1)^\top
	\begin{bmatrix}
	A R^\MLE_1 & A t^\MLE_1 + c & \hdots & A R^\MLE_n & A t^\MLE_n + c
	\end{bmatrix} \\
	&= 
	\big[
	\underbrace{I_d}_{R_1^\star} \quad 
	\underbrace{{R^\MLE_1}^\top t^\MLE_1 + {R^\MLE_1}^\top A^\top c }_{t^\star_1} \quad \hdots 
	\quad \underbrace{{R^\MLE_1}^\top R^\MLE_n}_{R_n^\star} \quad 
	\underbrace{{R^\MLE_1}^\top t^\MLE_n + {R^\MLE_1}^\top A^\top c }_{t^\star_1}
	\big].
	\end{align}
	In particular, 
	${\ZT^\star}_{(1:d,:)}$ can be obtained from $(R^\MLE, t^\MLE)$ 
	via a global rigid body transformation with rotation
	${R^\MLE_1}^\top$ and translation
	${R^\MLE_1}^\top A^\top c$. 
	Due to the global gauge symmetry of PGO, 
	${\ZT^\star}_{(1:d,:)}$ thus also is an optimal solution.

	So far, we have proved that the SDP relaxation (Problem~\ref{prob:se_sdp}) is exact if its rotation-only counterpart (Problem~\ref{prob:se_sdp_marg}) satisfies $\norm{\ConLapM - \underline{\ConLapM}}_2 < \beta$. To conclude the proof, let us consider the matrix $M$ defined in \eqref{eq:M_definition} 
	and its latent value $\underline{M}$ (i.e., constructed using noiseless relative transformation measurements). Note that $M$ and $\underline{M}$ only differ in certain blocks,
	\begin{equation}
		M = 
		\begin{bmatrix}
		\underline{L}(W^\tau) & \underline{\widetilde{V}} + \Delta_{12} \\
		(\underline{\widetilde{V}}+\Delta_{12})^\top & \underline{\ConLap} + \underline{\widetilde{\Sigma}} + \Delta_{22}
		\end{bmatrix}
		=
		\underline{M} + 
		\begin{bmatrix}
		0 &  \Delta_{12} \\
		\Delta_{12}^\top &  \Delta_{22}
		\end{bmatrix}.
	\end{equation}
	Once again, the underline notation denotes the latent value of each data matrix.
	Matrices $\Delta_{12}$ and $\Delta_{22}$ summarize the measurement noise. 
	Notice that the upper left block of $M$ is \emph{not} affected by noise, since by construction it is always the (constant) translation-weighted graph Laplacian.
	Using the notation above, we can also express $\ConLapM$ \eqref{eq:ConLapM_def} as a function of $\Delta_{12}$ and $\Delta_{22}$,
	\begin{equation}
		\ConLapM (\Delta_{12}, \Delta_{22}) =
		 \underline{\ConLap} + \underline{\widetilde{\Sigma}} + \Delta_{22} - (\underline{\widetilde{V}} + \Delta_{12})^\top \underline{L}(W^\tau)^{\pinv} (\underline{\widetilde{V}} + \Delta_{12}). 
	\end{equation}
	Crucially, note that $\ConLapM$ varies \emph{continuously} with respect to the noise terms $\Delta_{12}$ and $\Delta_{22}$. 
	Therefore, as the noise tends to zero (equivalently, as $M$ tends to $\underline{M}$), $\ConLapM$ converges to its latent value $\underline{\ConLapM}$. 
	By definition, this guarantees the existence of a constant $\delta > 0$, such that $\norm{M - \underline{M}}_2 < \delta$ implies $\norm{\ConLapM - \underline{\ConLapM}}_2 < \beta$. 
	Finally, since the connection Laplacian $\ConLapT$ and $M$ are related by a permutation of the rows and columns, it holds that, 
	\begin{equation}
		\norm{\ConLapT - \underline{\ConLapT}}_2 < \delta 
		\implies
		\norm{M - \underline{M}}_2 < \delta 
		\implies
		\norm{\ConLapM - \underline{\ConLapM}}_2 < \beta,
	\end{equation}
	which concludes the proof. 
\end{proof}

\section{Convergence of $\RBCD$ and $\ARBCD$}
\label{apx:convergence_proof}

\subsection{Proof of Lemma~\ref{lem:sufficient_decrease}}
\label{apx:prop2_proof}

\begin{proof}[Proof of Lemma~\ref{lem:sufficient_decrease}]
Our proof largely follows the proof of Lemma~3.11 in \cite{Boumal2018Convergence}.
By Assumption~\ref{as:lipschitz_gradient_pullback}, the reduced cost over each block $f_b$ satisfies \eqref{eq:lipschitz_gradient_pullback} globally with Lipschitz constant $c_b$. 
To simplify our notation, we drop the iteration counter $k$ and use $X_b$ to represent the input into \textsc{BlockUpdate} (Algorithm~\ref{alg:BlockUpdate}).
In addition, define $g_b \triangleq \rgrad f_b(X_b)$.
In the remaining proof, we also use $\eta$ 
to represent a tangent vector in $T_{X_b} \Mcal_b$.
Using the simplified notation, 
consider the first trust-region subproblem solved in \textsc{BlockUpdate} (Algorithm~\ref{alg:BlockUpdate}).
\begin{subequations}
\begin{align}
	\underset{\eta \in T_{X_b} \Mcal_b}{\text{minimize}}
	& \quad \mhat_{b}(\eta) \triangleq f_b(X_{b}) + \langle \eta, g_b \rangle + 
	\frac{1}{2} \langle \eta, H[\eta] \rangle, \\
	\text{subject to} 
	& \quad
	\norm{\eta} \leq \Delta_0.
\end{align}
\label{eq:subproblem_b2}
\end{subequations}
Recall that $\Delta_0$ above is the initial trust-region radius. 
By Assumption~\ref{as:H_bounded}, there exists $c_0 \geq 0$ such that
$\langle \eta, H[\eta] \rangle \leq c_0 \norm{\eta}^2$ for all $\eta \in T_{X_{b}} \Mcal_b$.
Define the constant, 
\begin{equation}
	\lambda_b \triangleq \frac{1}{4} \min\bigg(\frac{1}{c_0}, \, \frac{1}{2c_b +
	2c_0}\bigg) = \frac{1}{8(c_b + c_0)}.
	\label{eq:lambda_b_def}  
\end{equation}
Note that $\lambda_b$ is the same constant that appears in the lower bound of the initial trust-region radius (Assumption~\ref{as:trust_region_radius_bound}).

We show that the required cost decrement in Lemma~\ref{lem:sufficient_decrease} is achieved by taking the so-called \emph{Cauchy step} \cite{Boumal2018Convergence} in the first trust-region subproblem.
By definition, the Cauchy step $\eta^C$ minimizes the model function \eqref{eq:subproblem_b2} along the direction of the negative Riemannian gradient,
i.e., $\eta^C = -\alpha^C g_b$. The step size $\alpha^C \in [0, \Delta_0/\norm{g_b}]$ can be determined in closed-form as,
\begin{equation}
	\alpha^C = \begin{cases}
	\min\bigg(\frac{\norm{g_b}^2}{\langle g_b, H[g_b]\rangle},
	\frac{\Delta_0}{\norm{g_b}}\bigg), &\text{if } \langle g_b, H[g_b]\rangle > 0, \\
	\frac{\Delta_0}{\norm{g_b}}, &\text{otherwise.}
	\end{cases}
\end{equation}
A straightforward calculation (see \cite[Lemma~3.7]{Boumal2018Convergence}) shows that the Cauchy step reduces the model function $\mhat_{b}$ by at least,
\begin{equation}
	\mhat_{b}(0) - \mhat_{b}(\eta^C) \geq \frac{1}{2}\min\bigg(\Delta_0,
	\frac{\norm{g_b}}{c_0}\bigg)\norm{g_b}. 
	\label{eq:model_decrease}
\end{equation}
Next, we need to relate the above model decrement \eqref{eq:model_decrease} with the actual decrement of the cost function $f_b$.  For this,  we show that the following ratio,
\begin{equation}
	\bigg |\frac{\fhat_{b}(0)-\fhat_{b}(\eta^C)}{\mhat_{b}(0)-\mhat_{b}(\eta^C)} - 1 \bigg|
	= 
	\bigg |\frac{\mhat_{b}(\eta^C)-\fhat_{b}(\eta^C)}{\mhat_{b}(0)-\mhat_{b}(\eta^C)}\bigg|
	\label{eq:prop2_ratio}
\end{equation}
is close to zero. 
Note that in \eqref{eq:prop2_ratio} we use the fact that, by definition, $\mhat_{x_b}(0) = \fhat_{x_b}(0) = f_b(X_b)$. 
We derive an upper bound on the numerator of \eqref{eq:prop2_ratio} as follows,
\begin{align}
	|\mhat_{b}(\eta^C)-\fhat_{b}(\eta^C)| 
	&= \big |f_b(X_b) + \langle g_b, \eta^C \rangle + \frac{1}{2}\langle \eta^C, H[\eta^C] \rangle - \fhat_{b}(\eta^C) \big |
	\why{\text{Definition of $\mhat_{x_b}$}} \\
	&\leq \big |f_b(X_b) + \langle g_b, \eta^C \rangle - \fhat_{b}(\eta^C) \big |
	+ \frac{1}{2} \big| \langle \eta^C, H[\eta^C] \rangle \big| 
	\why{\text{Triangle inequality}}\\
	&\leq \frac{1}{2}(c_b + c_0) \norm{\eta^c}^2.
	\label{eq:prop2_numerator}
\end{align}
For the last inequality, we have used the Lipschitz-type gradient condition of $f_b$ for the first term, 
and the boundedness of $H$ for the second term. 
Plugging \eqref{eq:model_decrease} and \eqref{eq:prop2_numerator} into \eqref{eq:prop2_ratio}, we obtain, 
\begin{align}
	\bigg |\frac{\fhat_{b}(0)-\fhat_{b}(\eta^C)}{\mhat_{b}(0)-\mhat_{b}(\eta^C)} - 1 \bigg|
	&\leq \frac{\frac{1}{2}(c_b + c_0)
	\norm{\eta^c}^2}{\frac{1}{2}\min\big(\Delta_0, \frac{\norm{g_b}}{c_0}\big)\norm{g_b}}
	&& {\color{gray!80!black}\text{\eqref{eq:model_decrease} and \eqref{eq:prop2_numerator}}} \\
	&\leq \frac{(c_b + c_0) \Delta_0^2}{\min\big(\Delta_0,
	\frac{\norm{g_b}}{c_0}\big)\norm{g_b}}
	\label{eq:prop2_eq1} 
	\why{\norm{\eta^c} \leq \Delta_0}\\
	&\leq \frac{(c_b + c_0) \Delta_0}{\norm{g_b}}.
	\why{\min(\Delta_0, \norm{g_b}/c_0) \leq \Delta_0}
	\label{eq:prop2_eq2}
\end{align}

In order to proceed, let us first impose an additional assumption that 
the initial trust-region radius $\Delta_0$ is also bounded above by $\Delta_0 \leq 4\lambda_b \norm{g_b}$. 
Towards the end of this proof, we show how this assumption can be safely removed. 
Under this additional assumption, it holds that 
$\Delta_0 \leq 4\lambda_b \norm{g_b} \leq \norm{g_b}/(2c_b + 2c_0)$, and thus \eqref{eq:prop2_eq2} implies,
\begin{equation}
	\bigg |\frac{\fhat_{b}(0)-\fhat_{b}(\eta^C)}{\mhat_{b}(0)-\mhat_{b}(\eta^C)} - 1 \bigg| \leq \frac{1}{2}
	\implies
	\rho \triangleq
	\frac{\fhat_{b}(0)-\fhat_{b}(\eta^C)}{\mhat_{b}(0)-\mhat_{b}(\eta^C)} 
	\geq \frac{1}{2}.
	\label{eq:prop2_eq3}
\end{equation}
In particular, $\rho$ is bigger than the acceptance threshold of $1/4$ in \textsc{BlockUpdate} (see Algorithm~\ref{alg:BlockUpdate}), 
and thus $\eta^C$ is \emph{guaranteed} to be accepted.
Consider the candidate solution corresponding to the Cauchy step $X_{b}^\prime = \Retr_{X_{b}}(\eta^C)$. 
Using the definition of the pullback function, it holds that,
\begin{align}
	f_{b}(X_{b}) - f_b(X_{b}^\prime) 
	& =\fhat_{b}(0)-\fhat_{b}(\eta^C)
	\why{\text{definition of pullback}}  \\
	& \geq \frac{1}{2} \left(\mhat_{b}(0) -
	\mhat_{b}(\eta^C)\right) 
	\why{\text{equation } \eqref{eq:prop2_eq3}}\\
	& \geq \frac{1}{4} \min\bigg(\Delta_0,
	\frac{\norm{g_b}}{c_0}\bigg) \norm{g_b}
	\why{\text{equation $\eqref{eq:model_decrease}$}} \\
	& \geq \frac{1}{4} \min\bigg (\lambda_b,
	\frac{1}{c_0}\bigg) \norm{g_b}^2
	\why{\text{$\Delta_0 \geq \lambda_b\norm{g_b}$ by Assumption~\ref{as:trust_region_radius_bound}}}\\
	& = \frac{1}{4} \lambda_b \norm{g_b}^2.
	\why{\text{$\lambda_b \leq 1/c_0$ by definition}}
\end{align}
Therefore, we have proven that under the additional assumption that  
$\Delta_0 \leq 4\lambda_b \norm{g_b}$, the desired cost reduction can be
achieved simply by taking the Cauchy step in the first trust-region subproblem. As we mention in Section~\ref{sec:bcd}, in practice, we use the truncated conjugate-gradient (tCG) algorithm to improve upon the initial Cauchy step. 
However, since each additional tCG iteration will strictly decrease the model function (see \cite[Proposition~7.3.2]{absil2009optimization}), we can show that the inequality \eqref{eq:prop2_eq2} holds at all times, and thus the desired cost reduction is always achieved.

To complete the proof, we need to show that Algorithm~\ref{alg:BlockUpdate} still achieves the desired cost reduction, even after removing the additional assumption that $\Delta_0 \leq 4\lambda_b \norm{g_b}$.
To do so, note that after dropping this assumption, inequality \eqref{eq:prop2_eq3} might fail to hold 
and as a result the Cauchy step can be rejected in the first iteration.
However, by the mechanism of \textsc{BlockUpdate} (Algorithm~\ref{alg:BlockUpdate}), after each rejection the trust-region radius will be divided by four in the next iteration.
Therefore, in the worst case, the trust-region radius will be within the interval $[\lambda_b \norm{g_b},4\lambda_b \norm{g_b}]$ after
$\Ocal(\log(4\lambda_b\norm{g_b}\Delta_0))$ consecutive rejections,
after which the Cauchy step is guaranteed to be accepted in the next
trust-region subproblem.

\end{proof}

\subsection{Proof of Theorem~\ref{thm:general_convergence}}
\begin{proof}[Proof of Theorem~\ref{thm:general_convergence}]
	We begin by using the sufficient descent property proved in Lemma~\ref{lem:sufficient_decrease}, which states that after each $\textsc{BlockUpdate}$ operation, the reduced cost is decreased by at least, 
	\begin{equation}
		f_{b_k}(X_{b_k}^k) - f_{b_k}(X_{b_k}^{k+1}) 
		\geq  \frac{1}{4}\lambda_{b_k} \norm{\rgrad f_{b_k}(X^k_{b_k})}^2
		=     \frac{1}{4}\lambda_{b_k} \norm{\rgrad_{b_k} f(X^k)}^2.
	\end{equation}
	Recall that by design of the $\RBCD$ algorithm, at each iteration the values of all blocks other than $b_k$ remain unchanged. 
	Denote the values of these fixed blocks as $X^k_c$.
	By definition of the reduced cost $f_{b_k}$, it is straightforward to verify that the decrease in $f_{b_k}$ {exactly equals} the decrease in the global cost $f$,
	\begin{equation}
		f_{b_k}(X_{b_k}^k) - f_{b_k}(X_{b_k}^{k+1}) 
		=
		f(X_{b_k}^k, X_c^k) - f(X_{b_k}^{k+1}, X_c^k)
		=
		f(X^k) - f(X^{k+1}).
	\end{equation}
	The above result directly implies that each iteration of $\RBCD$ reduces the global cost function by at least,
	\begin{align}
	f(X^k) - f(X^{k+1}) = f_{b_k}(X_{b_k}^k) - f_{b_k}(X_{b_k}^{k+1}) \geq \frac{1}{4}\lambda_{b_k} \norm{\rgrad_{{b_k}} f(X^k)}^2.
	\label{eq:sufficient_descent}
	\end{align}
	In the remaining proof, we use \eqref{eq:sufficient_descent} to prove the stated convergence rate for each block selection rule.
	\newline\newline
	\textbf{Uniform Sampling.}
	In this case, the updated block $b_k \in [N]$ is selected uniformly at random at each iteration.
	Conditioned on all previously selected blocks $b_{0:k-1}$, we can take expectation on both sides of \eqref{eq:sufficient_descent} with respect to the current block $b_k$.
	\begin{equation}
		\mathbb{E}_{b_k | b_{0:k-1}} \Big[f(X^k) - f(X^{k+1})\Big]
		\geq \sum_{b \in N} \text{Prob}(b_k = b) \cdot \frac{1}{4} \lambda_{b} \norm{\rgrad_b f(X^k)}^2
	\end{equation}
	 Recall that all block are selected with equal probability, i.e., 
	 $\text{Prob}(b_k = b) = 1/N$ for all $b \in [N]$. 
	 Using this, we arrive at a simplified inequality that 
	 lower bounds the expected cost decrease by the squared gradient norm with respect to all variables $X$.
	\begin{equation}
	\mathbb{E}_{b_k | b_{0:k-1}} \Big[f(X^k) - f(X^{k+1})\Big]
	\geq
	\sum_{b \in N} \frac{1}{N} \cdot \frac{1}{4} \lambda_{b}
	\norm{\rgrad_b f(X^k)}^2
	\geq \frac{\min_{b \in [N]} \lambda_b}{4N} \norm{\rgrad f(X^k)}^2. 
	\label{eq:thm1_pf_eq1}
	\end{equation}
	We can use the above inequality \eqref{eq:thm1_pf_eq1} to bound the expected cost decrease over all iterations from $k=0$ to $k=K$. 
	To do so, we first write the overall cost reduction as a cascading sum, and then apply the law of total expectation on each expectation term that appears inside the sum. 
	\begin{equation}
		f(X^0) - \mathbb{E}_{b_{0:K-1}} f(X^K) = \sum_{k=0}^{K-1} \mathbb{E}_{b_{0:k}} \Big[f(X^k) -
		f(X^{k+1})\Big]
		=\sum_{t=0}^{K-1} \mathbb{E}_{b_{0:k-1}} 
		\Big[\mathbb{E}_{b_{k} | b_{0:k-1}}
		[f(X^k) - f(X^{k+1})]\Big].
	\end{equation}
	Observe that each innermost conditional expectation is already bounded by our previous inequality \eqref{eq:thm1_pf_eq1}.
	Plugging in this lower bound, we arrive at,
	\begin{align}
		f(X^0) - \mathbb{E}_{b_{0:K-1}} f(X^K)
		&\geq 
		\sum_{k=0}^{K-1} \mathbb{E}_{b_{0:k-1}} \bigg [\frac{\min_{b \in [N]} \lambda_b}{4N} \norm{\rgrad f(X^k)}^2 \bigg]  \\
		&\geq 
		K \cdot \frac{\min_{b \in [N]} \lambda_b}{4N} \min_{0 \leq k
			\leq K-1} \mathbb{E}_{b_{0:k-1}}\norm{\rgrad f(X^k)}^2.
	\end{align}
	To conclude the proof, we note that since $f^\star$ is the global minimum, we necessarily have 
	$f(X^0) - f^\star \geq f(X^0) - \mathbb{E}_{b_{0:K-1}} f(X^K)$.
	This directly implies,
	\begin{equation}
		f(X^0) - f^\star \geq
		 K \cdot \frac{\min_{b \in [N]} \lambda_b}{4N} \min_{0 \leq k
		 	\leq K-1} \mathbb{E}_{b_{0:k-1}}\norm{\rgrad f(X^k)}^2.
	\end{equation}
	Rearranging the last inequality gives the desired convergence rate \eqref{eq:general_convergence_uniform}.
	\newline\newline
	\textbf{Importance Sampling.}
	We start with the same procedure of taking conditional expectations on both sides of \eqref{eq:sufficient_descent}. 
	The only difference is that with importance sampling, at each iteration the block is selected with probability proportional to the squared gradient norm, i.e.,
	$p_b = \norm{\rgrad_{b} f(X^k)}^2/\norm{\rgrad f(X^k)}^2$. 
	Using this to expand the conditional expectation gives:
	\begin{align}
			\mathbb{E}_{b_k|b_{0:k-1}}  [f(X^k) - f(X^{k+1})]  &\geq \sum_{b \in [N]} \frac{\norm{\rgrad_{b} f(X^k)}^2}{\norm{\rgrad f(X^k)}^2} \cdot \frac{1}{4} \lambda_{b} \norm{\rgrad_{b} f(X^k)}^2 \\
	& \geq \frac{\min_{b \in [N]} \lambda_b}{4} \cdot \frac{\sum_{b \in [N]} \norm{\rgrad_{b} f(X^k)}^4}{\sum_{b \in [N]} \norm{\rgrad_{b} f(X^k)}^2}  \\
	& = 
	\frac{\min_{b \in [N]} \lambda_b}{4} \cdot \frac{\sum_{b \in [N]} a_b^2 }{\sum_{b \in [N]} a_b}, 
	\label{eq:thm1_pf_eq2}
	\end{align}
	where we define $a_b \triangleq \norm{\rgrad_{b} f(X^k)}^2$ for brevity.
	Applying the Cauchy-Schwarz inequality gives,
	\begin{equation}
		\frac{1}{N^2} \bigg(\sum_{b \in [N]} a_b \bigg)^2 \leq \frac{1}{N} \sum_{b \in [N]} a_b^2.
	\end{equation}	
	Rearranging the above inequality yields, 
	\begin{equation}
		\frac{\sum_{b \in [N]} a_b^2 }{\sum_{b \in [N]} a_b} \geq \frac{1}{N} \sum_{b \in [N]} a_b = \frac{1}{N} \norm{\rgrad f(X^k)}^2.
		\label{eq:thm1_pf_eq3}
	\end{equation}
	Combining \eqref{eq:thm1_pf_eq2} and \eqref{eq:thm1_pf_eq3} gives,
	\begin{equation}
			f(X^k) - \mathbb{E}_{b_k|b_{0:k-1}} f(X^{k+1}) \geq \frac{\min_{b \in [N]} \lambda_b}{4N} \norm{\rgrad f(X^k)}^2. 
	\end{equation}	
	The rest of the proof is identical to the proof for uniform sampling.
	\newline\newline
	\textbf{Greedy Selection.}
	With greedy selection, we can perform a deterministic analysis. 
	Recall that at each iteration, the block with the largest squared gradient norm is selected. 
	Using this information inside our inequality \eqref{eq:sufficient_descent}, we arrive at,
	\begin{align}
		f(X^k) - f(X^{k+1})  &\geq \frac{1}{4}\lambda_{b_k} \norm{\rgrad_{b_k}
			f(X^k) }^2 \\
		&\geq 
		\frac{1}{4} \left(\min_{b \in [N]} \lambda_b\right) \cdot
		\left(\max_{b\in [N]}\norm{\rgrad_{b} f(X^k)}^2\right) \\
		& 
		\geq \frac{1}{4N} \min_{b \in [N]} \lambda_b \cdot \norm{\rgrad
			f(X^k)}^2.
	\end{align}
	A telescoping sum of the above inequalities from $k=0$ to $k=K-1$ gives,
	\begin{align}
	f(X^0) - f^\star 
	&\geq \sum_{k=0}^{K-1} [f(X^k) - f(X^{k+1})] \\
	&\geq \sum_{k=0}^{K-1} \frac{1}{4N} \min_{b \in [N]} \lambda_b \cdot \norm{\rgrad f(X^k)}^2 \\
	&\geq \frac{K}{4N} \cdot \min_{b \in [N]} \lambda_b \cdot \min_{0 \leq k
	\leq K-1} \norm{\rgrad f(X^k)}^2.
	\end{align}
	Rearranging the last inequality gives \eqref{eq:general_convergence_greedy}.
\end{proof}

\subsection{Proof of Theorem~\ref{thm:accelerated_convergence}}
\label{proof_of_convergence_of_ARBCD_sec}
\begin{proof}
	For the purpose of proving global first-order convergence, we only focus on the evolution of the $\{X^k\}$ sequence in $\ARBCD$ (Algorithm~\ref{alg:ARBCD}). 
	After each iteration, there are two possibilities depending on whether the adaptive restart condition (line~\ref{alg:restart_start}) is triggered. 
	If restart is not triggered, then by construction 
	the current iteration must achieve a cost reduction that is at least,
	\begin{equation}
		f(X^k) - f(X^{k+1}) \geq c_1 \norm{\rgrad_{b_k} f(X^k)}^2.
	\end{equation}
	On the other hand, if restart is triggered, then the algorithm would instead update the selected block using the default \textsc{BlockUpdate} method. In this case, using the same argument as the beginning of proof for Theorem~\ref{thm:general_convergence},
	we can establish the following lower bound on global cost reduction.
	\begin{equation}
		f(X^k) - f(X^{k+1}) = f_{b_k}(X_{b_k}^k) - f_{b_k}(X_{b_k}^{k+1}) \geq \frac{1}{4}\lambda_{b_k} \norm{\rgrad_{{b_k}} f(X^k)}^2.
	\end{equation}
	Combining the two cases, we see that each iteration of $\ARBCD$ decreases the global cost by at least, 
	\begin{align}
		f(X^k) - f(X^{k+1}) &\geq \min\left(c_1, \lambda_{b_k}/4\right) \norm{\rgrad_{{b_k}} f(X^k)}^2 \\
		&\geq \min\left(c_1, \min_{b\in [N]}\lambda_{b}/4\right) \norm{\rgrad_{{b_k}} f(X^k)}^2 \\
		&\geq C \norm{\rgrad_{{b_k}} f(X^k)}^2 
		\why{\text{definition of $C$}}.
	\end{align}
	The above inequality serves as a lower bound on the global cost reduction after each iteration. Note that this lower bound is very similar to \eqref{eq:sufficient_descent} used in the proof of Theorem~\ref{thm:general_convergence}. The only difference is the change of constant on the right hand side. 
	From this point on, we can employ the same steps used in the proof of Theorem~\ref{thm:general_convergence} to prove the desired convergence rates for all three block selection rules.
\end{proof}

\subsection{Convergence on Problem~\ref{prob:Rank_restricted_SDP}}
\label{sec:convergence_pgo}

In this section, we address the convergence of \RBCD\ and \ARBCD\ when solving the rank-restricted semidefinite relaxations of PGO (Problem~\ref{prob:Rank_restricted_SDP}).
To apply the general convergence results in Section~\ref{sec:convergence}, we need to check that the required technical conditions are satisfied.
Among these, we note that the Riemannian Hessian operator (which we use as the user-specified map $H$ in Algorithm~\ref{alg:BlockUpdate}) satisfies Assumption~\ref{as:global_radial_linearity} due to the linearity of affine connections \cite[Section~5.2]{absil2009optimization}.
In addition, Assumption~\ref{as:trust_region_radius_bound} can be satisfied by choosing a sufficiently large initial trust-region radius.
Therefore, to establish convergence, we are left to verify Assumption~\ref{as:lipschitz_gradient_pullback} (Lipschitz-type gradients for pullback) and Assumption~\ref{as:H_bounded} (boundedness of Riemannian Hessian).

\subsubsection{Verifying Assumption~\ref{as:lipschitz_gradient_pullback}}
\label{sec:convergence_pgo:lipschitz_gradient}
We now show that Assumption~\ref{as:lipschitz_gradient_pullback} is satisfied in Problem~\ref{prob:Rank_restricted_SDP}, i.e., the reduced cost function \eqref{eq:reduced_cost} corresponding to any block 
has Lipschitz-type gradient for pullbacks along the iterates of \RBCD\ and \ARBCD.
In \cite{Boumal2018Convergence}, \mbox{Boumal et al.} prove a simple condition for a function $f: \Mcal \to \Real$ defined on a matrix submanifold to have Lipschitz-type gradient for pullbacks.
For convenience, we include their result below.

\begin{lemma}[Lemma~2.7 in \cite{Boumal2018Convergence}]
	\normalfont
	Let $\Ecal$ be a Euclidean space and let $\Mcal$ be a compact Riemannian submanifold of $\Ecal$.
	Let $\Retr$ be a retraction on $\Mcal$ (globally defined). If $f: \Ecal \to \Real$ has Lipschitz continuous gradient
	then the pullbacks $f \circ \Retr_x$ satisfy
	\eqref{eq:lipschitz_gradient_pullback} globally with some constant $c_g$
	independent of $x$.
	\label{lem:boumal_lipschitz}
\end{lemma}

In our case, we need a generalized version of Lemma~\ref{lem:boumal_lipschitz} that extend the result to product manifold with Euclidean spaces.

\begin{lemma}[Extension of Lemma~\ref{lem:boumal_lipschitz} to product manifolds with Euclidean spaces]
	\normalfont
	Let $\Ecal_1$ and $\Ecal_2$ be Euclidean spaces, and define $\Ecal \triangleq \Ecal_1 \times \Ecal_2$.
	Let $\Mcal \triangleq \Mcal_1 \times \Ecal_2$, where $\Mcal_1$ is a compact Riemannian submanifold of $\Ecal_1$.
	Given $x = \left[ x_1 \,\, x_2 \right] \in \Mcal$ and
	$\eta = \left[ \eta_1 \,\, \eta_2 \right] \in T_x \Mcal$, 
	define a retraction operator $\Retr_x: T_x \Mcal \to \Mcal$ as:
	$
	\Retr_x (\eta) 	=
	\left[
	\Retr_{x_1} (\eta_1) \,\,
	\Retr_{x_2} (\eta_2)
	\right]
	=
	\left[
	\Retr_{x_1} (\eta_1) \,\,
	x_2 + \eta_2
	\right],
	$
	where $\Retr_{x_1}$ is a globally defined retraction on $\Mcal_1$ and we employ the standard retraction for Euclidean space.
	If $f: \Ecal \to \Real$ has bounded and Lipschitz continuous gradient,
	then the pullbacks $f \circ \Retr_x$ satisfy
	\eqref{eq:lipschitz_gradient_pullback} globally with some constant $c_g$
	independent of $x$; i.e.,
	\begin{equation}
	\big|\fhat_x(\eta) - [f(x) + \langle \eta, \rgrad_x f \rangle]\big| \leq \frac{c_g}{2} \norm{\eta}^2_2.
	\end{equation} 
	\label{lem:extended_lipschitz}
\end{lemma}
\begin{proof}
	This proof is a straightforward generalization of the proof of Lemma~\ref{lem:boumal_lipschitz}.
	By assumption, the Euclidean gradient $\nabla f$ is Lipschitz continuous, which implies the existence of a constant $L \geq 0$ such that for any $x, y \in \Mcal$, 
	\begin{equation}
	\left| f(y) - \left[f(x) + \langle \nabla_x f, y-x \rangle \right]\right| \leq \frac{L}{2} \norm{y-x}^2.
	\label{eq:euclidean_lipschitz}
	\end{equation}
	The above inequality is true in particular for any $y = \Retr_x(\eta), \eta \in T_x \Mcal$. 
	In this case, the inner product that appears in the LHS of \eqref{eq:euclidean_lipschitz} can be expanded as, 
	\begin{align}
	\langle \nabla_x f, \Retr_x (\eta) - x \rangle 
	&= 
	\left\langle
	\begin{bmatrix}
	\nabla_{x_1} f &
	\nabla_{x_2} f
	\end{bmatrix},
	\begin{bmatrix}
	\Retr_{x_1} (\eta_1) - x_1 &
	(x_2 + \eta_2) - x_2
	\end{bmatrix}
	\right\rangle \\
	&=
	\langle \nabla_{x_1} f, \Retr_{x_1} (\eta_1) - x_1\rangle
	+ 
	\langle \nabla_{x_2} f, \eta_2 \rangle 
	\\
	&= 
	\langle \nabla_{x_1} f, \Retr_{x_1} (\eta_1) - x_1 - \eta_1 + \eta_1\rangle
	+ 
	\langle \nabla_{x_2} f, \eta_2 \rangle 
	\why{\text{Add and subtract $\eta_1$}}\\
	&= 
	\langle \nabla_{x_1} f, \eta_1 \rangle + 
	\langle \nabla_{x_2} f, \eta_2 \rangle + 
	\langle \nabla_{x_1} f, \Retr_{x_1} (\eta_1) - x_1 - \eta_1 \rangle.
	\end{align}
	Next, we use two facts: (1) the Riemannian gradient in Euclidean space is just the standard (Euclidean) gradient;
	and (2) the Riemannian gradient of submanifolds \emph{embedded} in a Euclidean space is the orthogonal projection of the Euclidean gradient onto the tangent space (\cite[Equation~3.37]{absil2009optimization}). 
	With these, the above equality can be further simplified to,
	\begin{align}
	\langle \nabla_x f, \Retr_x (\eta) - x \rangle 
	&= 
	\langle \rgrad_{x_1} f, \eta_1 \rangle + 
	\langle \rgrad_{x_2} f, \eta_2 \rangle + 
	\langle \nabla_{x_1} f, \Retr_{x_1} (\eta_1) - x_1 - \eta_1 \rangle \\
	&= 
	\langle \rgrad_{x} f, \eta \rangle + 
	\langle \nabla_{x_1} f, \Retr_{x_1} (\eta_1) - x_1 - \eta_1 \rangle.
	\label{eq:lem_eq2}
	\end{align}
	Plugging \eqref{eq:lem_eq2} into \eqref{eq:euclidean_lipschitz} gives,
	\begin{align}
	\big| f(y) - [f(x) + \langle \nabla_x f, y-x \rangle]\big|
	&=
	\big|f(\Retr_x(\eta)) - [f(x) + \langle \rgrad_x f, \eta \rangle + \langle
	\nabla_{x_1} f, \Retr_{x_1} (\eta_1) - x_1 - \eta_1 \rangle]\big| \nonumber \\
	& \leq \frac{L}{2} \norm{\Retr_x(\eta)-x}^2.
	\end{align}
	Applying the triangle and Cauchy-Schwarz inequalities and expanding
	$\norm{\Retr_x(\eta)-x}^2$ yields,
	\begin{align}
	|f(\Retr_x(\eta)) - [f(x) + \langle \rgrad_x f, \eta \rangle]| 
	&\leq  \frac{L}{2} \norm{\Retr_x(\eta)-x}^2 + \lvert \langle \nabla_{x_1} f,\Retr_{x_1}(\eta_1) - x_1 - \eta_1 \rangle \rvert \\
	&\leq  \frac{L}{2} \norm{\Retr_x(\eta)-x}^2 + \norm{\nabla_{x_1} f} \norm{\Retr_{x_1} (\eta_1) - x_1 - \eta_1} \\
	&= \frac{L}{2} \norm{\eta_2}^2 + \frac{L}{2} \norm{\Retr_{x_1}(\eta_1)-x_1}^2 + \norm{\nabla_{x_1} f} \norm{\Retr_{x_1} (\eta_1) - x_1 - \eta_1}.
	\label{eq:lem_eq3}
	\end{align}
	Since we assume that the Euclidean gradient is bounded, there exists
	constant $G_1 \geq 0$ such that
	$\norm{\nabla_{x_1} f} \leq G_1$ for all $x_1 \in \Mcal_1$. 
	In equations (B.3) and (B.4) in \cite{Boumal2018Convergence},
	Boumal et al. show that for the compact submanifold $\Mcal_1$, the following inequalities hold,
	\begin{align}
	\norm{\Retr_{x_1}(\eta_1)-x_1} &\leq \alpha_1 \norm{\eta_1}, \label{eq:lem_eq4} \\
	\norm{\Retr_{x_1} (\eta_1) - x_1 - \eta_1} &\leq \beta_1 \norm{\eta_1}^2, \label{eq:lem_eq5}
	\end{align}
	for some $\alpha_1, \beta_1 \geq 0$.
	Plugging \eqref{eq:lem_eq4} and \eqref{eq:lem_eq5} in \eqref{eq:lem_eq3} gives,
	\begin{align}
	|f(\Retr_x(\eta)) - [f(x) + \langle \rgrad_x f, \eta \rangle]| \leq 
	\frac{L}{2} \norm{\eta_2}^2 + 
	\frac{L}{2} \alpha_1^2 \norm{\eta_1}^2 + 
	G_1 \beta_1 \norm{\eta_1}^2.
	\label{eq:lem_eq6}
	\end{align}
	Let $c_g \triangleq \max\Big(\frac{L}{2}, \frac{L \alpha_1^2}{2}+G_1
	\beta_1\Big)$. From \eqref{eq:lem_eq6} it trivially holds that,
	\begin{align}
	|f(\Retr_x(\eta)) - [f(x) + \langle \rgrad_x f, \eta \rangle]| \leq c_g (\norm{\eta_1}^2 + \norm{\eta_2}^2) = c_g \norm{\eta}^2.
	\end{align}
\end{proof}

Let us consider the reduced cost $f_b$ \eqref{eq:reduced_cost} as a function defined on the product manifold $\Stiefel(d,r)^{n_b} \times \Real^{r \times n_b}$, where $n_b$ is the number of poses contained in block $b$.
Note that in the ambient (Euclidean) space, $f_b$ is a quadratic function of $X_b$. Hence, its Euclidean gradient $\nabla f_b(X_b)$ is Lipschitz continuous, with Lipschitz constant given by the maximum eigenvalue of $Q_b$.
Next, we show that within any sublevel set of the global cost function $f$, $\nabla f_b(X_b)$ is also bounded for any block $b \in [N]$. 
Since $\nabla f_b(X_b) = \nabla_b f(X)$ by construction of the reduced cost function, it is equivalent to bound the global Euclidean gradient $\nabla f(X)$. 
Write out the individual rotation and translation components of $X \in \Manifold(r,n)$,
\begin{equation}
X = \begin{bmatrix}
Y_1 & p_1 & \hdots & Y_n  & p_n
\end{bmatrix}.
\end{equation}
For any $X$ in the $\bar{f}$-sublevel-set, expanding $f(X) \leq \bar{f}$ yields,
\begin{equation}
f(X) =
\sum_{(i,j) \in \DirEdges} \kappa_{ij} \norm{Y_j - Y_i \Rtilde_{ij}}_F^2 
+ \sum_{(i,j) \in \DirEdges} \tau_{ij} \norm{p_j - p_i - Y_i \ttilde_{ij}}_2^2
\leq \bar{f}. 
\end{equation}
Since the objective consists of a sum over squared terms, we can obtain a \emph{uniform} upper bound on each single term. Specifically, for each edge $(i,j) \in \DirEdges$, the associated relative translation cost is bounded by, 
\begin{equation}
\norm{p_j - p_i - Y_i \ttilde_{ij}}_2^2 \leq \bar{f} / \min_{(i,j) \in \DirEdges} \tau_{ij}
\implies 
\norm{p_j - p_i - Y_i \ttilde_{ij}}_2 \leq \sqrt{\bar{f} / \min_{(i,j) \in \DirEdges} \tau_{ij}}.
\label{eq:pgo_translation_bound}
\end{equation}
Note that when forming $\nabla f(X)$, the only terms 
that are potentially unbounded correspond to the translation terms
$h_{ij} \triangleq \tau_{ij} \norm{p_j - p_i - Y_i \ttilde_{ij}}_2^2$, 
since the rotation terms are defined on compact search spaces.
The Euclidean gradients of $h_{ij}$ with respect to $Y_i$, $p_i$, $p_j$ are given by,
\begin{align}
	&\nabla_{Y_i} h_{ij} = -2(p_j - p_i - Y_i \ttilde_{ij}) \ttilde_{ij}^\top, \label{eq:Euclidean_gradient_Yi} \\
	&\nabla_{p_i} h_{ij} = -2(p_j - p_i - Y_i \ttilde_{ij}),
	\label{eq:Euclidean_gradient_pi} \\
	&\nabla_{p_j} h_{ij} = 2(p_j - p_i - Y_i \ttilde_{ij}),
	\label{eq:Euclidean_gradient_pj}
\end{align}
By \eqref{eq:pgo_translation_bound}, the norms of \eqref{eq:Euclidean_gradient_Yi}-\eqref{eq:Euclidean_gradient_pj} are bounded, which implies that the overall Euclidean gradient $\nabla f(X)$ is bounded within the sublevel set.
Applying Lemma~\ref{lem:extended_lipschitz}, we have shown that $f_b$ has Lipschitz-type gradient for pullbacks within the sublevel set of $f$. 

Finally, note that both \RBCD\ and \ARBCD\ are by construction \emph{descent methods}, i.e., the sequence of iterates satisfies $f(X^k) \leq f(X^{k-1}) \leq \hdots \leq f(X^1) \leq f^0$ for any $k$. 
For $\RBCD$, this is true since we explicitly enforce that each block update 
should produce a function decrease that is at least a constant fraction of the model decrease (Algorithm~\ref{alg:BlockUpdate}, line~\ref{alg:block_update_termination}).  
For $\ARBCD$, this is also guaranteed by the adaptive restart condition (see Section~\ref{sec:acceleration}).
This property ensures that the sequence of iterates generated by \RBCD\ or \ARBCD\ never leaves the initial sublevel set,
and therefore Assumption~\ref{as:lipschitz_gradient_pullback} is satisfied.

\subsubsection{Verifying Assumption~\ref{as:H_bounded}}
\label{sec:convergence_pgo:H_bounded}
We can follow a similar strategy and show that the Riemannian Hessian operator is bounded at any iterate generated by \RBCD\ and \ARBCD, hence satisfying Assumption~\ref{as:H_bounded}.  
Let $f$ denote the cost function in Problem~\ref{prob:se_sync_riemannian}, and $X^k \in \Manifold(r,n)$ denote the solution at iteration $k$.
From \eqref{eq:riemannian_hessian}, it may be straightforwardly verified that for Problem~\ref{prob:se_sync_riemannian}, the Riemannian Hessian operator has the following explicit form: 
\begin{equation}
\Hess f(X)[\eta] = 2 \proj_{T_X} (\eta \ST(X)).
\label{eq:riemannian_hessian_explicit_form}
\end{equation}
Above, $S(X)$ is the ``dual certificate'' matrix \eqref{certificate_matrix} defined in Theorem~\ref{thm:verification},  repeated below for convenience,
\begin{equation}
	S(X) \triangleq \ConLapT - \SymBlockDiag_d^{+}(X^\top X \ConLapT).
	\label{eq:dual_certificate}
\end{equation}
In particular, note that \eqref{eq:riemannian_hessian_explicit_form} implies the following equality, 
\begin{equation}
	\langle \eta, \Hess f(X) [\eta] \rangle = 2 \langle \eta, S(X) \eta \rangle.
	\label{eq:hessian_and_dual_certificate}
\end{equation}
Following the strategy in Section~\ref{sec:convergence_pgo:lipschitz_gradient}, we now show that $S(X)$ is bounded within the sublevel sets of $f$.
To make the presentation clear, we present the proof in three steps.

\paragraph{Step 1: Eliminating Global Translation Symmetry from $S(X)$.}
To begin, we observe that the connection Laplacian $\ConLapT$ in Problem~\ref{prob:se_sync_riemannian} always has a null vector  \cite{Briales17CartanSync},
\begin{equation}
	\vnull \triangleq 1_n \otimes \begin{bmatrix}
	0_d \\ 1
	\end{bmatrix},
	\label{eq:null_vector}
\end{equation}
where $1_n$ stands for the vector of all ones. 
Intuitively, $\vnull$ arises as a result of the global translation symmetry inherent in pose-graph optimization. 
We note that although $\vnull$ is a single vector, it actually accounts for the \emph{complete} set of translational symmetries which has dimension $r$. 
More precisely, adding a multiple of $\vnull$ on the $k$-th row of $X$ ($1 \leq k \leq r$) 
corresponds to applying a global translation along the $k$-th unit direction. 
Let us consider projecting each row of our decision variable $X$ onto the subspace orthogonal to $\vnull$. In matrix notation, this operation can be written as a matrix product $XP$, where $P$ is a projection matrix defined as follows. 
\begin{equation}
	P \triangleq I - \vnull {\vnull}^\top  / {\norm{\vnull}_2^2}.
\end{equation}
Note that after projection, 
it holds that $X = XP + u \vnull^\top$ for some vector $u \in \Real^r$. 
Substitute this decomposition into the second term of \eqref{eq:dual_certificate},
\begin{align}
	\SymBlockDiag_d^{+}(X^\top X \ConLapT) &= \SymBlockDiag_d^{+}((XP + u \vnull^\top)^\top (XP + u \vnull^\top) \ConLapT) \\
	&=\SymBlockDiag_d^{+}((XP + u \vnull^\top)^\top XP \ConLapT) \\
	&=\SymBlockDiag_d^{+}((XP)^\top XP \ConLapT) +
	  \SymBlockDiag_d^{+}(\underbrace{\vnull u^\top XP \ConLapT}_{A}).
	\label{eq:dual_certificate_centered1}
\end{align}
In \eqref{eq:dual_certificate_centered1}, we have used the linearity of the $\SymBlockDiag_d^{+}$ operator \eqref{eq:SymBlockDiagPlus}.
Consider $A$ defined in \eqref{eq:dual_certificate_centered1} as a $(d+1)$-by-$(d+1)$ block-structured matrix. 
Using the special structure of $\vnull$, we can verify that for each diagonal block of $A$, its top-left $d$-by-$d$ submatrix is always zero. Thus, applying $\SymBlockDiag_d^{+}$ zeros out this term and it holds that,
\begin{equation}
	S(X) = \ConLapT - \SymBlockDiag_d^{+}(X^\top X \ConLapT) = \ConLapT - \SymBlockDiag_d^{+}((XP)^\top XP \ConLapT) = S(XP).
	\label{eq:dual_certificate_centered2}
\end{equation}
Equation~\eqref{eq:dual_certificate_centered2} proves the intuitive result that the dual certificate matrix $S(X)$ is invariant to any global translations applied to $X$. 

\paragraph{Step 2: Bounding $S(X)$ within the sublevel set of $f$.}
We now prove that for all $X$ for which $f(X) \leq \bar{f}$, 
the spectral norm $\norm{S(X)}_2$ is upper bounded by a constant which only depends on $\bar{f}$.
In light of \eqref{eq:dual_certificate_centered2}, for the purpose of bounding $S(X)$, we can assume without loss of generality that
$X = XP$. This simply means that each row of $X$ is orthogonal to $\vnull$, which implies that $\sum_{i=1}^{n} p_i = 0$, i.e., the translations are \emph{centered} at zero.
Recall the bound \eqref{eq:pgo_translation_bound} we have obtained for translation terms in the previous section. 
Using triangle inequality, we can move $Y_i \ttilde_{ij}$ to the right hand side, so that we obtain an upper bound on the relative translation between $p_i$ and $p_j$,
\begin{align}
	\norm{p_j - p_i}_2 
	&\leq \sqrt{\bar{f} / \min_{(i,j) \in \DirEdges} \tau_{ij}} + \norm{Y_i \ttilde_{ij}}_2
	\why{\text{triangle inequality}} \\
	&= \sqrt{\bar{f} / \min_{(i,j) \in \DirEdges} \tau_{ij}} + \norm{\ttilde_{ij}}_2 \why{Y_i \in \Stiefel(d,r)} \\
	&\leq \sqrt{\bar{f} / \min_{(i,j) \in \DirEdges} \tau_{ij}} + \max_{(i,j) \in \DirEdges}\norm{\ttilde_{ij}}_2.
\end{align}
Taking the square of the above bound and summing over all edges in the pose graph, it holds that, 
\begin{equation}
	\sum_{(i,j) \in \DirEdges} \norm{p_j - p_i}_2^2 \leq 
	|\DirEdges| \bigg(\sqrt{\bar{f} / \min_{(i,j) \in \DirEdges} \tau_{ij}} + \max_{(i,j) \in \DirEdges}\norm{\ttilde_{ij}}_2 \bigg)^2 \triangleq c_f.
	 \label{eq:dual_certificate_bound_1}
\end{equation}
Note that the constant $c_f$ only depends on $\bar{f}$ and the input measurements.
Let $\vectorize(p) \in \Real^{rn}$ be the vector formed by concatenating all $p_1, \hdots, p_n$. The left hand side of \eqref{eq:dual_certificate_bound_1} can be written in matrix form as $\vectorize(p)^\top (L \otimes I_r) \vectorize(p)$, where $L$ is the (unweighted) Laplacian of the pose graph $\Gcal = (\Vcal, \DirEdges)$. Since $\Gcal$ is connected, the Laplacian has a rank-1 null space spanned by the all-ones vector $1_n$. Correspondingly, $L \otimes I_r$ has a rank-$r$ null space spanned by the columns of $1_n \otimes I_r$. Crucially, using our assumption that $\sum_{i=1}^n p_i = 0$, it can be readily checked that $\vectorize(p)$ is \emph{orthogonal} to this null space. Therefore, we can obtain a lower bound on $\vectorize(p)^\top (L \otimes I_r) \vectorize(p)$ using the smallest \emph{positive} eigenvalue of $L \otimes I_r$, which coincides with the \emph{algebraic connectivity} of $\Gcal$, denoted as $\lambda_2(L)$,
\begin{equation}
	\vectorize(p)^\top (L \otimes I_r) \vectorize(p) \geq \lambda_2(L) \norm{\vectorize(p)}_2^2. 
\end{equation}
Combine this inequality with \eqref{eq:dual_certificate_bound_1}, we have thus shown that,
\begin{equation}
	\lambda_2(L) \norm{\vectorize(p)}_2^2 \leq c_f. 
\end{equation}
Since $\lambda_2(L)$ is guaranteed to be positive as the graph is connected, we can divide both sides by $\lambda_2(L)$. After taking the square root, we obtain the following bound on the translations,
\begin{equation}
	\norm{\vectorize(p)}_2 \leq \sqrt{ c_f / \lambda_2(L) }. 
	\label{eq:dual_certificate_bound_2}
\end{equation}
Recall that $X$ contains translations in addition to $n$ ``lifted'' rotations $Y_1, \hdots, Y_n \in \Stiefel(d,r)$. Since $Y_i$ is an element of the Stiefel manifold, the norm of $Y_i$ in the ambient space is always a constant. 
This implies that the Frobenius norm $\norm{X}_F$ is bounded. 
Finally, using the fact that $S(X)$ as defined in \eqref{eq:dual_certificate} is a continuous operator together with \eqref{eq:hessian_and_dual_certificate},
it holds that the induced operator norm of $\Hess f(X)$ on the tangent space is bounded within the sublevel set. 

\paragraph{Step 3: Bounding the Riemannian Hessian along iterates of $\RBCD$ and $\ARBCD$. }
Once again, since both \RBCD\ and \ARBCD\ are descent methods by construction, any sequence of iterates will remain in the initial sublevel set. 
Therefore, by the result obtained in \textbf{Step 2}, 
there exists a constant $c_0$ (whose value only depends on $f(X^0)$) that bounds the Riemannian Hessian at all iterations:
\begin{equation}
		\max_{\eta \in T_{X^{k}} \Manifold(r,n), \norm{\eta}=1} |\langle \eta, \Hess f(X^k) [\eta] \rangle| \leq c_0, \; \forall k.
	\label{eq:H_bounded_2}
\end{equation}

To complete the proof, we show that the reduced Riemannian Hessian, $\Hess f_b(X_b^k)$, at an arbitrary block $b$ is also bounded by $c_0$. 
This is needed as during distributed local search, we apply \textsc{BlockUpdate} (Algorithm~\ref{alg:BlockUpdate}) on individual blocks instead of the full problem. 
Suppose for the sake of contradiction that there exists a tangent vector $\eta_b \in T_{X^k_b} \Manifold(r, n_b)$ such that $\norm{\eta_b} = 1$ and $|\langle \eta_b, \Hess f_b(X_b^k) [\eta_b] \rangle| > c_0$. 
Let $\gamma_b: [-\epsilon, \epsilon] \to \Manifold(r, n_b)$ be the corresponding geodesic 
such that $\gamma_b(0) = X_b$ and $\gamma_b'(0) = \eta_b$. 
Define the scalar function $h_b \triangleq f_b \circ \gamma_b$. By standard results in differential geometry (e.g,. see \cite{absil2009optimization}), it holds that
$h_b''(0) = \langle \eta_b, \Hess f_b(X_b^k) [\eta_b] \rangle$.

Using the product structure of our manifold, we can associate $\gamma_b$ with a corresponding geodesic $\gamma: [-\epsilon, \epsilon] \to \Manifold(r,n)$ on the overall manifold such that $\gamma$ agrees with $\gamma_b$ at the selected block coordinate $b$ and stays constant at all other blocks. Consider the tangent vector $\eta \in T_{\gamma(0)} \Manifold(r,n)$. By construction, the blocks of $\eta$ are given by,
\begin{equation}
	\eta_{b'} = \begin{cases}
	\eta_b, & \text{if $b' = b$},  \\
	0, & \text{otherwise.}
	\end{cases}
\end{equation}
As a result, we have $\norm{\eta} = 1$. Define the scalar function $h \triangleq f \circ \gamma$ and we have by construction that $h''(0) = h_b''(0)$. This would then imply,
\begin{equation}
	|h''_b(0)| = |h''(0)| = |\langle \eta, \Hess f(X^k) [\eta] \rangle| > c_0,
\end{equation}
which is a contradiction.

\section{Proof of Theorem \ref{Convergence_of_first_order_Riemannian_staircase_Theorem}}
\label{staircase_convergence_appendix}

In this subsection we establish the convergence properties of the (first-order) distributed Riemannian Staircase (Algorithm \ref{alg:riemannian_staircase}) described in Theorem \ref{Convergence_of_first_order_Riemannian_staircase_Theorem}.

If Algorithm \ref{alg:riemannian_staircase} terminates finitely (case (i)) then
there is nothing to prove, so henceforward let us assume that Algorithm
\ref{alg:riemannian_staircase} generates an infinite sequence $\lbrace X^{(r)}
\rbrace$ of factors in line \ref{alg:rs_lift_rank}.  Our overall strategy is to
exploit the correspondence between the critical points $X^{(r)}$ of Problem
\ref{prob:se_sync_riemannian} (rank-restricted full SDP) and the critical points
$Y^{(r)}$ of Problem \ref{prob:se_sync_riemannian_marg} (rank-restricted
rotation-only SDP) provided by Lemma \ref{lem:rank_restricted_SDP_equivalence}, together with the compactness of the feasible set of the SE-Sync relaxation Problem \ref{prob:se_sdp_marg}, to control the behavior of this sequence. To that end, define:
\begin{equation}
Z^{(r)} \triangleq (X^{(r)})^\top X^{(r)}, \quad \quad Z_{\text{R}}^{(r)} \triangleq (Y^{(r)})^\top Y^{(r)}, \quad \quad \forall r \ge r_0,
\end{equation}
where $Y^{(r)}$ is the first-order critical point of Problem
\ref{prob:se_sync_riemannian_marg} obtained from $X^{(r)}$ as described by Lemma
\ref{lem:rank_restricted_SDP_equivalence}.  
In addition, let $\Lambda^{(r)}$ and $\Lambda_{\text{R}}^{(r)}$ denote the Lagrange multiplier matrices corresponding to $Z^{(r)}$ and $Z_{\text{R}}^{(r)}$, respectively.

Since $\lbrace Z_{\text{R}}^{(r)} \rbrace$ is an infinite sequence contained in the (compact) feasible set of Problem \ref{prob:se_sdp_marg}, it must contain a convergent subsequence $\lbrace  Z_{\text{R}}^{(r_k)} \rbrace$, with limit point $Z_{\text{R}}^\star$. 
Since the Lagrange multiplier $\Lambda_{\text{R}}^{(r_k)}$ is a continuous function of $Z_{\text{R}}^{(r_k)}$ (see \eqref{eq:Lagrange_multipliers_SE-Sync} and \cite[eq.\ (107)]{Rosen19IJRR}), it follows that $\lbrace \Lambda_{\text{R}}^{(r_k)} \rbrace$ likewise converges to a limit $\Lambda_{\textnormal{R}}^\star$. 
By Lemma \ref{lem:rank_restricted_SDP_equivalence}(iii) (cf.\ \eqref{relation_between_Lagrange_multipliers_of_SESync_fact_and_full_PGO_fact}), it follows that the subsequence of Lagrange multipliers $\lbrace \Lambda^{(r_k)} \rbrace$ 
for the full (translation-explicit) problem also converges 
to a limit point $\Lambda^\star$, and consequently so does the sequence of certificate matrices $\ST^{(r_k)} \triangleq \ST(X^{(r_k)})$ computed in line \ref{alg:rs_dual_certificate} of Algorithm \ref{alg:riemannian_staircase}:
\begin{equation}
\label{convergence_of_certificate_matrix}
 \lim_{k \to \infty} S^{(r_k)} = S^\star.
\end{equation}
Now, let us consider two cases, corresponding to whether $S^\star$ is positive semidefinite.

\paragraph{Case 1: $S^\star \succeq 0$.}
It is easy to check that $S^{(r_k)} (X^{(r_k)})^\top = 0$ since $X^{(r_k)}$ is by
definition a
first-order critical point. This shows that $S^{(r_k)}$ has a zero eigenvalue.
Using this property, together with the fact that the eigenvalues of a matrix $S$
are continuous functions of $S$ and that $S^\star \succeq 0$, it holds that
$\lim_{k \to \infty} \lambda_{\textnormal{min}}\left(S(X^{(r_k)}) \right) = 0$,
thus proving \eqref{subsequence_converging_to_nonnegative_min_eig}.

To prove \eqref{subsequence_converging_to_optimal_value}, note that by Lemma
\ref{lem:rank_restricted_SDP_equivalence}(iv), 
the certificate matrix $S_{\text{R}}^\star \triangleq \ConLapM -
\SymBlockDiag\left(\ZR^\star \ConLapM \right)$
associated with the limit point
$Z_{\text{R}}^\star$ is likewise positive semidefinite, and therefore
$Z_{\text{R}}^\star$ is a minimizer of Problem \ref{prob:se_sdp_marg}
\cite[Thm.\ 7]{Rosen19IJRR}. Since $f(X^{(r)}) = \langle \ConLapM,
{Y^{(r)}}^\top Y^{(r)} \rangle = \langle
\ConLapM, Z_{\text{R}}^{(r)}\rangle$ for all $r \ge r_0$ (Lemma \ref{lem:rank_restricted_SDP_equivalence}(ii)), it follows that:
\begin{equation}
\label{convergence_of_objective_values_from_subsequence}
\lim_{k \to \infty} f(X^{(r_k)}) = \lim_{k \to \infty} \langle \ConLapM,
{Y^{(r_k)}}^\top Y^{(r_k)} \rangle =
\lim_{k \to \infty} \langle \ConLapM, Z_{\text{R}}^{(r_k)} \rangle = \langle
\ConLapM,Z_{\text{R}}^\star \rangle = f_{\textnormal{SDP}}^\star.
\end{equation}
But then in fact we must have:
\begin{equation}
\label{convergence_of_objective_values}
\lim_{r \to \infty} f(X^{(r)})
 = f_{\textnormal{SDP}}^\star
\end{equation}
since the sequence of objective values $\lbrace f(X^{(r)}) \rbrace$ is \emph{monotonically decreasing}.  This establishes that Theorem \ref{Convergence_of_first_order_Riemannian_staircase_Theorem}(ii) holds for the case $S^\star \succeq 0$.

\paragraph{Case 2: $S^\star \not \succeq 0$.} To finish the proof, we now show that in fact $S^\star  \not \succeq 0$ cannot occur.  To do so, suppose for contradiction that $\lambda_{\textnormal{min}}(S^\star) < 0$, and define $\mu \triangleq \lvert \lambda_{\textnormal{min}}(S^\star) \rvert$.  Then there exists $k_1 > 0$ sufficiently large that $\lambda_{\textnormal{min}}(S^{(r_k)}) < -\mu / 2$ for all $k > k_1$.  Our aim is to show that this \emph{upper} bound (away from $0$) on the minimum eigenvalue of $S^{(r_k)}$ implies a \emph{lower} bound  $\delta > 0$ on the achievable decrease in the objective value each time the saddle escape procedure in line \ref{alg:rs_escape} of Algorithm \ref{alg:riemannian_staircase} is invoked with any $k > k_1$; since this occurs infinitely many times, this would imply that the optimal value of Problem \ref{prob:SDP_relaxation_for_PGO} is $-\infty$, a contradiction.

We note that the following analysis holds at any rank $r_k$ of  Algorithm~\ref{alg:riemannian_staircase} with $k > k_1$. 
For the ease of reading, however, we will suppress $k$ in our notation and
assume that $k$ and $r_k$ are clear from the context. 
Let $X$ be the critical point generated by the distributed Riemannian Staircase at iteration rank $r_k$ (Algorithm~\ref{alg:riemannian_staircase}, line~\ref{alg:rs_lift_rank}). 
Denote $\lambda$ as the minimum eigenvalue of the corresponding certificate matrix (Algorithm~\ref{alg:riemannian_staircase}, line~\ref{alg:min_eigenpair_computation}), and 
$\dot{X} \in T_X \Manifold(r_k, n)$ as the corresponding second-order descent direction constructed (Algorithm~\ref{alg:riemannian_staircase}, line~\ref{alg:descent_direction}).
Recall that since $k > k_1$, we have $\lambda < -\mu/2$. 

Consider the geodesic emanating from point $X$ with initial velocity $\dot{X} =
\dot{X}_+$ (see Theorem~\ref{thm:verification}). With a slight abuse of notation, we denote this geodesic as $X(t)$ and it may be parameterized as follows,\footnote{Note that this map is in fact well-defined on all of $\Real$ because both the Stiefel manifold and $\Real^n$ are \emph{geodesically complete.}}
\begin{align}
X \colon \Real  \to \Manifold(r_k, n) : t  \mapsto \exp_{X} \left(t \dot{X} \right)
\end{align}
In addition, the cost function $f$ of Problem~\ref{prob:se_sync_riemannian} restricted to $X(t)$:
\begin{align}
g\colon \Real  \to \Real_{\geq 0} : t \mapsto f \circ X(t).
\end{align}
From the integral form of Taylor's Theorem, we have:
\begin{equation}
\label{Taylor_expansion_for_objective_along_geodesic_retraction}
g(t) = g(0) + t \dot{g}(0) + \frac{t^2}{2} \ddot{g}(0) + \int_0^t \frac{\alpha^2}{2} g^{(3)}(\alpha) \: d\alpha.
\end{equation}
Differentiating $g(t)$ and applying the chain rule, we arrive at,
\begin{equation}
\label{first_derivative_of_g}
\dot{g}(t) = \mathrm{D} f(X(t)) [ \dot{X}(t)] = \langle \rgrad f(X(t)),
\dot{X}(t) \rangle,
\end{equation}
where $\langle \cdot, \cdot \rangle$ denotes the Frobenius inner product. 
Differentiating \eqref{first_derivative_of_g} again, we obtain,
\begin{equation}
\label{setup_for_second_derivative_of_g}
\begin{split}
\ddot{g}(t) &= \frac{d}{dt} \big \langle \rgrad f(X(t)), \dot{X}(t) \big \rangle \\
&= \big \langle \frac{d}{dt} \left[ \rgrad f(X(t)) \right], \dot{X}(t) \big \rangle + \big \langle \rgrad f(X(t)), \ddot{X}(t) \big \rangle \\
&= \big \langle \Hess f(X(t))[\dot{X}(t)], \dot{X}(t) \big \rangle + \big \langle  \rgrad f(X(t)), \ddot{X}(t) \big \rangle.
\end{split}
\end{equation}
Equations~\eqref{first_derivative_of_g}-\eqref{setup_for_second_derivative_of_g}
hold by applying standard results for embedded manifolds; see
\cite[Chapters 3 and 5]{boumal2020intromanifolds} and \cite{absil2009optimization}. Furthermore, since $X(t)$ is a geodesic, we may
further show (see, e.g. \cite[Section~5.4]{absil2009optimization}) that its
extrinsic acceleration $\ddot{X}(t)$ is always orthogonal to the tangent space
at $X(t)$. This means that the last term in
\eqref{setup_for_second_derivative_of_g} is identically zero since the
Riemannian gradient at $X(t)$ belongs to the tangent space at that
point. Therefore, the second-order derivative further simplifies to,
\begin{equation}
\label{second_derivative_of_g}
 \ddot{g}(t) = \left \langle \Hess f(X(t))[\dot{X}(t)], \dot{X}(t) \right \rangle
 = 2 \langle \dot{X}(t) S(X(t)), \dot{X}(t) \rangle.
\end{equation}
The second equality makes use of \eqref{eq:hessian_and_dual_certificate}. 
In particular, \eqref{first_derivative_of_g} and \eqref{second_derivative_of_g}
show that $\dot{g}(0) = 0$ (since $X(0) = X$ is a critical point of $f$) and
$\ddot{g}(0) = 2\lambda$ (since $\dot{X}(0) = \dot{X}$ is nothing but $\dot{X}_+$ in
Theorem~\ref{thm:verification} and Remark \ref{saddle_escape_remark}).  It follows from \eqref{Taylor_expansion_for_objective_along_geodesic_retraction} that:
\begin{equation}
\label{quadratic_model_at_first_order_critical_point}
g(0) - g(t) = -\lambda t^2 - \int_0^t \frac{\alpha^2}{2} g^{(3)}(\alpha) \: d\alpha.
\end{equation}
Thus, if we can exhibit scalars $\kappa, \tau > 0$ such that:
 \begin{equation}
 \label{bound_on_third_derivative}
\lvert g^{(3)}(t) \rvert \le \kappa \quad \forall t \in [0,\tau],
 \end{equation}
then \eqref{quadratic_model_at_first_order_critical_point} implies:
\begin{equation}
\label{lower_bound_on_reduction_in_objective_value}
g(0) - g(t) \ge  \underbrace{-\lambda t^2 - \frac{\kappa}{6} t^3}_{\delta(t)}  \quad \forall t \in [0,\tau].
\end{equation}
Since:
\begin{equation}
 \delta'(t) = - 2\lambda t - \frac{\kappa}{2}t^2, \quad \delta''(t) = -2\lambda - \kappa t,
\end{equation}
a straightforward calculation shows that the lower-bound $\delta(t)$ for the reduction in the objective value attained in \eqref{lower_bound_on_reduction_in_objective_value} is maximized by the steplength:
\begin{equation}
 t^* = - \frac{4\lambda}{\kappa},
\end{equation}
and that:
\begin{equation}
 \delta'(t) > 0 \quad \forall t \in (0, t^*).
\end{equation}
In consequence, by choosing a stepsize of: 
\begin{equation}
t = \min \left \lbrace -\frac{4\lambda}{\kappa}, \: \tau \right \rbrace \ge \min \left \lbrace \frac{2\mu}{\kappa}, \: \tau \right \rbrace \triangleq \sigma
\end{equation}
when performing the retraction in the saddle escape procedure, we can guarantee that the objective is decreased by at least:
\begin{equation}
\delta(t) \ge \delta(\sigma) = -\lambda \sigma^2 - \frac{\kappa}{6}\sigma^3 \ge  \frac{\mu}{2} \sigma^2 - \frac{\kappa}{6} \sigma^3 \triangleq \delta > 0
\end{equation}
each time the saddle escape procedure is invoked, producing our desired contradiction.  

It therefore suffices to exhibit fixed constants $\kappa, \tau > 0$  that will satisfy \eqref{bound_on_third_derivative}  for \emph{all} invocations of the saddle escape procedure in line \ref{alg:rs_escape} with $k > k_1$. 
Before proceeding, we make one simplifying assumption that will remain in effect throughout the remainder of the proof: we assume without loss of generality that all iterates $X^{(r)}$ are translated so that the sum of the translational components $p_i^{(r)}$ is $0$, as in Section \ref{sec:convergence_pgo:H_bounded}.\footnote{Note that this map leaves the objective $f$ invariant, and is an isometry of both the ambient Euclidean space and $\Manifold(r,n)$.  Consequently, applying this transformation leaves all objective values and derivative norms invariant -- we perform this transformation simply to achieve a parameterization that is more convenient for calculation.} We now proceed in two stages: first we will identify an upper bound $\kappa$ for the magnitude of $g^{(3)}(t)$ on the sublevel set of the initial value $f^{(0)}$ of $f$, and \emph{then} a minimum distance $\tau$ from a given suboptimal critical point that we can retract along while still remaining inside this sublevel set.

\paragraph{Calculation of $\kappa$.} Recall the following notation for the sublevel sets of the cost function $f$:
\begin{equation}
\label{definition_of_sublevel_set_for_convergence_of_first_order_RS}
\Lcal_f(r, n; \bar{f}) \triangleq \left \lbrace X \in \Manifold(r,n) | f(X) \leq
\bar{f} \right \rbrace.
\end{equation}
Our goal is to identify a fixed constant $\kappa > 0$ that upper bounds the magnitude of $g^{(3)}(t)$ within the \emph{initial} sublevel set, i.e., 
\begin{equation}
\lvert g^{(3)}(t) \rvert \le \kappa \quad \forall X(t) \in \Lcal_f(r, n; f^{(r_0)}), \: r \ge r_0,
\end{equation}
where $f^{(r_0)} \triangleq f(X^{(r_0)})$ is the \emph{initial} value of the objective at the starting point $X^{(r_0)}$ supplied to Algorithm \ref{alg:riemannian_staircase}.  
From the second-order derivative \eqref{second_derivative_of_g}, differentiate once again to obtain an explicit expression for  $g^{(3)}(t)$. 
Here, we make use of the equality derived in
\eqref{eq:hessian_and_dual_certificate} that relates the Riemannian Hessian and the certificate matrix $S(X)$,
\begin{equation}
\begin{split}
g^{(3)}(t) = \frac{d}{dt} \big \langle \Hess f(X(t))[\dot{X}(t)], \: \dot{X}(t) \big \rangle 
= 2  \frac{d}{dt} \big \langle  \dot{X}(t) S(X(t)) , \: \dot{X}(t) \big \rangle.
\end{split}
\end{equation}
Fully expanding the differentiation yields, 
\begin{equation}
\label{setup_for_computation_of_third_derivative}
\begin{split}
g^{(3)}(t) &= 
2 \big \langle \frac{d}{dt} [ \dot{X}(t) S(X(t)) ] , \: \dot{X}(t) \big \rangle +
2 \big \langle \dot{X}(t) S(X(t))  , \: \ddot{X}(t) \big \rangle \\
&=
2 \big \langle \ddot{X}(t) S(X(t)) + \dot{X}(t)\frac{d}{dt} [S(X(t))] , \: \dot{X}(t) \big \rangle +
2 \big \langle \dot{X}(t) S(X(t))  , \: \ddot{X}(t) \big \rangle \\
&=
2 \big \langle \dot{X}(t)\frac{d}{dt} [S(X(t))] , \: \dot{X}(t) \big \rangle +
4 \big \langle \dot{X}(t) S(X(t))  , \: \ddot{X}(t) \big \rangle 
\end{split}
\end{equation}
where in the last step, we have used the symmetry of the certificate matrix $S(X(t))$ to combine terms. We proceed to bound the two inner products in \eqref{setup_for_computation_of_third_derivative} separately. 

To bound the first inner product, we first use the Cauchy-Schwarz inequality followed by applying the submultiplicative property of the Frobenius norm:
\begin{equation}
	\big \lvert \big \langle \dot{X}(t)\frac{d}{dt} [S(X(t))] , \: \dot{X}(t) \big \rangle \big \rvert
	\leq 
	\big \lVert \dot{X}(t)\frac{d}{dt} [S(X(t))] \big \rVert_F 
	\big \lVert \dot{X}(t) \big \rVert_F
	\leq 
	\big \lVert \frac{d}{dt} [S(X(t))] \big \rVert_F 
	\big \lVert \dot{X}(t) \big \rVert_F^2.
\end{equation}
By construction of the descent direction \eqref{second_order_descent_direction}, the velocity vector at $t=0$ satisfies $\lVert \dot{X}(0) \rVert_F = 1$. Since geodesics have constant speed, it thus holds that $\lVert \dot{X}(t) \rVert_F = 1$ for all $t$. Therefore, it suffices to bound the derivative of $S(X(t))$. 
Using the chain rule and repeatedly applying the submultiplicative property
yields,
\begin{equation}
\label{setup_for_bounding_first_inner_product}
\begin{split}
\left \lVert \frac{d}{dt} \left[ S(X(t)) \right] \right \rVert_F &= \left \lVert \SymBlockDiag_d^{+} \left(\dot{X}(t)\transpose X(t) Q + X(t)\transpose \dot{X}(t)Q \right)  \right \rVert_F \\
& \le \left \lVert \SymBlockDiag_d^{+} \right \rVert_{Fop}  \left \lVert \dot{X}(t)\transpose X(t) Q + X(t)\transpose \dot{X}(t) Q  \right \rVert_F \\
&\le 2 \left \lVert \SymBlockDiag_d^{+} \right \rVert_{Fop} \lVert Q \rVert_F  \lVert X(t) \rVert_F \lVert \dot{X}(t) \rVert_F.
\end{split}
\end{equation}
Above, $\left \lVert \SymBlockDiag_d^{+} \right \rVert_{Fop}$ denotes the operator norm of the linear map $\SymBlockDiag_d^{+}$ with respect to the Frobenius norm.  A closer examination of the definition of $\SymBlockDiag_d^{+}$ in \eqref{eq:SymBlockDiagPlus} reveals that in fact:\footnote{This conclusion follows from the observation that $\SymBlockDiag_d^{+}$ sets off-diagonal block elements of its argument to zero, and extracts the symmetric parts of the on-diagonal blocks; consequently $\SymBlockDiag_d^{+}$  cannot increase the Frobenius norm of its argument, so that $ \lVert \SymBlockDiag_d^{+} \rVert_{Fop} \le 1$.  Equality follows from the fact that $\SymBlockDiag_d^{+}$ \emph{fixes} any symmetric block-diagonal matrix with the sparsity pattern appearing in \eqref{eq:SymBlockDiagPlus}.}
\begin{equation}
\label{Frobenius_norm_of_SymBlockDiag}
\lVert \SymBlockDiag_d^{+} \rVert_{Fop} = 1.
\end{equation}
Furthermore, since  $X(t) \in \Lcal_f(r, n; f^{(r_0)})$ and is assumed to be
centered, applying the results in Appendix~\ref{sec:convergence_pgo:H_bounded} (Step 2) shows that, 
\begin{equation}
\label{Upper_bound_for_norm_of_Xt}
\begin{split}
\lVert X(t) \rVert_F &= \sqrt{\lVert Y(t) \rVert_F^2 + \lVert p(t) \rVert_F^2} \le \sqrt{dn + \frac{c_f}{\lambda_2 (L)}}.
\end{split}
\end{equation}
Finally, using that fact that the geodesic has constant speed $\lVert \dot{X}(t) \rVert_F = 1$, we establish the following constant bound for \eqref{setup_for_bounding_first_inner_product},
\begin{equation}
\label{bound_on_first_inner_product}
\big \lvert \big \langle \dot{X}(t)\frac{d}{dt} [S(X(t))] , \: \dot{X}(t) \big \rangle \big \rvert \leq
2 \lVert Q \rVert_F \sqrt{dn + \frac{c_f}{\lambda_2 (L)}}. 
\end{equation}

Now, let us return to \eqref{setup_for_computation_of_third_derivative} to bound the second inner product. Once again, we start by invoking the Cauchy-Schwarz inequality and the submultiplicative property of the Frobenius norm, 
\begin{equation}
	\label{setup_for_bounding_second_inner_product}
	\begin{split}
	\big \lvert \big \langle \dot{X}(t) S(X(t))  , \: \ddot{X}(t) \big \rangle \big \rvert \leq 
	\left \| \dot{X}(t) \right \|_F \left \| S(X(t)) \right \|_F 
	\left \| \ddot{X}(t) \right \|_F
	= \left \| S(X(t)) \right \|_F 
	\left\| \ddot{X}(t) \right\|_{F}.
	\end{split}
\end{equation}
Since $X(t) \in \Lcal_f(r, n; f^{(r_0)})$, applying the result of
Appendix~\ref{sec:convergence_pgo:H_bounded} (Step 2) shows that there exists a constant $c_s > 0$ such that $\lVert S(X(t)) \rVert_F \leq c_s$. Next, we obtain an upper bound on the extrinsic acceleration. To this end, we note that in order to bound $\lVert \ddot{X}(t) \rVert_F$, it suffices to bound the extrinsic acceleration of each Stiefel element $\lVert \ddot{Y}_i(t) \rVert_F$ (since translations have zero extrinsic acceleration). 
Here, we use a result that gives an explicit characterization of $\ddot{Y}_i(t)$
along a geodesic on the Stiefel manifold \cite[Eq.\ (2.7)]{Edelman1999}:
\begin{equation}
	\ddot{Y}_i(t) + Y_i(t)[\dot{Y}_i(t)^\top \dot{Y}_i(t)] = 0.
\end{equation}
From the above characterization, it holds that,
\begin{equation}
	\lVert \ddot{Y}_i(t) \rVert_F \leq 
	\lVert Y_i(t) \rVert_F 
	\lVert \dot{Y}_i(t) \rVert_F^2
	\leq 
	\lVert Y_i(t) \rVert_F 
	=
	\sqrt{d},
\end{equation}
where in the second inequality, we use the fact that $\lVert \dot{Y}_i(t) \rVert_F \leq \lVert \dot{X}(t) \rVert_F = 1$. Therefore, $\lVert \ddot{X}(t) \rVert_F$ is upper bounded as follows, 
\begin{equation}
	\lVert \ddot{X}(t) \rVert_F^2 = \sum_{i=1}^n \lVert \ddot{Y}_i(t) \rVert_F^2 \leq dn
	\implies
	\lVert \ddot{X}(t) \rVert_F \leq \sqrt{dn}. 
\end{equation}
Combining these results yields a constant bound for the second inner product \eqref{setup_for_bounding_second_inner_product}:
\begin{equation}
	\label{bound_on_second_inner_product}
	\big \lvert \big \langle \dot{X}(t) S(X(t))  , \: \ddot{X}(t) \big \rangle \big \rvert
	\leq c_s \sqrt{dn}. 
\end{equation}
Finally, using the upper bounds \eqref{bound_on_first_inner_product} and \eqref{bound_on_second_inner_product} in \eqref{setup_for_computation_of_third_derivative} yields the final fixed upper bound $\kappa$,
\begin{equation}
	\begin{split}
	\lvert g^{(3)}(t) \rvert &= 
	2 \big \lvert \big \langle \dot{X}(t)\frac{d}{dt} [S(X(t))] , \: \dot{X}(t) \big \rangle  \big \rvert +
	4 \big \lvert \big \langle \dot{X}(t) S(X(t))  , \: \ddot{X}(t) \big \rangle \big \rvert \\
	&\leq 4 \lVert Q \rVert_F \sqrt{dn + \frac{c_f}{\lambda_2 (L)}} + 
	4 c_s \sqrt{dn} \\
	&\triangleq
	\kappa.
	\end{split}
\end{equation}

\paragraph{Calculation of $\tau$:}  Now we derive a $\tau > 0$ such that $X(t) \in \Lcal_f(r_k, n; f^{(r_0)})$ for all $t \in [0, \tau]$ and $k \ge k_1$.  Our approach is based upon deriving a simple upper bound for the magnitude of the change in objective value between two points $X_1$ and $X_2$. To that end, consider:
\begin{equation}
\label{difference_in_objective_value_in_terms_of_distance_between_points}
\begin{split}
\lvert f(X_1) - f(X_2) \rvert &= \left \lvert \tr(Q X_1^\top X_1) - \tr(QX_2^\top X_2) \right \rvert \\
&= \left \lvert \tr\left[Q \left(X_1^\top X_1 - X_2^\top X_2 \right)\right] \right \rvert \\
&\le \lVert Q \rVert_F \left \lVert X_1^\top X_1 - X_2^\top X_2 \right \rVert_F.
\end{split}
\end{equation}
Defining:
\begin{equation}
 \Delta \triangleq X_2 - X_1,
\end{equation}
we may substitute $X_2 = X_1 + \Delta$ in \eqref{difference_in_objective_value_in_terms_of_distance_between_points} to obtain:
\begin{equation}
\label{bound_for_function_value_in_terms_of_X1_and_Delta}
\begin{split}
 \lvert f(X_1) - f(X_2) \rvert &\le \lVert Q \rVert_F \left \lVert X_1^\top X_1 - (X_1 + \Delta)^\top (X_1 + \Delta) \right \rVert_F \\
 &= \lVert Q \rVert_F  \left \lVert X_1^\top \Delta + \Delta^\top X_1 + \Delta^\top \Delta  \right \rVert_F  \\
 &\le \lVert Q \rVert_F  \left( 2 \lVert X_1 \rVert_F \lVert \Delta \rVert_F + \lVert \Delta \rVert_F^2 \right).
 \end{split}
\end{equation}
We will use inequality \eqref{bound_for_function_value_in_terms_of_X1_and_Delta}
to derive a \emph{lower} bound $\tau$ on the admissible steplength $t$ of the
retraction applied at a critical point $X^{(r_k)}$ while still remaining inside
the original sublevel set $\Lcal_f(r_k, n; f^{(r_0)})$. Given the fact that
$f^{(r_k)} < f^{(r_0)}$ and since, by definition, $X(t) \in \Lcal_f(r_k, n; f^{(r_0)})$
iff $f(X(t)) \le f^{(r_0)}$, it suffices to ensure that:
\begin{equation}
\label{upper_bound_on_admissible_change_in_function_value}
\lvert f^{(r_k)} - f(X(t))\rvert \le f^{(r_0)} - f^{(r_k)}.
\end{equation}
Applying inequality \eqref{bound_for_function_value_in_terms_of_X1_and_Delta} with $X_1 = X^{(r_k)}$, $X_2 = X(t)$ and $\Delta = X(t) - X^{(r_k)}$, 
\begin{equation}
\label{bound_on_change_in_function_value_in_terms_of_Delta}
\lvert f^{(r_k)} - f(X(t)) \rvert \le  \lVert Q \rVert_F  \left( 2 \left\|
X^{(r_k)} \right\|_F \lVert \Delta \rVert_F + \lVert \Delta \rVert_F^2 \right).
\end{equation}
From \eqref{upper_bound_on_admissible_change_in_function_value} and \eqref{bound_on_change_in_function_value_in_terms_of_Delta}, for $X(t)$ to remain in the sublevel set, it suffices to ensure that,
\begin{equation}
	\lVert Q \rVert_F  \left( 2 \left\| X^{(r_k)} \right\|_F \lVert \Delta \rVert_F + \lVert \Delta \rVert_F^2 \right)
	\leq 
	f^{(r_0)} - f^{(r_k)}.
\end{equation}
After rearranging and simplification, we can arrive at the following equivalent condition:
\begin{equation}
\label{bounding_deltaF_intermediate}
\lVert X(t) - X^{(r_k)} \rVert_F =
\lVert \Delta \rVert_F \le   \sqrt{ \left \lVert X^{(r_k)} \right \rVert_F^2 + \frac{f^{(r_0)} - f^{(r_k)}} {\lVert Q \rVert_F} } - \left \lVert X^{(r_k)} \right \rVert_F.
\end{equation}
It is worth noting that each invocation of escaping procedure strictly reduces the objective,
and thus here $f^{(r_0)} > f^{(r_k)}$. We identify a constant lower bound for the right hand side of \eqref{bounding_deltaF_intermediate}. To this end, let us view the right hand side as a function of $\lVert X^{(r_k)} \rVert_F$. It can be straightforwardly verified that this function is \emph{nonincreasing}. 
Since by \eqref{Upper_bound_for_norm_of_Xt} we have $\lVert X^{(r_k)} \rVert_F \leq \sqrt{dn + \frac{c_f}{\lambda_2 (L)}}$, it follows that the right hand side of \eqref{bounding_deltaF_intermediate} is always lower bounded the following constant,
\begin{equation}
\label{tau_definition}
\tau \triangleq  \sqrt{ dn + \frac{c_f}{\lambda_2(L)} + \frac{f^{(r_0)} - f^{(r_{k_1})}} {\lVert Q \rVert_F} }  - \sqrt{dn + \frac{c_f}{\lambda_2(L)}}.
\end{equation}
Recalling \eqref{bounding_deltaF_intermediate}, we have shown that $X(t)$ remains in the sublevel set if the following holds:
\begin{equation}
	\lVert X(t) - X^{(r_k)} \rVert_F \leq \tau.
\end{equation}
Note that the above bound is expressed in terms of the chordal distance between $X(t)$ and $X^{(r_k)}$. To complete the proof, we note that the geodesic distance $d \left(X^{(r_k)}, X(t) \right) = t$ is always at least as large as the chordal distance. This means that for all $t \in [0, \tau]$, we have:
\begin{equation}
	\lVert X(t) - X^{(r_k)} \rVert_F \leq d \left(X^{(r_k)}, X(t) \right) = t \leq \tau,
\end{equation}
and thus $X(t) \in \Lcal_f(r_k, n; f^{(r_0)})$, as claimed.

\end{appendices}

\end{document}